\documentclass[11pt,a4paper]{article}
 \usepackage[UTF8,scheme=plain]{ctex}
 \usepackage{indentfirst, latexsym,bm}
 
 \usepackage{amsmath}
 \usepackage{mdframed}
 \usepackage{pifont}
 \usepackage{amsfonts}
 \usepackage{mathrsfs}
 \usepackage{array}
 \usepackage{multirow} 
 \usepackage[normalem]{ulem}
 \usepackage{tikz}
 \usepackage{graphicx}
 \usepackage{subfigure}
 \usepackage{picinpar}
 \usepackage{setspace}
 \usepackage[textsize=footnotesize]{todonotes}
 \RequirePackage[colorlinks,citecolor=blue,urlcolor=blue,linkcolor=blue]{hyperref}
 \usepackage[top=2cm,bottom=2cm, outer=2cm, inner=2cm]{geometry}
 
 \usepackage{caption}
 \usepackage{subcaption}
 \usepackage[labelfont={bf,small},textfont={small}]{caption}
 
 \RequirePackage[numbers]{natbib}
 \usepackage{enumitem} % math notations i.e. \mathscr
 
 \usepackage{booktabs}
 \usepackage{authblk}
 
 \usepackage{threeparttable}
 \usepackage{mathrsfs,amsfonts,amsmath}
 
 \usepackage{color}
 
 \makeatletter
 \def\namedlabel#1#2{\begingroup
 	#2%
 	\def\@currentlabel{#2}%
 	\phantomsection\label{#1}\endgroup
 }
 \makeatother

 \numberwithin{figure}{section}

 \newcommand\email[1]{\href{mailto:#1}{ \nolinkurl{#1}}}

 \newtheorem{theorem}{Theorem}[section]
 \newtheorem{definition}[theorem]{Definition}
 \newtheorem{lemma}[theorem]{Lemma}
 \newtheorem{corollary}[theorem]{Corollary}
 \newtheorem{proposition}[theorem]{Proposition}
 \newtheorem{remark}[theorem]{Remark}
 \newtheorem{condition}[theorem]{Condition}
 \newtheorem{example}{Example}[section]

 \def\blemma{\begin{lemma}}\def\elemma{\end{lemma}}
 \def\bproposition{\begin{proposition}}\def\eproposition{\end{proposition}}
 \def\ttheorem{\begin{theorem}}\def\etheorem{\end{theorem}}
 \def\bcorollary{\begin{corollary}}\def\ecorollary{\end{corollary}}
 \def\bremark{\begin{remark}}\def\eremark{\end{remark}}
 \def\bcondition{\begin{condition}}\def\econdition{\end{condition}}

 \def\benumerate{\begin{enumerate}}\def\eenumerate{\end{enumerate}}
 \def\bitemize{\begin{itemize}}\def\eitemize{\end{itemize}}

 \def\beqlb{\begin{eqnarray}}\def\eeqlb{\end{eqnarray}}
 \def\beqnn{\begin{eqnarray*}}\def\eeqnn{\end{eqnarray*}}
 \def\ar{\!\!\!&}

 \def\proof{\noindent{\it Proof.~~}}\def\qed{\hfill$\Box$\medskip}

\begin{document} 

 \title{\bf  \Large Large Excursions of Reflected L\'evy Processes: Local Asymptotics}
 
 \author{ Hao Wu\footnote{School of Mathematical Sciences, Nankai University, China; email: wuhao.math@outlook.com} 
 \quad\ and\quad 
 Wei Xu\footnote{School of Mathematics and Statistics, Beijing Institute of Technology,  China; email: xuwei.math@gmail.com. Xu gratefully acknowledges financial support from the Fundamental Research Funds for the Central Universities, the National Natural Science Foundation of China (No. 11531001) and the National Key R\&D Program of China (No. 2023YFA1010103).} 
 }  
 \maketitle

 \begin{abstract}
 	
 This paper establishes uniform local large-deviation asymptotics for excursions of a L\'evy process $X$ with negative drift $-\beta$ and $\alpha$-locally regularly varying L\'evy measure $\nu$, reflected at its running infimum. 
 Under the excursion measure $\underline n$, we first prove that the lifetime $\zeta$ and height $\overline\epsilon$ satisfy
 \beqnn
 \underline{n}\big(\zeta\in(t,t+\delta] \big) \sim  \underline{n}(\overline\epsilon\in\big (\beta t,\beta t+\beta\delta]\big)
 \sim\underline{n}(\zeta)\cdot\nu\big((\beta t,\beta t+\beta \delta]\big),
 \quad \mbox{as } t\to\infty, 
 \eeqnn
 uniformly in $\delta\in[\delta_0,\infty]$ for every $\delta_0>0$. 
 In contrast to the asymptotic equivalence of the corresponding tail events, fixed-width lifetime and height windows have asymptotically negligible relative overlap. 
 Nevertheless, for any positive function $\theta$ satisfying $|X_t+\beta t|/\theta_t\to0$ in probability, 
  \beqnn
 \underline n\big( |\overline\epsilon-\beta t|  \leq \theta_t \,\big|\, \zeta\in (t,t+\delta] \big)  \sim 
 \underline n\big( |\zeta-t|  \leq \theta_t \,\big|\, \overline\epsilon \in (\beta t,\beta t+\beta\delta] \big) \to 1  . 
 \eeqnn
 For conditioning windows of sublinear width bounded away from zero, both conditional excursion laws converge under spatial rescaling to a single-jump process with deterministic jump size and random arrival time determined by the integrated excursion survival function. 
 An additional time rescaling yields the same deterministic linear descent killed at zero. 
 Finally, we identify the joint limiting distribution of the unscaled pre-terminal position and terminal overshoot, including continuous and discontinuous termination.

 	\bigskip
 	
 	
 	\smallskip
 	

 \end{abstract} 
 
 %\tableofcontents
 
 \section{Introduction and main results}
 \label{Sec.Introduction}
 \setcounter{equation}{0}
 
 Consider a one-dimensional real-valued L\'evy process $X=\{X_t:t\geq0\}$ with negative drift $\beta:=-\mathbf E[X_1]\in(0,\infty)$, diffusion coefficient $\sigma\geq0$ and L\'evy measure $\nu(dy)$ satisfying
 \begin{equation}\label{eqn.LevyMeasure}
 	\int_{\mathbb{R}}\big(|y|\wedge |y|^2\big)\,\nu(dy)<\infty.
 \end{equation}
 For $x\in \mathbb{R}$, we write $ \tau^-_x$ and $ \tau^+_x$ for the first passage times of $X$ into $(-\infty,x)$ and $(x,\infty)$, respectively. 
 Denote by $\underline{X}$ and $\overline{X}$ the \textsl{running infimum} and \textsl{supremum processes}, respectively. The reflected processes $X-\underline{X}$ and $\overline{X}-X$ are non-negative strong Markov processes.
 Their excursions away from zero form two Poisson point processes on the space $\mathcal{E}$ of excursions, with respective intensities $\underline{n}(d\epsilon)$ and $\overline{n}(d\epsilon)$ that are well-known as \textsl{excursion measures}.  
 Here, each excursion $\epsilon \in \mathcal{E}$ is a non-negative c\`adl\`ag function on $\mathbb{R}_+$ with \textsl{lifetime} $\zeta$ and \textsl{height} $\overline\epsilon$ defined by 
 \beqnn
 \zeta:= \inf\big\{t>0: \epsilon_s= 0 \mbox{ for all } s\geq t \big\}
 \quad\mbox{and}\quad 
 \overline\epsilon:=\sup_{0\leq s<\zeta}\epsilon_s, 
 \eeqnn
 with $\inf \emptyset =\infty$ for convention. 
 For path convergence, each excursion is killed at zero at and after its lifetime. 
 To describe its termination, we also retain the unreflected terminal position, denoted by $\epsilon_\zeta\leq0$, as an additional mark. Thus $\epsilon_{\zeta-}$ and $-\epsilon_\zeta$ are the pre-terminal position and terminal overshoot, respectively.  
 
 Our earlier work~\cite{CuiWuXu2025} investigated the asymptotic shapes of excursions with long lifetime or large height. 
 When the tail law $\overline\nu(x):=\nu\big((x,\infty)\big)$ is regularly varying at infinity, we proved that
 \beqlb\label{CWX01}
 \underline n(\zeta>t)\sim\underline n(\overline\epsilon>\beta t)
 \sim\underline n(\zeta)\cdot \overline\nu(\beta t),\quad \mbox{as } t\to\infty.
 \eeqlb
 Moreover, the events $\{\zeta>t\}$ and $\{\overline\epsilon>\beta t\}$ are asymptotically equivalent, in the sense that
 \beqlb\label{eqn.tail.equivalence}
 \underline n(\overline\epsilon>\beta t\mid\zeta>t)\sim
 \underline n(\zeta>t\mid\overline\epsilon>\beta t)\to 1.
 \eeqlb
 Conditioned on either event, the spatially rescaled excursion converges to a single-jump process with a Pareto-distributed jump size and an independent arrival time $\mathcal T$ satisfying
 \begin{equation}\label{CWX02}
 	\mathbf P(\mathcal T\leq s)=
 	\frac{\underline n(\zeta\wedge s)}{\underline n(\zeta)},\quad s\geq0.
 \end{equation}
 After rescaling time by $t$, the effect of the negative mean becomes visible through the linear descent following the jump. 
 
 These results describe excursions selected by a tail event. 
 A finer question is how the excursion behaves when its lifetime or height is constrained to lie in a local interval.
 More precisely, for $\delta\in(0,\infty]$, $x\in \mathbb{R}$ and $a>0$, write
 \beqlb 
 \varDelta_\delta :=(0,\delta],\quad x+\varDelta_\delta:=(x,x+\delta]
 \quad \mbox{and}\quad   a\cdot (x+\varDelta_\delta):=(ax,ax+a\delta], 
 \eeqlb
 with $x+\varDelta_\infty:=(x,\infty)$.
 We are interested in the local masses
 \beqlb\label{eqn.001}
 \underline n\big(\zeta\in t+\varDelta_\delta\big)
 \quad\text{and}\quad
 \underline n\big(\overline\epsilon\in x+\varDelta_\delta\big),
 \eeqlb
 as well as in the corresponding excursion laws obtained by conditioning on either of these events.
 Under a local regular-variation assumption on the L\'evy measure $\nu(dy)$, this paper aims to address the following questions:
 \begin{enumerate}
 	\item[$\bullet$] Does the asymptotic equivalence \eqref{eqn.tail.equivalence} remain valid at the local level for the masses in~\eqref{eqn.001}\,?
 	
 	\item[$\bullet$]  Do local lifetime conditioning and local height conditioning lead to the same scaling limits\,?
 	
 	\item[$\bullet$]  How does a long excursion terminate when its lifetime is constrained to a local interval\,?
 \end{enumerate}

 \subsection{Literature review}
 
 Local first-passage asymptotics have been studied extensively for a random walk $S=\{S_n:n=0,1,\cdots \}$. Let $\tau^-_0=\inf\{n\geq1:S_n\leq0\}$. 
 When $S$ belongs to the attraction domain of a stable law with positivity parameter $\rho \in (0,1)$ and norming sequence $c=\{c_n:n=1,2,\cdots\}$, Vatutin and Wachtel~\cite{VatutinWachtel2009} proved that  
 \beqnn
 \mathbf{P}\big(\tau^-_0=n\big)\sim(1-\rho)\cdot \frac{\mathbf{P}(\tau^-_0>n)}{n},\quad \mbox{as $n\to\infty$}.
 \eeqnn
 Later, Doney~\cite{Doney2012} obtained related local first-passage estimates allowing the barrier to vary with time. In particular, the probability of first crossing a level $x=o(c_n)$ is asymptotic to the probability for level zero multiplied by the ascending ladder-height renewal function. These results connect local first-passage laws with the spatial structure of paths conditioned to stay positive. 
 When $S$ has negative drift $\mathbf{E}[S_1]=-a<0$ and
 step-size  has subexponential density $b$, Denisov et al.~\cite{DenisovVatutinWachtel2014} proved that for each fixed $x\geq 0$,
 \beqnn
 \mathbf{P}_x\big (\tau^-_0=n\big)\sim a\cdot \mathbf{E}_x[\tau^-_0]\cdot b(an) . 
 \eeqnn
 In particular, if $b$ is regularly varying at infinity with index $-\alpha-1$ for some $\alpha>1$,   this and the corresponding tail estimate give 
 \beqlb\label{eqn.003}
 \mathbf{P}_x\big (\tau^-_0=n\big) \sim \alpha \cdot \frac{\mathbf{P}_x\big (\tau^-_0>n\big)}{n}. 
 \eeqlb 
 
 For L\'evy processes, Doney and Rivero~\cite{DoneyRivero2013,DoneyRivero2016Erratum} developed the corresponding local theory under the assumption that $X$ belongs to the domain of attraction of a strictly stable law without centering, with scaling index $\gamma \in (0,2]$, positivity parameter $\rho \in (0,1)$ and norming function $c$. 
 Their excursion-lifetime estimate gives, for bounded positive window widths, that as $t\to\infty$, 
 \beqlb\label{eqn.004}
 \underline n\big(\zeta\in t+\varDelta_\delta\big) \sim (1-\rho)\cdot \frac{\underline n(\zeta>t)}{t}\cdot \delta.
 \eeqlb
 The original uniformity over widths tending to zero is understood with the additional hypotheses in their erratum~\cite{DoneyRivero2016Erratum}.
 They also distinguished continuous termination from termination by a jump as follows: as $t\to\infty$, 
 \beqlb\label{eqn.005}
 \underline{n}\big(\epsilon_{\zeta-}=0\,\big|\, \zeta\in t+\varDelta_\delta\big)  \to \frac{\underline{\tt a}}{\mathbf{E}\big[\widehat{H}_1\big]}  
 \quad \mbox{and}\quad
 \underline{n}\big(\epsilon_{\zeta-}>0\,\big|\, \zeta\in t+\varDelta_\delta\big) \to 1- \frac{\underline{\tt a}}{\mathbf{E}\big[\widehat{H}_1\big]}  ,
 \eeqlb
 where $\widehat{H}$ is the descending ladder height process that has drift $\underline{\tt a}$. 
 When $\gamma(1-\rho)<1$, one has $\mathbf{E}\big[\widehat{H}_1\big]=\infty$, so continuous termination is negligible on the scale of the total local lifetime mass. In the boundary case $\gamma(1-\rho)=1$, $\mathbf{E}\big[\widehat{H}_1\big]$ can be finite or infinite and both mechanisms must be retained. In the absence of negative jumps, termination is continuous.

 \subsection{Main results}
 
 We now state our main results on the uniform local asymptotics for the lifetime and height of large excursions, together with the corresponding scaling limits describing their geometric shapes. 
 Throughout, we write $f\sim g$ whenever $f(x)/g(x)\to1$ as $x\to\infty$. 
 For $\kappa\in\mathbb{R}$, a positive measurable function $f$ on $\mathbb{R}$ is said to be \textsl{regularly varying} at infinity with index $\kappa$, written $f \in \mathrm{RV}^\infty_\kappa$, if $f(\lambda t)/f(t) \to \lambda^\kappa$ as $t\to\infty$ for any $\lambda>0$. 
 We refer to \cite{BinghamGoldieTeugels1987} for the basic theory of regularly varying functions. We also introduce the class
 \beqnn
 \boldsymbol{\varTheta}=\big\{\theta:[0,\infty)\to(0,\infty):
 \theta\text{ is non-decreasing and }\theta_t/t\to0 \mbox{ as }t\to\infty\big\}.
 \eeqnn
 
 If the state $0$ is irregular for $(0,\infty)$, for $X$, then $\underline{n}(\epsilon_0>0)>0$.
 Under the conditional measure $\underline n(\cdot\,|\,\epsilon_0>0)$, the canonical process has the same law as $X$ started from an initial position with distribution $\underline{n}(\epsilon_0\in dx\,|\,\epsilon_0>0)$ and killed upon first downward passing $0$.  
 Consequently, the asymptotic behavior of excursions starting from a positive position can be deduced from the corresponding asymptotics for the L\'evy process started from a random positive initial state; see \cite{WuXu2026c}.
 We therefore restrict throughout to the complementary case and impose the standing assumption that 
 \begin{center}
 	\bf the state $0$ is regular for $(0,\infty)$, for $X$.
 \end{center}
 
 All our results are established under the following local regular-variation  condition on the L\'evy measure $\nu(dy)$. 
 This condition strengthens the heavy-tail assumption imposed in \cite{CuiWuXu2025} by specifying the asymptotic mass of bounded intervals, while still allowing $\nu(dy)$ to be non-absolutely continuous.
 
 \begin{condition}\label{Main.Condition}
 	For some $\alpha>1$, assume that $x\mapsto\nu\big(x+\varDelta_\delta\big)
 	\in \mathrm{RV}^\infty_{-\alpha-1}$ for any 
 	$\delta>0$. 
 \end{condition}
 
 \begin{remark} \label{Remark01}
 	By Proposition~3.3 in \cite{WuXu2026a}, Condition~\ref{Main.Condition} is equivalent to requiring that for every $\delta>0$,
 	\beqlb\label{eqn.30}
 	\overline{\nu}\in\mathrm{RV}_{-\alpha}^{\infty}
 	\quad \mbox{and} \quad
 	\nu\big(x+\varDelta_\delta\big)
 	\sim
 	\alpha\cdot\frac{\overline{\nu}(x)}{x}\cdot\delta,
 	\quad \mbox{as } x\to\infty.
 	\eeqlb
 	In particular, this yields that Condition~\ref{Main.Condition} need only be verified for $\delta$ in a neighborhood of $0$. 
 	Moreover, Lemma~\ref{Lemma.301} extends the asymptotic equivalence in~\eqref{eqn.30} uniformly to $\delta\in[\theta_0,\theta_x]$ for every $\theta\in\boldsymbol{\varTheta}$. 
 	Since $\theta_x=o(x)$, this corresponds to interval widths that are sublinear in $x$. For intervals of arbitrary width, however, the local approximation in~\eqref{eqn.30} is no longer uniform in general, and the interval mass $\nu(x+\varDelta_\delta)$ must be retained explicitly. 
 \end{remark}

 \textit{\textbf{Local asymptotics for lifetime and height.}}
 Our first two results provide local versions of~\eqref{CWX01}, uniformly over all interval widths bounded away from zero. The same constant $\underline n(\zeta)<\infty$ appears in both asymptotics. It accounts for the possible times at which the large jump occurs while the excursion is alive. Recall that $\beta(t+\varDelta_\delta)=(\beta t,\beta t+\beta\delta]$.
 
 \begin{theorem}\label{MainThm01}
 	For any constant $\delta_0>0$,
 	we have as $t\to\infty$, uniformly for
 	$\delta\in[\delta_0,\infty]$,
 	\beqlb\label{eqn.MainThm01}
 	\underline n\big(\zeta\in t+\varDelta_\delta\big)
 	\sim
 	\underline n(\zeta)\cdot
 	\nu\big(\beta(t+\varDelta_\delta)\big).
 	\eeqlb
 \end{theorem}
 
 \begin{theorem}\label{MainThm02}
 	For any constant $\delta_0>0$,
 	we have as $x\to\infty$, uniformly for
 	$\delta\in[\delta_0,\infty]$,
 	\beqlb\label{eqn.MainThm02}
 	\underline n\big(\overline\epsilon\in x+\varDelta_\delta\big)
 	\sim
 	\underline n(\zeta)\cdot\nu\big(x+\varDelta_\delta\big).
 	\eeqlb
 \end{theorem}

 Our window estimates concern $\delta\in[\delta_0,\infty]$, where $\delta_0>0$, and do not assert uniformity as $\delta\to0+$. 
 It follows directly from Theorems~\ref{MainThm01}--\ref{MainThm02} together with Remark~\ref{Remark01} that,  for every $\theta\in\boldsymbol{\varTheta}$, as $t\to\infty$, uniformly in $\delta\in[\theta_0,\theta_t]$,
 \beqnn
 \underline n\big(\zeta\in t+\varDelta_\delta\big)\sim  
 \underline n\big(\overline\epsilon\in \beta(t+\varDelta_\delta)\big)
 \sim \alpha \cdot \frac{\underline n(\zeta> t)}{t} \cdot \delta . 
 \eeqnn
 The lifetime asymptotic is the negative-mean analogue of \eqref{eqn.004} and the continuous-time counterpart of \eqref{eqn.003}, with uniformity extending to growing sublinear windows.
 The use of the full interval mass in Theorems~\ref{MainThm01} and \ref{MainThm02} provides a unified formulation that covers both bounded local windows and tail events. 
 As explained in Remark~\ref{Remark01}, when the interval width is sublinear in the reference level, local regular variation further reduces the interval mass to an asymptotic quantity proportional to $\delta$. 
 Such a reduction is, however, neither necessary nor generally valid uniformly over the full range of interval widths allowed in the theorem statements.

 \textit{\textbf{Failure of local equivalence.}} The preceding results also imply that corresponding lifetime and height windows have asymptotically equal masses:
 \beqnn
 \underline n\big(\zeta\in t+\varDelta_\delta\big) \sim\underline{n}\big(\overline\epsilon\in\beta(t+\varDelta_\delta)\big).
 \eeqnn
 It is therefore natural to ask whether these local events are asymptotically equivalent, as their tail counterparts are in~\eqref{eqn.tail.equivalence}. Our next theorem shows that the answer is negative. In fact, each fixed-window event becomes unlikely under conditioning on the other.
 
 \begin{theorem}\label{MainThm03}
 	For any constants $\delta_1,\delta_2>0$, we always have as $t\to\infty$,
 	\beqnn
 	\underline{n}\big(  \overline\epsilon\in \beta (t+\varDelta_{\delta_1}) \,|\,  \zeta\in t+\varDelta_{\delta_2} \big) 
 	\sim  \frac{\delta_1}{\delta_2}\cdot \underline{n}\big(   \zeta\in t+\varDelta_{\delta_2} \,|\, \overline\epsilon\in \beta (t+\varDelta_{\delta_1} )\big) \to 0.
 	\eeqnn 
 \end{theorem}
 
 The vanishing conditional probabilities in Theorem~\ref{MainThm03} do not preclude concentration in a wider window. Our next theorem shows that local lifetime conditioning forces the height to lie near $\beta t$, and local height conditioning forces the lifetime to lie near $t$, on scales dominating the centered fluctuations of the underlying process.
 
 \begin{theorem} \label{MainThm04}
 	For any constant $\delta>0$ and any positive function $\theta$ on $\mathbb{R}_+$ satisfying $|X_t+\beta t|/\theta_t \overset{\rm p}\to 0$,  we have   as $t\to\infty$, 
 	\beqlb\label{eqn.MainThm04}
 	\underline n\big( |\overline\epsilon-\beta t|  \leq \theta_t \,\big|\, \zeta\in t+\varDelta_\delta \big) \to 1
 	\quad \mbox{and}\quad 
 	\underline n\big( |\zeta-t|  \leq \theta_t \,\big|\, \overline\epsilon \in \beta(t+\varDelta_\delta ) \big) \to 1 . 
 	\eeqlb
 \end{theorem}

 Theorems~\ref{MainThm03} and~\ref{MainThm04} describe the relation between long lifetime and large height at two different spatial and temporal resolutions. The former concerns localization in fixed-width windows, whereas the latter establishes concentration in windows that dominate the fluctuations of $X_t+\beta t$; see Figure \ref{Fig.01}. In this sense, the finer the localization window, the more precise the asymptotic correspondence between a long lifetime and a large height becomes.
 Obviously, the law of large numbers yields that the convergence in probability in Theorem~\ref{MainThm04} can be satisfied whenever $\liminf_{t\to\infty}\theta_t/t  >0$.
 Moreover, if $\mathbf{E}\big[|X_1|^p\big]<\infty$ for some $p \in(1,2]$, the standard moment inequalities yield $\mathbf{E}\big[|X_t+\beta t|^p\big]=O(t)$ and hence $|X_t+\beta t|/t^\gamma \overset{\rm p}\to 0$  as $t\to\infty$ for any $\gamma >1/p$.

 \begin{figure}[htbp]
 	\centering
 	% 第一张图
 	\begin{minipage}{0.49\textwidth}
 		\centering
 		\resizebox{\linewidth}{!}{%
 			\begin{tikzpicture}[
 				font=\Large,
 				baseline=(current bounding box.south)
 				]
 				\path[use as bounding box]
 				(-2.7,-0.9) rectangle (10.3,7.3);
 				
 				% 水平参考线
 				\draw[black,densely dashed,thin]
 				(0,6.2) -- (9.75,6.2);
 				\draw[black,densely dashed,thin]
 				(0,4.6) -- (9.75,4.6);
 				\draw[black,densely dashed,thin]
 				(0,3.0) -- (9.75,3.0);
 				
 				% 辅助线
 				\draw[black,dashed,thin]
 				(2.15,0) -- (2.15,7);
 				\draw[orange,dashed,thin]
 				(2.15,6.2) -- (8.05,0);
 				\draw[orange,dashed,thin]
 				(2.15,3.0) -- (7.65,0);
 				
 				% 插入第一张 R 路径图
 				% 图片对应绘图区：[0,10] x [0,7]
 				\node[anchor=south west,inner sep=0,outer sep=0] at (0,0)
 				{\includegraphics[width=10cm,height=7cm]{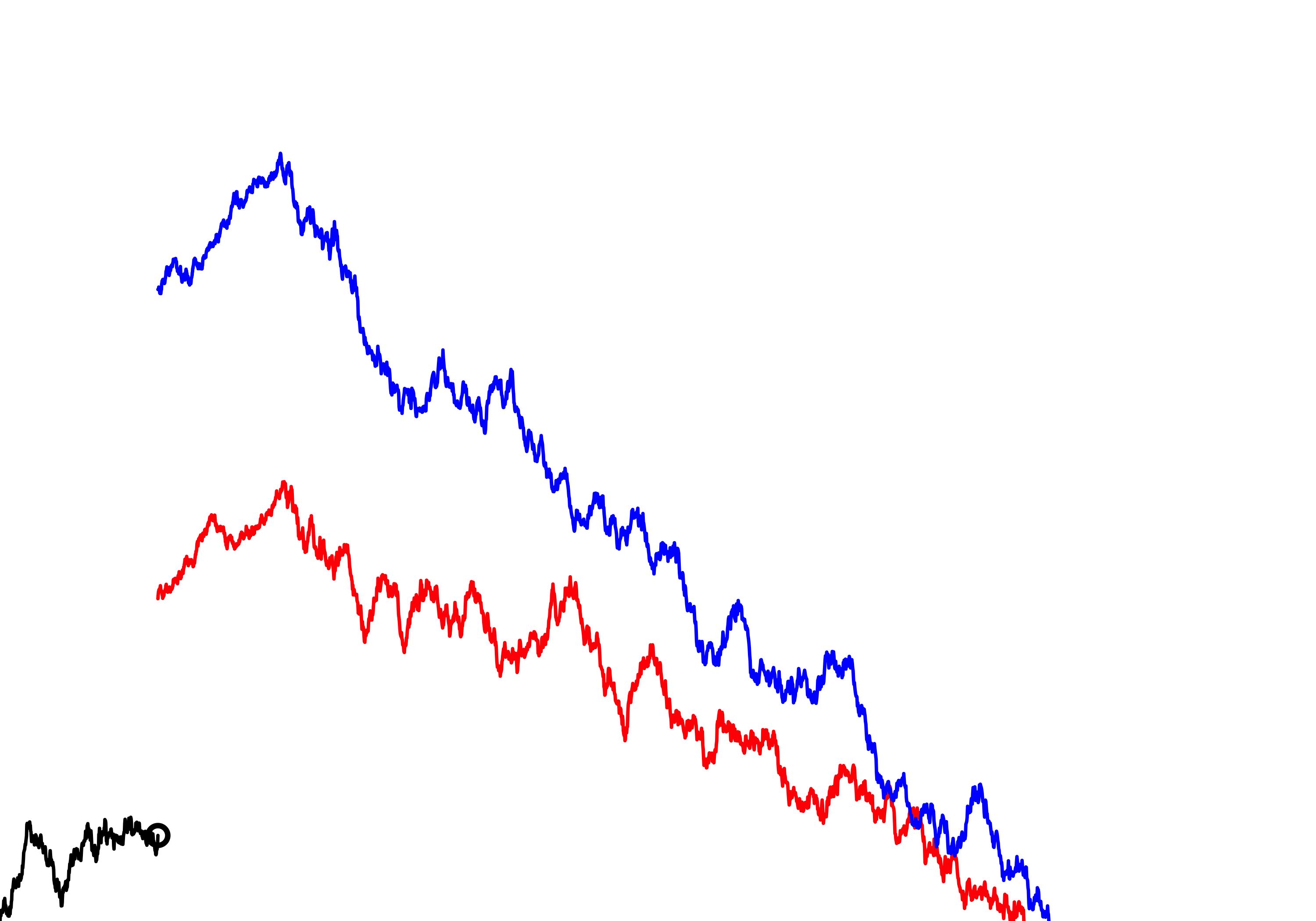}};
 				
 				% 坐标轴
 				\draw[->,thick] (0,0) -- (10,0);
 				\draw[->,thick] (0,0) -- (0,7);
 				
 				% 纵轴刻度与标注
 				\draw[thick] (0,6.2) -- (0.08,6.2);
 				\draw[thick] (0,4.6) -- (0.08,4.6);
 				\draw[thick] (0,3.0) -- (0.08,3.0);
 				\node[left=3pt,font=\small] at (0,6.2)
 				{$\beta t+\theta_t$};
 				\node[left=5pt,font=\small] at (0,4.6)
 				{$\beta t$};
 				\node[left=5pt,font=\small] at (0,3.0)
 				{$\beta t-\theta_t$};
 				
 				% J 的位置
 				\draw[thick] (1.2,0) -- (1.2,0.08);
 				\node[below=5pt,font=\small] at (1.2,0) {$\mathcal{J}$};
 				
 				% 寿命区间：左开右闭，长度为 0.4
 				\draw[line width=2pt]
 				(7.65,0) -- (8.05,0);
 				\draw[black,fill=white,thick]
 				(7.65,0) circle (2pt);
 				\fill[black]
 				(8.05,0) circle (2pt);
 				\node[below=7pt,font=\small] at (7.85,0)
 				{$t+\Delta_\delta$};
 				
 				% 右上角标题：两个标题的基线均位于 y=6.8
 				\node[anchor=base east,inner sep=0,outer sep=0] at (8.75,6.8)
 				{$\underline n\big( \overline\epsilon\in\cdot \,\big|\, \zeta\in t+\varDelta_\delta \big)$};
 			\end{tikzpicture}%
 		}
 	\end{minipage}%
 	\hfill
 	% 第二张图
 	\begin{minipage}{0.49\textwidth}
 		\centering
 		\resizebox{\linewidth}{!}{%
 			\begin{tikzpicture}[
 				font=\Large,
 				baseline=(current bounding box.south)
 				]
 				\path[use as bounding box]
 				(-2.7,-0.9) rectangle (10.3,7.3);
 				
 				% 高度区间边界：上边界与第一张图的 6.2 对齐
 				\draw[black,densely dashed,thin]
 				(0,6.2) -- (9.75,6.2);
 				\draw[black,densely dashed,thin]
 				(0,5.8) -- (9.75,5.8);
 				
 				% 辅助线
 				\draw[black,dashed,thin]
 				(1.85,0) -- (1.85,7);
 				\draw[orange,dashed,thin]
 				(1.85,5.8) -- (4.8,0);
 				\draw[orange,dashed,thin]
 				(1.85,6.2) -- (9.1,0);
 				
 				% 插入第二张 R 路径图
 				% 图片对应绘图区：[0,10] x [0,7]
 				\node[anchor=south west,inner sep=0,outer sep=0] at (0,0)
 				{\includegraphics[width=10cm,height=7cm]{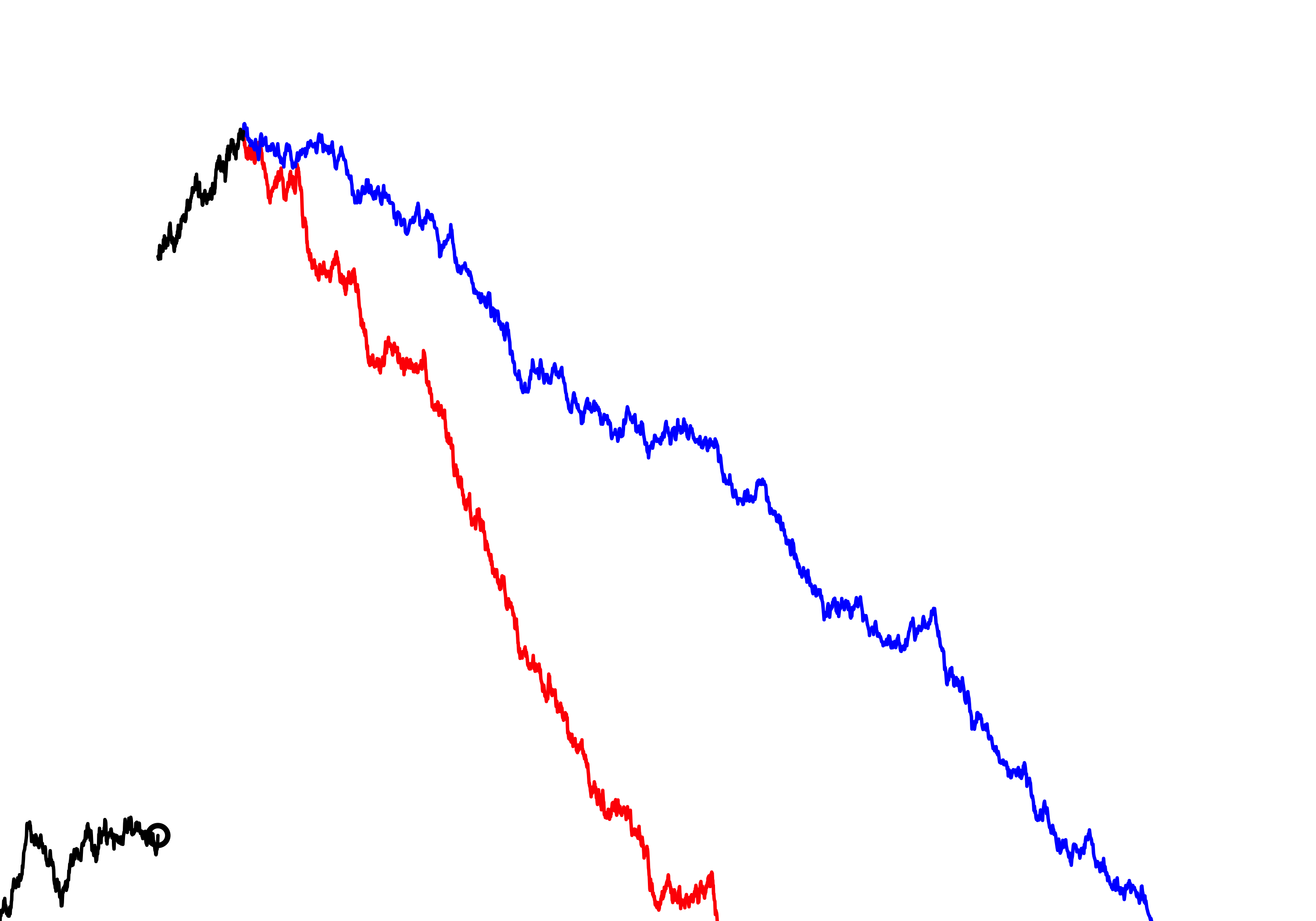}};
 				
 				% 坐标轴
 				\draw[->,thick] (0,0) -- (10,0);
 				\draw[->,thick] (0,0) -- (0,7);
 				
 				% 高度区间：左开右闭，长度为 0.4
 				\draw[line width=2pt]
 				(0,5.8) -- (0,6.2);
 				\draw[black,fill=white,thick]
 				(0,5.8) circle (2pt);
 				\fill[black]
 				(0,6.2) circle (2pt);
 				\node[left=5pt,font=\small] at (0,6.0)
 				{$\beta(t+\Delta_\delta)$};
 				
 				% J 的位置
 				\draw[thick] (1.2,0) -- (1.2,0.08);
 				\node[below=5pt,font=\small] at (1.2,0) {$\mathcal{J}$};
 				
 				% 横轴刻度与标注
 				\draw[thick] (4.8,0) -- (4.8,0.08);
 				\draw[thick] (6.95,0) -- (6.95,0.08);
 				\draw[thick] (9.1,0) -- (9.1,0.08);
 				\node[below=5pt,font=\small] at (4.8,0)
 				{$t-\theta_t$};
 				\node[below=5pt,font=\small] at (6.95,0)
 				{$t$};
 				\node[below=5pt,font=\small] at (9.1,0)
 				{$t+\theta_t$};
 				
 				% 右上角标题：两个标题的基线均位于 y=6.8
 				\node[anchor=base east,inner sep=0,outer sep=0] at (8.75,6.8)
 				{$\underline n\big( \zeta\in\cdot \,\big|\,
 					\overline\epsilon\in\beta(t+\varDelta_\delta) \big)$};
 			\end{tikzpicture}%
 		}
 	\end{minipage}
 	\caption{Conditional localization of excursion height and lifetime. 
 		Left: under the local lifetime constraint $\zeta\in t+\varDelta_\delta$, the height concentrates in $[\beta t-\theta_t,\beta t+\theta_t]$. 
 		Right: under the local height constraint $\overline\epsilon\in\beta(t+\varDelta_\delta)$, the lifetime concentrates in $[t-\theta_t,t+\theta_t]$. 
 		The coloured curves illustrate possible excursion trajectories.} \label{Fig.01}
 \end{figure}
 
 \begin{remark}\label{Remark.02}
 	There always exists some function $\theta \in \boldsymbol{\varTheta}$ such that  the convergence in probability in Theorem~\ref{MainThm04} is satisfied. For instance, choose an increasing sequence $\{T_n\}_{n\geq0}$ with $T_0=0$ and $T_1=1$ such that $T_{n}/T_{n-1}\geq n$ and $\mathbf{P}\big(|X_t+\beta t|/t \geq 1/n^2\big) \leq 1/n$ for all $n\geq 2$ and $t\geq T_n$.  
 	Consider a continuous piecewise linear function $\ell $ on $\mathbb{R}_+$ defined by 
 	\beqnn
 	\ell_t=  n+ \frac{t-T_n}{T_{n+1}-T_n} \in[n,n+1], \qquad \mbox{for }t\in[T_n,T_{n+1}).
 	\eeqnn
 	Now set $\theta_0=1$ and $\theta_t=t/\ell_t$ for $ t>0$. 
 	It is straightforward to verify that $\ell_t\to\infty$ as $t\to\infty$ and hence $\theta \in \boldsymbol{\varTheta}$.    
 	It remains to verify the convergence in probability in Theorem~\ref{MainThm04}.
 	Fix $\varepsilon\in(0,1)$.
 	For any $n\geq 2/\varepsilon$ and $t\in[T_n,T_{n+1})$, we have $n\leq \ell_t\leq n+1$  and then 
 	\beqnn
 	\mathbf{P}\big( |X_t+\beta t|/\theta_t\geq \varepsilon \big)
 	\ar=\ar   \mathbf{P}\big( |X_t+\beta t|/t\geq \varepsilon/\ell_t \big) 
 	\leq   \mathbf{P}\big( |X_t+\beta t|/t\geq 1/n^2  \big) \leq 1/n,
 	\eeqnn
 	which goes to $0$ as $t\to\infty$. 
 	This implies $\overline\epsilon/t\to\beta$ under local lifetime conditioning and $\zeta/t\to1$ under local height conditioning. Thus the local events remain linked on a macroscopic scale, despite their negligible relative overlap on a fixed scale.
 \end{remark}

 \textit{\textbf{Conditional scaling limits.}}
 We now turn to the entire excursion under local conditioning.  
 In the tail-conditioned scaling limits of~\cite{CuiWuXu2025}, the principal large jump retains a non-degenerate Pareto-distributed size after rescaling. 
 As explained in Remark~\ref{Remark.02}, a sublinear lifetime or height windows instead removes this macroscopic size randomness and forces the rescaled jump size to concentrate at a deterministic value; see Figure~\ref{Fig.02}.  
 By contrast, the occurrence time of the principal large jump remains random on the original time scale and converges in distribution to the random variable $\mathcal T$, whose law is given in~\eqref{CWX02}. 
 
 For a locally compact set $\mathcal{I}\subset \mathbb{R}_+$, let $D(\mathcal{I};\mathbb{R})$ be the space of all c\`adl\`ag functions on $\mathcal{I}$ that is equipped with the Skorohod $J_1$-topology. 
 A sequence of probability laws on $D(\mathcal{I};\mathbb{R})$ is said to   \textit{converge weakly} if it converges weakly on $D(I;\mathbb{R})$ for any compact subset $I\subset \mathcal{I}$; see Chapter 3 in \cite{Billingsley1999}.
 
 \begin{theorem}\label{MainThm05}
 	For any  $\theta \in  \boldsymbol{\varTheta}$, as $t\to\infty$, uniformly in $\delta\in[\theta_0,\theta_t]$, the following assertions hold: 
 	\begin{enumerate}
 		\item[(1)] The pushforward measures of both  $\underline n\big(\cdot\,\big|\,\zeta\in t+\varDelta_\delta\big)$ and $\underline n\big(\cdot\,\big|\,\overline{\epsilon}\in\beta(t+\varDelta_\delta)\big)$ under the scaling map $ \epsilon\mapsto \big\{\epsilon_s/t:s\geq0\big\}$ converge weakly on $D\big([0,\infty);\mathbb{R}\big)$ to the law of the single-step process $\big\{\beta\cdot\mathbf 1_{\{s\geq\mathcal T\}}:s\geq0\big\}$. 
 		
 		\item[(2)] The pushforward measures of both $\underline n\big(\cdot\,\big|\,\zeta\in t+\varDelta_\delta\big)$ and $\underline n\big(\cdot\,\big|\,\overline{\epsilon}\in\beta(t+\varDelta_\delta)\big)$ under the scaling map $\epsilon\mapsto \big\{\epsilon_{ts}/t:s>0\big\}$ converge weakly on $D\big((0,\infty);\mathbb R\big)$ to the Dirac measure concentrated at the path $\big\{\beta(1-s)^+:s>0\big\}$.
 	\end{enumerate}
 \end{theorem}

 The first limit retains the arrival time of the large jump, while its size becomes deterministic. 
 The negative mean produces a negligible displacement on each bounded interval of the original time scale after division by $t$. 
 On the accelerated scale, the arrival time collapses to zero and the accumulated negative mean gives the deterministic path $\beta(1-s)^+$; see Figure~\ref{Fig.02}. 
 Convergence is formulated on $D\big((0,\infty);\mathbb R\big)$ in the second assertion because the limiting path has right limit $\beta$ at zero, whereas the excursions start at zero.  
 
 \begin{figure}[htbp]
 	\centering
 	\def\shrink{0.5}
 	
 	% 第一张图
 	\begin{minipage}{0.49\textwidth}
 		\centering
 		\resizebox{\linewidth}{!}{%
 			\begin{tikzpicture}[
 				font=\Large,
 				baseline=(current bounding box.south)
 				]
 				\path[use as bounding box]
 				(-2.7,-0.9) rectangle (10.3,7.3);
 				
 				\pgfmathsetmacro{\xvertical}{2.15*\shrink}
 				\pgfmathsetmacro{\yupper}{4.6+1.6*\shrink}
 				\pgfmathsetmacro{\ylower}{4.6-1.6*\shrink}
 				\pgfmathsetmacro{\xend}{7.65+0.4*\shrink}
 				
 				% 水平参考线
 				\draw[black,densely dashed,thin]
 				(0,\yupper) -- (9.75,\yupper);
 				\draw[black,densely dashed,thin]
 				(0,4.6) -- (9.75,4.6);
 				\draw[black,densely dashed,thin]
 				(0,\ylower) -- (9.75,\ylower);
 				
 				% 辅助线
 				\draw[black,dashed,thick]
 				(\xvertical,0) -- (\xvertical,7);
 				\draw[orange,dashed,thin]
 				(\xvertical,\yupper) -- (\xend,0);
 				\draw[orange,dashed,thin]
 				(\xvertical,\ylower) -- (7.65,0);
 				\draw[black,thick]
 				(\xvertical,4.6) -- (7.65,0);
 				
 				% 同步压缩后的 R 路径图
 				\node[anchor=south west,inner sep=0,outer sep=0] at (0,0)
 				{\includegraphics[width=10cm,height=7cm]
 					{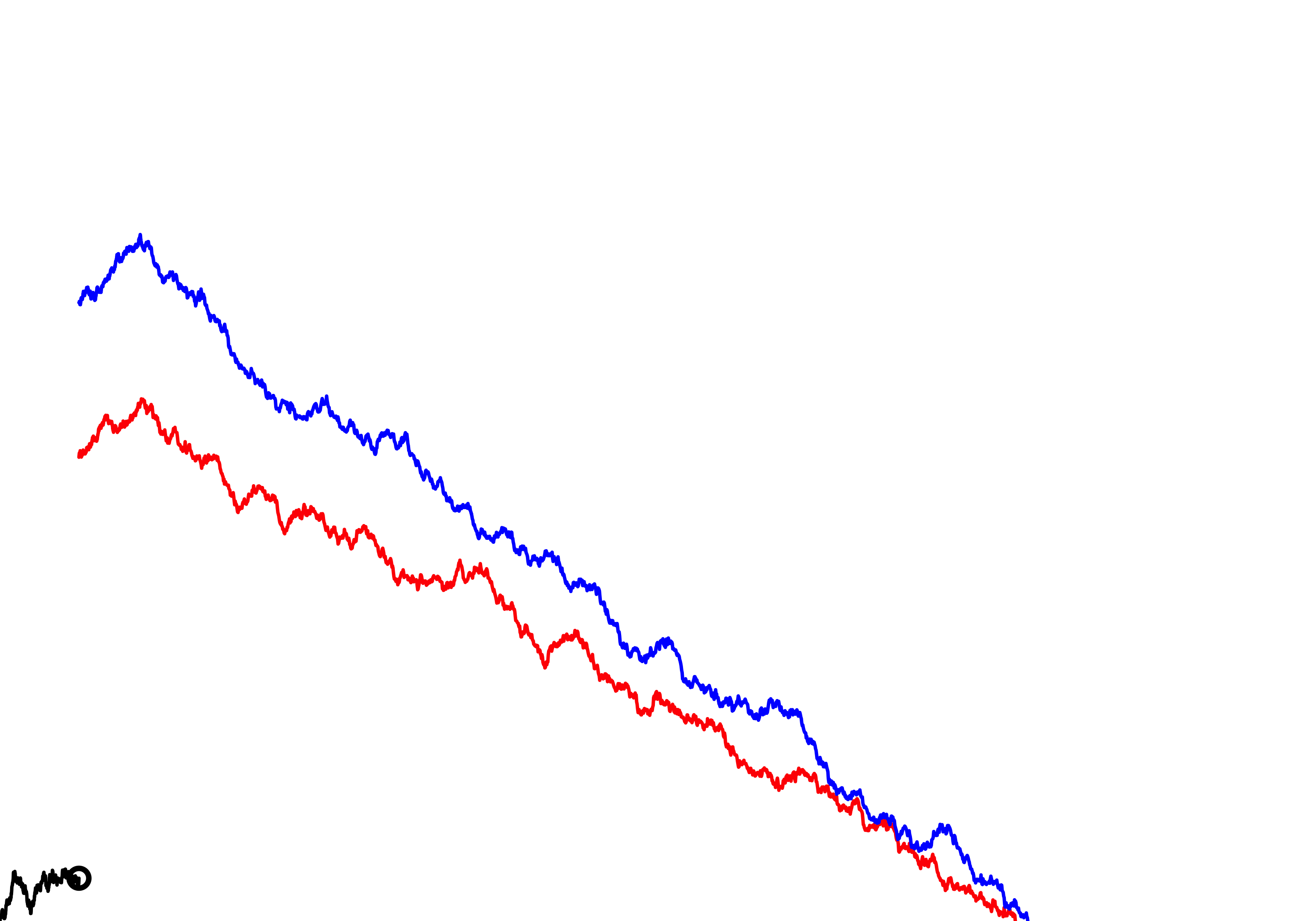}};
 				
 				% 坐标轴
 				\draw[->,thick] (0,0) -- (10,0);
 				\draw[->,thick] (0,0) -- (0,7);
 				
 				% 纵轴刻度与标注
 				\draw[thick] (0,\yupper) -- (0.08,\yupper);
 				\draw[thick] (0,4.6) -- (0.08,4.6);
 				\draw[thick] (0,\ylower) -- (0.08,\ylower);
 				\node[left=5pt,font=\small] at (0,\yupper)
 				{$\beta+\theta_t/t$};
 				\node[left=5pt,font=\small] at (0,4.6) {$\beta$};
 				\node[left=5pt,font=\small] at (0,\ylower)
 				{$\beta-\theta_t/t$};
 				
 				\draw[thick] (0.6,0) -- (0.6,0.08);
 				\node[below=5pt,font=\small] at (0.6,0) {$\mathcal{J}/t$};
 				
 				% 寿命区间同步缩小，左开右闭
 				\draw[line width=2pt]
 				(7.65,0) -- (\xend,0);
 				\draw[black,fill=white,thick]
 				(7.65,0) circle (2pt);
 				\fill[black] (\xend,0) circle (2pt);
 				\node[below=7pt,font=\small] at (7.65,0) {$1$};
 				
 				% 向左箭头
 				\draw[<-,thick]
 				({0.8*\shrink},1.5) -- (\xvertical,1.5);
 				
 				% 下方水平线旁：向上
 				\draw[->,thick]
 				({-0.8*\shrink},\ylower+0.2)
 				-- ({-0.8*\shrink},{4.6-0.5*\shrink});
 				
 				% 上方水平线旁：向下
 				\draw[->,thick]
 				({-0.8*\shrink},\yupper-0.1)
 				-- ({-0.8*\shrink},{4.6+0.7*\shrink});
 				
 				% 右上角标题
 				\node[anchor=base east,inner sep=0,outer sep=0] at (8.75,6.5)
 				{$\underline n\big( \epsilon_{t\cdot}/t\in\cdot
 					\,\big|\, \zeta\in t+\varDelta_\delta \big)$};
 			\end{tikzpicture}%
 		}
 	\end{minipage}%
 	\hfill
 	% 第二张图
 	\begin{minipage}{0.49\textwidth}
 		\centering
 		\resizebox{\linewidth}{!}{%
 			\begin{tikzpicture}[
 				font=\Large,
 				baseline=(current bounding box.south)
 				]
 				\path[use as bounding box]
 				(-2.7,-0.9) rectangle (10.3,7.3);
 				
 				\pgfmathsetmacro{\xvertical}{1.85*\shrink}
 				\pgfmathsetmacro{\yupper}{5.8+0.4*\shrink}
 				\pgfmathsetmacro{\xlower}{6.95-2.15*\shrink}
 				\pgfmathsetmacro{\xupper}{6.95+2.15*\shrink}
 				
 				% 高度区间边界
 				\draw[black,densely dashed,thin]
 				(0,\yupper) -- (9.75,\yupper);
 				\draw[black,densely dashed,thin]
 				(0,5.8) -- (9.75,5.8);
 				
 				% 辅助线
 				\draw[black,dashed,thick]
 				(\xvertical,0) -- (\xvertical,7);
 				\draw[orange,dashed,thin]
 				(\xvertical,5.8) -- (\xlower,0);
 				\draw[orange,dashed,thin]
 				(\xvertical,\yupper) -- (\xupper,0);
 				\draw[black, thick]
 				(\xvertical,5.8) -- (6.95,0);
 				
 				% 同步压缩后的 R 路径图
 				\node[anchor=south west,inner sep=0,outer sep=0] at (0,0)
 				{\includegraphics[width=10cm,height=7cm]
 					{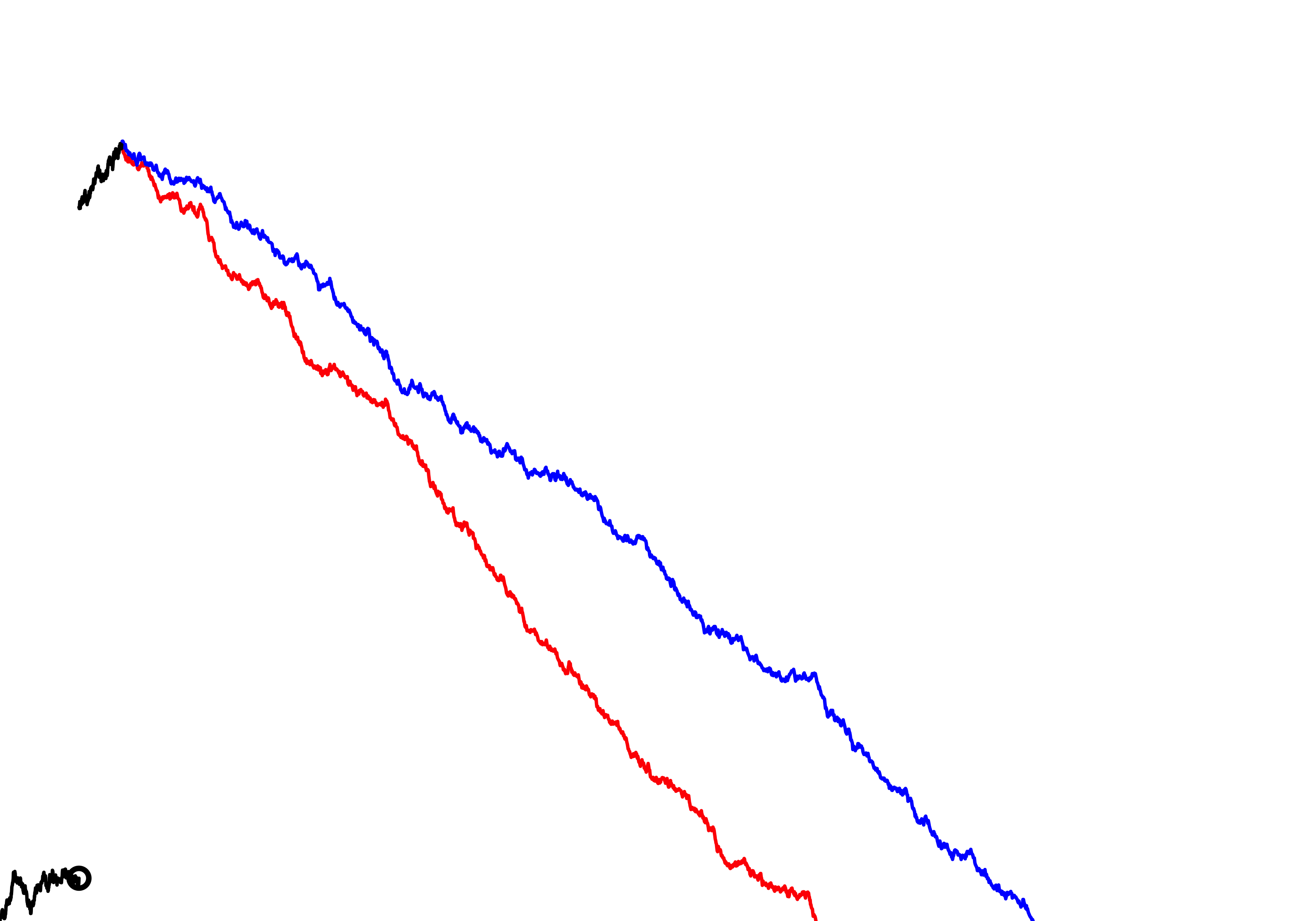}};
 				
 				% 坐标轴
 				\draw[->,thick] (0,0) -- (10,0);
 				\draw[->,thick] (0,0) -- (0,7);
 				
 				% 高度区间：左开右闭
 				\draw[line width=2pt]
 				(0,5.8) -- (0,\yupper);
 				\draw[black,fill=white,thick]
 				(0,5.8) circle (2pt);
 				\fill[black] (0,\yupper) circle (2pt);
 				\node[left=5pt,font=\small] at (0,5.8) {$\beta$};
 				
 				\draw[thick] (0.6,0) -- (0.6,0.08);
 				\node[below=5pt,font=\small] at (0.6,0) {$\mathcal{J}/t$};
 				
 				% 横轴刻度与标注
 				\draw[thick] (\xlower,0) -- (\xlower,0.08);
 				\draw[thick] (6.95,0) -- (6.95,0.08);
 				\draw[thick] (\xupper,0) -- (\xupper,0.08);
 				\node[below=5pt,xshift=-12pt,font=\small]
 				(timeLower) at (\xlower,0)
 				{$1-\theta_t/t$};
 				
 				\node[below=5pt,font=\small]
 				(timeCenter) at (6.95,0) {$1$};
 				
 				\node[below=5pt,xshift=12pt,font=\small]
 				(timeUpper) at (\xupper,0)
 				{$1+\theta_t/t$};
 				
 				% 向左箭头
 				\draw[<-,thick]
 				({0.5*\shrink},1.5) -- (\xvertical,1.5);
 				
 				% 左侧向右，指向 1
 				\draw[->,thick]
 				([xshift={-\shrink cm}]timeCenter.west)
 				-- (timeCenter.west);
 				
 				% 右侧向左，指向 1
 				\draw[->,thick]
 				([xshift={\shrink cm}]timeCenter.east)
 				-- (timeCenter.east);
 				
 				% 右上角标题
 				\node[anchor=base east,inner sep=0,outer sep=0] at (8.75,6.5)
 				{$\underline n\big( \epsilon_{t\cdot}/t\in\cdot \,\big|\,
 					\overline\epsilon\in\beta(t+\varDelta_\delta) \big)$};
 			\end{tikzpicture}%
 		}
 	\end{minipage}
 	\caption{A common macroscopic excursion shape under local lifetime and height conditioning. The panels illustrate the rescaled paths ${\epsilon_{ts}/t>0}$, conditioned on $\zeta\in t+\varDelta_\delta$ (left) and $\overline\epsilon\in\beta(t+\varDelta_\delta)$ (right). 
 		Under either conditioning, the distinguished large jump occurs at a rescaled time tending to zero, the coloured trajectories approach the deterministic profile $\beta(1-s)^+$, represented by the black reference line. } 
 	\label{Fig.02}
 \end{figure}

 Although the fixed-width lifetime and height conditioning events have asymptotically negligible relative overlap, the two conditioning schemes nevertheless lead to the same scaling limits on both time scales, through different asymptotic mechanisms. 
 More precisely, local lifetime conditioning fixes the terminal time near $t$ and forces the principal jump size to be $\beta t+o(t)$, whereas local height conditioning fixes the maximum near $\beta t$ and forces the lifetime to be $t+o(t)$.  
 These finer-scale differences account for the negligible overlap but disappear under rescaling, yielding the same single-step limit on the original time scale and deterministic descent on the accelerated time scale; see Figure~\ref{Fig.02}. 
 
 \textit{\textbf{The terminal passage.}}
 Our final result describes the termination of an excursion with lifetime in a distant local interval. The positions immediately before and at termination are not resolved by the preceding macroscopic scalings. We therefore keep these positions unscaled and identify their joint conditional limit.
 
 Let $V$ be the renewal function of the ascending ladder height process $H$, with the paired normalization specified in Section~\ref{Sec.Preliminaries}. Denote by $\underline{\tt a}$ the drift coefficient of the descending ladder height process $\widehat H$. By Theorem~7.8 in \cite[p.~213]{Kyprianou2014},
 \beqlb\label{eqn.EH}
 \mathbf{E}\big[\widehat{H}_1\big] \ar=\ar \underline{\tt a}+\int_{0}^{\infty}dy\int_{[0,\infty)} \nu\big((-\infty,-z-y]\big) \,V(dz)  = \underline{\tt a}+\int_0^\infty V(z)\cdot \nu\big((-\infty,-z)\big)\,dz.
 \eeqlb

 \begin{theorem}\label{MainThm06}  
 	For any $\delta_0>0$ and any bounded, continuous function $f$ on $\mathbb{R}_+^2$, we have as $t\to\infty$,	 uniformly for $\delta  \in[\delta_0,\infty]$,
 	\beqlb\label{eqn.MainThm06}
 	\underline{n}\big( f (\epsilon_{\zeta-},-\epsilon_\zeta ) \,\big|\, \zeta\in t+\varDelta_\delta \big) 
 	\to \frac{1}{\mathbf{E}[\widehat{H}_1]} \Big(  \underline{\tt a}  \cdot f(0,0)+  \int_0^\infty dx\int_0^\infty f(x,y)V(x)\,\nu(-x-dy) \Big) . 
 	\eeqlb
 	In particular, the two limits in \eqref{eqn.005} hold uniformly for $\delta \in[\delta_0,\infty]$ as  $t\to\infty$. 
 \end{theorem}
 
 The identity~\eqref{eqn.EH} ensures that the limiting measure in Theorem~\ref{MainThm06} has total mass one. 
 Its atom at $(0,0)$ corresponds to continuous termination, and the jump term combines the renewal profile $V(x)$ with the rate of a negative jump from $x$ to $-y$. 
 The result connects the local entrance-law asymptotics in~\cite{CuiWuXu2026b} with the terminal passage of a long excursion. 
 It also complements the continuous and discontinuous passage estimates of Doney and Rivero~\cite{DoneyRivero2013,DoneyRivero2016Erratum} in the present negative-mean regime.  
 Whether the same conditional limit holds under weaker assumptions on the L\'evy measure, while retaining $\mathbf E\big[\widehat H_1\big]<\infty$, remains an open question. 
 
 \subsection{Methodologies}
 \label{Subsec.Methodologies}
 
 The proofs combine uniform local large-deviation estimates, excursion compensation and fluctuation identities. 
 The main difficulty is that local probabilities are smaller than the corresponding tail probabilities, requiring error estimates on the scale of a local interval mass.
 For example, for fixed $\delta>0$, the local lifetime mass is of order $t^{-1}\cdot \underline n(\zeta>t)$. 
 Subtracting tail equivalents at $t$ and $t+\delta$ does not provide the required precision, and scenarios negligible under tail conditioning may contribute substantially under local conditioning. 
 We therefore need to refine the single-big-jump arguments of~\cite{CuiWuXu2025} to account for the remaining trajectory and the local lifetime or height constraint. 
 These estimates also distinguish asymptotic equality of the marginal masses from asymptotic equivalence of the conditioning events. 
 We outline the principal steps below.

 The analysis of the local lifetime mass starts from the fluctuation identity \eqref{eqn.209}.
 Subtracting the identities at $t$ and $t+\delta$ reduces the analysis to probability differences of the form
 \beqnn
 \mathbf{P}(X_{t}>0)-\mathbf{P}(X_{t+\delta}>0)
 \ar=\ar \mathbf{P}(X_{t}>0, X_{t+\delta}\leq 0) - \mathbf{P}(X_{t}\leq 0, X_{t+\delta}>0).
 \eeqnn
 Uniform local large-deviation estimates from~\cite{WuXu2026a}, applied to the centered process $X_s+\beta s$, identify the contributions from positions near zero. 
 A further decomposition of the upward-crossing probability captures the contribution of a new large jump from the region of the typical position $-\beta t$. 
 Together, these estimates yield precise asymptotics for both crossing probabilities and their difference.

 The key step in deriving the local height asymptotics is to identify jumps that carry the excursion close to its eventual maximum. 
 More precisely, for a fixed cutoff $K>0$,  we observe that every candidate jump  has pre-jump position at most $K$, subsequent maximum in $x+\varDelta_\delta$, and further ascent after the jump at most $K$. 
 These constraints imply that its size belongs to $(x-2K,x+\delta]$. 
 Writing $N_\delta(x,K)$ for the number of candidates, we approximate the event ${\overline\epsilon\in x+\varDelta_\delta}$ by ${N_\delta(x,K)\geq1}$, with error negligible relative to $\nu(x+\varDelta_\delta)$ as $x\to\infty$ and then $K\to\infty$. We also verify that the contribution of multiple candidates is also negligible on this local scale. 
 The estimates combine truncation of positive jumps, exponential bounds for upward passage by the truncated process, and compensation for configurations involving large jumps. 
 To evaluate the expected candidate count, we first integrate out the post-jump trajectory using the killed semigroup. 
 Compensation then yields an integral over the jump time, pre-jump position and jump size. 
 
 For Theorem~\ref{MainThm03}, we split candidate jumps according to whether they occur before or after a fixed time $R$. 
 The late contribution is negligible on the local scale as $t\to\infty$ and then $R\to\infty$. 
 For an early candidate, the fixed height window confines its size to a bounded interval around $\beta t$. Imposing the lifetime constraint additionally requires the post-jump process, started near $\beta t$, to cross zero within a bounded time window near $t$. We show that this first-passage probability tends to zero uniformly over bounded shifts in the starting height and time, yielding the negligible relative overlap of the two fixed windows. 
 For Theorem~\ref{MainThm04}, we use the same distinguished jump but retain wider error windows. 
 The estimates control its arrival time, pre-jump position, size and subsequent upward displacement, together with the centered fluctuations accumulated during the descent. 
 With a scale dominating those fluctuations, these errors give the conditional localization of height and lifetime. 
 The two theorems thus follow from complementary estimates: anti-concentration on fixed scales and concentration on sufficiently large fluctuation scales. 
 
 The proof of the conditional scaling limits builds on and extends the argument for Theorem~1.8 in~\cite{CuiWuXu2025}. For $\lambda\in(0,\beta)$, under either local conditioning in Theorem~\ref{MainThm05}, the first jump time $\mathcal J^{\lambda t}$ above the threshold $\lambda t$ and the rescaled jump size $\Delta\epsilon_{\mathcal J^{\lambda t}}/t$ converge jointly to $(\mathcal T,\beta)$. On bounded intervals of the original time scale, the rescaled excursion is approximated by $t^{-1}\cdot \Delta\epsilon_{\mathcal J^{\lambda t}}\cdot \mathbf1_{\{s\geq\mathcal J^{\lambda t}\}}$.
 Estimates for the pre-jump segment and post-jump increments make the remainder negligible, yielding the single-step limit. 
 On the accelerated time scale,  a uniform law-of-large-numbers approximation controls the post-jump trajectory, while lifetime localization controls the killing time. These estimates along with $\mathcal J^{\lambda t}/t\to0$ in conditional probability yield convergence to $\beta(1-s)^+$ on compact time intervals bounded away from zero.
 
 To prove Theorem~\ref{MainThm06}, we express termination by a negative jump through the L\'evy system, i.e., for suitable test functions, the resulting integrals have the form
 \beqnn
 \int_t^{t+\delta}ds\int_0^\infty\underline{n} \big(\epsilon_s\in dx,\zeta>s\big)
 \int_0^\infty f(x,y)\,\nu(-x-dy).
 \eeqnn
 The time-integrated entrance-law limit established in \cite{CuiWuXu2026b} identifies the contribution on truncated spatial regions. 
 Additional bounds near zero and at infinity permit passage to the terminal jump kernel $V(x)\,dx\,\nu(-x-dy)$. The lifetime asymptotic and the identity  $\mathbf{E}\big[\widehat H_1\big]=\beta\underline n(\zeta)$ 
 provide the normalization. 
 The ladder-height balance~\eqref{eqn.EH} identifies the remaining mass at the origin and completes the weak limit.  
 
 \medskip
 
 \textit{\textbf{Organization of the paper.}}
 Section~\ref{Sec.Preliminaries} introduces further notation, recalls the fluctuation identities and collects the auxiliary estimates. Sections~\ref{Sec.ProofThm01} and~\ref{Sec.ProofThm02} prove the local lifetime and height laws, respectively. Section~\ref{Sec.ProofThm03} establishes the negligible overlap of fixed windows. Section~\ref{Sec.ProofThm04} proves the conditional localization estimates. The path limits are proved in Section~\ref{Sec.ProofThm05}, and the terminal-pair limit in Section~\ref{Sec.ProofThm06}.
 
 \medskip
 
 \textit{\textbf{Notation.}} 
 Let $\mathbb{N}$ be the space of all natural numbers including $0$. 
 For any $x\in\mathbb{R}$, let $x^+:=x\vee 0$, $x^-=-(x\wedge 0)$ and $[x]$ be the integer part of $x$. 
 Let $\overset{\rm f.d.d.}\longrightarrow$,  $\overset{\rm a.s.}\to$, $\overset{\rm d}\to$  and  $\overset{\rm p}\to$ be the convergence in the sense of finite dimensional distributions, almost sure convergence,  convergence in distribution and convergence in probability, respectively.  
 We make the conventions that for $x, y\in \mathbb{R}$ with $y\geq x$,
 \beqnn
 \int_x^y=-\int_y^x= \int_{(x,y]},\quad \int_{x-}^{y-}=  \int_{[x,y)}
 \quad\mbox{and}\quad
 \int_x^\infty = \int_{(x,\infty)}.
 \eeqnn
 Integrals with respect to an excursion measure are written in expectation notation, for example $\underline n(\zeta)$ and $\underline n(\zeta\wedge t)$. 
 We use $C$ to denote a positive constant whose value might change from line to line.

 \section{Preliminaries} 
 \label{Sec.Preliminaries}
 \setcounter{equation}{0}

 This section begins by introducing additional notation and recalling basic fluctuation theory for Lévy processes, for which we refer to \cite{Bertoin1996,Doney2007,Kyprianou2014}. We then present several auxiliary lemmas that will be essential for proving our main results.
 All processes in this work are defined on a complete probability space
 $(\Omega,\mathscr{F},\mathbf{P})$ equipped with a filtration $\{\mathscr{F}_t\}_{t\geq 0}$ satisfying the usual hypotheses. 
 For $x\in\mathbb{R}$, let $\mathbf{P}_x$ and $\mathbf{E}_x$ denote the law and expectation of the  process starting from $x$. 
 For simplicity, we also write $\mathbf{P}= \mathbf{P}_0$ and $\mathbf{E}=\mathbf{E}_0$.

 \subsection{Fluctuation theory}
 
 Let $\Delta X_t := X_t-X_{t-}$ denote the jump size of $X$ at time $t  > 0$. 
 The point measure of all jumps 
 \beqnn
 \sum_{t>0} \mathbf{1}_{\{ \Delta X_t\neq 0 \}}\cdot \delta_{(t,\Delta X_t)}(ds,dy),
 \eeqnn
 is a Poisson random measure on $(0,\infty)\times \mathbb{R}$ with intensity $ds\, \nu(dy)$.  
 By the compensation formula for Poisson random measures (see \cite[p.7]{Bertoin1996}), we have that for any $(\mathscr{F}_t)$-predictable process $\mathbf{F}=\{ \mathbf{F}_t:t\geq 0 \}$ taking values in the space of non-negative measurable functions on $\mathbb{R}$, 
 \beqlb\label{eqn.CompenstationX}
 \mathbf{E}\bigg[ \sum_{t>0} \mathbf{F}_t(\Delta X_t) \cdot \mathbf{1}_{\{ \Delta X_t\neq 0 \}}  \bigg]= \mathbf{E}\bigg[\int_0^\infty dt \int_{\mathbb{R}\setminus\{0\}}\mathbf             {F}_t(y)\, \nu(dy)\bigg].
 \eeqlb
 
 The reflected process $X-\underline{X}$ is a Markov process with Feller transition semigroup given by $\mathbf{P}_x(X_t\in dy, \underline{X}_t>0)$ for any $x,y>0$; see, e.g., Proposition~1 in \cite[p.156]{Bertoin1996}. 
 Its excursions away from $0$ form a Poisson point process on $\mathcal{E}$ with intensity $\underline{n}(d\epsilon)$.  
 Under $\underline{n}(d\epsilon)$, the excursion process is Markovian and shares the same semigroup as $X$ killed upon entering $(-\infty,0]$, that is, for any $t,s>0$, any measurable functionals $F$ on $D\big([0,t];\mathbb{R}\big)$ and  $G$ on $D\big([0,s];\mathbb{R}\big)$, 
 \beqlb\label{eqn.n-Markov}
 \lefteqn{\underline{n}\Big(F\big(\epsilon_r, r \in[0,t]\big) \cdot G\big(\epsilon_{r+t}, r\in[0,s]\big),\, \zeta>t+s\Big)}\ar\ar\cr
 \ar\ar\cr
 \ar=\ar \underline{n}\Big(F\big(\epsilon_r, r \leq t\big)\cdot \mathbf{E}_{\epsilon_t}\big[G\big(X_r, r\in[0,s]\big),\,  \underline{X}_s>0\big],\, \zeta>t\Big)  .
 \eeqlb 
 In particular, the excursion process is stochastically continuous and has no fixed-time jumps, in the sense that	$\underline n(\Delta\epsilon_t\neq0,\zeta>t)=0$ for all $t\geq 0$. Consequently, its entrance law satisfies 
 \beqnn
 \underline n_t(dz) :=\underline{n}\big(\epsilon_t\in dz,\,\zeta>t \big) =\underline{n}\big(\epsilon_{t-}\in dz,\,\zeta>t \big),
 \eeqnn
 and $t\mapsto\underline n_t(dz)$ is continuous as a measure-valued function. 
 By Theorem 5 in \cite{Alili2001}, the entrance law admits the decomposition
 \beqlb\label{c12}
 t\cdot \underline{n}_t(dx) =\int_0^t d s \int_0^x  \mathbf{P}_z(X_{t-s} \in dx)\,\underline{n}_s(dz);
 \eeqlb
 see also \cite{DoneySavov2010} for further detailed proof. 
 Integrating both sides of \eqref{c12} over $x>0$, then applying
 Fubini's theorem and finally the spatial homogeneity of $X$, we obtain
 \beqlb\label{eqn.209}
 t\cdot\underline n(\zeta>t)
 \ar=\ar
 \int_0^\infty
 \int_0^t ds
 \int_0^x
 \underline n_s(dz)\,
 \mathbf P_z(X_{t-s}\in dx)
 \cr\cr
 \ar=\ar
 \int_0^t ds
 \int_0^\infty \underline n_s(dz)
 \int_z^\infty
 \mathbf P_z(X_{t-s}\in dx)  
 = \int_0^t
 \mathbf P(X_{t-s}>0)
 \underline n(\zeta>s)\,ds.
 \eeqlb

 An important consequence of \eqref{eqn.CompenstationX} and \eqref{eqn.n-Markov} is that the compensation formula also holds under $\underline{n}(d\epsilon)$, i.e., let $\mathscr{G}_t:=\sigma\big(\{\epsilon_s:s\in [0,t], \epsilon \in \mathcal{E}\}\big)$ for any $t\geq 0$, for any $(\mathscr{G}_t)$-predictable process $\mathbf{F}=\{ \mathbf{F}_t:t\geq 0 \}$ taking values in the space of non-negative measurable functions on $\mathbb{R}$, we have 
 \beqlb\label{eqn.Compensation}
 \underline{n}\bigg(\sum_{0<s\leq \zeta} \mathbf{F}_s(\Delta \epsilon_s) \cdot 1_{\{\Delta \epsilon_s \neq 0\}}\bigg)
 = \underline{n}\bigg(\int_0^\zeta ds \int_{\mathbb{R}\setminus\{0\}} \mathbf{F}_s(y)\,\nu(dy)\bigg).
 \eeqlb
 In particular,  Proposition 2.1 in \cite{CuiWuXu2026b} states that for any $t> 0$ and $x,y> 0$,  
 \beqlb\label{Prop.203}
 \underline{n}\bigg(  \sum_{0<s<r\leq t\wedge \zeta} \mathbf{1}_{\{\Delta \epsilon_s > x ,\, \Delta \epsilon_r > y \}} \bigg) 
 \leq t \cdot \underline{n}(\zeta ) \cdot \overline{\nu}(x)  \cdot \overline{\nu}(y).   
 \eeqlb

 Let $\widehat{L} = \{\widehat{L}_t: t\ge 0\}$ be the local time at zero of $X-\underline{X}$ in the sense of Definition 6.1 in \cite[p.154]{Kyprianou2014}; see also \cite[p.109]{Bertoin1996}. 
 Without loss of generality, we normalize it so that 
 \beqnn
 \mathbf{E}\bigg[ \int_0^\infty e^{-t} d \widehat{L}_t\bigg] =1. 
 \eeqnn
 Its right-inverse $\widehat{L}^{-1}$ is referred to as the \textsl{descending ladder time process}, and $\widehat{H}= \{ \widehat{H}_t :t\geq 0 \}$ with $\widehat{H}_t := -X_{\widehat{L}^{-1}_t}$ is called the \textsl{descending ladder height process}. 
 Lemma~2 in \cite[p.157]{Bertoin1996} states that $(\widehat{L}^{-1},\widehat{H})$ is a L\'evy process (possibly killed at an exponential rate), whose \textsl{bivariate exponent}, by \cite[p.166]{Bertoin1996}, is given by
 \beqnn
 \widehat\kappa(\lambda,u)
 =-\log \mathbf{E}\Big[e^{-\lambda \widehat{L}^{-1}_1-u\widehat{H}_1}\Big]
 \ar=\ar
 \exp\Big\{\int_0^\infty \frac{dt}{t}\int_{(-\infty,0]}
 (e^{-t}-e^{-\lambda t+ux})\, \mathbf{P}(X_t\in dx)\Big\},
 \quad \lambda,u\geq 0.
 \eeqnn
 In particular, both $\widehat L^{-1}$ and $\widehat H$ are (possibly killed) subordinators with drifts $\underline{\tt d},\underline{\tt a}\geq 0$.
 Since the state $0$ is assumed to be regular for $(0, \infty)$, for $X$, we have $\underline{\mathrm{d}}=0$.
 By (6.16) in \cite[p.170]{Kyprianou2014}, the bivariate exponent $\widehat\kappa(\lambda, u)$ also admits the  representation
 \beqlb\label{eqn.kappa}
 \widehat \kappa(\lambda,u) \ar=\ar  \underline{\tt a}\cdot u + \int_0^\infty \int_0^\infty (1- e^{-\lambda t-u x})\,\underline{n}(\zeta \in dt,-\epsilon_\zeta \in dx). 
 \eeqlb
 Denote by $\widehat{U}(dt,dy)$ the \textsl{potential measure} of $(\widehat{L}^{-1},\widehat{H})$. 
 It is a $\sigma$-finite measure on $\mathbb{R}^2_+$ defined by
 \beqlb\label{eqn.U}
 \widehat{U}(dt,dy):= \int_0^\infty \mathbf{P}\big( \widehat{L}^{-1}_s \in  dt,\widehat{H}_s\in dy \big)\,ds.
 \eeqlb
 In particular, its marginal cumulative function in space is also known as the \textsl{renewal function} associated with the ladder height process $\widehat{H}$, that is, 
 \beqlb\label{eqn.DefV}
 \widehat{V}(x) := \widehat{U}\big(\mathbb{R}_+,[0,x]\big)= \int_0^\infty \mathbf{P}\big( \widehat{H}_s \leq x \big)\,ds,\quad x\geq 0;
 \eeqlb
 see e.g.  \cite[Chapter VI]{Bertoin1996} and \cite[Chapter 6]{Kyprianou2014}. 
 Analogously, quantities associated with $\overline{X}$ and $\overline{X}-X$ are denoted without hats, e.g., $L^{-1}$, $H$, $\kappa(\lambda,u)$, $U(dt,dy)$ and $V(x)$,  and satisfy analogous properties.

 Applying integration by parts to \eqref{eqn.kappa} with $u=0$ gives that
 \beqnn
 \widehat \kappa(\lambda,0) =  \int_0^\infty (1- e^{-\lambda t})\,\underline{n}(\zeta \in dt)
 = \lambda \int_0^\infty  e^{-\lambda t}\underline{n}(\zeta > t)\, dt,\quad \lambda \geq 0.
 \eeqnn
 By Theorem 6.9 (ii) in \cite{Kyprianou2014},  we have that $\kappa(0,0) >0$ since $X$ drifts to  $-\infty$ and hence 
 \beqlb\label{eqn.501}
 \underline{n}(\zeta )= \int_0^\infty  \underline{n}(\zeta > t)dt=\lim_{\lambda\to 0}\frac{\widehat \kappa(\lambda,0)}{\lambda} =\frac{1}{ \kappa(0,0)} < \infty.
 \eeqlb
 Here the third equality follows from the identity $\widehat \kappa(\lambda,0) \cdot \kappa(\lambda,0)=\lambda$ for all $\lambda \geq 0$. 
 Moreover, by (4.2) in \cite{CuiWuXu2025} and (4) in \cite[p.166]{Bertoin1996} we have
 \beqlb\label{eqn.EH0}
 \mathbf{E}\big[ \widehat{H}_1 \big] = \frac{\partial \widehat{\kappa}(\lambda, u)}{\partial u}\Big|_{(\lambda, u)=(0,0)}=\frac{\beta}{\kappa(0,0)}=\beta \cdot \underline{n}(\zeta).
 \eeqlb
 By \eqref{eqn.kappa} and \eqref{eqn.U}, the function $\widehat{V}$ also satisfies that
 \beqnn
 \widehat{V}(x) = \overline{\tt d} +  \int_0^\infty  \overline{n}(  \epsilon_s\leq x, \zeta>s)\, ds ,\quad x\geq 0,
 \eeqnn
 with $\overline{\tt d}\geq 0$ being the drift of $L^{-1}$; see also (16) in \cite{DoneyRivero2013}. 
 Similarly, we also have for any $x\geq 0$, 
 \beqnn
 V(x) = \int_0^\infty \mathbf{P}\big( H_s \leq x \big)\,ds =   \int_0^\infty  \underline{n}(  \epsilon_s\leq x, \zeta>s)\, ds  
 \quad \mbox{and}\quad 
 V(\infty)=  \underline{n}(\zeta) <\infty. 
 \eeqnn
 Recall that $\tau_0^-$ is the first passage time of $X$ into $(-\infty,0)$.    For any $x>0$, Proposition 17 (ii) in \cite[p.172]{Bertoin1996} and Theorem 2.2  in \cite{DenisovShneer2013} give
 \beqlb  \label{Asytau} 
 \mathbf{E}_{x}[\tau_0^-] 
 = \underline{n}(\zeta ) \cdot \widehat{V}(x)
 \quad \mbox{and}\quad 
 \mathbf{P}_x\big(\tau_0^->t\big)  
 \sim  \underline{n}(\zeta ) \cdot \widehat{V}(x) \cdot \overline{\nu}(\beta t) \in \mathrm{RV}^\infty_{-\alpha} ,\quad \mbox{as }t\to\infty. 
 \eeqlb

 \subsection{Global auxiliary lemmas}
 % \label{Sec.AuxiliaryLemma}
 % \setcounter{equation}{0}

 This section provides four important auxiliary lemmas that will be frequently used in the proofs of our main theorems. 
 Lemma~\ref{Lemma.301} extends the range of uniformity in the asymptotic equivalence \eqref{eqn.30}.
 Lemma \ref{Lemma.203} is drawn from Corollary~1.4 in \cite{WuXu2026a}, which provides a unified treatment of uniform local large-deviation asymptotics for centered L\'evy processes with regularly varying jumps. Lemma~\ref{Lemma.3001} gives the almost sure and $L^1$ limits of $X_t/t$, $-X_t^-/t$, and $\underline X_t/t$. Lemma~\ref{Lemma.310} gives uniform integrability and the mean limit for $\underline X_t-\underline X_{t+\delta}$.

 \begin{lemma}\label{Lemma.301}
 	For any $\delta_0>0$, there exists some constant $C>0$ such that for all  $\lambda\in[0,1]$, 
 	\beqlb\label{eqn.31}
 	\limsup_{x\to\infty}
 	\sup_{\delta_0\leq\delta\leq \delta_0+ \lambda x}
 	\bigg| \frac{\nu\big(  x+ \varDelta_\delta\big)}{\alpha\cdot\overline{\nu}(x)/x\cdot \delta} -1\bigg|
 	\leq C\cdot \lambda.
 	\eeqlb  
 \end{lemma}
 \proof Without loss of generality, we prove the result for $\delta_0=1$; the general case follows by an identical argument. 
 For any $1\leq \delta\leq  1+ \lambda x$, we decompose the interval $x+\varDelta_\delta$ into unit subintervals:
 \beqnn
 \nu\big(  x+ \varDelta_\delta\big) 
 =
 \sum_{k=0}^{[\delta]-1}  \nu\big(x+k + \varDelta_1\big) 
 +
 \nu\big( x+[\delta]+ \varDelta_{\delta-[\delta]}\big) .
 \eeqnn
 By the triangle inequality, 
 \beqlb\label{eqn.202}
 \Big| \frac{\nu\big(  x+ \varDelta_\delta\big)}{\alpha\cdot\overline{\nu}(x)/x\cdot \delta} - 1 \Big|
 \ar=\ar \frac{1}{\delta} \cdot \bigg|\sum_{k=0}^{ [\delta] -1} \Big(\frac{\nu\big(x+k + \varDelta_1\big)}{\alpha\overline{\nu}(x)/x} -1\Big)  +  \Big(  \frac{\nu\big( x+[\delta] + \varDelta_{\delta-[\delta] }\big)}{\alpha\cdot\overline{\nu}(x)/x} - \big(\delta-[\delta] \big) \Big)\bigg|\cr
 \ar\leq\ar \frac{1}{\delta}\bigg( \sum_{k=0}^{[\delta] -1} \bigg|\frac{\nu\big(x+k + \varDelta_1\big)}{\alpha\cdot\overline{\nu}(x)/x} -1\bigg|   +   \bigg|  \frac{\nu\big( x+[\delta] + \varDelta_{\delta- [\delta] }\big)}{\alpha\overline{\nu}(x)/x} -\big(\delta- [\delta]  \big) \bigg| \bigg)\cr
 \ar \leq \ar  \sup_{k=0,\cdots, [\delta]-1} \bigg|\frac{\nu\big(x+k + \varDelta_1\big)}{\alpha\cdot \overline{\nu}(x)/x} -1\bigg| 
 + \bigg|  \frac{\nu\big( x+[\delta] + \varDelta_{\delta- [\delta] }\big)}{\alpha\cdot\overline{\nu}(x)/x} -\big(\delta- [\delta]  \big) \bigg|. \qquad
 \eeqlb
 Since $\overline{\nu}(x)/x$ is strictly decreasing on $(0,\infty)$, we have uniformly in $1\leq \delta \leq 1+\lambda x$ and $k=0,\cdots, [\delta]-1$, 
 \beqlb\label{eqn.201}
 \bigg|\frac{\nu\big(x+k + \varDelta_1\big)}{\alpha\cdot\overline{\nu}(x)/x} -1\bigg|  
 \ar\leq\ar \frac{\overline{\nu}(x+k) }{x+k} \frac{x}{\overline{\nu}(x)} \cdot \bigg|\frac{\nu\big(x+k + \varDelta_1\big)}{\alpha\cdot\overline{\nu}(x+k)/(x+k)} -1\bigg|   + \bigg|\frac{\overline{\nu}(x+k) }{x+k} \frac{x}{\overline{\nu}(x)} -1\bigg| \cr
 \ar\leq\ar  \frac{ \overline{\nu}(x+ \lambda x)}{ x+ \lambda x } \frac{x}{\overline{\nu}(x)}  \cdot \bigg|\frac{\nu\big(x+k + \varDelta_1\big)}{\alpha\cdot\overline{\nu}(x+k)/(x+k)} -1\bigg| + \bigg|\frac{ \overline{\nu}(x+\lambda x)}{ x+ \lambda x } \frac{x}{\overline{\nu}(x)}  -1\bigg| . \quad
 \eeqlb
 Firstly, since $\overline{\nu}(x) \in \mathrm{RV}^\infty_{-\alpha}$, we have, as $x\to\infty$,
 \beqnn
 \bigg|\frac{ \overline{\nu}(x+\lambda x)}{ x+\lambda x } \frac{x}{\overline{\nu}(x)} -1\bigg| \to      1-(1+\lambda)^{-\alpha-1} .
 \eeqnn
 Moreover, since the function $\frac{\nu( y+ \varDelta_r)}{ \alpha \overline{\nu}(y)/y}$ is non-decreasing in $r \in [0,1]$ and converges pointwise to $r$ as $y\to\infty$; see \eqref{eqn.30}, by P\'olya's theorem (see e.g.  (0.1) in \cite[p.1]{Resnick1987}) we have as $x\to\infty$, 
 \beqnn
 \sup_{k\geq 0} \bigg|\frac{\nu\big(x+k + \varDelta_1\big)}{\alpha\cdot \overline{\nu}(x+k)/(x+k)} -1\bigg| \leq   \sup_{y\geq x}\sup_{r\in[0,1]} \bigg|  \frac{\nu\big( y+ \varDelta_r\big)}{\alpha\cdot \overline{\nu}(y)/y}  -r \bigg|  \to 0.
 \eeqnn 
 Substituting these limits into \eqref{eqn.201}, we obtain, as $x\to\infty$,
 \beqnn
 \sup_{1\leq \delta \leq 1+\lambda x} \sup_{k=0,\cdots, [\delta]-1}\bigg|\frac{\nu\big(x+k + \varDelta_1\big)}{\alpha\cdot \overline{\nu}(x)/x} -1\bigg|   
 \to 1-(1+\lambda)^{-\alpha-1} \leq (1+\alpha)\cdot  \lambda, 
 \eeqnn
 for all $\lambda\in[0,1]$.
 Similarly, the same argument shows that, uniformly in $1\le \delta\le 1+\lambda x$, the last absolute value in \eqref{eqn.202} is bounded from above by $1-(1+\lambda)^{-\alpha-1}+o(1)$. 
 Combining these estimates with \eqref{eqn.202} yields the desired bound \eqref{eqn.31}.  
 \qed

 \begin{corollary}\label{Corollary.202}
 	There exist two constants $C,x_0> 0$ such that for all $x\geq x_0$ and  all  $\delta \geq0$,
 	\beqlb\label{eqn.212}
 	\sup_{y\geq 0}\nu \big( x+y+\varDelta_\delta \big)
 	\leq
 	C\cdot \frac{\overline{\nu}(x)}{x}  \cdot 
 	(1+ \delta).
 	\eeqlb
 \end{corollary}
 \proof By Lemma~\ref{Lemma.301} with $\lambda=1$ and $\delta_0=1$, there exists  a constant $x_0>0$ such that for all $x\geq x_0$,
 \beqnn
 \sup_{y\geq 0}
 \sup_{1\leq\delta\leq 1+x+y}
 \Big| \frac{\nu\big( x+ y+ \varDelta_\delta\big)}{\alpha\cdot \overline{\nu}(x+y)/(x+y)\cdot \delta} -1\Big|
 \leq 2+\alpha .
 \eeqnn
 Equivalently, for all $x\geq x_0$, $y\geq 0$, and $1\le \delta\le 1+x+y$,
 \beqnn
 \nu(x+y+\varDelta_\delta) \leq \alpha(3+\alpha)\cdot \frac{\overline{\nu}(x+y)}{x+y}\cdot \delta.
 \eeqnn
 Since $z\mapsto \overline{\nu}(z)/z$ is decreasing on $(0,\infty)$, we have, for all $x\ge x_0$, $y\ge 0$, and $1\le \delta\le 1+x$,
 \beqnn
 \sup_{y\ge 0} \nu(x+y+\varDelta_\delta)
 \le
 \alpha(3+\alpha)\,\frac{\overline{\nu}(x)}{x}\,\delta.
 \eeqnn
 For $\delta\in[0,1]$, using the monotonicity of $\delta\mapsto \nu(x+y+\varDelta_\delta)$, we obtain
 \beqnn
 \sup_{y\geq 0} \sup_{\delta \in [0,1]}  \nu\big(x+y+ \varDelta_\delta\big) \leq    \sup_{y\geq 0} \nu\big( x+ y+ \varDelta_1\big) \leq  \alpha(3+\alpha) \cdot \frac{\overline{\nu}(x)}{x}     . 
 \eeqnn
 It remains to treat the case $\delta\ge 1+x$. Since $\nu(x+y+\varDelta_\delta)\le \overline{\nu}(x+y)$ and $\overline{\nu}$ is decreasing, we have, uniformly in $\delta\ge 1+x$,
 \beqnn
 \sup_{y\geq 0}\nu\big( x+ y+ \varDelta_\delta\big) 
 \leq   \sup_{y\geq 0} \overline{\nu}(x+y) \leq \overline{\nu}(x)   = \frac{\overline{\nu}(x) }{x}\cdot x \leq  \frac{\overline{\nu}(x) }{x}\cdot  \delta.
 \eeqnn
 Combining these estimates for the three regimes $\delta\in[0,1]$, $\delta\in[1,1+x]$, and $\delta\ge 1+x$, we can obtain the desired bound \eqref{eqn.212} immediately. 
 \qed

 \begin{lemma}  \label{Lemma.203}
 	For any $\lambda,\delta_0>0$ and any non-negative function $\vartheta$ on $\mathbb{R}_+$ with $\vartheta_x/x\to0$ as $x\to\infty$, we have as $t\to\infty$, 
 	\beqnn
 	\sup_{0<s\leq  t} \sup_{x\geq \lambda t} \sup_{|y|\leq \vartheta_x} \sup_{\delta\in[\delta_0,\infty]} 
 	\bigg| \frac{\mathbf{P}\big(X_s + \beta s \in x -y+ \varDelta_\delta \big)}{s\cdot  \nu\big(x + \varDelta_\delta \big)   } -1\bigg|    
 	\to 0. 
 	\eeqnn
 \end{lemma}

 \begin{proposition}\label{Prop.206}
 	There exists a constant $C>0$ such that for all $t\geq 0$, 
 	\beqlb\label{eqn.206}
 	\mathbf{E}\bigg[ \sup_{s\in[0,t]} \big|X_s\big| \bigg] \leq C\cdot \big( \sqrt{t} +t  \big).
 	\eeqlb
 	Moreover, for any $t_0>0$, the family $\{|X_t/t| \}_{t\geq t_0}$ is uniformly integrable.
 \end{proposition}
 \proof
 The uniform upper bound follows from Proposition~2.3 in
 \cite{CuiWuXu2026b}. It remains to prove the uniform integrability of
 $\{|X_t/t|\}_{t\geq t_0}$. By the L\'evy--It\^o decomposition, we write
 $X_t=-\beta t+Y_t+M_t$, where
 \beqnn
 Y_t:=\sigma B_t+\int_0^t\int_{|y|\leq1}
 y\,\widetilde N(ds,dy)
 \quad\mbox{and}\quad 
 M_t:=\int_0^t\int_{|y|>1}
 y\,\widetilde N(ds,dy),\quad t\geq0.
 \eeqnn
 By the triangle inequality,
 \beqnn
 \mathbf{E}\big[\big|X_t/t+\beta\big|\big]
 \ar\leq\ar
 \mathbf{E}\big[\big|Y_t/t\big|\big]
 +\mathbf{E}\big[\big|M_t/t\big|\big].
 \eeqnn
 By Jensen's inequality, the It\^o isometry, and
 \eqref{eqn.LevyMeasure}, we have
 $\mathbf{E}[|Y_t/t|]\leq
 t^{-1/2}\big(\sigma^2+\int_{|y|\leq1}y^2\nu(dy)\big)^{1/2}
 \leq C/\sqrt{t}\to0$ as $t\to\infty$.
 Moreover, for any $K>1$, we also have
 \beqnn
 \mathbf{E}\big[\big|M_t/t\big|\big]
 \ar\leq\ar
 \mathbf{E}\Big[
 \frac{1}{t}\Big|
 \int_0^t\int_{1<|y|\leq K}y\widetilde{N}(ds,dy)
 \Big|\Big]
 +
 \mathbf{E}\Big[
 \frac{1}{t}\Big|
 \int_0^t\int_{|y|>K}y\widetilde{N}(ds,dy)
 \Big|\Big]
 \cr
 \ar\leq\ar
 \bigg(
 \frac{1}{t^2}\mathbf{E}\Big[
 \Big|
 \int_0^t\int_{1<|y|\leq K}y\widetilde{N}(ds,dy)
 \Big|^2\Big]
 \bigg)^{1/2}
 +
 \mathbf{E}\Big[
 \frac{1}{t}\Big|
 \int_0^t\int_{|y|>K}y\widetilde{N}(ds,dy)
 \Big|\Big]
 \cr
 \ar\leq\ar
 \Big(\frac{1}{t}\int_{1<|y|\leq K}y^2\nu(dy)\Big)^{1/2}
 +2\int_{|y|>K}|y|\nu(dy),
 \eeqnn
 which goes to $0$ as $t\to\infty$ and then $K\to\infty$. Combining these estimates together yields that $ \mathbf{E}\big[\big|X_t/t+\beta\big|\big]  \to 0$ as $t\to\infty$ and then $\{ |X_t/t| \}_{t\geq t_0}$ is uniformly integrable. 
 \qed

 \begin{proposition}\label{Prop.2011}
 	For any $\delta>0$, we have  $\sup_{y\in\mathbb R} \mathbf{P}\big( \underline{X}_t+\beta t\in y+\varDelta_\delta \big)\to 0$ as $t\to\infty$.
 \end{proposition}
 \proof Firstly, we have $\underline X_t+\beta t =  X_t+\beta t - (X_t-\underline X_t)$ and then for any $K>0$, 
 \beqlb\label{eqn.5020}
 \mathbf{P}\big( \underline{X}_t+\beta t\in y+\varDelta_\delta \big)
 \ar\leq\ar   \mathbf{P}\big(   X_t-\underline X_t>K \big) +  \mathbf{P}\big( \underline{X}_t+\beta t\in y+\varDelta_\delta,\, 0\leq X_t-\underline{X}_t\leq K \big).  
 \eeqlb
 The duality lemma induces that  $ X_t-\underline{X}_t\stackrel{d}{=}\overline{X}_t$ and then 
 \beqnn
 \mathbf P\big( X_t-\underline X_t>K \big)   = \mathbf P\big( \overline{X}_t>K \big)  \leq \mathbf P\big( \overline{X}_\infty>K \big)  \to 0 ,
 \eeqnn
 as $K\to\infty$, since   $\overline X_\infty<\infty$ a.s. 
 The last probability  in \eqref{eqn.5020} can be bounded by $\mathbf P\big(  X_t+\beta t \in y+\varDelta_{K+\delta} \big)$ and hence it remains to prove that for any fixed $K>0$,
 \beqlb\label{eqn.5021}
 \limsup_{t\to\infty} \sup_{y\in\mathbb R}   \mathbf P\big(  X_t+\beta t \in y+\varDelta_{K+\delta} \big) =0.
 \eeqlb
 By the L\'evy--It\^o decomposition,  we can decompose $X_t+\beta t$ into  two independent processes:
 \beqnn
 M_t:=- \int_1^\infty y\nu(dy)\cdot t+  \sigma\cdot B(t)  + \int_0^t \int_{-\infty}^1 y \widetilde{N}(ds,dy)
 \quad \mbox{and}\quad 
 Z_t:= \int_0^t \int_1^\infty y N(ds,dy).
 \eeqnn
 Hence $\mathbf{P}\big(  X_t+\beta t \in y+\varDelta_{K+\delta} \big) 
 = \mathbf{P}\big(  Z_t \in -M_t+ y+\varDelta_{K+\delta} \big) $ and 
 \beqnn
 \sup_{y\in \mathbb{R}}\mathbf{P}\big(  X_t+\beta t \in y+\varDelta_{K+\delta} \big) 
 \ar\leq\ar  \mathbf{E}\bigg[ \sup_{y\in \mathbb{R}}\mathbf{P}\Big(  Z_t \in -M_t+ y+\varDelta_{K+\delta}\,\Big|\, M_t \Big)  \bigg]
 = \sup_{y\in \mathbb{R}}\mathbf{P}\big(  Z_t \in   y+\varDelta_{K+\delta} \big) . 
 \eeqnn
 Note that $Z \overset{\rm d}=\{S_{N_t}:t\geq 0\}$ with $S$ being a random walk with step distribution $\mathbf{1}_{\{y>1\}}\cdot \nu(dy)/\overline{\nu}(1)$ and $N_t$ being a Poisson process with rate $\overline{\nu}(1)$. 
 Then, we have  for any integer $m\geq 1$ 
 \beqnn
 \limsup_{t\to\infty} \sup_{y\in \mathbb{R}}\mathbf{P}\big(  Z_t \in   y+\varDelta_{K+\delta} \big)  
 \ar\leq \ar  \limsup_{t\to\infty} \mathbf{P}\big(  N_t<m \big) + \limsup_{t\to\infty}  \sup_{y\in \mathbb{R}} \mathbf{P}\big( S_{N_t}\in y+\varDelta_{K+\delta}, N_t\geq m \big) .
 \eeqnn
 Obviously, the first limit on the right-hand side equals $0$. 
 Moreover, by Theorem~9 in \cite[p.49]{Petrov2012}, 
 \beqnn 
 \limsup_{t\to\infty}  \sup_{y\in \mathbb{R}} \mathbf{P}\big( S_{N_t}\in y+\varDelta_{K+\delta}, N_t\geq m \big)
 \leq   \sup_{n\geq m} \sup_{y\in\mathbb{R}} \mathbf{P}\big( S_n\in  y+\varDelta_{K+\delta} \big) ,
 \eeqnn 
 which goes to $0$ as $m\to\infty$. Then the limit \eqref{eqn.5021} holds. 
 \qed 
 
 \begin{lemma} \label{Lemma.3001}
 	As $t\to\infty$, each of the three random variables $X_t/t$, $-X_t^-/t$, and $\underline{X}_t/t$ converges to $-\beta$ both almost surely and in $L^1$. 
 \end{lemma}
 \proof It is clear that $\underline{X}_t \le -X_t^- \le X_t$. Hence it suffices to prove the desired convergence for $X_t/t$ and $\underline{X}_t/t$. 
 The almost sure convergence of $X_t/t$ to $-\beta$ follows directly from the law of large numbers for L\'evy processes; the $L^1$-convergence then follows from the uniform integrability of the family $\{|X_t/t|\}_{t\ge t_0}$, established in Proposition~\ref{Prop.206}. 
 We now prove $\underline{X}_t/t\overset{\rm a.s.}\to -\beta$. 
 Since $\underline X_t\leq X_t$, we first have  
 \beqlb \label{eqn.30111}
 \underline{X}_t/t 
 \leq X_t/t \overset{\rm a.s.}\to -\beta. 
 \eeqlb
 For the reverse inequality, fix $\varepsilon>0$. 
 The almost sure convergence $X_t/t \to -\beta$ implies that, for some finite random time $t_0>0$, we have $X_t \geq -(\beta+\varepsilon)t$ for all $t\geq t_0$. Consequently, for $t\geq t_0$, 
 \beqnn
 \underline X_t \ar=\ar  
 \underline X_{t_0} \wedge \inf_{t_0\leq s\leq t}X_s  \geq 
 \min\Big\{  \underline X_{t_0},\,
 -(\beta+\varepsilon)t
 \Big\} .
 \eeqnn
 Dividing by $t$ and taking the limit inferior gives
 \beqnn
 \liminf_{t\to\infty} \underline{X}_t/t  \geq -\beta -\varepsilon.
 \eeqnn
 Since $\varepsilon>0$ was arbitrary, combining this with \eqref{eqn.30111} yields  $\underline{X}_t/t \overset{\rm a.s.}\to -\beta$. 
 It remains to prove the $L^1$-convergence.  
 By \eqref{eqn.206} and the triangle inequality, 
 \beqnn
 \mathbf E\Big[ \big|  \underline X_t/t+\beta \big| \Big]
 \ar\leq\ar
 \mathbf E\Big[ \big| X_t/t+\beta \big| \Big]
 +
 \frac{\mathbf E[X_t-\underline X_t]}{t} 
 \eeqnn
 The first term on the right-hand side tends to $0$ as $t\to\infty$ by the $L^1$-convergence of $X_t/t$. For the second term, by the duality lemma and Spitzer's identity for Lévy processes (see, e.g., Theorem 5 in \cite[ p.~160]{Bertoin1996}), we have 
 \beqnn
 \frac{\mathbf E[X_t-\underline X_t]}{t} 
 =\frac{\mathbf E[\overline X_t]}{t} =\frac{1}{t}\int_0^t \frac{\mathbf E[X_s^+]}{s}\,ds,
 \eeqnn
 which also goes to $0$ as $t\to\infty$ by the continuous Ces\`aro theorem and $\mathbf E[X_t^+]/t \leq \mathbf E[X_t/t+\beta]\to 0$. Consequently, we have $ \mathbf{E}\big[ |  \underline X_t/t+\beta | \big]\to0$ as $t\to\infty$.  This completes the proof. 
 \qed  
 
 \begin{lemma}\label{Lemma.310}
 	For any $\delta>0$, the family $\big\{  \underline{X}_t-\underline{X}_{t+\delta}\big\}_{t\geq 0}$   is uniformly integrable and also $\mathbf{E}\big[ \underline{X}_t -\underline{X}_{t+\delta}\big] \to\beta\delta$ as $t\to\infty$. 
 \end{lemma}
 \proof The independent and stationary increments of $X$ yield that
 \beqnn
 0 \leq
 \underline X_t-\underline X_{t+\delta}
 \leq 
 -\inf_{0\leq s\leq\delta}
 \big(
 X_{t+s}-X_t
 \big) \overset{\rm d}= -\underline{X}_\delta.
 \eeqnn
 Hence, the uniform integrability of $\big\{  \underline{X}_t-\underline{X}_{t+\delta}\big\}_{t\geq 0}$ follows from \eqref{eqn.206}. 
 By   Spitzer's identity for L\'evy process; see Theorem~5 in \cite[p.160]{Bertoin1996}, we  have for all $t\geq 0$,
 \beqnn
 \mathbf{E}\big[ \underline{X}_t \big] = -\int_0^t \frac{\mathbf{E}[X^-_s]}{s}\,ds .
 \eeqnn
 This induces that  
 \beqnn
 \Big|   \mathbf{E}\big[ \underline{X}_t -\underline{X}_{t+\delta}\big] -\beta\delta \Big| 
 \leq   \int_{t}^{t+\delta} \Big| \frac{\mathbf E[X_s^-]}{s}-\beta\Big|\,ds 
 \leq \sup_{s\geq t} \Big| \frac{\mathbf E[X_s^-]}{s}-\beta \Big| \cdot \delta ,
 \eeqnn 
 which goes to $0$ as $t\to\infty$ by Lemma~\ref{Lemma.3001}.  
 \qed

 \section{Proof of Theorem \ref{MainThm01}} 
 \label{Sec.ProofThm01}
 \setcounter{equation}{0} 
 
 This section is devoted to the proof of the local limit theorem for the lifetime $\zeta$ stated in Theorem~\ref{MainThm01}.  
 For any $t,\delta >0$, we have 
 $\underline{n}\big(\zeta\in t+\varDelta_\delta\big)=  \underline{n}(\zeta >t)- \underline{n}(\zeta >t+\delta)$. 
 In view of \eqref{eqn.209}, this difference admits the representation
 \beqlb \label{eqn.4011}
 \underline{n}\big(\zeta\in t+\varDelta_\delta\big)
 %\ar=\ar  \underline{n}(\zeta >t)    - \underline{n}(\zeta >t+\delta)
 = \frac{1}{t}\int_0^t \mathbf{P}(X_{t-s}>0) \,\underline{n}(\zeta>s)\, ds
 -\frac{1}{t+\delta}\int_0^{t+\delta} \mathbf{P}(X_{t+\delta-s}>0) \,\underline{n}(\zeta>s)\, ds . 
 \eeqlb
 Thus, the long-time asymptotic behavior of
 $\underline n\big(\zeta\in t+\varDelta_\delta\big)$
 is closely related to that of the probability difference 
 \beqlb\label{eqn.402}
 \mathbf{P}(X_{t}>0)-\mathbf{P}(X_{t+\delta}>0)
 \ar=\ar \mathbf{P}(X_{t}>0, X_{t+\delta}\leq 0) - \mathbf{P}(X_{t}\leq 0, X_{t+\delta}>0).
 \eeqlb
 The two probabilities on the right-hand side describe the two possible changes of sign of $X$ over the time interval $(t,t+\delta]$: the first corresponds to a transition from the positive half-line to the nonpositive half-line, whereas the second corresponds to a transition from the nonpositive half-line to the positive half-line. Accordingly, the analysis of these two crossing probabilities will play a central role in deriving the desired local asymptotics for the lifetime $\zeta$.
 
 \begin{proposition}\label{Prop.301}
 	For any  $\delta>0$, we have as $t\to\infty$,
 	\beqlb\label{eqn.301}
 	\frac{\mathbf{P}\big(X_t>0,X_{t+\delta}\leq 0\big)}{\alpha/\beta\cdot\overline{\nu}(\beta t)\cdot\mathbf{E}\big[X_\delta^-\big]}\to 1.
 	\eeqlb
 \end{proposition}
 
 \proof
 Fix $\varepsilon>0$. By the Markov property and the stationary independent increments of $X$,
 \beqnn
 \frac{\mathbf{P}\big(X_t>0,X_{t+\delta}\leq 0\big)}
 {\alpha/\beta\cdot\overline{\nu}(\beta t)\cdot\mathbf{E}\big[X_\delta^-\big]}
 \ar=\ar
 \int_{0}^{\infty}
 \frac{\mathbf{P}(X_\delta\leq-y)}
 {\mathbf{E}\big[X_\delta^-\big]}
 \cdot
 \frac{\mathbf{P}(X_t\in dy)}
 {\alpha/\beta\cdot\overline{\nu}(\beta t)}
 \cr\cr
 \ar=\ar
 \sum_{k=0}^\infty
 \int_{k\varepsilon}^{(k+1)\varepsilon}
 \frac{\mathbf{P}(X_\delta\leq-y)}
 {\mathbf{E}\big[X_\delta^-\big]}
 \cdot
 \frac{\mathbf{P}(X_t\in dy)}
 {\alpha/\beta\cdot\overline{\nu}(\beta t)}.
 \eeqnn
 Since $y\mapsto\mathbf P(X_\delta\leq-y)$ is non-increasing, the last expression is bounded from above and below, respectively, by
 \beqlb\label{eqn.3011}
 \sum_{k=0}^\infty
 \frac{\varepsilon\cdot\mathbf{P}(X_\delta\leq-k\varepsilon)}
 {\mathbf{E}\big[X_\delta^-\big]}
 \cdot
 \frac{\mathbf{P}\big(X_t\in k\varepsilon+\varDelta_\varepsilon\big)}
 {\alpha/\beta\cdot\overline{\nu}(\beta t)\cdot\varepsilon}
 \quad\mbox{and}\quad
 \sum_{k=0}^\infty
 \frac{\varepsilon\cdot\mathbf{P}(X_\delta\leq-(k+1)\varepsilon)}
 {\mathbf{E}\big[X_\delta^-\big]}
 \cdot
 \frac{\mathbf{P}\big(X_t\in k\varepsilon+\varDelta_\varepsilon\big)}
 {\alpha/\beta\cdot\overline{\nu}(\beta t)\cdot\varepsilon}.
 \eeqlb
 Using Lemma~\ref{Lemma.203} with $\lambda=\beta$ and \eqref{eqn.30} yields that as $t\to\infty$,
 \beqnn
 \sup_{k\geq 0}
 \bigg|
 \frac{\mathbf{P}\big(X_t\in k\varepsilon+\varDelta_\varepsilon\big)}
 {t\cdot\alpha\cdot\overline{\nu}(\beta t+k\varepsilon)/(\beta t+k\varepsilon)\cdot\varepsilon}
 -1
 \bigg|
 \ar\leq\ar
 \sup_{x\geq 0}
 \bigg|
 \frac{\mathbf{P}\big(X_t+\beta t\in\beta t+x+\varDelta_\varepsilon\big)}
 {t\cdot\alpha\cdot\overline{\nu}(\beta t+x)/(\beta t+x)\cdot\varepsilon}
 -1
 \bigg|
 \to 0.
 \eeqnn
 Consequently, there exists a non-negative and non-increasing function $\rho_\varepsilon$ on $\mathbb{R}_+$ such that $\rho_\varepsilon(t)\to 0$ as $t\to\infty$, and, for all $t>0$, 
 \beqlb\label{eqn.3012}
 1-\rho_\varepsilon(t)\leq\inf_{k\geq0}\frac{\mathbf{P}\big(X_t\in k\varepsilon+\varDelta_\varepsilon\big)}
 {t\cdot\alpha\cdot\overline{\nu}(\beta t+k\varepsilon)/(\beta t+k\varepsilon)\cdot\varepsilon}\leq
 \sup_{k\geq0}
 \frac{\mathbf{P}\big(X_t\in k\varepsilon+\varDelta_\varepsilon\big)}
 {t\cdot\alpha\cdot\overline{\nu}(\beta t+k\varepsilon)/(\beta t+k\varepsilon)\cdot\varepsilon}
 \leq1+\rho_\varepsilon(t).
 \eeqlb
 Applying the upper bound to the first term in \eqref{eqn.3011} and then using the monotonicity of $\overline{\nu}(x)$, we obtain that uniformly in  $t>0$,
 \beqnn
 \frac{\mathbf{P}\big(X_t>0,X_{t+\delta}\leq0\big)}
 {\alpha/\beta\cdot\overline{\nu}(\beta t)\cdot\mathbf{E}\big[X_\delta^-\big]}
 \ar\leq\ar
 \big(1+\rho_\varepsilon(t)\big)
 \sum_{k=0}^\infty
 \frac{\varepsilon\cdot\mathbf{P}(X_\delta\leq-k\varepsilon)}
 {\mathbf{E}\big[X_\delta^-\big]}
 \cdot
 \frac{\beta t}{\beta t+k\varepsilon}
 \cdot
 \frac{\overline{\nu}(\beta t+k\varepsilon)}
 {\overline{\nu}(\beta t)}
 \cr\cr
 \ar\leq\ar
 \frac{1+\rho_\varepsilon(t)}
 {\mathbf{E}\big[X_\delta^-\big]}
 \sum_{k=0}^\infty
 \varepsilon\cdot\mathbf{P}(X_\delta\leq-k\varepsilon)
 \cr\cr
 \ar\leq\ar
 \frac{1+\rho_\varepsilon(t)}
 {\mathbf{E}\big[X_\delta^-\big]}
 \bigg(
 \varepsilon\cdot\mathbf{P}(X_\delta\leq0)
 +\int_0^\infty\mathbf{P}(X_\delta\leq-y)\,dy
 \bigg)
 \cr\cr
 \ar\leq\ar
 \big(1+\rho_\varepsilon(t)\big)
 \cdot
 \Big(
 \frac{\varepsilon}{\mathbf{E}\big[X_\delta^-\big]}+1
 \Big)
 \cr\cr
 \ar\leq\ar
 \big(1+\rho_\varepsilon(t)\big)
 \cdot
 \bigg(
 \frac{\varepsilon}{\beta\delta}+1
 \bigg).
 \eeqnn
 Here the last inequality follows from
 $\mathbf{E}\big[X_\delta^-\big]
 =\beta\delta+\mathbf{E}\big[X_\delta^+\big]
 \geq\beta\delta$. 
 For the lower bound, applying the lower bound in \eqref{eqn.3012} and retaining only the terms with $0\leq k\leq[t]$ yields, for all sufficiently large $t$,
 \beqnn
 \frac{\mathbf{P}\big(X_t>0,X_{t+\delta}\leq0\big)}
 {\alpha/\beta\cdot\overline{\nu}(\beta t)\cdot\mathbf{E}\big[X_\delta^-\big]}
 \ar\geq\ar
 \big(1-\rho_\varepsilon(t)\big)
 \sum_{k=0}^{[t]}
 \frac{\varepsilon\cdot\mathbf{P}(X_\delta\leq-(k+1)\varepsilon)}
 {\mathbf{E}\big[X_\delta^-\big]}
 \cdot
 \frac{\beta t}{\beta t+k\varepsilon}
 \cdot
 \frac{\overline{\nu}(\beta t+k\varepsilon)}
 {\overline{\nu}(\beta t)}
 \cr\cr
 \ar\geq\ar
 \big(1-\rho_\varepsilon(t)\big)
 \cdot
 \frac{\beta t}{\beta t+[t]\varepsilon}
 \cdot
 \frac{\overline{\nu}(\beta t+[t]\varepsilon)}
 {\overline{\nu}(\beta t)}
 \cdot
 \sum_{k=0}^{[t]}
 \frac{\varepsilon\cdot\mathbf{P}(X_\delta\leq-(k+1)\varepsilon)}
 {\mathbf{E}\big[X_\delta^-\big]}.
 \eeqnn
 Since
 $\mathbf{P}(X_\delta\leq-(k+1)\varepsilon)
 =\mathbf{P}(X_\delta^-\geq(k+1)\varepsilon)$,
 similarly as before we have
 \beqnn
 \sum_{k=0}^{[t]}
 \frac{\varepsilon\cdot\mathbf{P}(X_\delta\leq-(k+1)\varepsilon)}
 {\mathbf{E}\big[X_\delta^-\big]}
 \ar\geq\ar
 1-\frac{\varepsilon}{\beta\delta}
 -\sum_{k=[t]+1}^\infty
 \frac{\varepsilon\cdot\mathbf{P}\big(X_\delta^-\geq(k+1)\varepsilon\big)}
 {\beta\delta}
 \cr\cr
 \ar\geq\ar
 1-\frac{\varepsilon}{\beta\delta}
 -\frac{1}{\beta\delta}
 \int_{([t]+1)\varepsilon}^\infty
 \mathbf{P}(X_\delta^-\geq y)\,dy
 \cr\cr
 \ar\geq\ar
 1-\frac{\varepsilon}{\beta\delta}
 -\frac{1}{\beta\delta}
 \mathbf{E}\Big[
 X_\delta^-;\,
 X_\delta^-\geq([t]+1)\varepsilon
 \Big].
 \eeqnn
 Since $\mathbf{E}\big[X_\delta^-\big]<\infty$, the last expectation goes to $0$ as $t\to\infty$. Combining these two estimates together and then using the fact that
 $\overline{\nu}\in\mathrm{RV}_{-\alpha}^\infty$,
 \beqnn
 \big(1+\varepsilon/\beta\big)^{-\alpha-1}
 \cdot
 \bigg(1-\frac{\varepsilon}{\beta\delta}\bigg)
 \ar\leq\ar
 \liminf_{t\to\infty}
 \frac{\mathbf{P}\big(X_t>0,X_{t+\delta}\leq0\big)}
 {\alpha/\beta\cdot\overline{\nu}(\beta t)\cdot\mathbf{E}\big[X_\delta^-\big]}
 \cr\cr
 \ar\leq\ar
 \limsup_{t\to\infty}
 \frac{\mathbf{P}\big(X_t>0,X_{t+\delta}\leq0\big)}
 {\alpha/\beta\cdot\overline{\nu}(\beta t)\cdot\mathbf{E}\big[X_\delta^-\big]}
 \leq1+\frac{\varepsilon}{\beta\delta}.
 \eeqnn
 The desired limit \eqref{eqn.301} follows directly by letting $\varepsilon\to0+$.
 \qed
 
 \begin{proposition}\label{Prop.302}
 	For any $\delta>0$, we have as $t\to\infty$,
 	\beqlb\label{eqn.302}
 	\frac{\mathbf{P}\big(X_t\leq0,X_{t+\delta}>0\big)}{\alpha/\beta\cdot\overline{\nu}(\beta t)\cdot\big(\mathbf{E}\big[X_\delta^+\big]+\beta\delta/\alpha\big)}\to1.
 	\eeqlb
 \end{proposition}
 
 \proof
 For any fixed $\varepsilon\in(0,\beta/2)$, by the law of total probability we decompose $\mathbf{P}(X_t\leq0,X_{t+\delta}>0)$ as
 \beqlb\label{eqn.3021}
 \mathbf{P}\big(X_t\leq(\varepsilon-\beta)t,X_{t+\delta}>0\big)
 +\mathbf{P}\big((\varepsilon-\beta)t<X_t\leq-\varepsilon t,X_{t+\delta}>0\big)
 +\mathbf{P}\big(-\varepsilon t<X_t\leq0,X_{t+\delta}>0\big).
 \eeqlb
 We now analyze the asymptotic behaviour of the three probabilities one-by-one.
 
 \medskip
 
 \textbf{(i)} For the first probability in \eqref{eqn.3021}, using the Markov property, the stationarity of increments, and spatial homogeneity, we obtain
 \beqnn
 \mathbf{P}\big(X_t\leq (\varepsilon-\beta)t,X_{t+\delta}>0\big)
 \ar=\ar \int_{-\infty}^{(\varepsilon-\beta)t} \mathbf{P}_y\big( X_{\delta}>0\big)\, \mathbf{P}\big( X_t\in dy \big)\cr
 \ar=\ar \int_{-\infty}^{(\varepsilon-\beta)t} \mathbf{P}\big( X_{\delta}>-y\big)\, \mathbf{P}\big( X_t\in dy \big).
 \eeqnn
 Since $y\mapsto\mathbf{P}(X_\delta>-y)$ is non-decreasing, we have
 \beqnn
 \mathbf{P}\big( X_{\delta}>(\beta+\varepsilon)t\big) \cdot \mathbf{P}\big( |X_t+\beta t|\leq \varepsilon t \big) 
 \leq  \mathbf{P}\big(X_t\leq (\varepsilon-\beta)t,X_{t+\delta}>0\big)
 \leq \mathbf{P}\big( X_{\delta}>(\beta-\varepsilon)t\big).
 \eeqnn 
 For the lower bound, Lemma~\ref{Lemma.3001} induces that
 $\mathbf{P}\big(|X_t+\beta t|\leq\varepsilon t\big)\to1$
 as $t\to\infty$.
 Moreover, using Lemma~\ref{Lemma.203} with
 $\lambda=\beta-\varepsilon$ and $\varDelta_\infty=(0,\infty)$ and $\overline{\nu}\in\mathrm{RV}_{-\alpha}^\infty$, we have as $t\to\infty$,
 \beqnn
 \frac{\mathbf{P}\big(X_\delta>(\beta+\varepsilon)t\big)}
 {\delta\cdot\overline{\nu}(\beta t)}
 \ar=\ar
 \frac{\mathbf{P}\big(X_\delta+\beta\delta>
 	\beta\delta+(\beta+\varepsilon)t\big)}
 {\delta\cdot\overline{\nu}\big(\beta\delta+(\beta+\varepsilon)t\big)}
 \cdot
 \frac{\overline{\nu}\big(\beta\delta+(\beta+\varepsilon)t\big)}
 {\overline{\nu}(\beta t)}
 \to
 \bigg(1+\frac{\varepsilon}{\beta}\bigg)^{-\alpha}.
 \eeqnn
 For the upper bound, an analogous argument gives
 \beqlb\label{eqn.401}
 \lim_{t\to\infty}
 \frac{\mathbf{P}\big(X_\delta>(\beta-\varepsilon)t\big)}
 {\delta\cdot\overline{\nu}(\beta t)}
 =
 \lim_{t\to\infty}
 \frac{\mathbf{P}\big(X_\delta+\beta\delta>
 	(\beta-\varepsilon)t+\beta\delta\big)}
 {\delta\cdot\overline{\nu}\big((\beta-\varepsilon)t+\beta\delta\big)}
 \cdot
 \frac{\overline{\nu}\big((\beta-\varepsilon)t+\beta\delta\big)}
 {\overline{\nu}(\beta t)}
 =\bigg(1-\frac{\varepsilon}{\beta}\bigg)^{-\alpha}.
 \eeqlb
 Combining these estimates, we obtain
 \beqnn
 \bigg(1+\frac{\varepsilon}{\beta}\bigg)^{-\alpha}
 \ar\leq\ar
 \liminf_{t\to\infty}
 \frac{\mathbf{P}\big(X_t\leq(\varepsilon-\beta)t,X_{t+\delta}>0\big)}
 {\delta\cdot\overline{\nu}(\beta t)}
 \cr\cr
 \ar\leq\ar
 \limsup_{t\to\infty}
 \frac{\mathbf{P}\big(X_t\leq(\varepsilon-\beta)t,X_{t+\delta}>0\big)}
 {\delta\cdot\overline{\nu}(\beta t)}
 \leq
 \bigg(1-\frac{\varepsilon}{\beta}\bigg)^{-\alpha}.
 \eeqnn
 Furthermore, by the arbitrariness of $\varepsilon$ we finally obtain
 \beqlb\label{eqn.3022}
 \lim_{\varepsilon\to0+}\limsup_{t\to\infty}
 \bigg|
 \frac{\mathbf{P}\big(X_t\leq(\varepsilon-\beta)t,X_{t+\delta}>0\big)}
 {\delta\cdot\overline{\nu}(\beta t)}
 -1
 \bigg|=0.
 \eeqlb
 
 \medskip
 
 \textbf{(ii)} For the second probability in \eqref{eqn.3021}, using the Markov property and stationary increments gives
 \beqnn
 \frac{\mathbf{P}\big((\varepsilon-\beta)t<X_t\leq-\varepsilon t,
 	X_{t+\delta}>0\big)}
 {\delta\cdot\overline{\nu}(\beta t)}
 \ar\leq\ar
 \mathbf{P}\big(X_t>(\varepsilon-\beta)t\big)
 \cdot
 \frac{\mathbf{P}\big(X_\delta>\varepsilon t\big)}
 {\delta\cdot\overline{\nu}(\beta t)}.
 \eeqnn
 Arguing as in \eqref{eqn.401}, we have
 \beqnn
 \lim_{t\to\infty}
 \frac{\mathbf{P}\big(X_\delta>\varepsilon t\big)}
 {\delta\cdot\overline{\nu}(\beta t)}=
 \lim_{t\to\infty}
 \frac{\mathbf{P}\big(X_\delta+\beta\delta>
 	\varepsilon t+\beta\delta\big)}
 {\delta\cdot\overline{\nu}(\varepsilon t+\beta\delta)}
 \cdot
 \frac{\overline{\nu}(\varepsilon t+\beta\delta)}
 {\overline{\nu}(\beta t)}=
 \Big(\frac{\varepsilon}{\beta}\Big)^{-\alpha}.
 \eeqnn
 Since
 $\mathbf{P}(X_t>(\varepsilon-\beta)t)
 =\mathbf{P}(X_t+\beta t>\varepsilon t)\to0$
 as $t\to\infty$, it follows that
 \beqlb\label{eqn.3023}
 \lim_{t\to\infty}
 \frac{\mathbf{P}\big((\varepsilon-\beta)t<X_t\leq-\varepsilon t,
 	X_{t+\delta}>0\big)}
 {\delta\cdot\overline{\nu}(\beta t)}
 =0.
 \eeqlb
 
 \medskip
 
 \textbf{(iii)} For the third probability in \eqref{eqn.3021}, using the Markov property and stationary increments again,
 \beqnn
 \frac{\mathbf{P}\big(-\varepsilon t<X_t\leq0,X_{t+\delta}>0\big)}
 {\alpha/\beta\cdot\overline{\nu}(\beta t)\cdot
 	\mathbf{E}\big[X_\delta^+\big]}
 \ar=\ar
 \int_{-\varepsilon t}^{0}
 \frac{\mathbf{P}\big(X_\delta>-y\big)}
 {\mathbf{E}\big[X_\delta^+\big]}
 \cdot
 \frac{\mathbf{P}\big(X_t\in dy\big)}
 {\alpha/\beta\cdot\overline{\nu}(\beta t)}.
 \eeqnn
 Similarly to the proof of Proposition~\ref{Prop.301}, we have
 \beqnn
 \frac{\mathbf{P}\big(-\varepsilon t<X_t\leq0,X_{t+\delta}>0\big)}
 {\alpha/\beta\cdot\overline{\nu}(\beta t)\cdot
 	\mathbf{E}\big[X_\delta^+\big]}
 \ar\leq\ar
 \sum_{k=0}^{[t]}
 \frac{\varepsilon\cdot\mathbf{P}\big(X_\delta>k\varepsilon\big)}
 {\mathbf{E}\big[X_\delta^+\big]}
 \cdot
 \frac{\mathbf{P}\big(X_t\in-(k+1)\varepsilon+\varDelta_\varepsilon\big)}
 {\alpha/\beta\cdot\overline{\nu}(\beta t)\cdot\varepsilon}
 \cr\cr
 \ar\leq\ar
 \sup_{k=0,\cdots,[t]}
 \frac{\mathbf{P}\big(X_t\in-(k+1)\varepsilon+\varDelta_\varepsilon\big)}
 {\alpha/\beta\cdot\overline{\nu}(\beta t)\cdot\varepsilon}
 \cdot
 \sum_{j=0}^{[t]}
 \frac{\varepsilon\cdot\mathbf{P}\big(X_\delta>j\varepsilon\big)}
 {\mathbf{E}\big[X_\delta^+\big]}
 \eeqnn
 and
 \beqnn
 \frac{\mathbf{P}\big(-\varepsilon t<X_t\leq0,X_{t+\delta}>0\big)}
 {\alpha/\beta\cdot\overline{\nu}(\beta t)\cdot
 	\mathbf{E}\big[X_\delta^+\big]}
 \ar\geq\ar
 \sum_{k=0}^{[t]-1}
 \frac{\varepsilon\cdot\mathbf{P}\big(X_\delta>(k+1)\varepsilon\big)}
 {\mathbf{E}\big[X_\delta^+\big]}
 \cdot
 \frac{\mathbf{P}\big(X_t\in-(k+1)\varepsilon+\varDelta_\varepsilon\big)}
 {\alpha/\beta\cdot\overline{\nu}(\beta t)\cdot\varepsilon}
 \cr\cr
 \ar\geq\ar
 \inf_{k=0,\cdots,[t]-1}
 \frac{\mathbf{P}\big(X_t\in-(k+1)\varepsilon+\varDelta_\varepsilon\big)}
 {\alpha/\beta\cdot\overline{\nu}(\beta t)\cdot\varepsilon}
 \cdot
 \sum_{j=0}^{[t]-1}
 \frac{\varepsilon\cdot\mathbf{P}\big(X_\delta>(j+1)\varepsilon\big)}
 {\mathbf{E}\big[X_\delta^+\big]}.
 \eeqnn
 Using the monotonicity of $\overline{\nu}$,
 Lemma~\ref{Lemma.203} with $\lambda=\beta-2\varepsilon>0$ and \eqref{eqn.30}, we obtain
 \beqnn
 \lefteqn{
 	\liminf_{t\to\infty}\inf_{k=0,\cdots,[t]-1}
 	\frac{\mathbf{P}\big(X_t\in-(k+1)\varepsilon+\varDelta_\varepsilon\big)}
 	{\alpha/\beta\cdot\overline{\nu}(\beta t)\cdot\varepsilon}
 }
 \cr\cr
 \ar=\ar
 \liminf_{t\to\infty}\inf_{k=0,\cdots,[t]-1}
 \frac{\mathbf{P}\big(X_t+\beta t\in
 	\beta t-(k+1)\varepsilon+\varDelta_\varepsilon\big)}
 {t \alpha\cdot
 	\overline{\nu}\big(\beta t-(k+1)\varepsilon\big)/
 	\big(\beta t-(k+1)\varepsilon\big)\cdot\varepsilon}
 \cdot
 \frac{\overline{\nu}\big(\beta t-(k+1)\varepsilon\big)}
 {\beta t-(k+1)\varepsilon} 
 \frac{\beta t}{\overline{\nu}(\beta t)}
 \geq1.
 \eeqnn
 and
 \beqnn
 \lefteqn{
 	\limsup_{t\to\infty}\sup_{k=0,\cdots,[t]}
 	\frac{\mathbf{P}\big(X_t\in-(k+1)\varepsilon+\varDelta_\varepsilon\big)}
 	{\alpha/\beta\cdot\overline{\nu}(\beta t)\cdot\varepsilon} }
 \cr\cr
 \ar\leq\ar
 \limsup_{t\to\infty}\sup_{k=0,\cdots,[t]}
 \frac{\mathbf{P}\big(X_t+\beta t\in
 	\beta t-(k+1)\varepsilon+\varDelta_\varepsilon\big)}
 {t \alpha\cdot
 	\overline{\nu}\big(\beta t-(k+1)\varepsilon\big)/
 	\big(\beta t-(k+1)\varepsilon\big)\cdot\varepsilon}
 \cdot
 \frac{\overline{\nu}\big(\beta t-([t]+1)\varepsilon\big)}
 {\beta t-([t]+1)\varepsilon} 
 \frac{\beta t}{\overline{\nu}(\beta t)}=
 \bigg(\frac{\beta}{\beta-\varepsilon}\bigg)^{\alpha+1}.
 \eeqnn
 Since $y\mapsto\mathbf{P}\big(X_\delta>y\big)$ is non-increasing, we have
 \beqnn
 1-\varepsilon\cdot
 \frac{\mathbf{P}\big(X_\delta>0\big)}
 {\mathbf{E}\big[X_\delta^+\big]}
 \ar\leq\ar
 \liminf_{t\to\infty}
 \sum_{j=0}^{[t]-1}
 \frac{\varepsilon\cdot\mathbf{P}\big(X_\delta>(j+1)\varepsilon\big)}
 {\mathbf{E}\big[X_\delta^+\big]}
 \cr\cr
 \ar\leq\ar
 \limsup_{t\to\infty}
 \sum_{j=0}^{[t]}
 \frac{\varepsilon\cdot\mathbf{P}\big(X_\delta>j\varepsilon\big)}
 {\mathbf{E}\big[X_\delta^+\big]}
 \leq
 1+\varepsilon\cdot
 \frac{\mathbf{P}\big(X_\delta>0\big)}
 {\mathbf{E}\big[X_\delta^+\big]}.
 \eeqnn
 Combining the preceding estimates together yields
 \beqnn
 1-\varepsilon\cdot
 \frac{\mathbf{P}\big(X_\delta>0\big)}
 {\mathbf{E}\big[X_\delta^+\big]}
 \ar\leq\ar
 \liminf_{t\to\infty}
 \frac{\mathbf{P}\big(-\varepsilon t<X_t\leq0,X_{t+\delta}>0\big)}
 {\alpha/\beta\cdot\overline{\nu}(\beta t)\cdot
 	\mathbf{E}\big[X_\delta^+\big]}
 \cr\cr
 \ar\leq\ar
 \limsup_{t\to\infty}
 \frac{\mathbf{P}\big(-\varepsilon t<X_t\leq0,X_{t+\delta}>0\big)}
 {\alpha/\beta\cdot\overline{\nu}(\beta t)\cdot
 	\mathbf{E}\big[X_\delta^+\big]}
 \leq \bigg(\frac{\beta}{\beta-\varepsilon}\bigg)^{\alpha+1}
 \cdot
 \Big(
 1+\varepsilon\cdot
 \frac{\mathbf{P}\big(X_\delta>0\big)}
 {\mathbf{E}\big[X_\delta^+\big]}
 \Big).
 \eeqnn
 By the fact $\mathbf{E}\big[X_\delta^+\big]\leq \mathbf{E}\big[|X_\delta|\big]<\infty$ and the arbitrariness of $\varepsilon$, we obtain
 \beqlb\label{eqn.3024}
 \lim_{\varepsilon\to0+}\limsup_{t\to\infty}
 \bigg|
 \frac{\mathbf{P}\big(-\varepsilon t<X_t\leq0,X_{t+\delta}>0\big)}
 {\alpha/\beta\cdot\overline{\nu}(\beta t)\cdot
 	\mathbf{E}\big[X_\delta^+\big]}
 -1
 \bigg|=0.
 \eeqlb
 Finally, combining \eqref{eqn.3022}-\eqref{eqn.3024} together,
 we can get the desired limit \eqref{eqn.302} immediately.
 \qed

 \begin{corollary}\label{Corollary.303}
 	For any $\delta>0$, we have as $t\to\infty$,  
 	\beqlb\label{eqn.403}
 	\frac{\mathbf{P}(X_t>0)-\mathbf{P}(X_{t+\delta}>0)}{(\alpha-1)\cdot \overline{\nu}(\beta t)\cdot\delta}\to 1. 
 	\eeqlb 
 \end{corollary}
 \proof 
 Substituting the limits obtained in Propositions~\ref{Prop.301} and~\ref{Prop.302} into \eqref{eqn.402}, and then applying the elementary identity
 \beqnn
 \frac{\alpha}{\beta}\cdot  \overline{\nu} (\beta t)\cdot \mathbf{E}\big[X_\delta^-\big] -\frac{\alpha}{\beta}\cdot \overline{\nu} (\beta t)\cdot \Big(\mathbf{E}\big[X_\delta^+\big]+ \frac{\beta \delta}{\alpha} \Big)
 \ar=\ar (\alpha-1)\cdot \delta \cdot  \overline{\nu} (\beta t),
 \eeqnn
 which follows from $\mathbf{E}[X_\delta^+] -\mathbf{E}[X_\delta^-]= -\beta\delta$, yields the desired limit \eqref{eqn.403} immediately. 
 \qed   
 
 \textbf{\textit{Proof of Theorem~\ref{MainThm01}. }} We first prove the asymptotic equivalence \eqref{eqn.MainThm01} for any fixed $\delta>0$. In view of \eqref{eqn.4011}, for any $\varepsilon\in(0,1/2)$, we can further write it as
 \beqlb\label{eqn.305}
 \underline{n}\big(\zeta\in t+\varDelta_\delta\big)
 =\frac{\delta}{t+\delta}\cdot\underline{n}(\zeta>t)+\sum_{i=1}^3 J_i(t,\varepsilon,\delta)
 \eeqlb
 with the three terms in the last sum given by
 \beqnn 
 J_1(t,\varepsilon,\delta)\ar:=\ar \frac{1}{t+\delta}\int_{0}^{\varepsilon t}\big(\mathbf{P}(X_{t-s}>0)-\mathbf{P}(X_{t+\delta-s}>0)\big)\,\underline{n}(\zeta>s)\, ds, \cr
 J_2(t,\varepsilon,\delta)\ar:=\ar \frac{1}{t+\delta}\int_{\varepsilon t}^{t-\varepsilon t}\big(\mathbf{P}(X_{t-s}>0)-\mathbf{P}(X_{t+\delta-s}>0)\big)\,\underline{n}(\zeta>s)\, ds,\cr
 J_3(t,\varepsilon,\delta)\ar:=\ar  \frac{1}{t+\delta}\int_{t-\varepsilon t}^t  \mathbf{P}(X_{t-s}>0) \,\underline{n}(\zeta>s)\, ds -   \frac{1}{t+\delta} \int_{t-\varepsilon t}^{t+\delta}\mathbf{P}(X_{t+\delta-s}>0) \,\underline{n}(\zeta>s)\, ds. 
 \eeqnn
 We now analyze the asymptotic behaviour of the four terms on the right-hand side of \eqref{eqn.305}.
 
 \medskip
 
 \textbf{(i)} By \eqref{CWX01}, we have, as $t\to\infty$,
 \beqlb\label{eqn.311}
 \frac{\delta}{t+\delta}\cdot\underline{n}(\zeta>t)
 \sim\frac{\delta}{t}\cdot\underline{n}(\zeta)
 \cdot\overline{\nu}(\beta t).
 \eeqlb
 
 \textbf{(ii)} We now prove that
 $J_1(t,\varepsilon,\delta)
 \sim(\alpha-1)\cdot\underline{n}(\zeta)
 \cdot\overline{\nu}(\beta t)/t\cdot\delta$
 as $t\to\infty$.
 By the triangle inequality, we first have
 \beqlb\label{eqn.312}
 \lefteqn{
 	\bigg|
 	\frac{J_1(t,\varepsilon,\delta)}
 	{(\alpha-1)\underline{n}(\zeta)
 		\cdot\overline{\nu}(\beta t)/t\cdot\delta}
 	-1
 	\bigg|
 }
 \cr\cr
 \ar=\ar
 \bigg|
 \frac{t}{t+\delta}
 \int_0^{\varepsilon t}
 \frac{\mathbf{P}(X_{t-s}>0)-\mathbf{P}(X_{t+\delta-s}>0)}
 {(\alpha-1)\cdot\overline{\nu}\big(\beta(t-s)\big)\cdot\delta}
 \cdot
 \frac{\overline{\nu}\big(\beta(t-s)\big)}
 {\overline{\nu}(\beta t)}
 \cdot
 \frac{\underline{n}(\zeta>s)}{\underline{n}(\zeta)}
 \,ds-1
 \bigg|
 \cr\cr
 \ar\leq\ar
 \int_0^{\varepsilon t}
 \bigg|
 \frac{\mathbf{P}(X_{t-s}>0)-\mathbf{P}(X_{t+\delta-s}>0)}
 {(\alpha-1)\cdot\overline{\nu}\big(\beta(t-s)\big)\cdot\delta}
 -1
 \bigg|
 \cdot
 \frac{\overline{\nu}\big(\beta(t-s)\big)}
 {\overline{\nu}(\beta t)}
 \cdot
 \frac{\underline{n}(\zeta>s)}{\underline{n}(\zeta)}
 \,ds
 \cr\cr
 \ar\ar+
 \int_0^{\varepsilon t}
 \bigg|
 \frac{\overline{\nu}\big(\beta(t-s)\big)}
 {\overline{\nu}(\beta t)}-1
 \bigg|
 \cdot
 \frac{\underline{n}(\zeta>s)}{\underline{n}(\zeta)}
 \,ds+
 \bigg|
 \frac{t}{t+\delta}
 \int_0^{\varepsilon t}
 \frac{\underline{n}(\zeta>s)}{\underline{n}(\zeta)}
 \,ds-1
 \bigg|.
 \eeqlb
 Since $\overline{\nu}\in\mathrm{RV}_{-\alpha}^\infty$
 is non-increasing, for any fixed $K>0$, we have as $t\to\infty$,
 \beqnn
 \sup_{r\in[0,K]}
 \bigg|
 \frac{\overline{\nu}\big(\beta(t-r)\big)}
 {\overline{\nu}(\beta t)}-1
 \bigg|\to0
 \quad\mbox{and}\quad
 \sup_{s\in[0,\varepsilon t]}
 \frac{\overline{\nu}\big(\beta(t-s)\big)}
 {\overline{\nu}(\beta t)}
 \leq
 \frac{\overline{\nu}\big((1-\varepsilon)\beta t\big)}
 {\overline{\nu}(\beta t)}
 \to(1-\varepsilon)^{-\alpha}.
 \eeqnn
 The first term on the right-hand side of the inequality in \eqref{eqn.312} can be bounded, for all sufficiently large $t$, by
 \beqnn
 C\cdot
 \int_0^{\varepsilon t}
 \bigg|
 \frac{\mathbf{P}(X_{t-s}>0)-\mathbf{P}(X_{t+\delta-s}>0)}
 {(\alpha-1)\cdot\overline{\nu}\big(\beta(t-s)\big)\cdot\delta}
 -1
 \bigg|
 \cdot
 \frac{\underline{n}(\zeta>s)}{\underline{n}(\zeta)}
 \,ds,
 \eeqnn
 which vanishes as $t\to\infty$ by Corollary~\ref{Corollary.303}.
 Moreover, the second term can be decomposed as
 \beqnn
 \lefteqn{\int_{0}^K \bigg|  \frac{\overline{\nu}\big(\beta (t-s)\big)}{\overline{\nu}(\beta t)} -1\bigg| \cdot \frac{\underline{n}(\zeta>s)}{\underline{n}(\zeta)}\, ds 
 	+  \int_K^{\varepsilon t}  \bigg|  \frac{\overline{\nu}\big(\beta (t-s)\big)}{\overline{\nu}(\beta t)} -1\bigg| \cdot \frac{\underline{n}(\zeta>s)}{\underline{n}(\zeta)}\, ds }\ar\ar\cr
 \ar\leq\ar 
 \sup_{r\in [0,K]} \bigg|  \frac{\overline{\nu}\big(\beta (t-r)\big)}{\overline{\nu}(\beta t)} -1\bigg|\cdot \int_{0}^\infty   \frac{\underline{n}(\zeta>s)}{\underline{n}(\zeta)}\, ds  +\sup_{r\in [K, \varepsilon t]} \bigg|  \frac{\overline{\nu}\big(\beta (t-r)\big)}{\overline{\nu}(\beta t)} -1\bigg|\cdot \int_K^{\varepsilon t}    \frac{\underline{n}(\zeta>s)}{\underline{n}(\zeta)}\, ds,
 \eeqnn
 which tends to $0$ as $t\to\infty$ and then $K\to\infty$.
 For the third term in \eqref{eqn.312}, we have
 \beqnn
 \bigg|
 \frac{t}{t+\delta}
 \int_0^{\varepsilon t}
 \frac{\underline{n}(\zeta>s)}{\underline{n}(\zeta)}
 \,ds-1
 \bigg|
 \ar\leq\ar
 \frac{\delta}{t+\delta}
 \int_0^{\varepsilon t}
 \frac{\underline{n}(\zeta>s)}{\underline{n}(\zeta)}
 \,ds
 +
 \int_{\varepsilon t}^\infty
 \frac{\underline{n}(\zeta>s)}{\underline{n}(\zeta)}
 \,ds,
 \eeqnn
 which clearly goes to $0$ as $t\to\infty$.
 Combining these estimates yields
 \beqlb\label{eqn.314}
 \lim_{t\to\infty}
 \bigg|
 \frac{J_1(t,\varepsilon,\delta)}
 {(\alpha-1)\underline{n}(\zeta)
 	\cdot\overline{\nu}(\beta t)/t\cdot\delta}
 -1
 \bigg|=0.
 \eeqlb
 
 \textbf{(iii)} We now prove that
 $J_2(t,\varepsilon,\delta)
 =o\big(\overline{\nu}(\beta t)/t\cdot\delta\big)$
 as $t\to\infty$. Firstly,
 \beqnn
 \bigg|
 \frac{J_2(t,\varepsilon,\delta)}
 {\overline{\nu}(\beta t)/t\cdot\delta}
 \bigg|
 \ar\leq\ar
 \frac{t}{t+\delta}
 \int_{\varepsilon t}^{t-\varepsilon t}
 \frac{\big|\mathbf{P}(X_{t-s}>0)
 	-\mathbf{P}(X_{t+\delta-s}>0)\big|}
 {\overline{\nu}(\beta t)\cdot\delta}
 \,\underline{n}(\zeta>s)\,ds
 \cr\cr
 \ar\leq\ar
 \int_{\varepsilon t}^{t-\varepsilon t}
 \frac{\big|\mathbf{P}(X_{t-s}>0)
 	-\mathbf{P}(X_{t+\delta-s}>0)\big|}
 {\overline{\nu}(\beta t)\cdot\delta}
 \cdot\underline{n}(\zeta>s)\,ds.
 \eeqnn
 By using Corollary~\ref{Corollary.303},
 then \eqref{CWX01} and finally the monotonicity of
 $\overline{\nu}\in\mathrm{RV}_{-\alpha}^\infty$,
 we have, as $t\to\infty$,
 \beqlb\label{eqn.315}
 \bigg|
 \frac{J_2(t,\varepsilon,\delta)}
 {\overline{\nu}(\beta t)/t\cdot\delta}
 \bigg|
 \ar\leq\ar
 C\int_{\varepsilon t}^{t-\varepsilon t}
 \frac{\overline{\nu}\big(\beta(t-s)\big)}
 {\overline{\nu}(\beta t)}
 \cdot\underline{n}(\zeta>s)\,ds
 \leq 
 C\cdot
 \frac{\overline{\nu}(\varepsilon\beta t)}
 {\overline{\nu}(\beta t)}
 \int_{\varepsilon t}^\infty
 \overline{\nu}(\beta s)\,ds
 \to0.
 \eeqlb
 
 \textbf{(iv)} For $J_3(t,\varepsilon,\delta)$, using the change of variables, we may write
 \beqnn
 J_3(t,\varepsilon,\delta) 
 \ar=\ar \frac{1}{t+\delta} \int_{0}^{\varepsilon t}  \mathbf{P}(X_s>0) \,\underline{n}(\zeta>t-s)\, ds   -  \frac{1}{t+\delta}  \int_0^{\delta+\varepsilon t}\mathbf{P}(X_s>0) \,\underline{n}(\zeta>t+\delta-s)\, ds  \cr
 \ar=\ar  \frac{1}{t+\delta} \int_{0}^{\varepsilon t}  \mathbf{P}(X_s>0) \,\underline{n}\big(\zeta \in  t-s+\varDelta_\delta\big)\, ds -\frac{1}{t+\delta} \int_{\varepsilon t}^{\delta+\varepsilon t}\mathbf{P}(X_s>0) \,\underline{n}(\zeta>t+\delta-s)\, ds.
 \eeqnn
 Since $z\mapsto\underline{n}(\zeta>z)$ is non-increasing, we have
 \beqnn
 \big|J_3(t,\varepsilon,\delta)\big|
 \ar\leq\ar
 \frac{1}{t+\delta}
 \int_0^{\varepsilon t}
 \underline{n}\big(\zeta\in t-s+\varDelta_\delta\big)\,ds+
 \frac{\delta}{t+\delta}
 \cdot\underline{n}\big(\zeta>(1-\varepsilon)t\big)
 \cdot
 \sup_{\varepsilon t<s\leq\varepsilon t+\delta}
 \mathbf{P}(X_s>0),
 \eeqnn
 and also for all sufficiently large $t$,
 \beqnn
 \int_0^{\varepsilon t}
 \underline{n}\big(\zeta\in t-s+\varDelta_\delta\big)\,ds
 \ar=\ar
 \int_{t-\varepsilon t}^t
 \underline{n}(\zeta>s)\,ds
 -
 \int_{t-\varepsilon t+\delta}^{t+\delta}
 \underline{n}(\zeta>s)\,ds
 \cr\cr
 \ar=\ar
 \int_{t-\varepsilon t}^{t-\varepsilon t+\delta}
 \underline{n}(\zeta>s)\,ds
 -
 \int_t^{t+\delta}
 \underline{n}(\zeta>s)\,ds
 \cr\cr
 \ar\leq\ar
 \delta\cdot
 \big(
 \underline{n}(\zeta>t-\varepsilon t)
 -\underline{n}(\zeta>t+\delta)
 \big).
 \eeqnn
 Combining these estimates yields
 \beqnn
 \bigg|
 \frac{J_3(t,\varepsilon,\delta)}
 {\overline{\nu}(\beta t)/t\cdot\delta}
 \bigg|
 \ar\leq\ar
 \frac{t}{t+\delta}
 \cdot
 \frac{\underline{n}\big(\zeta>(1-\varepsilon)t\big)}
 {\overline{\nu}(\beta t)}
 \cdot
 \Big(
 1-
 \frac{\underline{n}(\zeta>t+\delta)}
 {\underline{n}\big(\zeta>(1-\varepsilon)t\big)}
 +
 \sup_{\varepsilon t<s\leq\varepsilon t+\delta}
 \mathbf{P}(X_s>0)
 \Big).
 \eeqnn
 Using the monotonicity of $\underline{n}(\zeta>z)$ in $z$ again,
 \beqnn
 \bigg|
 \frac{J_3(t,\varepsilon,\delta)}
 {\overline{\nu}(\beta t)/t\cdot\delta}
 \bigg|
 \ar\leq\ar
 \frac{\underline{n}\big(\zeta>(1-\varepsilon)t\big)}
 {\overline{\nu}(\beta t)}
 \cdot
 \Big(
 1-
 \frac{\underline{n}(\zeta>t+\delta)}
 {\underline{n}\big(\zeta>(1-\varepsilon)t\big)}
 +
 \sup_{s>\varepsilon t}\mathbf{P}(X_s>0)
 \Big).
 \eeqnn
 The fact $X_t\overset{\rm a.s.}\to-\infty$ induces that
 $\sup_{s>\varepsilon t}\mathbf{P}(X_s>0)
 \leq\mathbf{P}\big(\sup_{s>\varepsilon t}X_s>0\big)\to0$
 as $t\to\infty$.
 Furthermore, by \eqref{CWX01} we have as $t\to\infty$,
 \beqnn
 \frac{\underline{n}\big(\zeta>(1-\varepsilon)t\big)}
 {\overline{\nu}(\beta t)}
 \to
 \underline{n}(\zeta)\cdot(1-\varepsilon)^{-\alpha}
 \quad\mbox{and}\quad
 \frac{\underline{n}(\zeta>t+\delta)}
 {\underline{n}\big(\zeta>(1-\varepsilon)t\big)}
 \to(1-\varepsilon)^\alpha.
 \eeqnn
 Consequently,
 \beqlb\label{eqn.316}
 \lim_{\varepsilon\to0+}\limsup_{t\to\infty}
 \bigg|
 \frac{J_3(t,\varepsilon,\delta)}
 {\overline{\nu}(\beta t)/t\cdot\delta}
 \bigg|
 \ar\leq\ar
 \lim_{\varepsilon\to0+}
 \underline{n}(\zeta)
 \cdot\big((1-\varepsilon)^{-\alpha}-1\big)
 =0.
 \eeqlb
 Finally, substituting \eqref{eqn.311}, \eqref{eqn.314},
 \eqref{eqn.315} and \eqref{eqn.316} into \eqref{eqn.305}
 yields as $t\to\infty$,
 \beqnn
 \underline{n}\big(\zeta\in t+\varDelta_\delta\big)
 \sim
 \frac{\alpha\delta}{t}\cdot
 \underline{n}(\zeta)\cdot\overline{\nu}(\beta t)
 \sim
 \underline{n}(\zeta)\cdot
 \nu\big(\beta(t+\varDelta_\delta)\big),
 \eeqnn
 where the second asymptotic equivalence follows from
 \eqref{eqn.30}.
 
 We now extend \eqref{eqn.MainThm01}  to hold uniformly for $\delta\in[\delta_0,\infty]$.
 For any large integer $n\geq1/\delta_0$ and
 $\delta\in[\delta_0,\infty]$, we have
 \beqnn
 \sum_{k=0}^{[n\delta]-1}
 \underline{n}\big(\zeta\in t+k/n+\varDelta_{1/n}\big)
 \ar\leq\ar
 \underline{n}\big(\zeta\in t+\varDelta_\delta\big)
 \leq
 \sum_{k=0}^{[n\delta]}
 \underline{n}\big(\zeta\in t+k/n+\varDelta_{1/n}\big).
 \eeqnn
 Applying the  asymptotic equivalence \eqref{eqn.MainThm01} with $\delta=1/n$,
 we obtain a positive function $\rho_n$ on $\mathbb{R}_+$
 that decreases to $0$ at infinity and satisfies,
 for all sufficiently large $t$ and all $k\geq0$,
 \beqnn
 \big(1-\rho_n(t)\big)\cdot
 \nu\big(\beta(t+k/n+\varDelta_{1/n})\big)\leq 
 \frac{
 	\underline{n}\big(\zeta\in t+k/n+\varDelta_{1/n}\big)
 }{
 	\underline{n}(\zeta)
 }
 \leq 
 \big(1+\rho_n(t)\big)\cdot
 \nu\big(\beta(t+k/n+\varDelta_{1/n})\big).
 \eeqnn
 Combining these two estimates together yields
 \beqnn
 \frac{
 	\underline{n}\big(\zeta\in t+\varDelta_\delta\big)
 }{
 	\underline{n}(\zeta)\cdot
 	\nu\big(\beta(t+\varDelta_\delta)\big)
 }
 \leq
 \frac{1+\rho_n(t)}
 {\nu\big(\beta(t+\varDelta_\delta)\big)}
 \sum_{k=0}^{[n\delta]}
 \nu\big(\beta(t+k/n+\varDelta_{1/n})\big)
 \leq
 \big(1+\rho_n(t)\big)\cdot
 \frac{
 	\nu\big(\beta(t+\varDelta_{\delta+1/n})\big)
 }{
 	\nu\big(\beta(t+\varDelta_\delta)\big)
 }.
 \eeqnn
 In view of the monotonicity of $\nu(x+\varDelta_\delta)$ in $\delta$ and
 \eqref{eqn.30}, this yields
 \beqnn
 \limsup_{t\to\infty}
 \sup_{\delta\in[\delta_0,\infty]}
 \frac{
 	\underline{n}\big(\zeta\in t+\varDelta_\delta\big)
 }{
 	\underline{n}(\zeta)\cdot
 	\nu\big(\beta(t+\varDelta_\delta)\big)
 }\ar\leq\ar 
 \limsup_{t\to\infty}
 \sup_{\delta\in[\delta_0,\infty]}
 \frac{
 	\nu\big(\beta(t+\varDelta_{\delta+1/n})\big)
 }{
 	\nu\big(\beta(t+\varDelta_\delta)\big)
 }
 \cr\cr
 \ar\leq\ar
 1+
 \limsup_{t\to\infty}
 \sup_{\delta\in[\delta_0,\infty]}
 \frac{
 	\nu\big(\beta(t+\delta)+\varDelta_{\beta/n}\big)
 }{
 	\nu\big(\beta t+\varDelta_{\beta\delta_0}\big)
 }
 \leq 
 1+\frac{1}{n\delta_0}.
 \eeqnn
 On the other hand, an analogous argument gives
 \beqnn
 \liminf_{t\to\infty}
 \inf_{\delta\in[\delta_0,\infty]}
 \frac{
 	\underline{n}\big(\zeta\in t+\varDelta_\delta\big)
 }{
 	\underline{n}(\zeta)\cdot
 	\nu\big(\beta(t+\varDelta_\delta)\big)
 }\geq
 \liminf_{t\to\infty}
 \inf_{\delta\in[\delta_0,\infty]}
 \frac{
 	\nu\big(\beta(t+\varDelta_{\delta-1/n})\big)
 }{
 	\nu\big(\beta(t+\varDelta_\delta)\big)
 }\geq
 1-\frac{1}{n\delta_0}.
 \eeqnn
 Combining these two bounds together and then
 letting $n\to\infty$ yields the desired uniform
 asymptotic equivalence \eqref{eqn.MainThm01}.
 \qed

 \section{Proof of Theorem \ref{MainThm02}} 
 \label{Sec.ProofThm02}
 \setcounter{equation}{0}

 This section is dedicated to the proof of Theorem~\ref{MainThm02}, which rests on the following key observation.
 For an excursion attaining a large maximum, the dominant scenario is that a unique large jump carries the excursion to a level close to, but slightly below, its eventual maximum, after which only a relatively small additional upward displacement is needed for the excursion to reach its maximum.  
 To make this idea precise, for $x,\delta,K>0$, define
 \beqlb \label{eqn.5033}
 \mathcal{A}_{\delta}(x,K):= \Big\{ s\in (0,\zeta):  \Delta\epsilon_s>0,\, \epsilon_{s-}\leq K,\, \sup_{s\leq r<\zeta} \epsilon_r \in x+\varDelta_\delta, \,  \sup_{s\leq r<\zeta} \epsilon_r- \epsilon_s \leq K\Big\}. 
 \eeqlb
 Thus, $\mathcal A_\delta(x,K)$ consists of those positive jumps whose  pre-jump position is bounded by $K$, and for which the subsequent maximum lies in $x+\varDelta_\delta$ and exceeds the post-jump position by at most $K$. In particular, when $x$ is large compared with $K$, every jump belonging to $\mathcal A_\delta(x,K)$ must itself be large and carry the excursion close to its subsequent maximum. Indeed, for every $s\in\mathcal A_\delta(x,K)$,
 \beqlb\label{eqn.5018}
 \Delta\epsilon_s = \epsilon_s-\epsilon_{s-} = \sup_{s\leq r<\zeta} \epsilon_r-\Big( \sup_{s\leq r<\zeta} \epsilon_r- \epsilon_s\Big) -\epsilon_{s-} \in (x-2K,x+\delta].
 \eeqlb
 Consequently, when $x>2K$, all jumps in  $\mathcal A_\delta(x,K)$ have size at least $x-2K>0$, and hence 
 \beqnn
 \underline n\big[  \#\mathcal A_\delta(x,K) \big]
 \leq  \underline n\bigg[ \sum_{0<s\leq\zeta} \mathbf 1_{\{\Delta\epsilon_s>x-2K\}}\bigg]
 =\underline n(\zeta)\cdot  \overline\nu(x-2K) <\infty,
 \eeqnn
 which implies that  $\underline{n}\big( \# \mathcal A_\delta(x,K)=\infty \big)=0$. 
 Here and throughout,  $\#A$ denotes the cardinality of the set $A$.
 
 We now state four auxiliary lemmas that constitute the main ingredients in the proof of Theorem~\ref{MainThm02}; their proofs are given in the subsequent subsections. 
 The first lemma shows that, on the scale of $\nu(x+\varDelta_\delta)$, an excursion whose maximum lies in $x+\varDelta_\delta$ must contain a jump belonging to $\mathcal A_\delta(x,K)$, provided $K$ is subsequently taken sufficiently large. The second lemma controls the converse implication. The third lemma shows that the occurrence of two or more such large jumps is asymptotically negligible. Finally, the fourth lemma identifies the first-order asymptotics of the expected number of jumps in $\mathcal A_\delta(x,K)$ and, in view of the third lemma, of the probability that exactly one such jump occurs. 
 
 \begin{lemma} \label{Lemma.501}
 	For any $\delta >0$, we have  
 	\beqlb\label{eqn.5016}
 	\lim_{K\to\infty}\limsup_{x\to\infty}	\frac{	\underline{n}\big( \# \mathcal{A}_{\delta}(x,K)=0,\,\overline{\epsilon}\in x+\varDelta_\delta\big)}{\nu\big(x+ \varDelta_\delta\big)} =0. 
 	\eeqlb
 \end{lemma}

 \begin{lemma}\label{Lemma.502}
 	For any $\delta>0$  and $K>0$, we have as $x\to\infty$,
 	\beqlb\label{eqn.5017}
 	\frac{\underline{n}\big( \# \mathcal{A}_\delta(x,K)\geq 1,\, \overline{\epsilon}\notin x+\varDelta_\delta  \big)}{ \nu\big(x+ \varDelta_\delta\big)} 
 	\leq 
 	\frac{\underline{n}\big(\# \mathcal{A}_\delta(x,K), \overline{\epsilon}\notin x+\varDelta_\delta  \big)  }{ \nu\big(x+ \varDelta_\delta\big)}\to 0. 
 	\eeqlb
 \end{lemma}
 
 \begin{lemma}\label{Lemma.503}
 	For any $\delta>0$ and $K>0$, we have as $x\to\infty$,
 	\beqnn
 	\frac{\underline{n}\big( \# \mathcal{A}_\delta(x,K) \geq 2  \big) }{\nu\big(x+ \varDelta_\delta\big)} 
 	\leq \frac{\underline{n}\big(( \# \mathcal{A}_\delta(x,K) -1)^+ \big) }{\nu\big(x+ \varDelta_\delta\big)}  
 	\to 0. 
 	\eeqnn
 \end{lemma}
 
 \begin{lemma}\label{Lemma.504}
 	For any $\delta>0$, we have  
 	\beqnn
 	\lim_{K\to\infty}\limsup_{x\to\infty}\bigg|\frac{ \underline{n}\big( \#\mathcal{A}_\delta(x,K)\big) }{\nu\big(x+ \varDelta_\delta\big)} - \underline{n}(\zeta) \bigg| =0
 	\quad \mbox{and}\quad 
 	\lim_{K\to\infty}\limsup_{x\to\infty} \bigg|\frac{ \underline{n}\big( \#\mathcal{A}_\delta(x,K)=1\big) }{\nu\big(x+ \varDelta_\delta\big)} - \underline{n}(\zeta) \bigg| = 0 .
 	\eeqnn
 	
 \end{lemma}
 
 \medskip
 
 \textbf{\textit{Proof of Theorem~\ref{MainThm02}.}} We first prove  the asymptotic equivalence \eqref{eqn.MainThm02}  for any fixed $\delta>0$.  
 By the law of total probability, $\underline{n}(\overline{\epsilon}\in x+\varDelta_\delta)$ admits the decomposition
 \beqnn
 \underline{n}\big(\overline{\epsilon}\in x+\varDelta_\delta,  \# \mathcal{A}_\delta(x,K)=0\big)  + \underline{n}\big( \# \mathcal{A}_\delta(x,K)\geq 1\big)  -\underline{n}\big(\overline{\epsilon}\notin x+\varDelta_\delta, \# \mathcal{A}_\delta(x,K)\geq 1\big) .
 \eeqnn
 By Lemmas~\ref{Lemma.501} and~\ref{Lemma.502}, both the first and third terms are $o(\nu(x+\varDelta_\delta))$ as $x\to\infty$ and then $K\to\infty$. For the second term, using the elementary identity  $\mathbf{1}_{\{ \# \mathcal{A}_\delta(x,K)\geq 1 \}}=  \# \mathcal{A}_\delta(x,K)- \big( \# \mathcal{A}_\delta(x,K)-1 \big)^+$, 
 \beqnn
 \underline{n}\big( \# \mathcal{A}_\delta(x,K)\geq 1\big)  
 \ar=\ar \underline{n}\big( \# \mathcal{A}_\delta(x,K)\big)  - \underline{n}\big(  ( \# \mathcal{A}_\delta(x,K)-1  )^+\big)  . 
 \eeqnn
 By Lemma~\ref{Lemma.503}, the second term on the right-hand side is also $o(\nu(x+\varDelta_\delta))$ as $x\to\infty$. Consequently, for sufficiently large $x$ and $K$,  
 \beqnn 
 \bigg|\frac{ \underline{n}\big(\overline{\epsilon}\in x+\varDelta_\delta\big) }{\nu\big(x+ \varDelta_\delta\big)} - \underline{n}(\zeta) \bigg|
 \ar\leq\ar 
 \bigg|\frac{ \underline{n}\big( \#\mathcal{A}_\delta(x,K)\big) }{\nu\big(x+ \varDelta_\delta\big)} - \underline{n}(\zeta) \bigg|
 +
 \frac{	\underline{n}\big( \# \mathcal{A}_{\delta}(x,K)=0,\,\overline{\epsilon}\in x+\varDelta_\delta\big)}{\nu\big(x+ \varDelta_\delta\big)}\cr\cr
 \ar+\ar
 \frac{\underline{n}\big( \# \mathcal{A}_\delta(x,K)\geq 1,\, \overline{\epsilon}\notin x+\varDelta_\delta  \big)}{ \nu\big(x+ \varDelta_\delta\big)}
 +
 \frac{ \underline{n}\big(( \# \mathcal{A}_\delta(x,K) -1)^+ \big)}{\nu\big(x+ \varDelta_\delta\big)}.
 \eeqnn 
 The right-hand side vanishes as $x\to\infty$ and then $K\to\infty$ by  Lemma~\ref{Lemma.501}-\ref{Lemma.504}. Hence the asymptotic equivalence \eqref{eqn.MainThm02} holds for any fixed $\delta>0$. The extension to uniformity in $\delta\in[\delta_0,\infty]$
 follows by the same discretization argument as in the proof
 of Theorem~\ref{MainThm01}. 
 \qed
 
 \subsection{Proof of Lemma~\ref{Lemma.501}}

 \begin{proposition}\label{Prop.503}
 	For any $\vartheta \in \boldsymbol{\varTheta}$ such that $\vartheta_x\to\infty$ as $x\to\infty$, there exists a constant $x_0>0$ such that for all $x\geq x_0$ and $y\in [0,x]$, 
 	\beqnn
 	\mathbf{P}_y \Big( \tau_x^+<\tau_0^-,\,
 	\sup_{0<s\leq\tau_x^+}\Delta X_s\leq \vartheta_x \Big) 
 	\leq
 	\exp\Big\{
 	-\frac{x-y}{\vartheta_x} \Big\}.
 	\eeqnn  
 \end{proposition}
 \proof For any $x\geq 0$, we introduce an auxiliary L\'evy process $X^{(\vartheta_x)}$ defined by moving all jumps larger than $\vartheta_x$ away from $X$, i.e.,   
 \beqnn
 X_t^{(\vartheta_x)}:=X_t-\sum_{0\leq s\leq t}\Delta X_s\cdot \mathbf{1}_{\{\Delta X_s>\vartheta_x\}},\quad t\geq 0. 
 \eeqnn
 By Theorem~25.17 in \cite[p.165]{Sato1999}, the Laplace exponent $\varPsi^{(\vartheta_x)}(\lambda):=  \log\mathbf E\big[e^{\lambda X_1^{(\vartheta_x)}}\big] $ is well-defined  for all $\lambda \geq 0$  and admits the following representation  
 \beqlb\label{eqn.511}
 \varPsi^{(\vartheta_x)}(\lambda) = - \Big(\beta+ \int_{\vartheta_x}^\infty z\,\nu(dz) \Big)\cdot \lambda +  \frac{\sigma^2}{2}\cdot \lambda^2 + \int_{-\infty}^{\vartheta_x} \big( e^{\lambda z}-1-\lambda z \big)\,\nu(dz)  . 
 \eeqlb
 Denote by $\tau^{(\vartheta_x),-}_z$ and $\tau^{(\vartheta_x),+}_z$ the first passage times of $X^{(\vartheta_x)}$ into $(-\infty,z)$ and $(z,\infty)$ respectively. 
 For any sufficiently large $x$ and $y\in[0,x]$, under $\mathbf{P}_y$, conditioned on $\tau_x^+<\tau_0^-$ and $\sup_{s\leq\tau_x^+}\Delta X_s\leq \vartheta_x$, we have $X_s^{(\vartheta_x)}=X_s $ for $ 0\leq s \leq \tau_x^+$ and $X_{\tau_x^+}^{(\vartheta_x)}=X_{\tau_x^+}\geq x $, which induces that  $  \tau^{(\vartheta_x),+}_x \leq \tau^+_x < \tau^{(\vartheta_x),-}_0$. Consequently,   
 \beqnn
 \mathbf{P}_y\Big( \tau_x^+<\tau_0^-,\, \sup_{s\leq\tau_x^+}\Delta X_s\leq \vartheta_x \Big)
 \leq \mathbf{P}_y\Big(  \tau^{(\vartheta_x),+}_x   < \tau^{(\vartheta_x),-}_0  \Big). 
 \eeqnn
 It suffices to prove that  uniformly for sufficiently large $x$ and all $y\in[0,x]$,
 \beqlb\label{eqn.505}
 \mathbf{P}\Big(  \tau^{(\vartheta_x),+}_{x-y}  < \tau^{(\vartheta_x),-}_{-y}  \Big) \leq \exp\Big\{ -\frac{x-y}{\vartheta_x} \Big\}.
 \eeqlb
 
 We now start to prove \eqref{eqn.505}. 
 Let $\sigma^x:=\tau^{(\vartheta_x),+}_x  \wedge  \tau^{(\vartheta_x),-}_0$. 
 Note that  $X_{\sigma^x}^{(\vartheta_x)}  \geq x$ when $\tau^{(\vartheta_x),+}_x  < \tau^{(\vartheta_x),-}_0 $, then
 \beqlb\label{eqn.512}
 \mathbf{E}_y \Big[  \exp \Big\{ X_{\sigma^x}^{(\vartheta_x)}/\vartheta_x\Big\} \Big] 
 \geq \mathbf{E}_y \Big[  \exp \Big\{ X_{\sigma^x}^{(\vartheta_x)}/\vartheta_x\Big\}, \tau^{(\vartheta_x),+}_x  < \tau^{(\vartheta_x),-}_0 \Big] \geq e^{x/\vartheta_x}\cdot  \mathbf{P}\Big(  \tau^{(\vartheta_x),+}_{x-y}  < \tau^{(\vartheta_x),-}_{-y}  \Big).\ 
 \eeqlb
 We now claim and will prove later that for sufficiently large $x$, 
 \beqlb\label{eqn.513}
 \mbox{$\Big\{   \exp \big\{  X_t^{(\vartheta_x)}/\vartheta_x\big\}:t\geq 0 \Big\} $ is a $\mathbf{P}_y$-supermartingale.}
 \eeqlb 
 Using  Fatou's theorem and then applying the optional stopping theorem at the stopping time $\sigma^x\wedge n$ yields that  
 \beqnn
 \mathbf{E}_y \Big[  \exp \Big\{ X_{\sigma^x}^{(\vartheta_x)}/\vartheta_x\Big\} \Big] 
 \ar=\ar \mathbf{E}_y \Big[ \liminf_{n\to\infty} \exp \Big\{ X_{\sigma^x\wedge n}^{(\vartheta_x)}/\vartheta_x\Big\} \Big] 
 \leq  \liminf_{n\to\infty} \mathbf{E}_y \Big[ \exp \Big\{ X_{\sigma^x\wedge n}^{(\vartheta_x)}/\vartheta_x\Big\} \Big]  \leq e^{y/\vartheta_x}.
 \eeqnn
 Combining this together with \eqref{eqn.512} immediately yields that the upper bound \eqref{eqn.505} holds for all sufficiently large $x$.
 
 It remains to prove our preceding claim \eqref{eqn.513}. In view of \eqref{eqn.511}, we have 
 \beqlb\label{eqn.514}
 \varPsi^{(\vartheta_x)}(1/\vartheta_x) \leq - \frac{\beta}{\vartheta_x} +  \frac{\sigma^2}{2|\vartheta_x|^2}  + \int_{-\infty}^{\vartheta_x} \big( e^{z/\vartheta_x}-1-z/\vartheta_x \big)\,\nu(dz) .
 \eeqlb
 By the inequality $0\leq e^{w}-1-w\leq |w|^2\wedge |w| $ for all $w\leq 1$, we have for all sufficiently large $x$,
 \beqnn
 \vartheta_x\cdot \big( e^{z/\vartheta_x}-1-z/\vartheta_x \big) \leq \frac{|z|^2}{\vartheta_x}\wedge |z|, \quad  \forall\, z\leq \vartheta_x,
 \eeqnn
 which goes to $0$ as $x\to\infty$. 
 By the dominated convergence theorem and \eqref{eqn.LevyMeasure}, 
 \beqnn
 \lim_{x\to\infty} \vartheta_x\int_{-\infty}^{\vartheta_x} \big( e^{z/\vartheta_x}-1-z/\vartheta_x \big)\,\nu(dz)
 \ar\leq\ar \lim_{x\to\infty} \int_{-\infty}^{\vartheta_x}  \Big(\frac{|z|^2}{\vartheta_x}\wedge |z| \Big)\,\nu(dz)\cr
 \ar\leq\ar \lim_{x\to\infty} \int_{-\infty}^\infty \Big(\frac{|z|^2}{\vartheta_x}\wedge |z| \Big)\,\nu(dz) =0.
 \eeqnn
 Substituting this into the right-hand side of \eqref{eqn.514} induces that as $x\to\infty$,
 \beqnn
 \varPsi^{(\vartheta_x)}(1/\vartheta_x) \leq - \frac{\beta}{\vartheta_x}  + o\big(1/\vartheta_x\big). 
 \eeqnn
 This indicates that $\varPsi^{(\vartheta_x)}(1/\vartheta_x)<0$ for all sufficiently large $x$ and then for any $0\leq s\leq t$,
 \beqnn
 \mathbf{E}_y\Big[ \exp \big\{  X_t^{(\vartheta_x)}/\vartheta_x\big\}\,\big|\,\mathscr{F}_s\Big] 
 =\exp \big\{  X_s^{(\vartheta_x)}/\vartheta_x\big\} \cdot \exp\Big\{(t-s)\cdot \varPsi^{(\vartheta_x)}(1/\vartheta_x)  \Big\} 
 \leq  \exp \big\{  X_s^{(\vartheta_x)}/\vartheta_x\big\}.
 \eeqnn
 Hence our claim \eqref{eqn.513} holds. 
 \qed

 \begin{proposition}\label{Prop.501}
 	Let $\varPi(dy)$ denote the jump measure associated with the ascending ladder height process $H$.   
 	For any constant $\delta>0$, we have as $x\to\infty$,
 	\beqlb\label{eqn.5001}
 	\frac{\varPi \big(x+\varDelta_\delta\big) }{\overline{\nu}(x)}\to  \frac{\delta}{\beta \underline{n}(\zeta)}  .
 	\eeqlb
 \end{proposition}
 \proof
 By Theorem 7.8 of \cite[p.213]{Kyprianou2014},  we first have 
 \beqnn
 \varPi \big(x+\varDelta_\delta\big)  
 \ar=\ar\int_{0-}^\infty \nu \big(x+z+\varDelta_\delta\big)\,d\widehat{V}(z).
 \eeqnn
 By \eqref{eqn.30}, there exists a non-negative function $\rho$ on $\mathbb{R}_+$ that decreases to $0$ at infinity such that, for all $x>0$ and $z\geq 0$, 
 \beqnn
 (1-\rho_x) \cdot \frac{\overline\nu(x+z)}{x+z} \cdot \alpha \delta
 \leq  \nu \big(x+z+\varDelta_\delta\big) \leq (1+\rho_x) \cdot \frac{\overline\nu(x+z)}{x+z} \cdot \alpha \delta.
 \eeqnn
 Consequently, 
 \beqlb\label{eqn.3013} 
 (1-\rho_x)\cdot \frac{\alpha \delta }{\overline\nu(x)} \int_{0-}^\infty
 \frac{\overline\nu(x+z)}{x+z}\, d\widehat{V}(z) 
 \leq \frac{\varPi \big(x+\varDelta_\delta\big)}{\overline\nu(x)} \leq (1+\rho_x)\cdot \frac{\alpha \delta }{\overline\nu(x)} \int_{0-}^\infty
 \frac{\overline\nu(x+z)}{x+z}\, d\widehat{V}(z).
 \eeqlb
 For any integer $K\geq 1$, we decompose the last integral at the point $K$:
 \beqnn 
 \frac{\alpha \delta }{\overline\nu(x)} \int_{0-}^\infty
 \frac{\overline\nu(x+z)}{x+z}\, d\widehat{V}(z)
 =\frac{\alpha \delta }{\overline\nu(x)} \int_{0-}^K \frac{\overline\nu(x+z)}{x+z} \,d\widehat{V}(z)  +
 \frac{\alpha \delta }{\overline\nu(x)}\int_K^\infty \frac{\overline\nu(x+z)}{x+z}\, d\widehat{V} (z).
 \eeqnn
 For the first term on the right-hand side, by the monotonicity of $\overline\nu$, we have  as $x\to\infty$,
 \beqlb \label{eqn.502}
 \frac{\alpha \delta }{\overline\nu(x)} \int_{0-}^K \frac{\overline\nu(x+z)}{x+z} \,d\widehat{V}(z)  
 \leq \int_{0-}^K \frac{ \alpha \delta }{x+z} \,\widehat{V}(dz) \leq \frac{ \alpha \delta  }{x}\cdot \widehat{V}(K) \to 0. 
 \eeqlb
 For the second term,  using the monotonicity of $\overline\nu(y)$ twice again in the next first and third inequalities gives that
 \beqnn
 \frac{\alpha \delta }{\overline\nu(x)} \int_K^\infty \frac{\overline\nu(x+z)}{x+z}\,d\widehat{V}(z)
 \ar=\ar  \frac{\alpha \delta }{\overline\nu(x)} \cdot \sum_{k=K}^{\infty} \int_k^{k+1} \frac{\overline\nu(x+z)}{x+z}\,d\widehat{V}(z) \cr
 \ar\leq\ar
 \frac{\alpha \delta }{\overline\nu(x)} \cdot\sum_{k=K}^{\infty}
 \frac{\overline\nu(x+k)}{x+k} \big( \widehat{V}(k+1)-\widehat{V}(k) \big)\cr
 \ar\leq\ar \sup_{j\geq K} \big( \widehat{V}(j+1)-\widehat{V}(j) \big)\cdot \frac{\alpha \delta }{\overline\nu(x)} \sum_{k=K}^{\infty}
 \frac{\overline\nu(x+k)}{x+k}\cr
 \ar\leq\ar \sup_{j\geq K} \big( \widehat{V}(j+1)-\widehat{V}(j) \big)\cdot \frac{\alpha \delta }{\overline\nu(x)}   \int_{K-1}^\infty \frac{\overline{\nu}(x+z)}{x+z} \,dz .
 \eeqnn 
 By Karamata's theorem; see e.g. Proposition 1.5.10.  in \cite[p.27]{BinghamGoldieTeugels1987}, we have as $x\to\infty$,
 \beqnn
 \frac{\alpha \delta }{\overline\nu(x)}   \int_{K-1}^\infty \frac{\overline{\nu}(x+z)}{x+z} \,dz 
 =\frac{\alpha \delta }{\overline\nu(x)}   \int_{x+K-1}^\infty \frac{\overline{\nu}(z)}{z} \,dz 
 \sim \delta\cdot \frac{\overline{\nu}(x+K-1)}{\overline{\nu}(x)} \to  \delta.
 \eeqnn 
 Furthermore, Theorem 5.1 in \cite[p.117]{Kyprianou2014} yields that
 \beqnn
 \lim_{k\to\infty}\big(\widehat{V}(k+1)-\widehat{V}(k)\big) = \frac{1}{\beta \underline{n}(\zeta)}
 \quad \mbox{and hence}\quad
 \lim_{K\to\infty}\sup_{j\geq K} \big( \widehat{V}(j+1)-\widehat{V}(j) \big) =\frac{1}{\beta \underline{n}(\zeta)}.
 \eeqnn
 Combining the preceding estimates together gives that
 \beqnn
 \limsup_{K\to\infty}\limsup_{x\to\infty}  \frac{\alpha \delta }{\overline\nu(x)}  \int_K^\infty \frac{\overline\nu(x+z)}{x+z}\, d\widehat{V} (z) \leq \frac{\delta}{\beta \underline{n}(\zeta)}.
 \eeqnn
 On the other hand, a completely analogous argument gives the matching lower bound:
 \beqnn
 \liminf_{K\to\infty} \liminf_{x\to\infty}  \frac{\alpha \delta }{\overline\nu(x)} \int_K^\infty \frac{\overline\nu(x+z)}{x+z}\, d\widehat{V} (z) \geq \frac{\delta}{\beta \underline{n}(\zeta)}.
 \eeqnn
 Therefore,
 \beqnn
 \lim_{x\to\infty}\frac{\alpha \delta }{\overline\nu(x)} \int_{0-}^\infty
 \frac{\overline\nu(x+z)}{x+z}\, d\widehat{V}(z)
 = \frac{\delta}{\beta \underline{n}(\zeta)}. 
 \eeqnn
 Combining this limit with \eqref{eqn.502} and then substituting into \eqref{eqn.3013}, and finally letting $x\to\infty$ (and recalling that $\rho_x\to0$), yields the desired limit \eqref{eqn.5001}. 
 \qed

 \begin{proposition}\label{Prop.502}
 	For any $\delta>0$, there exists a constant $C>0$ such that, for all $x\geq 0$,
 	\beqnn
 	V(x+\delta)-V(x)\leq C\cdot \big( \overline{\nu}(x)  \wedge 1\big). 
 	\eeqnn 
 \end{proposition}
 \proof By the monotonicity of $V$ and $\overline{\nu}$, it suffices to prove the bound for all sufficiently large $x$. 
 Let $\xi$ be a subordinator with drift $\underline{a}$ and L\'evy measure $\varPi$. 
 Then $H$ is equal in law to $\xi$ killed at rate $q:=\kappa(0,0)>0$, and $V$ is the distribution function of the $q$-resolvent measure of $\xi$, i.e., 
 \beqnn
 V(x) = \int_0^\infty e^{-q t}\, \mathbf{P}\big( \xi_t \leq x \big)\,dt,\quad x\geq 0.
 \eeqnn
 Consequently, 
 \beqnn
 V(x+\delta)-V(x) = \int_0^\infty e^{-q t}\, \mathbf{P}\big( \xi_t  \in x+\varDelta_\delta \big)\,dt.
 \eeqnn
 The L\'evy-It\^o decomposition theorem allows us to split $\xi$ into  two independent subordinators: 
 \beqnn
 \xi^{(1)}_t:=\underline{\tt a}\cdot t+\int_0^t \int_0^1y\, N_H(ds, dy)
 \quad\text{and}\quad
 \xi^{(2)}_t:=\int_0^t \int_1^\infty y\, N_H(ds, dy),
 \eeqnn
 where $N_H(ds,dz)$ is a Poisson random measure on $(0,\infty)^2$ with intensity $ds\,\varPi(dz)$.  
 The law of total probability induces that $\mathbf{P}\big( \xi_t  \in x+\varDelta_\delta \big) \leq 
 \mathbf{P}\big( \xi^{(1)}_t> x/2 \big) + \mathbf{P}\big(  \xi^{(1)}_t + \xi^{(2)}_t  \in x+\varDelta_\delta ,  \xi^{(1)}_t\leq x/2\big) $
 and 
 \beqlb \label{eqn.503}
 V(x+\delta)-V(x) 
 \leq \int_0^\infty e^{-q t}\,   \mathbf{P}\big( \xi^{(1)}_t> x/2 \big)\,dt       
 + \int_0^\infty e^{-q t}\,\mathbf{P}\big(  \xi^{(1)}_t + \xi^{(2)}_t  \in x+\varDelta_\delta ,  \xi^{(1)}_t\leq x/2\big)\,dt. 
 \eeqlb
 We now show that any of the two integrals on the right-hand side is bounded by $C\cdot \overline{\nu}(x)$ uniformly for large $x$:
 
 \textbf{(i)} By Theorem~25.17 in \cite[p.165]{Sato1999},   $\xi^{(1)}$ has finite exponential moments of any order, i.e.,
 \beqnn
 \varPhi^{(1)}(\lambda) := \log \mathbf{E}\Big[ \exp\big\{ \lambda \cdot \xi^{(1)}_1 \big\}  \Big]<\infty, \quad \forall \,\lambda \in \mathbb{R}. 
 \eeqnn
 Since $ \varPhi^{(1)}$ is smooth and $\varPhi^{(1)}(0)=0$, we may choose  $\lambda_0>0$ such that $\varPhi^{(1)}(\lambda_0)<q $.   Chebyshev's inequality then gives
 \beqnn
 \mathbf{P}\big(\xi^{(1)}_t>x/2\big) \leq \exp\big\{-\lambda_0x/2 \big\} \cdot \mathbf{E}\big[ \exp\{\lambda_0\cdot \xi^{(1)}_t \} \big] = \exp\big\{-\lambda_0x/2 + t\cdot \varPhi^{(1)}(\lambda_0)\big\} . 
 \eeqnn
 Therefore, for some constant $C>0$, 
 \beqnn
 \int_0^\infty e^{-qt}\mathbf{P}\big(\xi^{(1)}_t>x/2\big)\,dt
 \leq e^{-\lambda_0 x/2}\int_0^\infty \exp\big\{(\varPhi^{(1)}(\lambda_0)-q)\cdot t\big\} \,dt 
 \leq  C \cdot \exp\big\{-\lambda_0 x/2 \big\},
 \eeqnn
 which is clearly $o\big(\overline{\nu}(x)\big)$ as $x\to\infty$. 
 
 \textbf{(ii)} For the second integral in \eqref{eqn.503}, we first note that
 \beqnn	
 \mathbf{P}\big(  \xi^{(1)}_t + \xi^{(2)}_t  \in x+\varDelta_\delta ,  \xi^{(1)}_t\leq  x/2\big)  
 \leq \sup_{z\in(0,x/2]}\mathbf{P}\big(    \xi^{(2)}_t  \in x-z+\varDelta_\delta \big)
 =\sup_{z\in[x/2,x)}\mathbf{P}\big(    \xi^{(2)}_t  \in z+\varDelta_\delta \big).
 \eeqnn
 Since $\gamma:= \varPi\big((1,\infty)\big)<\infty$, the subordinator $\xi^{(2)}$  is a compound Poisson process: 
 \beqnn
 \xi^{(2)}_t\overset{\rm d}= S_{N_t},\quad t\geq 0,
 \eeqnn
 where $N$  is a Poisson process with  rate $\gamma$, and $S$ is a random walk with i.i.d. steps $\{ Y_j \}_{j\geq 1}$ having distribution $	G(dz):=\mathbf{1}_{\{z>1\}}\cdot \gamma^{-1}\cdot  \varPi(dz) $. 
 Thus, 
 \beqnn
 \mathbf{P}\big(    \xi^{(2)}_t  \in z+\varDelta_\delta \big) 
 \ar=\ar e^{-\gamma t} \cdot\sum_{n=1}^\infty \frac{(\gamma t)^n}{n!}\cdot  \mathbf{P}\big( S_n \in z+\varDelta_\delta \big),\quad z\geq 0.
 \eeqnn
 By Proposition~\ref{Prop.501}, we have  $G(\cdot+\varDelta_\delta)\in\mathrm{RV}^\infty_{-\alpha}$. 
 Hence the law $G$ is $\varDelta_\delta$-subexponential by Proposition~9 in \cite{AsmussenFossKorshunov2003}. 
 For any $\varepsilon\in(0,q/\gamma)$,  Proposition~4 in \cite{AsmussenFossKorshunov2003} tells that 
 there exist two constants $C_0,z_0>0$ such that
 for all $z\geq z_0$ and $n\geq1$,
 \beqnn
 \mathbf{P}\big(S_n\in z+\varDelta_\delta\big)
 \leq C_0\cdot(1+\varepsilon)^n
 \cdot G\big(z+\varDelta_\delta\big)
 \leq C_0\cdot(1+\varepsilon)^n
 \cdot\overline\nu(z).
 \eeqnn
 Consequently, uniformly for $x\geq 2z_0$ and $t>0$,  
 \beqnn
 \mathbf{P}\big(  \xi^{(1)}_t + \xi^{(2)}_t  \in x+\varDelta_\delta ,  \xi^{(1)}_t\leq  x/2\big)  
 \ar\leq\ar \sup_{z\in [x/2,x)}\mathbf{P}\big(    \xi^{(2)}_t  \in z+\varDelta_\delta \big) \cr\cr
 \ar\leq \ar C_0\cdot \overline{\nu}(x/2)\cdot e^{-\gamma t} \cdot   \sum_{n=1}^\infty \frac{(\gamma t)^n}{n!}\cdot    (1+\varepsilon)^n
 \leq C\cdot e^{\varepsilon\gamma t } \cdot \overline{\nu}(x).
 \eeqnn
 The last inequality follows from the Taylor expansion of the exponential function and the fact that $\overline{\nu}(x/2)/\overline{\nu}(x)\to 2^{\alpha}$ as $x\to\infty$. 
 Substituting this estimate into the second integral in \eqref{eqn.503} yields, uniformly for $x\geq 2z_0$, 
 \beqnn
 \int_0^\infty e^{-q t}\,\mathbf{P}\big(  \xi^{(1)}_t + \xi^{(2)}_t  \in x+\varDelta_\delta ,  \xi^{(1)}_t\leq x/2\big)\,dt
 \ar\leq\ar C\cdot \overline{\nu}(x) \int_0^\infty e^{\varepsilon\gamma t-q t}\, dt \leq C\cdot \overline{\nu}(x) ,
 \eeqnn
 since $\varepsilon\gamma<q$.  
 \qed

 \begin{proposition}\label{Prop.504}
 	For any $\delta>0$, there exists a constant $C>0$ such that for all $x \geq 0$,
 	\beqlb\label{eqn.504}
 	\sup_{y\leq0} \mathbf P\big( \overline{X}_{\tau_{y}^-}\in x+\varDelta_\delta\big)
 	\leq C\cdot \overline\nu(x).
 	\eeqlb
 \end{proposition}
 \proof 
 Let $T_{x,\delta}^H:= \inf\{ t\geq 0: H_t\in x+\varDelta_\delta \}$ be the first entrance time of  the ascending ladder height process $H$ into the interval $x+\varDelta_\delta$ with $\inf\emptyset=\infty$ for convenience.  
 The monotonicity of $H$ shows that $\overline{X}_{\tau_{y}^-}\in x+\varDelta_\delta$ for some $y\leq0$ implies  $T_{x,\delta}^H<\infty$.
 This induces that 
 \beqnn
 \sup_{y\leq 0} \mathbf{P}\big( \overline{X}_{\tau_{y}^-}\in x+\varDelta_\delta\big) 
 \leq  \mathbf{P}\big(\cup_{y\leq 0} \big\{ \overline{X}_{\tau_{y}^-}\in x+\varDelta_\delta \big\}\big) 
 \leq \mathbf{P}\big( T_{x,\delta}^H<\infty\big).
 \eeqnn 
 Thus it suffices to prove that the last probability can be bounded by $C\cdot \overline{\nu}(x)$ uniformly for all $x\geq 0$.  
 By the defining formula \eqref{eqn.DefV} for $V$, we have
 \beqnn
 V(x+\delta+1)-V(x)= \mathbf{E}\bigg[\int_0^\infty \mathbf{1}_{\{H_t \in x+\varDelta_{\delta+1} \}} dt\bigg] 
 \geq \mathbf{E}\bigg[ \mathbf{1}_{\{T_{x,\delta}^H<\infty \}}\cdot  \int_{T_{x,\delta}^H}^\infty \mathbf{1}_{\{H_t \in x+\varDelta_{\delta+1} \}} dt\bigg] . 
 \eeqnn 
 Conditioned on  $T_{x,\delta}^H<\infty$,  the monotonicity of $H$ implies that $ \{H_t-H_{T_{x,\delta}^H}\leq 1\} \subset \{H_t\in x+\varDelta_{\delta+1}\}$ for all $t\geq T_{x,\delta}^H$.
 Applying the strong Markov property at $T_{x,\delta}^H$ and then the independent increments of $H$ to the last expectation,
 \beqnn
 V(x+\delta+1)-V(x) 
 \ar\geq\ar  \mathbf{E}\bigg[ \mathbf{1}_{\{T_{x,\delta}^H<\infty \}}\cdot  \int_{T_{x,\delta}^H}^\infty \mathbf{1}_{\{H_t-H_{T_{x,\delta}^H}\leq 1 \}} dt\bigg] \cr\cr
 \ar=\ar \mathbf{P}\big(  T_{x,\delta}^H<\infty\big) \cdot \mathbf{E}\bigg[  \int_0^\infty \mathbf{1}_{\{H_t \leq 1 \}} dt\bigg] =   \mathbf{P}\big(  T_{x,\delta}^H<\infty\big) \cdot V(1).
 \eeqnn 
 In conclusion, we have 
 \beqnn
 \sup_{y\leq 0} \mathbf{P}\big( \overline{X}_{\tau_{y}^-}\in x+\varDelta_\delta\big)  
 \leq \frac{V(x+\delta+1)-V(x) }{V(1)},
 \eeqnn
 and then the desired uniform upper bound \eqref{eqn.504} follows directly from Proposition~\ref{Prop.502}. 
 \qed

 \begin{proposition}\label{Prop.506}
 	For any $\delta>0$, there exists a constant $C>0$ such that, for all $x\geq 0$,
 	\beqnn
 	\int_0^\infty \underline{n}_s\big( x+\varDelta_\delta \big)\, ds \leq C\cdot\big( \overline{\nu}(x) \wedge 1\big).
 	\eeqnn
 \end{proposition}
 \proof Since $\underline{n}_s\big(x+\varDelta_\delta \big) \leq \underline{n}(\zeta>s)$ for all $x\geq0$, by \eqref{eqn.501} we have the trivial uniform bound
 \beqlb\label{eqn.3033}
 \int_0^\infty \underline{n}_s\big( x+\varDelta_\delta \big)\, ds 
 \leq \int_0^\infty \underline{n}(\zeta>s)\, ds =\underline{n}(\zeta)<\infty .
 \eeqlb
 Thus it remains to show that the integral is $O(\overline{\nu}(x))$ as $x\to\infty$.
 Define the stopping time $T_{x}^\delta := \inf\big\{ 0<s<\zeta: \epsilon_s\in x+\varDelta_\delta \big\}$ with the convention $\inf \emptyset =\infty$. 
 By the strong Markov property at $T_{x}^\delta $,
 \beqnn
 \int_0^\infty \underline{n}_s\big( x+\varDelta_\delta \big)\, ds
 \ar=\ar \underline{n}\bigg(\int_{T_{x}^\delta }^{\zeta}
 \mathbf{1}_{\{\epsilon_s\in x+\varDelta_\delta\}}\,ds,\,
 \zeta>T_{x}^\delta  \bigg) \cr
 \ar=\ar
 \underline{n}\bigg(
 \mathbf{1}_{\{T_{x}^\delta <\zeta\}}\cdot
 \mathbf{E}_{\epsilon_{T_{x}^\delta }}\Big[ \int_0^{\tau_0^-} \mathbf{1}_{\{X_r\in  x+\varDelta_\delta\}}\,dr
 \Big] \bigg) \cr 
 \ar\leq\ar \sup_{y\in x+\varDelta_\delta}\mathbf{E}_y\bigg[ \int_0^\infty \mathbf{1}_{\{X_r\in  x+\varDelta_\delta\}}\,dr
 \bigg] \cdot \underline{n}\big(  T_{x}^\delta <\zeta\big).
 \eeqnn 
 By spatial homogeneity of the Lévy process $X$,  
 \beqnn
 \sup_{y\in x+\varDelta_\delta}\mathbf{E}_y\bigg[ \int_0^\infty \mathbf{1}_{\{X_r\in  x+\varDelta_\delta\}}\,dr
 \bigg] 
 \ar\leq\ar \int_0^\infty \sup_{y\in x+\varDelta_\delta}\mathbf{P}_y\big(X_r\in  x+\varDelta_\delta\big)\,dr 
 \leq \int_0^\infty  \mathbf{P}\big(X_r\in   [-\delta,\delta]\big)\,dr .
 \eeqnn
 The last integral is finite, since $X_t\overset{\rm a.s.}\to-\infty$  and hence its potential measure is finite on compact sets; see, e.g., \cite[Chapter~I.4]{Bertoin1996}.  
 It remains to estimate $\underline{n}(T_x^\delta<\zeta)$. On the event $\{T_x^\delta<\infty\}$, we have $\overline{\epsilon}>x$. 
 Therefore, the asymptotic property \eqref{CWX01} induces that as $x\to\infty$,
 \beqnn
 \underline{n}\big(\zeta> T_{x}^\delta\big) \leq \underline{n}(\overline{\epsilon}>x) \sim \underline{n}(\zeta) \cdot \overline{\nu}(x)
 \quad \mbox{and then}\quad 
 \int_0^\infty \underline{n}_s\big( x+\varDelta_\delta \big)\, ds =O\big(\overline{\nu}(x)\big) . 
 \eeqnn 
 Together with the uniform bound \eqref{eqn.3033}, this gives the desired estimate. 
 \qed
 
 \medskip
 
 \textbf{\textit{Proof of Lemma~\ref{Lemma.501}.}}  Without loss of generality, we may assume that $K$ is a positive integer. 
 Given a constant  $\eta\in(1/\alpha,1)$, one can easily identify that 
 \beqlb\label{eqn.50122}
 x\cdot \overline{\nu}\big(x^\eta\big) \to 0,\quad \mbox{as $x\to\infty$.}
 \eeqlb 
 We introduce three classes of large jumps, distinguished according to the behavior of the excursion before and after the jump: 
 \beqnn
 \mathcal{A}_{\delta, \eta } (x)\ar:=\ar \Big\{0<s\leq\zeta:  \Delta\epsilon_s> x^\eta, \, \sup_{s\leq r<\zeta} \epsilon_{r}  \in x+ \varDelta_\delta \Big\},\cr
 \mathcal{A}_{\delta, \eta }^{(1)} (x,K)
 \ar:=\ar \Big\{0<s\leq\zeta:   \Delta\epsilon_s> x^\eta,\,  \sup_{s\leq r<\zeta} \epsilon_{r}  \in x+\varDelta_\delta ,\, \sup_{s\leq r<\zeta} \epsilon_{r}-\epsilon_s> K\Big\},\cr
 \mathcal{A}_{\delta, \eta }^{(2)} (x,K)
 \ar:=\ar \Big\{0<s\leq\zeta: \Delta\epsilon_s>x^\eta,\, \epsilon_{s-}>K,\, \sup_{s\leq r<\zeta} \epsilon_{r} 
 \in x+\varDelta_\delta,\, \sup_{s\leq r<\zeta} \epsilon_{r}-\epsilon_s\leq K \Big\}. 
 \eeqnn 
 In view of \eqref{eqn.5018}, for sufficiently large $x$, it is immediate from the definitions that $\mathcal{A}_{\delta}(x,K)$, $\mathcal{A}_{\delta, \eta }^{(1)} (x,K)$, $ \mathcal{A}_{\delta, \eta }^{(2)} (x,K)$ are mutually disjoint and together form a partition of $ \mathcal{A}_{\delta, \eta } (x)$.
 Consequently,  
 \beqnn
 \big\{\mathcal{A}_{\delta}(x,K) =\emptyset  \big\}
 = \big\{ \mathcal{A}_{\delta, \eta } (x) =\emptyset  \big\} \cup \big\{ \mathcal{A}_{\delta, \eta } (x)=\mathcal{A}_{\delta, \eta }^{(1)} (x,K)\cup \mathcal{A}_{\delta, \eta }^{(2)} (x,K) \neq  \emptyset \big\} . 
 \eeqnn 
 Then $ \underline{n}\big(\overline{\epsilon}\in x+\varDelta_\delta,   \mathcal{A}_{\delta}(x,K)= \emptyset\big) $ can be decomposed into the following two terms:
 \beqlb\label{eqn.5011}
 \underline{n}\big(\overline{\epsilon}\in x+\varDelta_\delta, \mathcal{A}_{\delta, \eta } (x)= \emptyset\big) 
 \quad \mbox{and}\quad \underline{n}\big(\overline{\epsilon}\in x+\varDelta_\delta, \mathcal{A}_{\delta, \eta }(x)=\mathcal{A}_{\delta, \eta }^{(1)}(x,K)\cup \mathcal{A}_{\delta, \eta }^{(2)}(x,K) \neq  \emptyset\big) . 
 \eeqlb
 
 \textbf{(i)}  For the first term, conditionally on $\overline{\epsilon}\in x+\varDelta_\delta $ and $ \mathcal{A}_{\delta, \eta }(x)= \emptyset$, we observe that $\tau^+_x<\zeta $ and $\Delta \epsilon_s \leq x^\eta $ for all $s\leq \tau^+_x$.
 This induces that 
 \beqnn
 \underline{n}\big(\overline{\epsilon}\in x+\varDelta_\delta,   \mathcal{A}_{\delta, \eta }(x)= \emptyset\big)  
 \leq \underline{n}\Big( \tau^+_x<\zeta ,\sup_{s\leq\tau_x^+} \Delta \epsilon_s \leq x^\eta \Big) . 
 \eeqnn
 By using the strong Markov property at the passage time $\tau_1^+$, 
 \beqnn
 \underline{n}\Big( \tau^+_x<\zeta ,\sup_{s\leq\tau_x^+} \Delta \epsilon_s \leq x^\eta \Big) 
 \ar=\ar \underline{n}\Big( \mathbf{P}_{\epsilon_{\tau_1^+}} \Big( \tau^+_x<\tau_0^-, \sup_{s\leq\tau_x^+} \Delta X_s \leq x^\eta \Big) ,\tau^+_1<\zeta ,\sup_{s\leq\tau_1^+} \Delta \epsilon_s \leq x^\eta \Big)\cr
 \ar\leq\ar \underline{n}\big( \tau^+_1<\zeta  \big) \cdot  \mathbf{P}_{x^\eta+1} \Big( \tau^+_x<\tau_0^-, \sup_{s\leq\tau_x^+} \Delta X_s \leq x^\eta \Big).
 \eeqnn
 Here, the inequality follows from the fact that the probability is non-decreasing as a function of $\epsilon_{\tau_1^+}$, together with the bound $\epsilon_{\tau_1^+} \leq x^\eta+1$.  
 By using the fact that $\underline{n}\big( \tau^+_1<\zeta  \big)<\infty$ and then Proposition~\ref{Prop.503}, we have as $x\to\infty$,
 \beqlb\label{eqn.5015}
 \frac{  \underline{n}\big(\overline{\epsilon}\in x+\varDelta_\delta,   \mathcal{A}_{\delta,\eta}(x)= \emptyset\big)
 }{\nu\big(x+ \varDelta_\delta\big)} 
 \ar\leq\ar C\cdot  \frac{\exp\big\{ - ( x- x^\eta-1 )/x^\eta\big\} }{\nu\big(x+ \varDelta_\delta\big)} \to0.
 \eeqlb
 
 \textbf{(ii)}   We now consider the second term in \eqref{eqn.5011}. Obviously, it can be bounded by  
 \beqnn
 \underline{n}\big( \mathcal{A}_{\delta, \eta }^{(1)}(x,K) \neq  \emptyset\big)  + \underline{n}\big( \mathcal{A}_{\delta, \eta }^{(2)}(x,K) \neq  \emptyset\big) 
 \leq \underline{n}\big(\# \mathcal{A}_{\delta, \eta }^{(1)}(x,K) \big)  + \underline{n}\big(\# \mathcal{A}_{\delta, \eta }^{(2)}(x,K) \big).
 \eeqnn
 Here, this inequality follows from the inequality $\mathbf{1}_{\{\mathcal{A}_{\delta, \eta }^{(i)}(x,K) \neq  \emptyset\}} \leq \# \mathcal{A}_{\delta, \eta }^{(i)}(x,K)  $ for $i=1,2$. 
 Using the Markov property and then the spatial homogeneity of $X$ to the first term on the right-hand side,
 \beqlb\label{eqn.301111}
 \underline{n}\big(\#\mathcal{A}_{\delta, \eta }^{(1)}(x,K)\big)
 \ar=\ar \underline{n}\bigg(\sum_{0<s\leq \zeta} \mathbf{1}_{\{ \Delta\epsilon_s>x^\eta,\,  \sup_{s\leq r<\zeta} \epsilon_{r}  \in x+\varDelta_\delta ,\, \sup_{s\leq r<\zeta} \epsilon_{r}-\epsilon_s> K \}}  \bigg)\cr
 \ar=\ar  \underline{n}\bigg(\sum_{0<s\leq \zeta} \mathbf{1}_{\{ \Delta\epsilon_s>x^\eta\}}\, \mathbf{P}_{\epsilon_s}\Big(  \overline{X}_{\tau_0^-}  \in x+\varDelta_\delta ,\, \overline{X}_{\tau_0^-}  -\epsilon_s> K \Big)   \bigg)\cr
 \ar=\ar  \underline{n}\bigg(\sum_{0<s\leq \zeta} \mathbf{1}_{\{ \Delta\epsilon_s>x^\eta\}}\, \mathbf{P}\Big(     \overline{X}_{\tau_{-\epsilon_s}^-} +\epsilon_s \in x+\varDelta_\delta ,\, \overline{X}_{\tau_{-\epsilon_s}^-}> K \Big)   \bigg) .
 \eeqlb
 By the compensation formula \eqref{eqn.Compensation},  the last expression equals 
 \beqlb\label{eqn.301112}
 \lefteqn{\int_0^\infty ds \int_0^\infty \underline{n}_s(dz) \int_{x^\eta}^\infty \mathbf{P}\Big(  \overline{X}_{\tau_{-z-w}^-}  \in x-z-w+\varDelta_\delta ,\,  \overline{X}_{\tau_{-z-w}^-}  > K \Big)\,\nu(dw)}\ar\ar\cr
 \ar=\ar\int_0^\infty ds \int_0^\infty \underline{n}_s(dz) \int_{x^\eta}^\infty  \sum_{j=K}^\infty\mathbf{P}\Big(   w  \in x-z- \overline{X}_{\tau_{-z-w}^-}  +\varDelta_\delta ,\,  \overline{X}_{\tau_{-z-w}^-}   \in j +\varDelta_1 \Big)\,\nu(dw).
 \eeqlb
 Note that $x -z- \overline{X}_{\tau_{-z-w}^-}  +\varDelta_\delta \subset x-z-j-1+\varDelta_{1+\delta}$ whenever $ \overline{X}_{\tau_{-z-w}^-}   \in j +\varDelta_1 $. 
 Consequently,  
 \beqlb\label{eqn.501331}
 \underline{n}\big(\# \mathcal{A}_{\delta, \eta }^{(1)} (x,K)\big)
 \ar\leq\ar  \int_0^\infty ds \int_0^\infty \underline{n}_s(dz) \sum_{j=K}^\infty \int_{(x-z-j-1+\varDelta_{1+\delta})\cap(x^\eta,\infty)} \mathbf{P}\big(\overline{X}_{\tau_{-z-w}^-} \in j +\varDelta_1 \big)\,\nu(dw) . \qquad
 \eeqlb
 Applying Proposition~\ref{Prop.504} to the integrand probability, there exists a constant $C>0$ depending only on $\delta$ such that 
 \beqlb\label{eqn.50133}
 \underline{n}\big(\# \mathcal{A}_{\delta, \eta }^{(1)}(x,K)\big) 
 \leq C  \int_0^\infty \sum_{j=K}^\infty \overline{\nu}(j)\cdot \nu\big(  (x-z-j-1+\varDelta_{1+\delta})\cap(x^\eta,\infty)\big) \,  \int_0^\infty  \underline{n}_s(dz) \,ds.
 \eeqlb
 For any  $\varepsilon\in(0,1/2)$, the last double integral can be bounded by the sum of the following three terms: 
 \beqnn
 I_1(x,K,\varepsilon)\ar:=\ar    \int_0^\infty \sum_{j=K}^\infty \mathbf{1}_{\{z+j\leq \varepsilon x\}}\cdot \overline{\nu}(j)\cdot \nu\big(  (x-z-j-1+\varDelta_{1+\delta})\cap(x^\eta,\infty)\big) \,  \int_0^\infty  \underline{n}_s(dz) \,ds,\cr
 I_2(x,K,\varepsilon)\ar:=\ar    \int_0^\infty \sum_{j=K}^\infty \mathbf{1}_{\{j> \varepsilon x/2\}}\cdot \overline{\nu}(j)\cdot \nu\big(  (x-z-j-1+\varDelta_{1+\delta})\cap(x^\eta,\infty)\big) \,  \int_0^\infty  \underline{n}_s(dz) \,ds,\cr
 I_3(x,K,\varepsilon)\ar:=\ar    \int_0^\infty \sum_{j=K}^\infty \mathbf{1}_{\{z>\varepsilon x/2\geq j\}}\cdot \overline{\nu}(j)\cdot \nu\big(  (x-z-j-1+\varDelta_{1+\delta})\cap(x^\eta,\infty)\big) \,  \int_0^\infty  \underline{n}_s(dz) \,ds.
 \eeqnn
 For $  I_1(x,K,\varepsilon)$, by using the equality $\underline{n}_s(\mathbb{R}_+)= \underline{n}(\zeta>s)$, then  \eqref{eqn.501}  and finally Lemma~\ref{Lemma.301}, we have uniformly for sufficiently large $x$, 
 \beqlb\label{eqn.5013}
 I_1(x,K,\varepsilon) 
 \ar\leq\ar \sum_{j=K}^\infty   \overline{\nu}(j)\cdot \sup_{0<y\leq \varepsilon x}\nu\big(x-y-1+\varDelta_{1+\delta}\big) \,  \int_0^\infty ds \int_0^\infty  \underline{n}_s(dz)  \cr
 \ar=\ar  \underline{n}(\zeta) \cdot \sum_{j=K}^\infty   \overline{\nu}(j)\cdot \sup_{0<y\leq \varepsilon x}\nu\big(x-y-1+\varDelta_{1+\delta}\big)  
 \leq C   \cdot  \nu\big(x+ \varDelta_{1+\delta}\big) \cdot \sum_{j=K}^\infty   \overline{\nu}(j).  
 \eeqlb
 For $  I_2(x,K,\varepsilon)$,  by the monotonicity of $\overline\nu$,  
 \beqnn
 I_2(x,K,\varepsilon) 
 \ar\leq\ar    \overline{\nu}\big(\varepsilon x/2\big)\cdot \int_0^\infty \sum_{j=K}^\infty \nu\big(  (x-z-j-1+\varDelta_{1+\delta})\cap(x^\eta,\infty)\big) \,  \int_0^\infty  \underline{n}_s(dz) \,ds \cr
 \ar=\ar  \overline{\nu}\big(\varepsilon x/2\big)\cdot \int_0^\infty   \int_0^\infty  \underline{n}_s(dz) \,ds \int_0^\infty \sum_{j=K}^\infty \mathbf{1}_{\{y\in (x-z-j-1+\varDelta_{1+\delta})\cap(x^\eta,\infty)\}}\, \nu(dy). 
 \eeqnn   
 Note that   $\sum_{j=K}^\infty  \mathbf{1}_{\{y\in (x-z-j-1+\varDelta_{1+\delta})\cap(x^\eta,\infty)\}} \leq C\cdot  \mathbf{1}_{\{y\in (x^\eta,\infty)\}}$ for some $C>0$ and all $y\geq 0$, we have 
 \beqlb\label{eqn.5012}
 I_2(x,K,\varepsilon) \leq  C \cdot   \underline{n}(\zeta) \cdot   \overline{\nu}\big(\varepsilon x/2\big)\cdot  \overline{\nu}\big(x^\eta \big) . 
 \eeqlb
 Finally, we consider  $I_3(x,K,\varepsilon)$, which can be discretized into
 \beqnn
 I_3(x,K,\varepsilon)\ar=\ar   \sum_{j=K}^{[\varepsilon x/2]} \overline{\nu}(j)\cdot  \int_{\varepsilon x/2}^\infty   \nu\big(  (x-z-j-1+\varDelta_{1+\delta})\cap(x^\eta,\infty)\big) \,  \int_0^\infty  \underline{n}_s(dz) \,ds\cr
 \ar\leq\ar   \sum_{j=K}^{[\varepsilon x/2]} \overline{\nu}(j)\cdot  \int_{\varepsilon x/2}^\infty  \sum_{i=[x^\eta]}^\infty  \nu\big(  (x-z-j-1+\varDelta_{1+\delta})\cap (i+\varDelta_1) \big) \,  \int_0^\infty  \underline{n}_s(dz) \,ds.
 \eeqnn
 Note that $(x-z-j-1+\varDelta_{1+\delta})\cap(i+\varDelta_1) = \emptyset$ if $z\notin x -j-i-2+\varDelta_{2+\delta} $, we have 
 \beqnn
 I_3(x,K,\varepsilon)
 \ar\leq\ar \sum_{j=K}^{[\varepsilon x/2]} \overline{\nu}(j)\cdot    \sum_{i=[x^\eta]}^\infty  \nu\big(  i+\varDelta_1 \big) \,  \int_0^\infty  \underline{n}_s\big((x -j-i-2+\varDelta_{2+\delta})\cap(\varepsilon x/2,\infty)\big) \,ds.
 \eeqnn 
 By Proposition~\ref{Prop.506} and the monotonicity of $\overline{\nu}$, the last integral always can be bounded by $\overline{\nu}\big(\varepsilon x/2 \big)$. 
 The asymptotic equivalence in \eqref{eqn.30} yields that $ \nu\big(  i+\varDelta_1 \big) \leq C\cdot\overline{\nu}(i)/i $ uniformly in $i\geq [x^\eta]$ for sufficiently large $x$.  
 Consequently, there exists some constant $C>0$ such that uniformly for sufficiently large $x$,
 \beqnn
 I_3(x,K,\varepsilon)
 \ar\leq\ar C\cdot \sum_{j=K}^{[\varepsilon x/2]} \overline{\nu}(j)\cdot    \sum_{i=[x^\eta]}^\infty  \frac{\overline{\nu}(i)}{i}\cdot\overline{\nu}\big(\varepsilon x/2 \big) 
 \leq C\cdot  \sum_{j=K}^\infty\overline{\nu}(j)\cdot  \overline{\nu}(x^\eta)\cdot \overline{\nu}\big(\varepsilon x/2 \big) .
 \eeqnn
 Here the second inequality follows from the discrete version of Karamata's theorem. 
 Taking this, \eqref{eqn.5013} and  \eqref{eqn.5012} back into \eqref{eqn.50133}, and then using \eqref{eqn.30} along with \eqref{eqn.50122}, we have as $x\to\infty$, 
 \beqlb\label{eqn.5014}
 \frac{  \underline{n}\big(\#\mathcal{A}_{\delta, \eta }^{(1)} (x,K)\big)
 }{\nu\big(x+ \varDelta_\delta\big) } 
 \ar\leq\ar C   \cdot  \frac{\nu\big(x+ \varDelta_{1+\delta}\big) }{\nu\big(x+ \varDelta_\delta\big) }  \sum_{j=K}^\infty   \overline{\nu}(j)  + C\cdot   \frac{\overline{\nu}\big(\varepsilon x/2\big)}{\overline{\nu}(x)}\cdot  x\cdot \overline{\nu}\big(x^\eta \big) \to  C   \cdot    \sum_{j=K}^\infty   \overline{\nu}(j),
 \eeqlb
 which goes to $0$ as $K\to\infty$, since $\overline\nu$ is integrable over $[1,\infty)$.

 We now consider
 $\underline n\big(\#\mathcal A_{\delta,\eta}^{(2)}(x,K)\big)$.
 Similarly as in \eqref{eqn.301111}-\eqref{eqn.301112}, we have
 \beqnn
 \lefteqn{\underline n\big(\#\mathcal A_{\delta,\eta}^{(2)}(x,K)\big)
 	= \int_0^\infty ds \int_K^\infty\underline n_s(dz)
 	\int_{x^\eta}^\infty \mathbf P\Big( \overline X_{\tau_{-z-w}^-}\leq K,\, \overline X_{\tau_{-z-w}^-}+z+w \in x+\varDelta_\delta
 	\Big)\nu(dw)}\qquad \cr\cr
 \ar=\ar \sum_{j=0}^{K-1} \int_0^\infty ds  \int_K^\infty\underline n_s(dz)  \int_{x^\eta}^\infty \mathbf P\Big( \overline X_{\tau_{-z-w}^-}\in j+\varDelta_1,\,
 \overline X_{\tau_{-z-w}^-}+z+w  \in x+\varDelta_\delta \Big)\nu(dw).
 \eeqnn
 Arguing as in \eqref{eqn.501331} and then applying
 Proposition~\ref{Prop.504} yield 
 \beqlb\label{eqn.400}
 \underline n\big(\#\mathcal A_{\delta,\eta}^{(2)}(x,K)\big)
 \ar\leq\ar  \sum_{j=0}^{K-1} \int_0^\infty ds \int_K^\infty\underline n_s(dz)
 \int_{(x-z-j-1+\varDelta_{1+\delta}) \cap(x^\eta,\infty)} \mathbf{P}\Big( \overline X_{\tau_{-z-w}^-}\in j+\varDelta_1 \Big)\nu(dw)\cr
 \ar\leq\ar C\cdot \sum_{j=0}^{K-1} a_j\cdot \int_0^\infty ds
 \int_K^\infty  \nu\big(
 (x-z-j-1+\varDelta_{1+\delta})
 \cap(x^\eta,\infty)
 \big)\, \underline n_s(dz),
 \eeqlb
 where $a_0=1$ and $a_j= \overline\nu(j)$ for $j=1,\cdots, K-1$.  
 For sufficiently large $x>8K$, we split the double integral in the last sum into the following two parts:
 \beqnn
 J_{j,1}(K,x)\ar:=\ar \int_0^\infty ds \int_K^{x/2}  \nu\big( (x-z-j-1+\varDelta_{1+\delta}) \cap(x^\eta,\infty) \big)\, \underline n_s(dz),\cr
 J_{j,2}(K,x)\ar:=\ar \int_0^\infty ds \int_{x/2}^\infty  \nu\big( (x-z-j-1+\varDelta_{1+\delta}) \cap(x^\eta,\infty)\big)\, \underline n_s(dz). 
 \eeqnn
 By \eqref{eqn.30} and
 $\overline{\nu}\in \mathrm{RV}^\infty_{-\alpha}$,
 there exists a constant $C>0$, independent of $K$,
 such that for any large $K$ and  $x\geq 8K$,
 \beqnn
 \sup_{0\leq j\leq K-1}\sup_{K<z\leq x/2}
 \frac{\nu\big(x-z-j-1+\varDelta_{1+\delta}\big)}
 {\nu\big(x+\varDelta_\delta\big)}\leq
 \sup_{x/4\leq y\leq x}
 \frac{\nu\big(y+\varDelta_{1+\delta}\big)}
 {\nu\big(x+\varDelta_\delta\big)}
 \leq C,
 \eeqnn
 which implies that 
 \beqlb\label{eqn.4000}
 \sup_{0\leq j\leq K-1} \frac{ J_{j,1}(K,x)}{\nu\big(x+\varDelta_\delta\big)} \leq C\cdot  \int_0^\infty
 \underline n(\epsilon_s>K,\zeta>s)\,ds . 
 \eeqlb
 For any sufficiently large $x$, using Fubini's theorem to $J_{j,2}(K,x)$ yields that 
 \beqnn
 J_{j,2}(K,x)
 \ar=\ar \int_{x^\eta}^{x/2+\delta}\nu(dw) \int_0^\infty \underline n_s\big( (x-w-j-1+\varDelta_{1+\delta})  \cap(x/2,\infty) \big)\,ds \cr
 \ar\leq\ar  \int_{x^\eta}^{x/2+\delta}\nu(dw) \sum_{k=0}^{[\delta]+1} \int_0^\infty
 \underline n_s\big( x-w-j-1+k+\varDelta_1 \big)\,ds,
 \eeqnn
 By Proposition~\ref{Prop.506} and monotonicity of $\overline\nu$,  there exists some constant $C>0$ independent of  $K$ and $x$ such that for all $j=0,\cdots, K-1$ and all sufficiently large $x>8K$, 
 \beqnn
 J_{j,2}(K,x) \leq  C\int_{x^\eta}^{x/2+\delta}\sum_{k=0}^{[\delta]+1}  \overline\nu(x-w-j-1+k)\, \nu(dw) 
 \leq C\cdot \overline\nu\big(x/4-\delta\big) \cdot \overline\nu\big(x^\eta\big).
 \eeqnn
 Substituting this and \eqref{eqn.4000} back into the last sum in \eqref{eqn.400} yields that uniformly for all large $K$ and $x\geq 8K$, 
 \beqnn
 \frac{\underline n\big(\#\mathcal A_{\delta,\eta}^{(2)}(x,K)\big) }{\nu\big(x+\varDelta_\delta\big)} \leq
 C\Big(1+\sum_{j=1}^\infty\overline\nu(j)\Big) \cdot 
 \bigg[
 \int_0^\infty \underline n(\epsilon_s>K,\zeta>s)\,ds + \frac{ \overline\nu(x/4-\delta)\overline\nu(x^\eta) }{ \nu\big(x+\varDelta_\delta\big)}
 \bigg],
 \eeqnn 
 which tends to $0$ as $x\to\infty$ and then $K\to\infty$ by $\overline\nu\in \mathrm{RV}^\infty_{-\alpha}$ and \eqref{eqn.501}. 
 Finally, combining this together with \eqref{eqn.5015} and \eqref{eqn.5014} induces \eqref{eqn.5016} immediately. The proof ends. 
 \qed 
 
 \subsection{Proof of Lemma~\ref{Lemma.502}}

 \begin{proposition}\label{Prop.508}
 	Given two constants $\delta, K>0$ and a function $\vartheta \in \boldsymbol{\varTheta}$,  consider the quantity
 	\beqlb\label{eqn.50123}
 	G_\delta(x,K,z):= \int_{\vartheta_x}^\infty  \mathbf{P}_{z+w}\Big( \overline{X}_{\tau_0^-} \in x+\varDelta_\delta,   \overline{X}_{\tau_0^-} - z-w\leq K \Big)\, \nu(dw),\quad x>0.
 	\eeqlb
 	We always have as $x\to\infty$,
 	\beqnn
 	\sup_{z\in[0,K]} \bigg|\frac{G_\delta(x,K,z)}{\nu\big(x+ \varDelta_\delta\big)} - \mathbf{P}\big(\overline{X}_\infty\leq K \big)\bigg| \to 0.
 	\eeqnn 
 \end{proposition}
 \proof By the spatial-homogeneity of $X$, we first have
 \beqnn
 G_\delta(x,K,z) \ar=\ar   \int_{\vartheta_x}^\infty  \mathbf{P}\Big(    \overline{X}_{\tau_{-z-w}^-} \in  x-z-w+\varDelta_\delta ,  \overline{X}_{\tau_{-z-w}^-} \in[0, K]\big]  \Big)\, \nu(dw)\cr
 \ar=\ar \int_{\vartheta_x}^\infty  \mathbf{P}\Big(\overline{X}_{\tau_{-z-w}^-} \in \big((x-z-w)^+, (x-z-w+\delta)^+ \wedge K\big]  \Big)\, \nu(dw).
 \eeqnn
 Note that the last integrand probability equals $0$ unless $x-z- K< w\leq x-z+\delta$. 
 Since $z\in (0,K]$ and $\vartheta_x=o(x)$, there exists some constant $x_{K}>0$ such that $x-z- K\geq \vartheta_x$ for all $x\geq x_{K}$ and then
 \beqnn
 \frac{G_\delta(x,K,z)}{\nu\big(x + \varDelta_\delta\big)} 
 \ar=\ar  \int_{x-z- K}^{x-z+\delta} \mathbf{P}\Big( \overline{X}_{\tau_{-z-w}^-} \in \big((x-z-w)^+, (x-z-w+\delta)^+ \wedge K\big]  \Big)\, \frac{\nu(dw)}{\nu\big(x+ \varDelta_\delta\big)}\cr
 \ar=\ar \int_{-\delta}^K \mathbf{P}\Big(   \overline{X}_{\tau_{-x+w}^-}\in \big(w^+, (w+\delta)^+ \wedge K\big]  \Big)\, \frac{\nu(x-z -dw)}{\nu\big(x+ \varDelta_\delta\big)},
 \eeqnn
 Let $\varepsilon(x,K,w):= \mathbf{P}\big(   \overline{X}_{\tau_{-x+w}^-}\in \big(w^+, (w+\delta)^+ \wedge K\big]  \big) - \mathbf{P}\big(   \overline{X}_\infty \in \big(w^+, (w+\delta)^+ \wedge K\big]  \big)$. Then the last integral can be written as 
 \beqlb\label{eqn.506}
 \int_{-\delta}^{K} \mathbf{P}\Big(   \overline{X}_\infty \in \big(w^+, (w+\delta)^+ \wedge K\big]  \Big)\, \frac{\nu(x-z -dw)}{\nu\big(x+ \varDelta_\delta\big)} 
 +  \int_{-\delta}^{K} \varepsilon(x,K,w)\, \frac{\nu(x-z -dw)}{\nu\big(x + \varDelta_\delta\big)} . 
 \eeqlb 
 Since $ \overline{X}_{\tau_{-x+w}^-} \overset{\rm a.s.}\to  \overline{X}_\infty$ as $x\to\infty$  uniformly in   $w\in[-\delta,K]$, we have 
 \beqnn
 \limsup_{x\to\infty}\sup_{w\in[-\delta,K]} \big| \varepsilon(x,K,w) \big| = 0. 
 \eeqnn 
 This along with Lemma~\ref{Lemma.301} yields that   as $x\to\infty$,
 \beqlb\label{eqn.507}
 \lefteqn{ \sup_{z\in[0,K]} \bigg|  \int_{- \delta}^{K} \varepsilon(x,K,w)\, \frac{\nu(x-z -dw)}{\nu\big(x + \varDelta_\delta\big)} \bigg|
 }\cr
 \ar\leq\ar   \sup_{w\in[-\delta,K]} \big| \varepsilon(x,K,w) \big|
 \cdot  \sup_{z\in[0,K]} \frac{\nu\big(x-z-K +\varDelta_{\delta+K}\big)}{\nu\big(x+ \varDelta_\delta\big)} \to 0. 
 \eeqlb
 We now consider the first term in \eqref{eqn.506}, which is denoted as $I(x,K,z)$ for convenience. 
 For any sufficiently large integer $n\geq 1$, we have 
 \beqnn
 I(x,K,z) 
 \ar=\ar \sum_{k=0}^{nK-1}  \int_{-\delta}^{K} \mathbf{P}\Big(   \overline{X}_\infty \in \big(w^+, (w+\delta)^+ \big]\cap\Big(\frac{k}{n},\frac{k+1}{n}\Big]\Big)\, \frac{\nu( x-z -dw)}{\nu\big( x+ \varDelta_\delta\big)} .
 \eeqnn
 It can be bounded from below and above respectively by 
 \beqnn 
 \sum_{k=0}^{nK-1}   \mathbf{P}\Big(   \overline{X}_\infty \in  \frac{k}{n}+\varDelta_{1/n}\Big) \cdot  \frac{\nu\big( x-z -\frac{k}{n}+\varDelta_{\delta-1/n}\big)}{\nu\big( x+ \varDelta_\delta\big)}  
 \eeqnn
 and 
 \beqnn
 \sum_{k=0}^{nK-1}   \mathbf{P}\Big(   \overline{X}_\infty \in  \frac{k}{n}+\varDelta_{1/n}\Big) \cdot  \frac{  \nu\big( x-z -\frac{k+1}{n}+\varDelta_{\delta+1/n}\big) }{\nu\big( x + \varDelta_\delta\big)} .  
 \eeqnn
 By Lemma~\ref{Lemma.301}, we have 
 \beqnn
 1-\frac{1/n}{\delta} 
 \ar\leq\ar \liminf_{x\to\infty}  
 \inf_{k=0,\cdots, nK}\inf_{z\in [0,K]}   \frac{\nu\big(x-z -\frac{k}{n}+\varDelta_{\delta-1/n}\big)}{\nu\big(x + \varDelta_\delta\big)}
 \cr
 \ar\leq\ar  \limsup_{x\to\infty} 
 \sup_{k=0,\cdots, nK}\sup_{z\in [0,K]} \frac{\nu\big(x-z -\frac{k+1}{n}+\varDelta_{\delta+1/n}\big)}{\nu\big(x + \varDelta_\delta\big)}
 \leq 1+\frac{1/n}{\delta},
 \eeqnn
 which follows that 
 \beqnn
 \Big( 1-\frac{1/n}{\delta} \Big) \cdot \mathbf{P}\Big(   \overline{X}_\infty \leq K\Big)
 \ar\leq\ar \liminf_{x\to\infty} \inf_{z\in [0,K]}I(x,K,z)  \cr
 \ar\leq\ar \limsup_{x\to\infty} \sup_{z\in [0,K]} I (x,K,z)  
 \leq  \Big( 1+\frac{1/n}{\delta} \Big) \cdot \mathbf{P}\Big(   \overline{X}_\infty \leq K\Big).
 \eeqnn
 This along with the arbitrariness of $n$ induces that 
 \beqnn
 \limsup_{x\to\infty} \sup_{z\in [0,K]}\bigg|I(x,K,z) -\mathbf{P}\Big(   \overline{X}_\infty \leq K\Big) \bigg| =0.
 \eeqnn
 The desired limit follows by taking this and \eqref{eqn.507}  back into \eqref{eqn.506}.  
 \qed

 \medskip
 
 \textbf{\textit{Proof of Lemma~\ref{Lemma.502}.}} Firstly, the inequality in \eqref{eqn.5017} follows directly from the fact   $\mathbf{1}_{\{\# \mathcal{A}_\delta(x,K)\geq 1\}} \leq \# \mathcal{A}_\delta(x,K)$.  
 We now prove the limit. 
 For $\eta\in(0,1)$, conditioned on $\mathcal{A}_\delta(x,K) \neq \emptyset$ and $\overline{\epsilon} \notin x +\varDelta_\delta$ for sufficiently large $x$, we  observe that $\overline{\epsilon}>x+\delta$ and the maximum is attained before the first jump in $\mathcal{A}_\delta(x,K) $. 
 This along with  \eqref{eqn.5018} indicates that $\underline{n}\big(\# \mathcal{A}_\delta(x,K), \overline{\epsilon}\notin x+\varDelta_\delta  \big) $ can be bounded by 
 \beqnn
 \underline{n}\bigg( \sum_{0<s\leq \zeta} \mathbf{1}_{\{\sup_{0<r<s} \epsilon_r>x+\delta\}} \cdot \mathbf{1}_{\{\epsilon_{s-}\leq K, \Delta\epsilon_s>x^\eta \}} \cdot  \mathbf{1}_{ \{ \sup_{s\leq r<\zeta} \epsilon_r \in x+\varDelta_\delta,   \sup_{s\leq r<\zeta} \epsilon_r- \epsilon_s \leq K\}}  \bigg).
 \eeqnn
 Recall the quantities $G_\delta(x,K,z)$ defined in \eqref{eqn.50123} with $\vartheta_x=x^\eta$. 
 Using the Markov property and the compensation formula \eqref{eqn.Compensation} to the last expression, it equals
 \beqnn
 \int_0^\infty ds \int_0^K G_\delta(x,K,z) \, \underline{n}\Big( \sup_{0<r<s} \epsilon_r> x+\delta, \epsilon_{s-}\in dz ,\,\zeta>s \Big).
 \eeqnn
 Consequently, 
 \beqnn
 \frac{\underline{n}\big(\# \mathcal{A}_\delta(x,K), \overline{\epsilon}\notin x+\varDelta_\delta  \big) }{\nu\big(x + \varDelta_\delta\big)} 
 \leq \int_0^\infty ds \int_0^K \frac{G_\delta(x,K,z)}{\nu\big(x+ \varDelta_\delta\big)} \, \underline{n}\Big( \sup_{0<r<s} \epsilon_r>x +\delta, \epsilon_{s-}\in dz ,\,\zeta>s
 \Big) .
 \eeqnn
 By Proposition~\ref{Prop.508} and  Lemma~\ref{Lemma.301}, it can further be bounded uniformly for sufficiently large $x$ by 
 \beqnn
 \lefteqn{C\cdot \int_0^\infty ds \int_0^K \frac{\nu\big(x -z-K +\varDelta_{K+\delta}\big)}{\nu\big(x + \varDelta_\delta\big)} \underline{n}\Big( \sup_{0<r<s} \epsilon_r> x+\delta , \epsilon_s\in dz ,\,\zeta>s
 	\Big)}\ar\ar\cr
 \ar\leq\ar C\cdot \sup_{z\in[0,K]} \frac{\nu(x -z-K +\varDelta_{K+\delta})}{\nu\big(x+ \varDelta_\delta\big)} \cdot \int_0^\infty  \underline{n}\Big( \sup_{0<r<s} \epsilon_r>x+\delta , \epsilon_s\leq K,\,\zeta>s
 \Big)\, ds  . 
 \eeqnn 
 The last supremum can be bounded uniformly for all sufficiently large  $x$ by  \eqref{eqn.30}. 
 The last integral  equals 	
 \beqnn
 \underline{n}\bigg(\int_0^\zeta \mathbf{1}_{\{s>\tau_{x+\delta}^+\}}\mathbf{1}_{\{\epsilon_s\leq K\}}ds\bigg)
 \ar=\ar
 \underline{n}\bigg(\mathbf{1}_{\{\zeta>\tau_{x+\delta}^+\}}\int_{\tau_{x+\delta}^+}^\zeta \mathbf{1}_{\{\epsilon_s\leq K\}}ds\bigg)\cr
 \ar=\ar 
 \underline{n}\bigg(\mathbf{1}_{\{\zeta>\tau_{x+\delta}^+\}}\mathbf{E}_{\epsilon_{\tau_{x+\delta}^+}}\Big[\int_{0}^{\tau_0^-}\mathbf{1}_{\{X_{s}\leq K\}}\,ds\Big]\bigg).
 \eeqnn
 By Proposition~\ref{Prop.509}, it can be bounded by
 $
 \underline{n}(\zeta) \cdot \widehat{V}(K)\cdot
 \underline {n}\big(\overline{\epsilon}>x\big),
 $
 which vanishes as $x\to\infty$ and then the limit in \eqref{eqn.5017} follows. 
 \qed

 \subsection{Proof of Lemma~\ref{Lemma.503}}
 
 \begin{proposition}\label{Prop.509}
 	The next inequality holds for all $K\geq 0$ and $x \in \mathbb{R}$,  
 	\beqnn
 	\mathbf{E}_x\bigg[\int_{0}^{\tau_0^-}\mathbf{1}_{\{X_{s}\leq K\}}\,ds\bigg] \leq \underline{n}(\zeta) \cdot \widehat{V}(K) <\infty .
 	\eeqnn
 \end{proposition}
 \proof If $x<0$, then the process $X$ starts below $0$ and $\tau_0^-=0$, so the expectation is trivially $0$. 
 For $0\leq x\leq K$, by \eqref{Asytau} and the monotonicity of $\widehat{V}$ we have 
 \beqlb\label{eqn9.4} 
 \mathbf{E}_x\Big[\int_{0}^{\tau_0^-}\mathbf{1}_{\{X_{s}\leq K\}}\,ds\Big]
 \leq  \mathbf E_x[\tau_0^-]
 = \underline{n}(\zeta) \cdot \widehat{V}(x)\leq  \underline{n}(\zeta) \cdot \widehat{V}(K) . 
 \eeqlb
 It remains to consider the case of $x>K$.  By the strong Markov property of $X$ at  the first passage time $\tau_K^-:=\inf\{ t\geq 0: X_t<K\}$,
 \beqnn
 \mathbf E_x\Big[  \int_0^{\tau_0^-}
 \mathbf 1_{\{X_{s-}\le K\}}\,ds\Big]
 \ar=\ar
 \mathbf E_x\Big[
 \mathbf 1_{\{\tau_K^-<\tau_0^-\}}\cdot 
 \int_{\tau_K^-}^{\tau_0^-}
 \mathbf 1_{\{X_{s-}\leq K\}}\,ds
 \Big] \cr
 \ar=\ar
 \mathbf{E}_x\bigg[
 \mathbf 1_{\{\tau_K^-<\tau_0^-\}} \cdot
 \mathbf E_{X_{\tau_K^-}}\Big[
 \int_0^{\tau_0^-}
 \mathbf 1_{\{X_{u-}\leq K\}}\,du  \Big]\bigg]
 \leq 
 \underline{n}(\zeta) \cdot \widehat{V}(K). 
 \eeqnn
 Here the last inequality follows from \eqref{eqn9.4} and the fact that  $X_{\tau_K^-}\in [0,K]$ when $\tau_K^-<\tau_0^-$.    
 \qed

 \textbf{\textit{Proof of Lemma~\ref{Lemma.503}.}}  Firstly, one can easily identify that 
 \beqnn
 \underline{n}\big( \# \mathcal{A}_\delta(x,K) \geq 2  \big) \leq \underline{n}\big(\big( \# \mathcal{A}_\delta(x,K) -1\big)^+ \big) \leq  \underline{n}\big(\# \mathcal{A}_\delta(x,K)\cdot\big(\# \mathcal{A}_\delta(x,K)-1\big) \big).
 \eeqnn
 Moreover, in view of \eqref{eqn.5018}, we also have 
 \beqnn
 \mathcal{A}_\delta(x,K)\subset \widetilde{\mathcal{A}}_\delta(x,K) := \big\{ 0<s\leq \zeta: \epsilon_{s-}\leq K, x -2K <\Delta \epsilon_s  \leq x +\delta \big\}.
 \eeqnn
 Hence it suffices to prove that for any $K>0$, 
 \beqlb\label{eqn.510}
 \limsup_{x\to\infty} \frac{\underline{n}\big( \#\widetilde{\mathcal{A}}_\delta(x,K) \cdot \big(\#\widetilde{\mathcal{A}}_\delta(x,K) -1\big) \big)}{\nu\big(x+ \varDelta_\delta\big)} =0.
 \eeqlb
 An application of the square of a sum formula shows that 
 \beqnn
 \underline{n}\Big( \#\widetilde{\mathcal{A}}_\delta(x,K) \cdot \big(\#\widetilde{\mathcal{A}}_\delta(x,K) -1\big) \Big)=  
 2\cdot
 \underline{n}\Big( \sum_{0<s<r\leq \zeta}  \mathbf{1}_{\{ \epsilon_{s-}\leq K,\, x -2K <\Delta \epsilon_s  \leq x +\delta \}} \cdot  \mathbf{1}_{\{ \epsilon_{r-}\leq K, \,x -2K <\Delta \epsilon_r  \leq x +\delta   \}} \Big).
 \eeqnn
 By using the Markov property,  the two compensation formulas \eqref{eqn.CompenstationX} and \eqref{eqn.Compensation}, it equals
 \beqnn
 \lefteqn{ 2\cdot \underline{n}\bigg(  \sum_{0<s\leq \zeta}  \mathbf{1}_{\{ \epsilon_{s-}\leq K,\, x -2K <\Delta \epsilon_s  \leq x +\delta \}}  \cdot
 	\mathbf{E}_{\epsilon_s}\Big[ \sum_{0<r\leq\tau_0^-}  \mathbf{1}_{\{ X_{r-}\le K,\, x -2K <\Delta X_r  \leq x +\delta \}}
 	\Big] \bigg)}\cr
 \ar=\ar 2\cdot \nu\big((x-2K,x+\delta]\big)\cdot \underline{n}\bigg(
 \sum_{0<s\leq\zeta} \mathbf 1_{\{ \epsilon_{s-}\leq K,\, x -2K <\Delta \epsilon_s  \leq x +\delta\}} \cdot \mathbf{E}_{\epsilon_s}\Big[ \int_0^{\tau_0^-} \mathbf 1_{\{X_{r-}\le K\}}\,dr\Big]\bigg).
 \eeqnn
 By Proposition~\ref{Prop.509} and the compensation formula \eqref{eqn.Compensation} again, it can be bounded by
 \beqnn
 C \cdot \nu\big((x-2K,x+\delta]\big)\cdot \underline{n}\bigg(
 \sum_{0<s\leq\zeta}   \mathbf 1_{\{  \Delta\epsilon_s\in(x-2K,x+\delta]\}}
 \bigg) 
 = C \cdot \underline{n}(\zeta) \cdot \big| \nu\big((x-2K,x+\delta]\big) \big|^2,
 \eeqnn
 for some constant $C>0$ depending only on $K$. 
 By \eqref{eqn.30}, the right-hand side is $o\big(\nu(x+ \varDelta_\delta) \big)$ as $x\to\infty$.  The limit \eqref{eqn.510} follows.
 The proof ends.
 \qed

 \subsection{Proof of Lemma~\ref{Lemma.504}}
 
 Recall \eqref{eqn.5018}. 
 For any $\eta\in(0,1)$, by the Markov property and then the compensation formula, we obtain that for sufficiently large $x$, 
 \beqlb\label{eqn.509}
 \underline{n}\big( \# \mathcal{A}_\delta(x,K) \big) \ar=\ar \underline{n}\bigg( \sum_{s\in( 0, \zeta]} \mathbf{1}_{\{\epsilon_{s-}\leq K, \Delta\epsilon_s>x^\eta \}} \mathbf{P}_{\epsilon_{s-}+\Delta \epsilon_s}\Big( \overline{X}_{\tau_0^-} \in x+\varDelta_\delta,   \overline{X}_{\tau_0^-} - X_0 \leq K\Big) \bigg)\cr
 \ar=\ar \int_0^\infty  ds \int_0^K\underline{n}_s(dz) \int_{x^\eta}^\infty  \mathbf{P}_{z+w}\Big( \overline{X}_{\tau_0^-} \in x+\varDelta_\delta,   \overline{X}_{\tau_0^-} - z-w\leq K \Big)\, \nu(dw).
 \eeqlb
 Recall the quantities $G_\delta(x,K,z)$ defined in \eqref{eqn.50123} with $\vartheta_x=x^\eta$. We have 
 \beqlb\label{eqn.5091}
 \bigg|\frac{\underline{n}\big( \# \mathcal{A}_\delta(x,K) \big)}{\nu\big(x+ \varDelta_\delta\big)}-\underline{n}(\zeta)\bigg|
 \ar=\ar \bigg|  \int_0^\infty  ds \int_0^K \frac{G_\delta(x,K,z)}{\nu\big(x+ \varDelta_\delta\big)}\,\underline{n}_s(dz) -\underline{n}(\zeta)\bigg|.  
 \eeqlb
 By the triangle inequality, it can be bounded by 
 \beqnn
 \int_0^\infty  ds \int_0^K \Big|\frac{G_\delta(x,K,z)}{\nu\big(x+ \varDelta_\delta\big)} - \mathbf{P}\big(\overline{X}_\infty\leq K \big)\Big|\,\underline{n}_s(dz) + \bigg| \mathbf{P}\big(\overline{X}_\infty\leq K \big) \int_0^\infty \underline{n}\big( \epsilon_s\leq K,\zeta >s \big)\,ds -\underline{n}(\zeta)\bigg|.
 \eeqnn
 Since $\overline{X}_\infty<\infty$ a.s., the second absolute value goes to $0$ as $K\to\infty$. 
 Moreover, the first absolute value can be bounded by 
 \beqnn
 \sup_{z\in[0,K]} \bigg|\frac{G_\delta(x,K,z)}{\nu\big(x+ \varDelta_\delta\big)} - \mathbf{P}\big(\overline{X}_\infty\leq K \big)\bigg| \cdot  \int_0^\infty  \underline{n}_s([0,K]) \,ds 
 \leq \underline{n}(\zeta)\cdot \sup_{z\in[0,K]} \bigg|\frac{G_\delta(x,K,z)}{\nu\big(x+ \varDelta_\delta\big)} - \mathbf{P}\big(\overline{X}_\infty\leq K \big)\bigg|,
 \eeqnn
 which goes to $0$ as $x\to\infty$ by Proposition~\ref{Prop.508}. Hence the first desired limit holds.
 The second one follows directly from the first one and Lemma~\ref{Lemma.503}. The proof ends. 
 \qed 
 
 \section{Proof of Theorem \ref{MainThm03}} 
 \label{Sec.ProofThm03}
 \setcounter{equation}{0} 
 
 In this section, we prove Theorem~\ref{MainThm03}, which states that, as $t\to\infty$, the simultaneous occurrence of the two events
 \beqnn
 \big\{ \overline{\epsilon} \in \beta (t+\varDelta_{\delta_1}) \big\}
 \quad \mbox{and}\quad 
 \big\{ \zeta \in t+\varDelta_{\delta_2} \big\}
 \eeqnn
 becomes asymptotically negligible. 
 The proof follows a strategy similar to that used for Theorem~\ref{MainThm02}, but requires an additional decomposition according to the occurrence time of the large jump. 
 
 Recall the set $\mathcal A_\delta(x,K)$ defined in \eqref{eqn.5033}. 
 The main idea is as follows.
 When the two events above occur simultaneously,  the dominant scenario necessarily involves a large jump belonging to $\mathcal A_{\beta\delta_1}(\beta t,K)$. Moreover, such a large jump must occur sufficiently early. 
 On the other hand, we show that the occurrence of such an early large jump is asymptotically incompatible with the excursion having lifetime close to $t$. Together, these observations explain why the two events above are asymptotically unlikely to occur simultaneously.
 The following three lemmas make this argument precise. To this end, for any $R>0$, we decompose $\mathcal A_\delta(x,K)$ according to the occurrence times of its jumps by setting
 \beqlb\label{eqn.600}
 \mathcal{A}_\delta^{\leq R}(x,K) :=  \mathcal{A}_\delta(x,K)\cap(0,R]
 \quad \mbox{and}\quad 
 \mathcal{A}_\delta^{>R}(x,K)  :=  \mathcal{A}_\delta(x,K)\cap(R,\infty).
 \eeqlb
 
 \begin{lemma}\label{Lemma.600}
 	For any constant $\delta>0$, we have 
 	\beqnn
 	\limsup_{R\to\infty}  \limsup_{K\to\infty} \limsup_{t\to\infty}\frac{\underline{n}\big( \#\mathcal{A}_{\beta\delta}^{> R} (\beta t,K )\big) }{ \nu\big(\beta t+\varDelta_{\beta\delta}\big)}  =0 .
 	\eeqnn 
 \end{lemma}
 \proof  For any $\eta\in(0,1)$, in view of \eqref{eqn.5018} we have
 $\Delta \epsilon_s\in(\beta t-2K,\beta t+\beta\delta] \subset \big((\beta t)^\eta,\infty\big)$
 and $\epsilon_{s-}\leq K$ for sufficiently large $t$, which follows that 
 \beqnn
 \underline{n}\Big( \#\mathcal{A}_{\beta\delta}^{> R} (\beta t,K )\Big) 
 \ar=\ar \underline{n}\bigg(  \sum_{R<s< \zeta} \mathbf{1}_{\{\Delta\epsilon_s>(\beta t)^\eta,\,\epsilon_{s-}\leq K,\, \sup_{s\leq r<\zeta} \epsilon_r \in \beta(t+\varDelta_{\delta}), \,  \sup_{s\leq r<\zeta} \epsilon_r- \epsilon_s \leq K \}} \bigg).
 \eeqnn
 Define   $G_{\beta\delta}(x,K,z)$ as in \eqref{eqn.50123} with $\vartheta_x=x^\eta$.  
 An argument similar as in \eqref{eqn.509}-\eqref{eqn.5091} yields  
 \beqnn
 \frac{\underline{n}\big( \#\mathcal{A}_{\beta\delta}^{> R} (\beta t,K )\big) }{\nu\big(\beta t+\varDelta_{\beta\delta}\big)} 
 \ar=\ar  \int_R^\infty ds \int_0^K \frac{G_{\beta\delta}(\beta t,K,z)}{\nu\big(\beta t+\varDelta_{\beta\delta}\big)} \, \underline{n}_s (dz)
 \leq  \sup_{z\in[0,K]}  \frac{G_{\beta\delta}(\beta t,K,z)}{\nu\big(\beta t+\varDelta_{\beta\delta}\big)} \cdot \int_R^\infty   \underline{n}(\zeta >s)\, ds.
 \eeqnn 
 By  Proposition~\ref{Prop.508},  there exists a constant $C>0$ independent of $R$ and $K$ such that 
 \beqnn
 \limsup_{K\to\infty}\limsup_{t\to\infty} 	\frac{\underline{n}\big( \#\mathcal{A}_{\beta\delta}^{> R} (\beta t,K )\big) }{ \nu\big(\beta t+\varDelta_{\beta\delta}\big)}  
 \leq C\cdot 	\limsup_{K\to\infty}\mathbf{P}\big( \overline{X}_\infty\leq K \big)\cdot \int_R^\infty    \underline{n}(\zeta >s)\, ds 
 \leq C \int_R^\infty    \underline{n}(\zeta >s)\, ds ,
 \eeqnn
 which vanishes as $R\to\infty$ by \eqref{eqn.501}. 
 \qed 
 
 \begin{lemma}\label{Lemma.601}
 	For any constants $\delta_1,\delta_2>0$, we have 
 	\beqlb\label{eqn.601}
 	\limsup_{R\to\infty} \limsup_{K\to\infty} \limsup_{t\to\infty} 	
 	\underline{n}\Big( \#\mathcal{A}_{\beta\delta_1}^{\leq R}\big(\beta t,K\big)=0 ,\, \overline{\epsilon}\in \beta \big(t+\varDelta_{\delta_1}\big) \,\Big|\, \zeta \in t+\varDelta_{\delta_2}\Big) =0.
 	\eeqlb
 \end{lemma}
 \proof Firstly, the law of total probability induces that 
 \beqnn
 \big\{\#\mathcal{A}_{\beta\delta_1}^{\leq R}\big(\beta t,K\big)=0 \big\} 
 \ar=\ar
 \big\{\#\mathcal{A}_{\beta\delta_1}\big(\beta t,K\big)=0 \big\} \cup \big\{\#\mathcal{A}_{\beta\delta_1}\big(\beta t,K\big)=\#\mathcal{A}_{\beta\delta_1}^{> R}\big(\beta t,K\big)\geq 1 \big\} \cr\cr
 \ar\subset\ar  \big\{\#\mathcal{A}_{\beta\delta_1}\big(\beta t,K\big)=0 \big\} \cup \big\{ \#\mathcal{A}_{\beta\delta_1}^{> R}\big(\beta t,K\big)\geq 1 \big\}.
 \eeqnn
 By using Theorem~\ref{MainThm01} along with \eqref{eqn.30}, and then Lemma~\ref{Lemma.501}, we have 
 \beqlb\label{eqn.6011}
 \lefteqn{\lim_{K\to\infty}\limsup_{t\to\infty}\underline{n}\Big( \#\mathcal{A}_{\beta\delta_1}\big(\beta t,K\big)=0 ,\, \overline{\epsilon}\in \beta \big(t+\varDelta_{\delta_1}\big) \,\Big|\, \zeta \in t+\varDelta_{\delta_2}\Big)}\cr
 \ar\leq\ar \lim_{K\to\infty}\limsup_{t\to\infty}\frac{ \underline{n}\big( \#\mathcal{A}_{\beta\delta_1}\big(\beta t,K\big)=0 ,\, \overline{\epsilon}\in \beta \big(t+\varDelta_{\delta_1}\big)\big)}{ \underline{n}\big(\zeta \in t+\varDelta_{\delta_2}\big)}\cr
 \ar=\ar  \limsup_{t\to\infty} \frac{  \nu\big(\beta (t+\varDelta_{\delta_1} )\big)}{\underline{n}\big(\zeta \in t+\varDelta_{\delta_2}\big)} \cdot \lim_{K\to\infty}\limsup_{t\to\infty}\frac{ \underline{n}\big( \#\mathcal{A}_{\beta\delta_1}\big(\beta t,K\big)=0 ,\, \overline{\epsilon}\in \beta \big(t+\varDelta_{\delta_1}\big)\big)}{   \nu\big(\beta (t+\varDelta_{\delta_1} )\big)}
 =0.
 \eeqlb
 Moreover, since $\mathbf{1}_{\#\mathcal{A}_{\beta\delta_1}^{> R}\big(\beta t,K\big)\geq 1 } \leq \#\mathcal{A}_{\beta\delta_1}^{> R}\big(\beta t,K\big)$, we have 
 \beqnn
 \underline{n}\Big( \#\mathcal{A}_{\beta\delta_1}^{> R} (\beta t,K )\geq 1 ,\, \overline{\epsilon}\in \beta (t+\varDelta_{\delta_1} ) \,\Big|\, \zeta \in t+\varDelta_{\delta_2}\Big)
 \leq \frac{\underline{n}\big( \#\mathcal{A}_{\beta\delta_1}^{> R} (\beta t,K )\big)}{ \underline{n}\big(\zeta \in t+\varDelta_{\delta_2}\big)}
 \eeqnn
 By using Theorem~\ref{MainThm01} along with \eqref{eqn.30} and then Lemma~\ref{Lemma.600}, we have as $t\to\infty$ and then $K,R\to\infty$, 
 \beqnn
 \frac{\underline{n}\big( \#\mathcal{A}_{\beta\delta_1}^{> R} (\beta t,K )\big) }{\underline{n}\big(\zeta \in t+\varDelta_{\delta_2}\big)} 
 \ar=\ar \frac{ \nu\big(\beta t+ \varDelta_{\beta\delta_1}\big)}{\underline{n}\big(\zeta \in t+\varDelta_{\delta_2}\big)} 	\frac{\underline{n}\big( \#\mathcal{A}_{\beta\delta_1}^{> R} (\beta t,K )\big) }{\nu\big(\beta t+\varDelta_{\beta\delta_1}\big)  }  \to 0.
 \eeqnn
 This together with \eqref{eqn.6011} yields the desired limit \eqref{eqn.601} directly.
 \qed  
 
 \begin{lemma}\label{Lemma.602}
 	For any constants $\delta_1,\delta_2>0$, we have 
 	\beqnn
 	\limsup_{R\to\infty}\limsup_{K\to\infty} \limsup_{t\to\infty} 	
 	\underline{n}\Big(\# \mathcal{A}_{\beta \delta_1}^{\leq R}\big(\beta t,K\big)  \,\Big|\, \zeta \in t+\varDelta_{\delta_2}\Big) =0.
 	\eeqnn
 \end{lemma}
 \proof For any $s \in \mathcal{A}_{\beta \delta_1}(\beta t,K) $, in view of \eqref{eqn.5018}  we have $\Delta \epsilon_s \in(\beta t-2K,\beta t+\beta\delta_1] $ and $\epsilon_{s-}\leq K$. This induces that for  sufficiently large $t$, 
 \beqnn
 \underline{n}\big(\#\mathcal{A}_{\beta \delta_1}^{\leq R}(\beta t,K) \,\big|\, \zeta \in t+\varDelta_{\delta_2}\big) 
 \leq \frac{\underline{n}\big( \sum_{0<s\leq R} \mathbf{1}_{\{ \Delta \epsilon_s \in(\beta t-2K,\beta t+\beta\delta_1],\,\epsilon_{s-}\leq K,\, \zeta \in t+\varDelta_{\delta_2} \}}\big)}{\underline{n}\big( \zeta \in t+\varDelta_{\delta_2} \big)},
 \eeqnn
 Similarly as in \eqref{eqn.509}, the right-hand side equals
 \beqnn
 \frac{1}{\underline{n}\big( \zeta \in t+\varDelta_{\delta_2} \big)}\cdot\int_0^R ds \int_0^K	\underline{n}_s (dz ) \int_{\beta t-2K}^{\beta t+\beta\delta_1} \mathbf{P}\big(  \tau_{-z-w}^- \in t-s+\varDelta_{\delta_2}
 \big)\,\nu(dw),
 \eeqnn
 which, for any fixed $K,R>0$ and sufficiently large $t$, is bounded by
 \beqnn
 \frac{\nu\big(\beta t-2K +\varDelta_{2K+\beta \delta_1}\big) }{\underline{n}\big( \zeta \in t+\varDelta_{\delta_2} \big)}\cdot   \sup_{r\in[t-R,t]}  \sup_{\beta t-2K\leq y\leq \beta t+K+\beta\delta_1}
 \mathbf{P} \big(  \tau_{-y}^- \in r+\varDelta_{\delta_2} \big) \cdot  	\int_0^R \underline{n}(\epsilon_s\leq K, \zeta >s) \, ds.
 \eeqnn
 By Theorem~\ref{MainThm01} and \eqref{eqn.30}, the fraction is bounded uniformly for sufficiently large $t$. Moreover, the last integral is smaller than $\underline{n}(\zeta)<\infty$.  
 Then it remains to prove that for any $R,K>0$,
 \beqlb\label{eqn.5023}
 \lim_{t\to\infty} \sup_{r\in[t-R,t]}\sup_{\beta t-2K\leq y\leq \beta t+K+\beta\delta_1}
 \mathbf{P} \big(  \tau_{-y}^- \in r+\varDelta_{\delta_2} \big)  =0.
 \eeqlb
 Firstly, for any $\eta>0$ we have
 \beqlb\label{eqn.5024}
 \mathbf{P} \big(  \tau_{-y}^- \in r+\varDelta_{\delta_2} \big)
 \ar\leq \ar \mathbf{P} \big(  \underline{X}_r - \underline{X}_{r+\delta_2} >\eta \big) + \mathbf{P} \big(  \tau_{-y}^- \in r+\varDelta_{\delta_2},\underline{X}_r - \underline{X}_{r+\delta_2} \leq \eta \big) . 
 \eeqlb
 For the first probability on the right-hand side, by the stationary increments of $X$ we have 
 \beqlb\label{eqn.7003}
 \underline{X}_r - \underline{X}_{r+\delta_2} 
 = \Big(  \underline{X}_{r} -   \inf_{r\leq u\leq r+\delta_2}X_u  \Big)^+ 
 \ar \leq\ar X_r  - \inf_{r\leq u\leq r+\delta_2}X_u  \cr
 \ar=\ar    -\inf_{0\leq s\leq\delta_2} \big( X_{r+s}-X_{r} \big)
 \overset{\rm d}=\  -\underline{X}_{\delta_2}  ,
 \eeqlb
 and then
 \beqnn
 \limsup_{\eta\to\infty}\sup_{r\in[t-R,t]}   \mathbf{P} \big( \underline{X}_r - \underline{X}_{r+\delta_2} >\eta\big) 
 \ar\leq \ar \limsup_{\eta\to\infty} \mathbf{P} \big(  -\underline{X}_{\delta_2} > \eta \big)  =0.
 \eeqnn
 For the second probability,   conditioned on $ \tau_{-y}^- \in r+\varDelta_{\delta_2}$ and $0<\underline{X}_r - \underline{X}_{r+\delta_2} \leq \eta$, we have  $\underline{X}_r \geq -y> \underline{X}_{r+\delta_2}$  and $\underline{X}_r    \in  -y +\varDelta_\eta$, which yields that 
 \beqnn
 \lefteqn{\sup_{r\in[t-R,t]}\sup_{\beta t-2K\leq y\leq \beta t+K+\beta\delta_1}  
 	\mathbf{P} \big(  \tau_{-y}^- \in r+\varDelta_{\delta_2},\underline{X}_r - \underline{X}_{r+\delta_2} \leq \eta \big) }\ar\ar\cr
 \ar\leq\ar \sup_{r\in[t-R,t]}\sup_{\beta t-2K\leq y\leq \beta t+K+\beta\delta_1} 
 \mathbf{P}\big( \underline{X}_r    \in  -y +\varDelta_\eta \big)
 \leq \sup_{r\in[t-R,t]}\sup_{y\in \mathbb{R}} \mathbf{P}\big( \underline{X}_r +\beta r   \in  y +\varDelta_\eta \big),
 \eeqnn
 which goes to $0$ as $t\to\infty$ by Proposition~\ref{Prop.2011}.  
 The limit \eqref{eqn.5023} follows directly by taking the preceding two limits back into \eqref{eqn.5024}. The proof ends. 
 \qed  
 
 \medskip
 
 \textbf{\textit{Proof of Theorem \ref{MainThm03}.}} 
 Firstly, by Theorem~\ref{MainThm01} and \ref{MainThm02} we have 
 \beqnn
 \lim_{t\to\infty}\frac{	\underline{n}\big(  \overline\epsilon\in \beta (t+\varDelta_{\delta_1}) \,|\,  \zeta\in t+\varDelta_{\delta_2} \big)  }{  \underline{n}\big(   \zeta\in t+\varDelta_{\delta_2} \,|\, \overline\epsilon\in \beta (t+\varDelta_{\delta_1})\big)}
 \ar=\ar \lim_{t\to\infty} \frac{\underline{n}\big(   \overline\epsilon\in \beta (t+\varDelta_{\delta_1})\big)}{\underline{n}\big(   \zeta\in t+\varDelta_{\delta_2} \big)}
 = \frac{\delta_1}{\delta_2}. 
 \eeqnn 
 Moreover,  the law of total probability induces that 
 \beqnn
 \lefteqn{\underline{n}\Big(  \overline\epsilon\in \beta \big(t+\varDelta_{\delta_1}\big) \,\Big|\,  \zeta\in t+\varDelta_{\delta_2} \Big)}\ar\ar\cr\cr
 \ar\leq \ar  \underline{n}\Big( \#\mathcal{A}_{\beta \delta_1}^{\leq R}\big(\beta t,K\big)=0, \, \overline\epsilon\in \beta \big(t+\varDelta_{\delta_1}\big) \,\Big|\,  \zeta\in t+\varDelta_{\delta_2} \Big)  +  \underline{n}\Big( \#\mathcal{A}_{\beta \delta_1}^{\leq R}\big(\beta t,K\big)\geq 1 \,\Big|\,  \zeta\in t+\varDelta_{\delta_2} \Big),
 \eeqnn 
 which vanishes  as $t\to\infty$, then $K\to\infty$ and finally $R\to\infty$ by Lemma~\ref{Lemma.601} and~\ref{Lemma.602}.  
 \qed

 \section{Proof of Theorem~\ref{MainThm04}} 
 \label{Sec.ProofThm04}
 \setcounter{equation}{0}

 This section is devoted to the proof of Theorem~\ref{MainThm04}, which provides second-order localization estimates for the height and lifetime of a long excursion.
 
 To prove the first limit in \eqref{eqn.MainThm04}, we recall a key observation from our study of the tail asymptotics of the excursion lifetime in \cite{CuiWuXu2025}: for an excursion to survive for a long time, a large jump must typically occur at an early stage of the excursion. Moreover, owing to the negative drift of the underlying L\'evy process, the excursion is unlikely to reach a high level before the occurrence of this first large jump.
 
 On the other hand, the proof of Theorem~\ref{MainThm02} reveals a complementary feature of excursions attaining a large maximum. Namely, when the maximum of an excursion lies in a high-level interval, the dominant scenario is that a large jump carries the excursion close to its eventual maximum, after which only a relatively small additional upward displacement occurs. More precisely, the difference between the subsequent maximum and the post-jump position is asymptotically negligible relative to the size of the large jump.
 
 These two observations suggest that the first large jump is the key quantity linking the lifetime and the maximum of a long excursion: its occurrence time determines, to first order, the remaining lifetime of the excursion, while its size determines, to first order, the height eventually attained by the excursion. To formulate this mechanism precisely, for $x\geq0$ and $\epsilon$,  define 
 \beqlb\label{eqn.70}
 \mathcal{B}^x(\epsilon):= \big\{ 0<s\leq \zeta: \Delta \epsilon_s>x \big\} 
 \quad \mbox{and}\quad 
 \mathcal{J}^x(\epsilon) := \inf\, \mathcal{B}^x(\epsilon) ,
 \eeqlb 
 with the convention that $\inf\emptyset=\infty$. When there is no risk of confusion, we suppress the dependence on $\epsilon$, and simply write $\mathcal{B}^x$ and $\mathcal{J}^x$. 
 For any $\lambda\in(0,\beta)$ and sufficiently large constants $R,K>0$, the preceding two observations suggest that a long-lived excursion typically exhibits the following behavior: its first jump of size greater than $\lambda t$ occurs before time $R$; prior to this jump, the excursion remains below level $K$; the size of the jump is within $\theta_t/2$ of $\beta t$; and, after the jump, the excursion makes at most an additional upward displacement of size $K$. Accordingly, we introduce the event
 \beqnn
 \boldsymbol{J}(t,K,R):=\Big\{  \mathcal{J}^{\lambda t}\leq R, \, \sup_{0\leq r< \mathcal{J}^{\lambda t}}\epsilon_r\leq K, \, \big| \Delta\epsilon_{ \mathcal{J}^{\lambda t}}-\beta t \big|
 \leq \frac{\theta_t}{2} ,\,
 \sup_{ \mathcal{J}^{\lambda t}\leq r<\zeta}  \epsilon_{r} - \epsilon_{ \mathcal{J}^{\lambda t}}  \leq K \Big\}. 
 \eeqnn
 To prove the first limit in \eqref{eqn.MainThm04},  in the next four lemmas we show that, conditional on $\zeta\in t+\varDelta_\delta$, the complementary event  of  $\boldsymbol{J}(t,K,R)$ occurs with probability tending to zero.

 \begin{lemma}\label{Lemma.701}
 	For any constants $\delta, R>0$ and $\lambda\in (0,\beta)$, we have  as $t\to\infty$,
 	\beqnn
 	\bigg| \underline{n}\big(\mathcal{J}^{\lambda t}\leq R \,\big|\,\zeta\in t+\varDelta_\delta   \big) -  \frac{\underline{n}(R\wedge \zeta)}{\underline{n}(\zeta)} \bigg| \to0.
 	\eeqnn 
 \end{lemma}

 \begin{lemma}\label{Lemma.702}
 	For any $\delta, R>0$ and $\lambda \in (0,\beta)$, we have 
 	\beqnn
 	\limsup_{K\to\infty}\limsup_{t\to\infty} \underline n\Big( \mathcal{J}^{\lambda t}\leq R,\, \sup_{0\leq r<\mathcal{J}^{\lambda t}}\epsilon_r>K \,\Big|\,\zeta\in t+\varDelta_\delta  \Big) =0.
 	\eeqnn  
 \end{lemma}

 \begin{lemma}\label{Lemma.703}
 	For any $\delta, R>0$ and $\lambda \in (0,\beta)$, we have 
 	\beqnn
 	\limsup_{K\to\infty}\limsup_{t\to\infty}  \underline n\Big( \mathcal{J}^{\lambda t}\leq R,\, \sup_{0\leq r<\mathcal{J}^{\lambda t}}\epsilon_r\leq K,\, \sup_{\mathcal{J}^{\lambda t}\leq r<\zeta} \epsilon_{r} - \epsilon_{\mathcal{J}^{\lambda t}} >K\,\Big|\, \zeta\in t+\varDelta_\delta  \Big)=0.
 	\eeqnn 
 \end{lemma}

 \begin{lemma}\label{Lemma.704}
 	For any constants $\delta,R,K>0$, $\lambda \in (0,\beta)$ and any positive function $\theta$ on $\mathbb{R}_+$ such that $\big|X_t+\beta t\big|/\theta_t\overset{\rm p}\to0$, we have 
 	\beqnn
 	\limsup_{t\to\infty}  \underline n\Big( \mathcal{J}^{\lambda t}\leq R,\, \sup_{0\leq r<\mathcal{J}^{\lambda t}}\epsilon_r\leq K,\,   \big| \Delta\epsilon_{\mathcal{J}^{\lambda t}}-\beta t \big| > \frac{\theta_t}{2} \,\Big|\, \zeta\in t+\varDelta_\delta  \Big)=0.
 	\eeqnn  
 \end{lemma}

 \textbf{\textit{Proof of Theorem~\ref{MainThm04}.}  Part~I.} 
 For all sufficiently large $t$, on the event $\boldsymbol{J}(t,K,R)$,
 we have $\mathcal{J}^{\lambda t}\leq R<t$, and 
 \beqnn
 \big|\overline{\epsilon}-\beta t\big| 
 \ar\leq\ar \epsilon_{\mathcal{J}^{\lambda t}-} + \Big|  \Delta\epsilon_{\mathcal{J}^{\lambda t}}-\beta t \Big|+  \sup_{\mathcal{J}^{\lambda t}\leq r<\zeta}  \epsilon_{r} - \epsilon_{\mathcal{J}^{\lambda t}}  
 \leq \frac{\theta_t}{2} +2K \leq \theta_t ,
 \eeqnn
 where the last inequality holds for all sufficiently large $t$, since $\theta_t\to\infty$. 
 Hence $\boldsymbol J(t,K,R) \subset  \big\{|\overline\epsilon-\beta t|\leq\theta_t\big\}$ for all sufficiently large $t$. 
 It follows that 
 \beqlb\label{eqn.702}
 1\geq \underline n\Big( \big|\overline\epsilon-\beta t \big|  \leq \theta_t \,\Big|\, \zeta\in t+\varDelta_\delta \Big)
 \ar\geq\ar  \underline n\Big(  \boldsymbol{J}(t,K,R) \,\Big|\, \zeta\in t+\varDelta_\delta \Big)
 = 1-  I_1(t,R)- \sum_{i=2}^4  I_i(t,R,K) , \quad 
 \eeqlb
 where  $ I_1(t,R) := \underline n\big( \mathcal{J}^{\lambda t}>R \,\big|\, \zeta\in t+\varDelta_\delta \big)$ and
 \beqnn 
 I_2(t,R,K) \ar:=\ar \underline n\Big( \mathcal{J}^{\lambda t}\leq R,\, \sup_{0\leq r<\mathcal{J}^{\lambda t}}\epsilon_r>K \,\Big|\,\zeta\in t+\varDelta_\delta  \Big) ,\cr
 I_3(t,R,K) \ar:=\ar \underline n\Big( \mathcal{J}^{\lambda t}\leq R,\, \sup_{0\leq r<\mathcal{J}^{\lambda t}}\epsilon_r\leq K,\, \sup_{\mathcal{J}^{\lambda t}\leq r<\zeta} \epsilon_{r} - \epsilon_{\mathcal{J}^{\lambda t}} >K\,\Big|\, \zeta\in t+\varDelta_\delta  \Big),\cr
 I_4(t,R,K) \ar:=\ar	\underline n\Big( \mathcal{J}^{\lambda t}\leq R,\, \sup_{0\leq r<\mathcal{J}^{\lambda t}}\epsilon_r\leq K,\, \sup_{\mathcal{J}^{\lambda t}\leq r<\zeta} \epsilon_{r} - \epsilon_{\mathcal{J}^{\lambda t}} \leq K, \big| \Delta\epsilon_{\mathcal{J}^{\lambda t}}-\beta t \big| > \frac{\theta_t}{2} \,\Big|\, \zeta\in t+\varDelta_\delta  \Big).
 \eeqnn
 By Lemma~\ref{Lemma.701}-\ref{Lemma.703}, we have $I_1(t,R) \to 0$ as $t\to\infty$ and then $R\to\infty$, and also $I_2(t,R,K)  + I_3(t,R,K) \to 0$ as $t\to\infty$ and $K\to\infty$.  
 Finally, $I_4(t,R,K) $ can be bounded by 
 \beqnn
 \underline n\Big( \mathcal{J}^{\lambda t}\leq R,\, \sup_{0\leq r<\mathcal{J}^{\lambda t}}\epsilon_r\leq K,\, \big| \Delta\epsilon_{\mathcal{J}^{\lambda t}}-\beta t \big| > \frac{\theta_t}{2} \,\Big|\, \zeta\in t+\varDelta_\delta  \Big) \to 0,
 \eeqnn
 as $t\to\infty$, by Lemma~\ref{Lemma.704}.  
 Substituting these estimates into \eqref{eqn.702}, and then taking the successive limits in $t$, $K$, and $R$, we can immediately get the first desired limit in \eqref{eqn.MainThm04}. 
 \qed

 We now turn to the proof of the second limit in Theorem~\ref{MainThm04}. Recall the sets $\mathcal A_\delta^{\leq R}(x,K)$ and $\mathcal A_\delta^{>R}(x,K)$ defined in \eqref{eqn.600}, which consist, respectively, of the jumps in $\mathcal A_\delta(x,K)$ occurring no later than time $R$ and those occurring after time $R$.
 
 Lemmas~\ref{Lemma.501}, \ref{Lemma.503}, and \ref{Lemma.600} together show that, conditional on ${\overline\epsilon\in\beta(t+\varDelta_\delta)}$, the dominant scenario for large $t$ is the occurrence of a unique large jump at an early stage of the excursion, which carries the excursion close to its eventual maximum. 
 Furthermore, owing to the early occurrence of the large jump, the contribution of the path segment preceding the large jump to the total lifetime of the excursion is asymptotically negligible. Consequently, the lifetime of the excursion is asymptotically determined by its post-jump path, which,  by the strong Markov property, evolves as the underlying L\'evy process started from a large initial position and run until it first enters the negative half-line. 
 
 The following lemma therefore provides the key estimate for proving the second limit in \eqref{eqn.MainThm04}. It shows that, uniformly over initial positions within a bounded distance of $\beta t$, the first-passage time into the negative half-line is asymptotically concentrated within a small window  around $t$.
 
 \begin{lemma}\label{Lemma.705}
 	For any constant $K>0$ and any positive function $\theta$ on $\mathbb{R}_+$ such that $\big|X_t+\beta t\big|/\theta_t\overset{\rm p}\to0$, we have as $t\to\infty$,
 	\beqlb\label{eqn.7051}
 	\sup_{|y-\beta t|\leq K}  \mathbf{P}_y \big(  |\tau_0^- -t| >\theta_t \big) \to 0.
 	\eeqlb
 \end{lemma}

 \textbf{\textit{Proof of Theorem~\ref{MainThm04}.} Part~II.}  It remains to prove the second limit in \eqref{eqn.MainThm04}, or equivalently,
 \beqlb\label{eqn.7013}
 \underline n\big( |\zeta-t|> \theta_t \,\big|\, \bar\epsilon\in \beta (t+\varDelta_\delta)\big) \to 0. 
 \eeqlb
 Fix $K,R>0$. We decompose the event
 $\{|\zeta-t|>\theta_t\}$ into the following three mutually disjoint events:
 \beqnn
 \quad B_1(t,K) :=\{|\zeta-t|> \theta_t,\, \#\mathcal{A}_{\beta \delta}(\beta t,K)=0  \}, \quad 
 B_2(t,K,R):= \{|\zeta-t|> \theta_t,\,   \# \mathcal{A}_{\beta \delta}^{>R}(\beta t,K) \geq 1 \}
 \eeqnn
 and 
 \beqnn
 B_3(t,K,R) :=  \{|\zeta-t|> \theta_t,\, \# \mathcal{A}_{\beta \delta}^{\leq R}(\beta t,K)\geq 1,  \#  \mathcal{A}_{\beta \delta}^{>R}(\beta t,K)=0 \}.
 \eeqnn 
 By Lemmas~\ref{Lemma.501} and \ref{Lemma.600}, respectively, we have as $t\to\infty$ and then $K\to\infty$,
 \beqlb\label{eqn.7011}
 \underline{n}\big( B_1(t,K)   \,\big|\, \bar\epsilon\in \beta (t+\varDelta_\delta) \big) 
 \leq   \underline n\big( \#\mathcal{A}_{\beta \delta}(\beta t,K)=0  \,\big|\, \bar\epsilon\in \beta (t+\varDelta_\delta) \big) \to 0 , 
 \eeqlb
 and also as $t\to\infty$, then $K\to\infty$ and finally $R\to\infty$,
 \beqlb\label{eqn.7012}
 \underline n\big(  B_2(t,K,R) \,\big|\, \bar\epsilon\in \beta (t+\varDelta_\delta) \big)  
 \ar\leq\ar  \underline n\big(  \# \mathcal{A}_{\beta \delta}^{>R}(\beta t,K)  \,\big|\, \bar\epsilon\in \beta (t+\varDelta_\delta) \big) \to0.
 \eeqlb 
 It remains to control $B_3(t,K,R)$. On this event, there exists a jump time $s\leq R$ belonging to $\mathcal A_{\beta\delta}^{\leq R}(\beta t,K)$. 
 By \eqref{eqn.5018}, for all sufficiently large $t$, 
 \beqnn
 \#\big\{0<s\leq R:  \Delta \epsilon_s\in(\beta t-2K,\beta t +\beta\delta] ,\,  \epsilon_{s-}\leq K,\,   |\zeta -t|>\theta_t \big\} \geq 1. 
 \eeqnn
 Consequently,
 \beqnn
 \underline{n}\big( B_3(t,K,R) \big)
 \ar\leq\ar  \underline{n}\bigg( \sum_{0<s\leq R} \mathbf{1}_{\{ \Delta \epsilon_s\in(\beta t-2K,\beta t +\beta\delta] ,\,  \epsilon_{s-}\leq K,\,   |\zeta -t|>\theta_t \}} \bigg).
 \eeqnn
 Similarly as in \eqref{eqn.509}, the right-hand side is equal to
 \beqnn
 \int_0^R ds \int_0^K \underline{n}_s(dz) \int_{\beta t-2K}^{\beta (t+\delta)} \mathbf{P}_{z+w} \big(   |\tau_0^- -(t-s)| >\theta_t \big)\,\nu(dw) .
 \eeqnn
 This along with \eqref{eqn.501} yields that 
 \beqnn
 \underline{n}\big( B_3(t,K,R) \big)
 \ar\leq\ar \underline{n}(\zeta)\cdot\nu\big(  \beta t-2K +\varDelta_{2K+\beta \delta} \big)\cdot \sup_{y\in (\beta t-2K, \beta(t+\delta)+K]}\sup_{s\in[t-R,t]} \mathbf{P}_y \big(  |\tau_0^- -s| >\theta_t \big) . 
 \eeqnn 
 By Theorem~\ref{MainThm02}, Lemma~\ref{Lemma.301}, and the regular variation $\overline\nu\in\mathrm{RV}_{-\alpha}^{\infty}$, for each fixed $K>0$ there exists $t_0>0$ such that
 \beqnn
 \sup_{t\geq t_0}\frac{ \nu\big(  \beta t-2K +\varDelta_{2K+\beta \delta} \big)}{ \underline{n}\big(  \bar\epsilon\in \beta (t+\varDelta_\delta) \big) }<\infty. 
 \eeqnn 
 Moreover, since $\theta_t\to\infty$, for fixed $R>0$ we have $R\leq\theta_t/2$ for all sufficiently large $t$ and hence $ \{ |\tau_0^- -s| >\theta_t  \} \subset \{ |\tau_0^- -t| >\theta_t/2\}$ for every $s\in[t-R,t]$.  Therefore, 
 \beqnn
 \sup_{y\in (\beta t-2K, \beta(t+\delta)+K]}\sup_{s\in[t-R,t]} \mathbf{P}_y \big(  |\tau_0^- -s| >\theta_t \big) 
 \leq  \sup_{|y-\beta t|\leq \beta \delta+2K}  \mathbf{P}_y \big(  |\tau_0^- -t| >\theta_t/2 \big).
 \eeqnn
 which goes to $0$ as $t\to\infty$ by Lemma~\ref{Lemma.705}.  Consequently, for every fixed $K,R>0$, we have as $t\to\infty$,
 \beqnn
 \underline{n}\big( B_3(t,K,R)  \,\big|\, \bar\epsilon\in \beta (t+\varDelta_\delta) \big)  
 \leq \frac{ \underline{n}\big( B_3(t,K,R)\big)}{\underline{n}\big(  \bar\epsilon\in \beta (t+\varDelta_\delta) \big) }  \to 0.
 \eeqnn
 Combining this estimate with \eqref{eqn.7011} and \eqref{eqn.7012}, and then letting $K\to\infty$ and $R\to\infty$ in the prescribed order, yields \eqref{eqn.7013}. This completes the proof.
 \qed

 \subsection{Proofs of Lemma~\ref{Lemma.701} and \ref{Lemma.702}}
 
 \begin{proposition}\label{Prop.7011}
 	For any constants $\delta>0$ and $\lambda \in(0,\beta)$, there exist two constants $C,t_0\geq0$ such that for any $t\geq t_0$,
 	\beqlb\label{eqn.70111}
 	\sup_{z\geq 0}  \int_{\lambda t}^\infty  \mathbf{P} \big(\tau_{-z-w}^-\in  t+\varDelta_\delta\big)\,\nu(dw)  
 	\leq C\cdot \nu\big(\beta (t+\varDelta_\delta) \big) . 
 	\eeqlb 
 	Moreover, for any constant $K>0$,  we have as $t\to\infty$,
 	\beqlb\label{eqn.70112}
 	\sup_{z\in [0,K]} \bigg| \frac{1}{ \nu\big(\beta (t +\varDelta_\delta)\big)} \int_{\lambda t}^\infty 	\mathbf{P}\big( \tau_{-z-w}^-\in  t+\varDelta_\delta \big)\,\nu(dw) -1 \bigg| \to 0.
 	\eeqlb 
 \end{proposition} 
 \proof  Note that $\tau_{-z-w}^-\in t+\varDelta_\delta$ if and only if $-\underline{X}_{t}\leq z+w < -\underline X_{t+\delta}$. 
 By Fubini's theorem,  
 \beqnn
 \int_{\lambda t}^\infty  	\mathbf{P}\big(\tau_{-z-w}^-\in t+\varDelta_\delta \big)\nu(dw)
 \ar=\ar  \int_{\lambda t}^\infty  \mathbf{P}\big( -\underline{X}_{t}-z\leq  w < -\underline X_{t+\delta}-z \big)\,\nu(dw) \cr\cr
 \ar=\ar \mathbf{E}\Big[  \nu\Big( [ -\underline{X}_{t}-z, -\underline{X}_{t+\delta}-z ) \cap (\lambda t,\infty) \Big) \Big] . 
 \eeqnn
 The intersection $[ -\underline{X}_{t}-z,\, -\underline{X}_{t+\delta} -z) \cap (\lambda t,\infty)$ has left endpoint strictly larger than $\lambda t$ and  length is bounded above by $\underline{X}_{t} -\underline{X}_{t+\delta}$.
 By Corollary~\ref{Corollary.202}, we have for all sufficiently large $t$,
 \beqlb\label{eqn.70113}
 \sup_{z\geq 0} \nu\big( [ -\underline{X}_{t}-z, -\underline{X}_{t+\delta} -z) \cap (\lambda t,\infty) \big)  
 \leq C\cdot \frac{\overline{\nu}(\lambda t)}{\lambda t} \cdot \big( \underline{X}_{t} -\underline{X}_{t+\delta}+1  \big). 
 \eeqlb
 Taking expectations on both sides of this inequality, then using Lemma~\ref{Lemma.310} and finally Lemma~\ref{Lemma.301} along with the fact that $\overline{\nu} \in \mathrm{RV}^\infty_{-\alpha}$, we can get the uniform upper bound \eqref{eqn.70111} immediately.

 We now start to prove the limit \eqref{eqn.70112}. 
 In view of \eqref{eqn.30} and Lemma~\ref{Lemma.310}, it suffices to prove that as $t\to\infty$, 
 \beqlb\label{eqn.70442}
 \sup_{z\in [0,K]}  \bigg| \frac{ \mathbf{E}\big[  \nu\big( [ -\underline{X}_{t}-z, -\underline{X}_{t+\delta}-z ) \cap (\lambda t,\infty) \big) \big] }{\alpha\cdot \frac{\overline{\nu}(\beta t)}{\beta t}\cdot \mathbf{E}\big[ \underline{X}_t -\underline{X}_{t+\delta}\big]} -1 \bigg| \to 0
 \eeqlb
 The absolute value on the left-hand side can  be bounded by 
 \beqlb\label{eqn.70441}
 \frac{ \mathbf{E}\big[  \big|\nu\big( [ -\underline{X}_{t}-z, -\underline{X}_{t+\delta}-z ) \cap (\lambda t,\infty) \big) -\alpha\cdot\frac{\overline{\nu}(\beta t)}{\beta t}\cdot \big(\underline{X}_t -\underline{X}_{t+\delta}\big)\big|\big] }{\alpha\cdot \frac{\overline{\nu}(\beta t)}{\beta t}\cdot \mathbf{E}\big[ \underline{X}_t -\underline{X}_{t+\delta}\big]}   . 
 \eeqlb
 In order to decompose the numerator into several tractable components, we introduce the following four mutually exclusive events: for any constants $\eta\in(0,\beta-\lambda)$, $\varepsilon\in(0,1\wedge \delta)$ and $R\geq1$, 
 \beqnn
 A^{(1)}_z(t,\eta)\ar:=\ar  \Big\{ \big|-\underline X_t -z-\beta t\big|>\eta t \Big\},\cr
 A^{(2)}_z(t,\eta,\varepsilon)\ar:=\ar  \Big\{  \big|-\underline X_t-z-\beta t\big|\leq\eta t,\,  \underline X_t-\underline X_{t+\delta} <\varepsilon \Big\}, \cr
 A^{(3)}_z(t,\eta,\varepsilon,R)\ar:=\ar  \Big\{  \big|-\underline X_t-z-\beta t\big|\leq\eta t,\,  \underline X_t-\underline X_{t+\delta} \in [\varepsilon,R] \Big\}, \cr 
 A^{(4)}_z(t,\eta,R)\ar:=\ar  \Big\{ \big|-\underline X_t -z-\beta t\big|\leq \eta t,\,\underline X_t-\underline X_{t+\delta} >R \Big\}.
 \eeqnn
 Note that these four events  together form the certain event. Then, by the triangle inequality  the numerator of \eqref{eqn.70441} can be bounded by the sum of the following three terms 
 \beqnn
 I^{(1)}_{z}(t,\eta,\varepsilon,R)\ar:=\ar  \mathbf{E}\bigg[  \Big|\nu\big( [ -\underline{X}_{t}-z,-\underline{X}_{t+\delta}-z )  \cap (\lambda t,\infty) \big) -\alpha\cdot \frac{\overline{\nu}(\beta t)}{\beta t}\cdot \big(\underline{X}_t -\underline{X}_{t+\delta}\big)\Big|, A^{(3)}_z(t,\eta,\varepsilon,R)\bigg] ,\cr\cr
 I^{(2)}_{z}(t,\eta,\varepsilon,R)\ar:=\ar  \mathbf{E}\Big[   \alpha\cdot \overline{\nu}(\beta t)/(\beta t)\cdot \big(\underline{X}_t -\underline{X}_{t+\delta}\big) ,\, A^{(1)}_z(t,\eta)\cup A^{(2)}_z(t,\eta,\varepsilon)\cup A^{(4)}_z(t,\eta,R) \Big] ,\cr\cr
 I^{(3)}_{z}(t,\eta,\varepsilon,R)\ar:=\ar  \mathbf{E}\Big[   \nu\big( [ -\underline{X}_{t}-z, -\underline{X}_{t+\delta} -z) \cap (\lambda t,\infty) \big)  ,\, A^{(1)}_z(t,\eta)\cup A^{(2)}_z(t,\eta,\varepsilon)\cup A^{(4)}_z(t,\eta,R)\Big]  . 
 \eeqnn
 By Lemma~\ref{Lemma.310}, we also have $\mathbf{E}\big[ \underline{X}_t -\underline{X}_{t+\delta}\big] \to \beta \delta $ as $t\to\infty$. 
 Consequently, the limit \eqref{eqn.70442} follows if we can prove that 
 \beqlb\label{eqn.70443}
 \limsup_{\eta\to0+} \limsup_{t\to \infty}\sup_{z\in[0,K]} \frac{I^{(1)}_{z}(t,\eta,\varepsilon,R)}{\overline{\nu}(\beta t)/t} =0,\quad \forall\, R>0,\, \varepsilon\in (0,1\wedge \delta),
 \eeqlb
 and 
 \beqlb\label{eqn.70444}
 \limsup_{R\to\infty} \limsup_{\varepsilon\to 0+}   \limsup_{t\to \infty}\sup_{z\in[0,K]} \frac{I^{(i)}_{z}(t,\eta,\varepsilon,R) }{\overline{\nu}(\beta t)/t} =0 ,\quad \forall \eta \in (0,\beta-\lambda), i=2,3.
 \eeqlb
 
 \textbf{Case $i=1$.} We now prove the limit \eqref{eqn.70443}. On the event $A^{(3)}_z(t,\eta,\varepsilon,R)$, we observe that 
 \beqnn
 \lambda t < -\underline{X}_t-z \quad 	(\beta-\eta)t \leq -\underline X_t-z \leq (\beta+\eta)t 
 \quad\mbox{and}\quad 
 \varepsilon \leq \underline X_t-\underline X_{t+\delta} \leq R,
 \eeqnn
 which follows that 
 \beqnn
 I^{(1)}_{z}(t,\eta,\varepsilon,R)
 \ar= \ar \, \mathbf{E}\bigg[  \Big|\nu\big( [ -\underline{X}_{t}-z,-\underline{X}_{t+\delta}-z )  \big) -\alpha\cdot \frac{\overline{\nu}(\beta t)}{\beta t}\cdot \big(\underline{X}_t -\underline{X}_{t+\delta}\big)\Big|,\, A^{(3)}_z(t,\eta,\varepsilon,R)\bigg] \cr
 \ar=\ar \mathbf{E}\bigg[  \bigg|\frac{\nu\big( [ -\underline{X}_{t}-z,-\underline{X}_{t+\delta}-z )  \big)}{\alpha\cdot \frac{\overline{\nu}(\beta t)}{\beta t}\cdot \big(\underline{X}_t -\underline{X}_{t+\delta}\big) } -1\bigg| \cdot \alpha\cdot \frac{\overline{\nu}(\beta t)}{\beta t}\cdot \big(\underline{X}_t -\underline{X}_{t+\delta}\big),\, A^{(3)}_z(t,\eta,\varepsilon,R)\bigg] \cr
 \ar\leq \ar\,\sup_{|y-\beta t|\leq \eta t} \sup_{\delta' \in [\varepsilon, R]} \bigg|\frac{\nu\big( y +\varDelta_{\delta'}  \big)}{\alpha\cdot \frac{\overline{\nu}(\beta t)}{\beta t}\cdot \delta' } -1\bigg| \cdot \alpha\cdot \frac{\overline{\nu}(\beta t)}{\beta t}\cdot \mathbf{E}\Big[   \big(\underline{X}_t -\underline{X}_{t+\delta}\big),\, A^{(3)}_z(t,\eta,\varepsilon,R)\Big] . 
 \eeqnn 
 The last expectation can be bounded by $\mathbf{E}\big[   \big(\underline{X}_t -\underline{X}_{t+\delta}\big) \big] \to \beta \delta$ as $t\to\infty$; see Lemma~\ref{Lemma.310}.  
 Consequently, uniformly for all sufficiently large $t$, 
 \beqnn
 \sup_{z\in[0,K]}\frac{I^{(1)}_{z}(t,\eta,\varepsilon,R)}{\overline{\nu}(\beta t)/t}\ar\leq \ar C\cdot \sup_{|y-\beta t|\leq \eta t} \sup_{\delta' \in [\varepsilon, R]} \bigg|\frac{\nu\big( y +\varDelta_{\delta'}  \big)}{\alpha\cdot \frac{\overline{\nu}(\beta t)}{\beta t}\cdot \delta' } -1\bigg| .
 \eeqnn
 By the triangle inequality, this supremum can further  be bounded by
 \beqlb\label{eqn.704431}
 \sup_{|y-\beta t|\leq \eta t} \sup_{\delta' \in [\varepsilon, R]} \bigg|\frac{\nu\big( y +\varDelta_{\delta'}  \big)}{\alpha\cdot \overline{\nu}(y)/y\cdot \delta' } -1\bigg|  \cdot \frac{\overline{\nu}(y)/y}{\overline{\nu}(\beta t)/(\beta t)} +  \sup_{|y-\beta t|\leq \eta t} \bigg|  \frac{\overline{\nu}(y)/y}{\overline{\nu}(\beta t)/(\beta t)}  -1\bigg| . 
 \eeqlb
 By Lemma~\ref{Lemma.301} and \eqref{eqn.30}, we have as $t\to\infty$,
 \beqnn
 \sup_{|y-\beta t|\leq \eta t} \sup_{\delta' \in [\varepsilon, R]} \bigg|\frac{\nu\big( y +\varDelta_{\delta'}  \big)}{\alpha\cdot \overline{\nu}(y)/y\cdot \delta' } -1\bigg|  \to0
 \quad \mbox{and}\quad 
 \sup_{|y-\beta t|\leq \eta t} \bigg|  \frac{\overline{\nu}(y)/y}{\overline{\nu}(\beta t)/(\beta t)}  -1\bigg| \to \Big( 1-\frac{\eta}{\beta} \Big)^{-\alpha-1} -1,
 \eeqnn
 which goes to $0$ as $\eta\to 0+$.  Taking them back into \eqref{eqn.704431} yields the limit \eqref{eqn.70443} immediately. 
 
 \textbf{Case $i=2$.} By the definition of $I^{(2)}_{z}(t,\eta,\varepsilon,R)$, we have 
 \beqnn
 \frac{I^{(2)}_{z}(t,\eta,\varepsilon,R) }{\overline{\nu}(\beta t)/t}
 \ar=\ar \frac{\alpha}{\beta} \cdot  \mathbf{E}\big[  \underline{X}_t -\underline{X}_{t+\delta},\,   A^{(1)}_z(t,\eta)\cup A^{(2)}_z(t,\eta,\varepsilon)\cup A^{(4)}_z(t,\eta,R)  \big]  . 
 \eeqnn
 Firstly, we have $ \sup_{z\in[0,K]}\mathbf{E}\big[  \underline{X}_t -\underline{X}_{t+\delta},\,     A^{(2)}_z(t,\eta,\varepsilon) \big] \leq  \mathbf{E}\big[  \underline{X}_t -\underline{X}_{t+\delta},\,   \underline{X}_t -\underline{X}_{t+\delta}<\varepsilon  \big] \leq   \varepsilon   $ uniformly in  $t\geq 0$. 
 Secondly, since the family $\{ \underline{X}_{t} -\underline{X}_{t+\delta} \}_{t\geq 0}$ is uniformly integrable  and $\sup_{z\in[0,K]}\mathbf{P}\big(A^{(1)}_z(t,\eta)\big) \to0$ as $t\to\infty$; see  Lemma~\ref{Lemma.310}, we have 
 \beqlb\label{eqn.70445}
 \limsup_{t\to\infty} \sup_{z\in[0,K]} \mathbf{E}\big[  \underline{X}_t -\underline{X}_{t+\delta},\, A^{(1)}_z(t,\eta)  \big] =0.
 \eeqlb
 Finally, by Chebyshev's inequality and Lemma~\ref{Lemma.310}, we have as $t\to\infty$, 
 \beqnn
 \sup_{z\in[0,K]} \mathbf{P}\big(A^{(4)}_z(t,\eta,R)\big) \leq  \mathbf{P}\big(\underline X_t-\underline X_{t+\delta} >R \big)  
 \leq  \frac{1}{R} \cdot   \mathbf{E}\big[\underline{X}_t-\underline X_{t+\delta} \big] \to \frac{\beta \delta }{R},
 \eeqnn
 which goes to $0$ as $R\to\infty$. This along with the uniform integrability of  $\{ \underline{X}_{t} -\underline{X}_{t+\delta} \}_{t\geq 0}$ yields 
 \beqlb\label{eqn.70446}
 \limsup_{R\to\infty}\limsup_{t\to\infty}  \sup_{z\in[0,K]} \mathbf{E}\big[  \underline{X}_t -\underline{X}_{t+\delta},\, A^{(4)}_z(t,\eta,R)  \big] =0. 
 \eeqlb
 Combining these limits together, we can obtain the limit \eqref{eqn.70444} with $i=2$ immediately. 
 
 \textbf{Case $i=3$.} 
 By \eqref{eqn.70113} and the fact that $\overline\nu \in \mathrm{RV}^\infty_{-\alpha}$, we have for all sufficiently large $t$, 
 \beqnn
 \sup_{z\in[0,K]}\frac{\mathbf{E}\big[   \nu\big( [ -\underline{X}_{t}-z, -\underline{X}_{t+\delta} -z) \cap (\lambda t,\infty) \big)  ,\, A^{(1)}_z(t,\eta)  \big]}{\overline{\nu}(\beta t)/t}  
 \ar\leq\ar C  \sup_{z\in[0,K]}\mathbf{E}\big[   \underline{X}_{t} -\underline{X}_{t+\delta}+1,\, A^{(1)}_z(t,\eta)   \big] ,
 \eeqnn
 which goes to $0$ as $t\to\infty$; see \eqref{eqn.70445}. Similarly, this inequality still holds with $A^{(1)}_z(t,\eta)$ replaced by $A^{(4)}_z(t,\eta,R) $, and also goes to  $0$ as $t\to\infty $ and then $R\to\infty$; see \eqref{eqn.70446}. 
 Finally, on the event $A^{(2)}_z(t,\eta,\varepsilon)$, we have $	(-\underline{X}_t-z)\vee\lambda t = -\underline{X}_t-z$. 
 Since $\delta'\mapsto \nu(x+\varDelta_{\delta'})$ is non-decreasing, we have 
 \beqnn
 \sup_{z\in[0,K]}\mathbf{E}\big[   \nu\big( [ -\underline{X}_{t}-z, -\underline{X}_{t+\delta} -z) \cap (\lambda t,\infty) \big)  ,\, A^{(2)}_z(t,\eta,\varepsilon)  \big]
 \leq \sup_{|y|\leq \eta t}\nu(\beta t+y+\varDelta_{\varepsilon}).
 \eeqnn
 Moreover, by \eqref{eqn.30} and the monotonicity of $\overline\nu$ we have as $t\to\infty$,
 \beqnn
 \sup_{|y|\leq \eta t} \frac{\nu(\beta t+y+\varDelta_{\varepsilon})}{\overline{\nu}(\beta t)/t} 
 \leq C\cdot \frac{\varepsilon}{\beta-\eta} \cdot \frac{\overline{\nu}(\beta t-\eta t)}{\overline{\nu}(\beta t)} \to C\cdot \frac{\varepsilon}{\beta-\eta} \cdot \big(  1-\eta/\beta\big)^{-\alpha},
 \eeqnn
 which goes to $0$ as $\varepsilon \to 0+$. 
 Combining these limits together yields the limit \eqref{eqn.70444} with $i=3$.
 \qed

 \begin{proposition}\label{Prop.7015}
 	For any constants  $R,\delta>0$ and $\lambda\in (0,\beta)$, we have 
 	\beqnn
 	\lim_{K\to\infty} \limsup_{t\to\infty} 
 	\underline{n}\Big( \mathcal{J}^{\lambda t}\leq R,\, \epsilon_{\mathcal{J}^{\lambda t}-}>K \,\Big|\, \zeta\in t+\varDelta_\delta \Big) =0. 
 	\eeqnn
 \end{proposition}
 \proof Firstly, we have for sufficiently large $t$,
 \beqnn
 \underline{n}\big( \mathcal{J}^{\lambda t}\leq R,\, \epsilon_{\mathcal{J}^{\lambda t}-}>K,\, \zeta\in t+\varDelta_\delta \big) 
 \ar\leq\ar   \underline{n}\bigg( \sum_{0<s\leq R\wedge\zeta} \mathbf 1_{\{\epsilon_{s-}>K,\,\Delta\epsilon_s>\lambda t,\,\zeta\in t+\varDelta_{\delta}\}}
 \bigg). 
 \eeqnn
 Similarly as in \eqref{eqn.509},   the right-hand side equals 
 \beqnn
 \int_0^R ds \int_K^\infty \underline{n}_s(dz) \int_{\lambda t}^\infty
 \mathbf{P}_{z+w}\big( \tau_0^-\in t-s+\varDelta_\delta \big)\,\nu(dw),
 \eeqnn
 which can further be bounded by 
 \beqnn
 \lefteqn{\sup_{s\in[0,R]} \sup_{z\geq K} \int_{\lambda t}^\infty
 	\mathbf{P}_{z+w}\big( \tau_0^-\in t-s+\varDelta_\delta \big)\,\nu(dw) \cdot \int_0^R
 	\underline n\big(
 	\epsilon_{s-}>K,\zeta>s
 	\big)ds}\cr
 \ar\leq\ar  \sup_{z\geq K} \int_{\lambda t}^\infty
 \mathbf{P}_{z+w}\big( \tau_0^-\in t-R+\varDelta_{R+\delta} \big)\,\nu(dw) \cdot \int_0^\infty
 \underline n\big(
 \epsilon_{s-}>K,\zeta>s
 \big)ds. 
 \eeqnn
 By Proposition~\ref{Prop.7011}, we have for sufficiently large $t$,
 \beqnn
 \underline{n}\Big( \mathcal{J}^{\lambda t}\leq R,\, \epsilon_{\mathcal{J}^{\lambda t}-}>K \,\Big|\, \zeta\in t+\varDelta_\delta \Big)
 \leq C\cdot \frac{\nu\big(\beta t+\varDelta_\delta \big) }{\underline{n}\big(\zeta\in t+\varDelta_\delta \big)} \cdot \int_0^\infty
 \underline n\big(
 \epsilon_{s-}>K,\zeta>s
 \big)ds.
 \eeqnn
 Since $\underline{n}(\zeta)<\infty$, the last integral vanishes as $K\to\infty$.  
 The desired limit follows from this and Theorem~\ref{MainThm01}.
 \qed

 \begin{proposition}\label{Prop.7012}
 	For any constants $R,K,\delta>0$ and $\lambda \in(0,\beta)$, we have as $t\to\infty$, 
 	\beqnn
 	\bigg| \underline{n}\Big( \mathcal{J}^{\lambda t}\leq R,\, \epsilon_{\mathcal{J}^{\lambda t}-}\leq K \,\Big|\, \zeta\in t+\varDelta_\delta \Big) - \frac{1}{\underline{n}(\zeta)}  \int_0^R\underline{n}\big(\epsilon_{s-}\leq K,\zeta>s \big)ds  \bigg|
 	\to 0.
 	\eeqnn
 \end{proposition}
 \proof Recall the set $\mathcal{B}^x$ defined in \eqref{eqn.70}.   For $R,K,\delta>0$, we first introduce two sets of jumps with different characteristics:
 \beqlb\label{eqn.7001}
 \mathcal{B}^{\lambda t}_{\leq R}:=\mathcal{B}^{\lambda t} \cap (0,R]
 \quad \mbox{and}
 \quad
 \mathcal{B}^{\lambda t}_{\leq R}(K) := \mathcal{B}^{\lambda t}_{\leq R} \cap   \big\{ 0<s\leq  \zeta:   \epsilon_{s-}\leq K  \big\}. 
 \eeqlb
 On the event $ \big\{ \mathcal J^{\lambda t}\leq R,\, \epsilon_{\mathcal J^{\lambda t}-}\leq K,\, \zeta\in t+\varDelta_\delta \big\}$, the first jump of size greater than $\lambda t$ occurs before time $R$ and its pre-jump position is bounded by $K$. Hence this jump belongs to $ \mathcal{B}^{\lambda t}_{\leq R}(K) $, and therefore 
 \beqnn
 \big\{\mathcal{J}^{\lambda t}\leq R,\, \epsilon_{\mathcal{J}^{\lambda t}-}\leq K ,\, \zeta\in t+\varDelta_\delta \big\}
 \subset   \big\{ \#\mathcal{B}^{\lambda t}_{\leq R}(K)\geq 1,\, \zeta\in t+\varDelta_\delta \big\}. 
 \eeqnn
 Conversely, suppose that $\# \mathcal{B}^{\lambda t}_{\leq R}(K)\geq1$, but $ \epsilon_{\mathcal J^{\lambda t}-}>K.$ 
 Since the existence of a jump in $\mathcal{B}^{\lambda t}_{\leq R}(K)$ already implies that $\mathcal J^{\lambda t}\leq R$, there must then be another jump of size greater than $\lambda t$, occurring after $\mathcal J^{\lambda t}$ but no later than time $R$, whose pre-jump position is at most $K$. Consequently, there are at least two jumps of size greater than $\lambda t$ before time $R$. Thus 
 \beqnn
 \big\{  \# \mathcal{B}^{\lambda t}_{\leq R}(K)\geq 1 ,\, \zeta\in t+\varDelta_\delta \big\}\setminus   \big\{\mathcal{J}^{\lambda t}\leq R,\, \epsilon_{\mathcal{J}^{\lambda t}-}\leq K ,\, \zeta\in t+\varDelta_\delta \big\} \subset \big\{\# \mathcal{B}^{\lambda t}_{\leq R} \geq 2  ,\, \zeta\in t+\varDelta_\delta \big\}. 
 \eeqnn
 The preceding two containment relations imply that 
 \beqnn
 \Big| \underline{n}\big( \mathcal{J}^{\lambda t}\leq R,\, \epsilon_{\mathcal{J}^{\lambda t}-}\leq K ,\, \zeta\in t+\varDelta_\delta \big)
 - \underline{n}\big( \# \mathcal{B}^{\lambda t}_{\leq R}(K) ,\, \zeta\in t+\varDelta_{\delta}\big)\Big| 
 \leq  \underline{n}\Big( \# \mathcal{B}^{\lambda t}_{\leq R} \cdot\big( \# \mathcal{B}^{\lambda t}_{\leq R} -1\big) \Big). 
 \eeqnn 
 By using the square of a sum formula and \eqref{Prop.203}, 
 \beqlb\label{eqn.701}
 \underline{n}\bigg( \# \mathcal{B}^{\lambda t}_{\leq R} \cdot\big( \# \mathcal{B}^{\lambda t}_{\leq R} -1\big) \bigg) =  2\cdot \underline{n}\bigg( \sum_{0< s<r\leq R\wedge \zeta} \mathbf{1}_{\{\Delta\epsilon_s> \lambda t,\,\Delta\epsilon_r> \lambda t \}}\bigg)
 \leq 
 2\cdot R\cdot\underline{n}(\zeta)\cdot 
 \big|\overline\nu ( \lambda t  )\big|^2,
 \eeqlb
 which is $o\big( \underline{n}(\zeta \in t+\varDelta_\delta) \big)$ as $t\to\infty$ by Theorem~\ref{MainThm01} and $\overline\nu \in \mathrm{RV}^\infty_{-\alpha}$ with $\alpha>1$.   
 Consequently,
 \beqnn
 \limsup_{t\to\infty}\bigg| \underline{n}\Big( \mathcal{J}^{\lambda t}\leq R,\, \epsilon_{\mathcal{J}^{\lambda t}-}\leq K \,\Big|\, \zeta\in t+\varDelta_\delta \Big) - \frac{\underline{n}\big( \# \mathcal{B}^{\lambda t}_{\leq R}(K) ,\,\zeta\in t+\varDelta_{\delta}\big) }{\underline{n}\big(\zeta \in t+\varDelta_\delta\big)}    \bigg|
 =0
 \eeqnn
 and hence it suffices to prove 
 \beqlb\label{eqn.7014}
 \lim_{t\to\infty}\frac{\underline{n}\big( \# \mathcal{B}^{\lambda t}_{\leq R}(K) ,\,\zeta\in t+\varDelta_{\delta} \big) }{\underline{n}\big(\zeta \in t+\varDelta_\delta\big)} 
 =\frac{1}{\underline{n}(\zeta)}  \int_0^R\underline{n}\big(\epsilon_{s-}\leq K,\zeta>s \big)ds . 
 \eeqlb
 Using the strong Markov property and the
 compensation formula \eqref{eqn.Compensation} again, 
 \beqnn
 \underline{n}\big( \# \mathcal{B}^{\lambda t}_{\leq R}(K) ,\, \zeta\in t+\varDelta_{\delta}\big)
 \ar=\ar \int_0^R ds  \int_0^K  \underline{n}_s(dz)
 \int_{\lambda t}^\infty 
 \mathbf{P}_{z+w}\big( \tau_0^-\in t-s+\varDelta_\delta \big)\nu(dw). 
 \eeqnn
 Then the limit \eqref{eqn.7014} follows from the uniform convergence in Proposition~\ref{Prop.7011} and Theorem~\ref{MainThm01}. 
 \qed 
 
 \medskip
 
 \textbf{\textit{Proof of Lemma~\ref{Lemma.701}.}} 
 Recall the distribution law of $\mathcal{T}$ in \eqref{CWX02}. By the triangle inequality,
 \beqnn
 \lefteqn{\Big| \underline{n}\big(\mathcal{J}^{\lambda t}\leq R \,\big|\,\zeta\in t+\varDelta_\delta   \big) - \mathbf{P}\big(\mathcal T\leq R \big) \Big| }\cr
 \ar=\ar \Big| \underline{n}\big(\mathcal{J}^{\lambda t}\leq R \,\big|\,\zeta\in t+\varDelta_\delta   \big) -\frac{ 1 }{\underline{n}(\zeta)}  \int_0^R ds \int_0^\infty \underline{n}\big( \epsilon_{s-}\in dz ,\, \zeta>s \big)\Big| \cr
 \ar\leq\ar \underline{n}\big(\mathcal{J}^{\lambda t}\leq R, \, \epsilon_{\mathcal{J}^{\lambda t} -} >K \,\big|\,\zeta\in t+\varDelta_\delta   \big) + \frac{ 1 }{\underline{n}(\zeta)}  \int_0^R   \underline{n}\big( \epsilon_{s-}>K ,\, \zeta>s \big)\,ds\cr
 \ar\ar + \Big| \underline{n}\big(\mathcal{J}^{\lambda t}\leq R ,\, \epsilon_{\mathcal{J}^{\lambda t} -} \leq K  \,\big|\,\zeta\in t+\varDelta_\delta   \big) -\frac{ 1 }{\underline{n}(\zeta)}  \int_0^R    \underline{n}\big( \epsilon_{s-}\leq K,\, \zeta>s \big)\,ds\Big| , 
 \eeqnn  
 which goes to $0$ as $t\to\infty$ and then $K\to\infty$ by Proposition~\ref{Prop.7015}, \ref{Prop.7012} and the fact that $\underline{n}(\zeta)<\infty$.   
 \qed

 \medskip
 
 \textbf{\textit{Proof of Lemma~\ref{Lemma.702}.}} Firstly, on the event $\{\mathcal{J}^{\lambda t}\leq R,\, \sup_{0\leq r<\mathcal{J}^{\lambda t}}\epsilon_r>K ,\, \zeta\in t+\varDelta_\delta  \} $,  there is at least one $s\in(0,R]$ such that  $  \Delta\epsilon_s> \lambda t$, $ \sup_{0\leq r<s}\epsilon_r>K $ and $\zeta\in t+\varDelta_\delta $. Consequently,    
 \beqnn
 \underline n\Big( \mathcal{J}^{\lambda t}\leq R,\, \sup_{0\leq r<\mathcal{J}^{\lambda t}}\epsilon_r>K ,\, \zeta\in t+\varDelta_\delta  \Big) 
 \leq \underline n\bigg(   \sum_{0<s\leq R} \mathbf{1}_{\{ \Delta\epsilon_s> \lambda t,\, \sup_{0\leq r<s}\epsilon_r>K, \, \zeta\in t+\varDelta_\delta \}}  \bigg).
 \eeqnn
 By  the compensation formula \eqref{eqn.Compensation}, the right-hand side equals 
 \beqnn
 \int_0^R ds \int_0^\infty \underline n\Big( \sup_{0\leq r<s}\epsilon_r>K,\,
 \epsilon_{s-}\in dz,\, \zeta>s \Big) \int_{\lambda t}^\infty \mathbf P_{z+w}\big(
 \tau_0^-\in t-s+\varDelta_\delta  \big)\,\nu(dw),
 \eeqnn
 which can be bounded by 
 \beqnn
 \lefteqn{\sup_{z\geq 0} \sup_{t-R\leq r\leq t} \int_{\lambda t}^\infty \mathbf P_{z+w}\big( \tau_0^-\in r+\varDelta_\delta  \big)\,\nu(dw) \cdot
 	\int_0^R \underline n\Big(  \sup_{0\leq r<s}\epsilon_r>K,\, \zeta>s \Big)\,ds}\cr
 \ar\leq\ar  \sup_{z\geq 0}   \int_{\lambda t}^\infty \mathbf P_{z+w}\big( \tau_0^-\in t-R+\varDelta_{R+\delta}  \big)\,\nu(dw) \cdot
 \int_0^\infty \underline n\Big(  \sup_{0\leq r<s}\epsilon_r>K,\, \zeta>s \Big)\,ds.
 \eeqnn
 Since $\underline{n}(\zeta)<\infty $, the last integral goes to $0$ as $K\to\infty$. 
 Moreover, by the uniform bound in Proposition~\ref{Prop.7011} and Theorem~\ref{MainThm01}, there exists some constant $t_0>0$ such that
 \beqnn
 \sup_{t\geq t_0}\frac{1}{\underline{n}\big( \zeta\in t+\varDelta_\delta \big)} \sup_{z\geq 0}   \int_{\lambda t}^\infty \mathbf P_{z+w}\big( \tau_0^-\in t-R+\varDelta_{R+\delta}  \big)\,\nu(dw) <\infty.
 \eeqnn
 Combining these estimates together yields the desired limit directly. 
 \qed

 \subsection{Proof of Lemma~\ref{Lemma.703}}
 
 For sufficiently large $t$, conditioned on $\zeta \in t+ \varDelta_\delta$, each jump with size larger than $\lambda t$ brings the excursion to levels above $K$, and hence there is at most one $s\in (0,\zeta]$ such that  $\Delta\epsilon_s > \lambda t $ and $ \sup_{0\leq r<s}\epsilon_r\leq K $. This yields that for any sufficiently large $t$, 
 \beqnn
 \lefteqn{\underline n\Big( \mathcal{J}^{\lambda t}\leq R,\, \sup_{0\leq r<\mathcal{J}^{\lambda t}}\epsilon_r\leq K,\, \sup_{\mathcal{J}^{\lambda t}\leq r<\zeta}  \epsilon_{r} - \epsilon_{\mathcal{J}^{\lambda t}} >K,\, \zeta\in t+\varDelta_\delta \Big)} \cr
 \ar=\ar \underline n\bigg( \sum_{0<s\leq R} \mathbf{1}_{ \{\Delta\epsilon_s > \lambda t,\,\sup_{0\leq r<s}\epsilon_r\leq K,\, \sup_{s\leq r<\zeta}  \epsilon_{r} - \epsilon_s >K,\, \zeta\in t+\varDelta_\delta \}} \bigg).
 \eeqnn
 Similarly as in \eqref{eqn.509},  the right-hand side equals  
 \beqnn
 \int_0^R ds \int_0^K \underline n\Big( \sup_{0\leq r<s}\epsilon_r\leq K,\, \epsilon_{s-}\in dz,\, \zeta>s \Big) \int_{\lambda t}^\infty \mathbf{P}_{z+w} \big( \tau_0^-\in t-s+\varDelta_\delta,\, \overline X_{\tau_0^-}-z-w >K \big)\, \nu(dw).
 \eeqnn
 By using the spatial homogeneity of $X$ and then the definition of $\tau_{-z-w}^-$, 
 \beqlb\label{eqn.702221}
 \lefteqn{\int_{\lambda t}^\infty \mathbf{P}_{z+w} \big( \tau_0^-\in t-s+\varDelta_\delta,\, \overline X_{\tau_0^-}-z-w >K \big)\, \nu(dw)}\cr
 \ar=\ar \int_{\lambda t}^\infty \mathbf{P} \big( \tau_{-z-w}^-\in t-s+\varDelta_\delta,\, \overline X_{\tau_{-z-w}^-}  >K \big)\, \nu(dw)\cr
 \ar=\ar \int_{\lambda t}^\infty \mathbf{P}\Big( \underline{X}_{t-s}\geq -z-w >\underline{X}_{t-s+\delta},\, \overline{X}_{\tau_{-z-w}^-}>K \Big)\, \nu(dw) .
 \eeqlb
 Conditionally on $\tau_{-z-w}^-\in t-s+\varDelta_\delta$, we see that  $\overline{X}_{\tau_{-z-w}^-} \leq \overline{X}_{ t-s+ \delta}$. This induces that 
 \beqlb\label{eqn9.74}
 \lefteqn{ \int_{\lambda t}^\infty \mathbf{P}_{z+w} \big( \tau_0^-\in t-s+\varDelta_\delta,\, \overline X_{\tau_0^-}-z-w >K \big)\, \nu(dw) }\cr
 \ar\leq\ar \int_{\lambda t}^\infty \mathbf{P}\Big( \underline{X}_{t-s}\geq -z-w >\underline{X}_{t-s+\delta},\, \overline{X}_{t-s+ \delta}>K \Big)\, \nu(dw) \cr
 \ar=\ar \mathbf E\Big[ \nu\Big( \big( (-\underline{X}_{t-s}-z)\vee\lambda t, -\underline{X}_{t-s+\delta}-z \big] \Big),  \, \overline{X}_{t-s+\delta}>K \Big],
 \eeqlb
 The interval $\big( (-\underline{X}_{t-s}-z)\vee\lambda t, -\underline{X}_{t-s+\delta}-z \big]$ has left endpoint larger than $\lambda t$ and length smaller than $\underline{X}_{t-s}-\underline{X}_{t-s+\delta} +1$. 
 By Corollary~\ref{Corollary.202} and \eqref{eqn.30}, we have uniformly for all sufficiently large $t$,
 \beqnn
 \int_{\lambda t}^\infty  \mathbf{P}\big( \tau_{-z-w}^-\in t-s+\varDelta_\delta,\, \overline X_{\tau_{-z-w}^-}>K \big)\,\nu(dw)
 \ar\leq\ar C\cdot\frac{\overline\nu(\lambda t)}{\lambda t}\cdot \mathbf{E}\Big[ \underline X_{t-s}-\underline X_{t-s+\delta}+1;\, \overline X_{t-s+\delta}>K 	\Big].
 \eeqnn
 Combining the preceding estimates together yields that uniformly for all sufficiently large $t$,
 \beqnn
 \lefteqn{\underline n\Big( \mathcal{J}^{\lambda t}\leq R,\, \sup_{0\leq r<\mathcal{J}^{\lambda t}}\epsilon_r\leq K,\, \sup_{\mathcal{J}^{\lambda t}\leq r<\zeta}\epsilon_r-\epsilon_{\mathcal{J}^{\lambda t}}>K,\, \zeta\in t+\varDelta_\delta \Big)}\cr\cr
 \ar\leq\ar C\cdot\frac{\overline\nu(\lambda t)}{\lambda t} \cdot\sup_{u\in[t-R,t]} \mathbf E\Big[ \underline X_u-\underline X_{u+\delta}+1;\, \overline X_{u+\delta}>K \Big] 	\cdot\int_0^R \underline n\Big( \sup_{0\leq r<s}\epsilon_r\leq K,\, \zeta>s \Big)\,ds.
 \eeqnn
 Obviously, the last integral can be bounded by $\underline n(\zeta)<\infty$.
 By the monotonicity of $\underline X$ and $\overline X$,
 \beqnn
 \sup_{u\in[t-R,t]} \mathbf E\Big[ \underline X_u-\underline X_{u+\delta}+1;\, \overline X_{u+\delta}>K \Big]
 \leq \mathbf E\Big[ \underline X_{t-R}-\underline X_{t+\delta}+1;\, \overline X_{t+\delta}>K \Big],
 \eeqnn
 which, for any $M>0$, can further be decomposed into the following two terms:
 \beqlb\label{eqn.7022}
 \mathbf E\Big[ \underline X_{t-R}-\underline X_{t+\delta}+1;\, \underline X_{t-R}-\underline X_{t+\delta}>M,\, \overline X_{t+\delta}>K \Big]
 \eeqlb
 and
 \beqlb\label{eqn.7023}
 \mathbf E\Big[ \underline X_{t-R}-\underline X_{t+\delta}+1;\, \underline X_{t-R}-\underline X_{t+\delta}\leq M,\, \overline X_{t+\delta}>K \Big].
 \eeqlb
 By using \eqref{eqn.7003} with $r=t-R$ and $\delta_2=R+\delta$, we can control the expectation \eqref{eqn.7022} by
 \beqnn
 \mathbf E\Big[ \underline X_{t-R}-\underline X_{t+\delta}+1;\, \underline X_{t-R}-\underline X_{t+\delta}>M \Big]
 \ar\leq\ar \mathbf E\Big[ 1-\underline X_{R+\delta};\, -\underline X_{R+\delta}>M \Big].
 \eeqnn
 Moreover, the expectation \eqref{eqn.7023} can be bounded by $(M+1)\cdot\mathbf P(\overline X_\infty>K)$.
 Consequently,
 \beqnn
 \lefteqn{\underline n\Big( \mathcal{J}^{\lambda t}\leq R,\, \sup_{0\leq r<\mathcal{J}^{\lambda t}}\epsilon_r\leq K,\, \sup_{\mathcal{J}^{\lambda t}\leq r<\zeta}\epsilon_r-\epsilon_{\mathcal{J}^{\lambda t}}>K\,\Big|\,\zeta\in t+\varDelta_\delta \Big)}\cr
 \ar\leq\ar  \frac{C\cdot\overline\nu(\lambda t)}{\lambda t\,\underline n(\zeta\in t+\varDelta_\delta)} \cdot \mathbf E\Big[1-\underline X_{R+\delta};\, -\underline X_{R+\delta}>M \Big]+  \frac{C\cdot\overline\nu(\lambda t)}{\lambda t\,\underline n(\zeta\in t+\varDelta_\delta)} \cdot(M+1)\cdot\mathbf P(\overline X_\infty>K).
 \eeqnn
 By Theorem~\ref{MainThm01} and \eqref{eqn.30}, we have uniformly for all sufficiently large $t$,
 \beqnn
 \lefteqn{\underline n\Big( \mathcal{J}^{\lambda t}\leq R,\, \sup_{0\leq r<\mathcal{J}^{\lambda t}}\epsilon_r\leq K,\, \sup_{\mathcal{J}^{\lambda t}\leq r<\zeta}\epsilon_r-\epsilon_{\mathcal{J}^{\lambda t}}>K\,\Big|\,\zeta\in t+\varDelta_\delta \Big)}\cr\cr
 \ar\leq\ar C\cdot\mathbf E\Big[ 2-\underline X_{R+\delta};\, -\underline X_{R+\delta}>M \Big] +C\cdot(M+1)\cdot\mathbf P(\overline X_\infty>K),
 \eeqnn
 which goes to $0$ as $K\to\infty$ and then $M\to\infty$, since $\overline X_\infty<\infty$ a.s. and $\mathbf E[-\underline X_{R+\delta}]<\infty$. This completes the proof. 
 \qed

 \subsection{Proof of Lemma~\ref{Lemma.704} and \ref{Lemma.705}} 
 
 \begin{proposition}\label{Prop.709}
  For any positive function $\theta$ on $\mathbb{R}_+$ such that $\big|X_t+\beta t\big|/\theta_t\overset{\rm p}\to0$, we have as $t\to\infty$,
 	\beqnn
 	\mathbf{P} \big(  |\underline{X}_t+\beta t |\geq \theta_t \big) \to 0. 
 	\eeqnn 
 \end{proposition}
 \proof Note that $\theta_t\to\infty$ as $t\to\infty$. By the triangle inequality  and the duality lemma, we have  
 \beqnn 
 \big|\underline{X}_t+\beta t \big| = \big|\underline{X}_t-X_t+X_t+\beta t \big| 
 \leq \big(X_t- \underline{X}_t\big)+\big|X_t+\beta t\big| 
 \quad \mbox{and}\quad 
 X_t- \underline{X}_t \overset{\rm d}= \overline{X}_t,\quad \forall\, t\geq 0. 
 \eeqnn
 Then   $\mathbf{P} \big(  |\underline{X}_t+\beta t |\geq \theta_t \big) 
 \leq  \mathbf{P} \big(  \overline{X}_t \geq \theta_t/2 \big) +  \mathbf{P} \big(  |X_t+\beta t|\geq \theta_t/2 \big)\to0$ as $t\to\infty$, since $ |X_t+\beta t|/\theta_t \overset{\rm p}\to0$ and $\overline{X}_\infty<\infty$ a.s. 
 \qed 
 
 \textbf{\textit{Proof of Lemma~\ref{Lemma.704}.}}  For sufficiently large $t$, we have 
 \beqnn
 \lefteqn{\underline n\Big( \mathcal{J}^{\lambda t}\leq R,\, \sup_{0\leq r<\mathcal{J}^{\lambda t}}\epsilon_r\leq K,\,   \big| \Delta\epsilon_{\mathcal{J}^{\lambda t}}-\beta t \big| > \frac{\theta_t}{2},\, \zeta\in t+\varDelta_\delta  \Big)}\cr
 \ar\leq\ar \underline n\bigg( \sum_{0<s\leq R} \mathbf{1}_{\{\Delta\epsilon_s> \lambda t,\,  \sup_{0\leq r<s}\epsilon_r\leq K, \,  | \Delta\epsilon_s-\beta t | > \theta_t/2 , \, \zeta\in t+\varDelta_\delta \}} \bigg) . 
 \eeqnn
 Similarly as in \eqref{eqn.509}, the right-hand side equals 
 \beqnn
 \int_0^R ds \int_0^K\underline{n}\Big( \sup_{0\leq r<s}\epsilon_r\leq K, \epsilon_{s-}\in dz, \,\zeta>s\Big)  \int_{\lambda t}^\infty \mathbf{P}_{z+w} \big( \tau_0^-\in t-s +\varDelta_{\delta} \big)\,\nu \Big(dw, |w-\beta t| >  \frac{\theta_t}{2}\Big). 
 \eeqnn
 An analogous argument as in \eqref{eqn.702221} and \eqref{eqn9.74} yields that  
 \beqnn
 \lefteqn{\underline n\Big( \mathcal{J}^{\lambda t}\leq R,\, \sup_{0\leq r<\mathcal{J}^{\lambda t}}\epsilon_r\leq K,\,   \big| \Delta\epsilon_{\mathcal{J}^{\lambda t}}-\beta t \big| > \frac{\theta_t}{2},\, \zeta\in t+\varDelta_\delta  \Big)}\cr\cr
 \ar\leq\ar \underline{n}(\zeta)\cdot    \sup_{s\in [0,R]}   \sup_{z\in [0,K]} \mathbf{E}\Big[ \nu\Big( \big[ -\underline{X}_{t-s}-z, -\underline{X}_{t-s+\delta}-z \big) \cap(\lambda t,\infty)  \cap\Big\{ w: |w-\beta t|>  \frac{\theta_t}{2} \Big\} \Big) \Big]\cr
 \ar\leq\ar \underline{n}(\zeta)\cdot   \sup_{z\in [0,K]} \mathbf{E}\Big[  \nu\Big( \big[ -\underline{X}_{t-R}-z, -\underline{X}_{t+\delta}-z \big) \cap(\lambda t,\infty)  \cap \Big\{ w: |w-\beta t| >  \frac{\theta_t}{2} \Big\}  \Big) \Big]\cr
 \ar\leq\ar \underline{n}(\zeta)\cdot \mathbf{E}\Big[   \nu\Big( \big[ -\underline{X}_{t-R}-K, -\underline{X}_{t+\delta} \big) \cap(\lambda t,\infty)  \cap \Big\{ w: |w-\beta t| >  \frac{\theta_t}{2} \Big\}  \Big) \Big]. 
 \eeqnn 
 Consequently, it suffices to prove that as $t\to\infty$, 
 \beqlb\label{eqn.7002}
 \frac{1}{\nu\big(\beta t +\varDelta_\delta\big)}\cdot \mathbf{E}\Big[ \nu\Big( \big[ -\underline{X}_{t-R}-K, -\underline{X}_{t+\delta} \big) \cap(\lambda t,\infty)  \cap \Big\{ w: |w-\beta t| >\frac{\theta_t}{2} \Big\}  \Big) \Big] \to0.
 \eeqlb
 When $\big[ -\underline{X}_{t-R}-K, -\underline{X}_{t+\delta} \big) \cap(\lambda t,\infty) \cap\big\{ w: |w-\beta t|> \theta_t/2 \big\} \neq\emptyset$,  we have that this interval has length at most $ \underline{X}_{t-R} - \underline X_{t+\delta} +K$ and then for all sufficiently large $t$, 
 \beqnn
 \big| \underline X_{t-R}+\beta t \big| >  \theta_t/4
 \quad \mbox{or}\quad
 \underline{X}_{t-R} - \underline X_{t+\delta} \geq \theta_t/4.
 \eeqnn
 Consequently, for all sufficiently large $t$, we can control the left-hand side of \eqref{eqn.7002} by 
 \beqnn
 \lefteqn{  \frac{1}{\nu\big(\beta t +\varDelta_\delta\big)}\cdot\mathbf{E}\Big[ \nu\big([ -\underline{X}_{t-R}-K, -\underline{X}_{t+\delta} ) \cap(\lambda t,\infty)  \big),  \, \big| \underline X_{t-R}+\beta t \big| >  \frac{\theta_t}{4} \Big]}\cr
 \ar\ar + \frac{1}{\nu\big(\beta t +\varDelta_\delta\big)}\cdot\mathbf{E}\Big[ \nu\big( \big[ -\underline{X}_{t-R}-K, -\underline{X}_{t+\delta} \big) \cap(\lambda t,\infty)  \big) ,\, \underline{X}_{t-R} - \underline X_{t+\delta} \geq \frac{\theta_t}{4}\Big].
 \eeqnn 
 By Corollary~\ref{Corollary.202}, this sum can be bounded by  
 \beqlb\label{eqn.7006}
 C\cdot \mathbf{E}\Big[ \big(\underline{X}_{t-R} - \underline X_{t+\delta} +K+1\big)  \cdot \big( \mathbf{1}_{\{| \underline X_{t-R}+\beta t | >   \theta_t/4 \}} + \mathbf{1}_{\{ \underline{X}_{t-R} - \underline X_{t+\delta} \geq  \theta_t/4 \} }\big) \Big]  ,
 \eeqlb
 uniformly for all sufficiently large $t$. 
 By \eqref{eqn.7003} with $r=t-R$ and $\delta_2=R+\delta$, 
 \beqnn  
 \underline{X}_{t-R} - \underline X_{t+\delta}
 \ar\leq\ar -\inf_{0\leq r\leq R+\delta}  \big(X_{t-R+r}-X_{t-R} \big).
 \eeqnn
 This along with the independent increments of $X$ yields  
 \beqnn
 \lefteqn{\mathbf{E}\Big[ \big(\underline{X}_{t-R} - \underline X_{t+\delta} +K+1\big)  \cdot   \mathbf{1}_{\{| -\underline X_{t-R}-\beta t | >   \theta_t/4 \}}   \Big] }\cr
 \ar\leq\ar \mathbf{E}\bigg[ \Big( -\inf_{0\leq r\leq R+\delta}  \big(X_{t-R+r}-X_{t-R} \big) +K+1\big)  \cdot   \mathbf{1}_{\{| \underline X_{t-R}+\beta t | >   \theta_t/4 \}}   \bigg] \cr
 \ar=\ar \mathbf{E}\bigg[  -\inf_{0\leq r\leq R+\delta}  \big(X_{t-R+r}-X_{t-R} \big) +K+1  \bigg] \cdot \mathbf{P}  \Big( | \underline X_{t-R}+\beta t | >   \frac{\theta_t}{4}  \Big).
 \eeqnn
 The stationary increments of $X$ induce that $\inf_{0\leq r\leq R+\delta}  \big(X_{t-R+r}-X_{t-R} \big)\overset{\rm d}=  \underline{X}_{R+\delta}$ and then by \eqref{eqn.206}, 
 \beqnn
 \mathbf{E}\bigg[  -\inf_{0\leq r\leq R+\delta}  \big(X_{t-R+r}-X_{t-R} \big) +K+1  \bigg] = \mathbf{E}\big[  -   \underline{X}_{R+\delta}  +K+1\big]<\infty. 
 \eeqnn
 By Proposition~\ref{Prop.709}, we also have as $t\to\infty$,
 \beqnn
 \mathbf{P}  \Big( | \underline X_{t-R}+\beta t | >   \frac{\theta_t}{4}  \Big)
 \leq  \mathbf{P}  \Big( \big| \underline X_{t-R}+\beta (t-R) \big| >   \frac{\theta_t}{8}  \Big)\to0. 
 \eeqnn 
 Putting these two estimates together yields that as $t\to\infty$,
 \beqlb\label{eqn.7007}
 \mathbf{E}\Big[ \big(\underline{X}_{t-R} - \underline X_{t+\delta} +K+1\big)  \cdot   \mathbf{1}_{\{| -\underline X_{t-R}-\beta t | >   \theta_t/4 \}}   \Big] \to0.
 \eeqlb
 Moreover, by using \eqref{eqn.7003} with $r=t-R$ and $\delta_2=R+\delta$, we have 
 \beqnn
 \mathbf{E}\Big[ \big(\underline{X}_{t-R} - \underline X_{t+\delta} +K+1\big)  \cdot   \mathbf{1}_{\{ \underline{X}_{t-R} - \underline X_{t+\delta} \geq  \theta_t/4 \} }  \Big] 
 \leq \mathbf{E}\Big[ \big(-\underline{X}_{R+\delta}   +K+1\big)  \cdot   \mathbf{1}_{\{ -\underline{X}_{R+\delta}   \geq  \theta_t/4 \} }  \Big] ,
 \eeqnn
 which goes to $0 $ as $t\to\infty$ by \eqref{eqn.206}.  
 Substituting this and \eqref{eqn.7007} into \eqref{eqn.7006}, we can get the limit \eqref{eqn.7002} immediately. 
 This completes the proof.
 \qed

 \textbf{\textit{Proof of Lemma~\ref{Lemma.705}.}} 
 In view of Remark~\ref{Remark.02}, it suffices to prove this lemma with $\theta_t=o(t)$ as $t\to\infty$. Other cases follow directly from the monotonicity of $s\mapsto\mathbf{P}_y\big(|\tau_0^--t|>s\big)$. 
 By the spatial homogeneity of $X$, we have  
 \beqnn
 \mathbf{P}_y \big(  |\tau_0^- -t| >\theta_t \big)=  \mathbf{P} \big(  |\tau_{-y}^- -t| >\theta_t \big) 
 = \mathbf{P} \big(  \tau_{-y}^- <t-\theta_t \big)  +  \mathbf{P} \big(  \tau_{-y}^- >t+\theta_t \big) . 
 \eeqnn
 On the event $\{\tau_{-y}^- <t-\theta_t \}$, we have $-\underline{X}_{t-\theta_t} >y$ and then  
 \beqnn
 \sup_{|y-\beta t|\leq K} \mathbf{P} \big(  -\underline{X}_{t-\theta_t} >y  \big) 
 \leq \mathbf{P} \big(  -\underline{X}_{t-\theta_t} >\beta t-K  \big)
 \ar=\ar \mathbf{P} \big(  -\underline{X}_{t-\theta_t} -\beta (t-\theta_t)>\beta \theta_t-K  \big), 
 \eeqnn
 which goes to $0$ as $t\to\infty$ by Proposition~\ref{Prop.709}.  
 On the event $\{  \tau_{-y}^- >t+\theta_t  \}$, we have $X_{t+\theta_t}\geq \underline{X}_{t+\theta_t} \geq -y$ and then $X_{t+\theta_t} + \beta (t+\theta_t)\geq \beta (t+\theta_t)-y\geq \beta\theta_t-K$ for all $|y-\beta t|\leq K$, which follows that 
 \beqnn
 \sup_{|y-\beta t|\leq K} \mathbf{P} \big(  \tau_{-y}^- >t+\theta_t \big) \leq    \mathbf{P} \big(  X_{t+\theta_t} + \beta (t+\theta_t)\geq  \beta\theta_t-K\big) , 
 \eeqnn
 which goes to $0$ as $t\to\infty$ by Lemma~\ref{Lemma.3001}.  The desired limit \eqref{eqn.7051} follows.
 \qed

 \section{Proof of Theorem \ref{MainThm05}} 
 \label{Sec.ProofThm05}
 \setcounter{equation}{0} 
 
 In this section, we prove the two conditional scaling limits stated in Theorem~\ref{MainThm05}. 
 As explained in detail in the Introduction, an excursion with a large lifetime or a large height is typically generated by a sufficiently large jump occurring at an early stage. 
 Moreover, the independent-increment property of the underlying L\'evy process suggests decomposing the excursion at this large jump and analyzing the corresponding pre- and post-jump contributions separately, i.e., for any fixed $\lambda \in (0,\beta)$,  
 \beqlb\label{eqn.81}
 \epsilon_s  =  \big( \epsilon_s- \Delta\epsilon_{\mathcal{J}^{\lambda t}} \cdot  \mathbf 1_{\{s\geq \mathcal{J}^{\lambda t}\}} \big) + \Delta\epsilon_{\mathcal{J}^{\lambda t}} \cdot  \mathbf 1_{\{s\geq \mathcal{J}^{\lambda t}\}},\quad s\in[0,\zeta] .
 \eeqlb
 The following lemmas establish the asymptotic behavior of the first occurrence time of the large jump, as well as that of the pre-jump and post-jump path segments, respectively. 
 
 \begin{lemma} \label{Lemma.901}
 	For any constants $\delta, R>0$ and $\lambda\in (0,\beta)$, we have  as $t\to\infty$,
 	\beqlb\label{eqn.7000}
 	\Big| \underline{n}\big(\mathcal{J}^{\lambda t}\leq R \,\big|\,\zeta\in t+\varDelta_\delta   \big) - \mathbf{P}\big(\mathcal T\leq R \big) \Big|
 	+\Big| \underline{n}\big(\mathcal{J}^{\lambda t}\leq R \,\big|\, \overline\epsilon\in \beta (t+\varDelta_\delta)   \big) - \mathbf{P}\big(\mathcal T\leq R \big) \Big| \to0.
 	\eeqlb
 	
 \end{lemma}
 
 \begin{lemma} \label{Lemma.902}
 	For any constants $\lambda \in (0,\beta)$ and   $\delta>0$, for every $\varepsilon>0$, we have as $t\to\infty$,  
 	\beqlb\label{eqn.9020}
 	\underline{n}\bigg(\Big|\frac{\Delta\epsilon_{ \mathcal{J}^{\lambda t}}}{t} -\beta\Big|\geq \varepsilon \,\bigg|\,\zeta\in t+\varDelta_\delta\bigg) 
 	+ \underline{n}\bigg(\Big|\frac{\Delta\epsilon_{ \mathcal{J}^{\lambda t}}}{t} -\beta\Big|\geq \varepsilon \,\bigg|\,  \overline\epsilon\in \beta (t+\varDelta_\delta)\bigg)  \to 0. 
 	\eeqlb
 \end{lemma}
 
 \begin{lemma} \label{Lemma.903}
 	For any constants $\lambda \in (0,\beta)$ and   $\delta,T>0$, for every $\varepsilon>0$, we have as $t\to\infty$,  
 	\beqlb\label{eqn.903}
 	\underline{n}\bigg(\sup_{s\leq T }  \frac{\epsilon_{(s\wedge  \mathcal{J}^{\lambda t})-}}{t}  \geq \varepsilon \,\bigg|\,\zeta\in t+\varDelta_\delta\bigg) 
 	+ \underline{n}\bigg(\sup_{s\leq T\wedge \zeta}\frac{\epsilon_{(s\wedge  \mathcal{J}^{\lambda t})-}}{t}  \geq  \varepsilon \,\bigg|\, \overline\epsilon\in \beta (t+\varDelta_\delta) \bigg)   \to 0. 
 	\eeqlb 
 \end{lemma}
 
 \begin{lemma} \label{Lemma.904}
 	For any constants $\lambda \in (0,\beta)$ and   $\delta,T>0$, for every $\varepsilon>0$, we have as $t\to\infty$,  
 	\beqlb\label{eqn.904}
 	\underline{n}\bigg(\sup_{s\leq T } \Big|\frac{\epsilon_{s\vee  \mathcal{J}^{\lambda t}}-\epsilon_{  \mathcal{J}^{\lambda t}}}{t} \Big|\geq \varepsilon \,\bigg|\,\zeta\in t+\varDelta_\delta\bigg) 
 	+ \underline{n}\bigg(\sup_{s\leq T\wedge \zeta}  \Big|\frac{\epsilon_{s\vee  \mathcal{J}^{\lambda t}}-\epsilon_{  \mathcal{J}^{\lambda t}}}{t} \Big|\geq  \varepsilon \,\bigg|\, \overline\epsilon\in \beta (t+\varDelta_\delta) \bigg)   \to 0. 
 	\eeqlb  
 \end{lemma}

 \textbf{\textit{Proof of Theorem~\ref{MainThm05}(1).}} Here we just give the detailed proof for the uniform weak convergence under $\underline{n}\big(\cdot\,|\, \zeta\in t+\varDelta_\delta\big)$.  
 The case of $\underline{n}\big(\cdot\,|\, \overline\epsilon\in  \beta(t+\varDelta_\delta)\big)$ can be proved similarly.   
 
 \textbf{Step 1.} We first prove the weak convergence for any fixed $\delta>0$.  
 By Lemma~\ref{Lemma.701}, we have $\underline n\big(  \mathcal{J}^{\lambda t} \geq \zeta \, \big|\, \zeta\in t+\varDelta_\delta \big)  \to 0$ as $t\to\infty$. 
 For any fixed $\lambda \in (0,\beta)$, recall the path-decomposition \eqref{eqn.81}. 
 By using Slutsky's theorem along with Lemma~\ref{Lemma.901} and \ref{Lemma.902}, we have $\big(\mathcal{J}^{\lambda t}, \Delta\epsilon_{\mathcal{J}^{\lambda t}}/t\big)$ under $\underline{n}(\cdot\,|\, \zeta \in t+\varDelta_\delta)$ converges to $(\mathcal{T},\beta)$ as $t\to \infty$, which induces that $\{t^{-1} \Delta\epsilon_{\mathcal{J}^{\lambda t}} \cdot  \mathbf 1_{\{s\geq \mathcal{J}^{\lambda t}\}}:s\geq 0 \}$  converges weakly to $\{ \beta \cdot \mathbf{1}_{s\geq \mathcal{T}}:s\geq 0 \}$ in $D(\mathbb{R}_+;\mathbb{R})$; see Example~1.19 in \cite[p.329]{JacodShiryaev2003}.  
 By Theorem 16.4 in \cite[p.~170]{Billingsley1999},  it remains to prove that for any $T,\varepsilon>0$, as $t\to\infty$, 
 \beqlb\label{eqn.906}
 \underline n\Big( \sup_{0\leq s\leq T}  \big| \epsilon_s- \Delta\epsilon_{\mathcal{J}^{\lambda t}} \cdot  \mathbf 1_{\{s\geq \mathcal{J}^{\lambda t}\}} \big| \geq \varepsilon t,\, \mathcal{J}^{\lambda t} <\zeta 
 \, \Big|\, \zeta\in t+\varDelta_\delta \Big) \to0. 
 \eeqlb 
 On the event $\{\mathcal{J}^{\lambda t}<\zeta\}$, we have for any $0\leq s\leq T$, 
 \beqnn
 \epsilon_s- \Delta\epsilon_{\mathcal{J}^{\lambda t}} \cdot  \mathbf 1_{\{s\geq \mathcal{J}^{\lambda t}\}}  = \big( \epsilon_{s\wedge \mathcal{J}^{\lambda t}}- \Delta\epsilon_{\mathcal{J}^{\lambda t}} \cdot  \mathbf 1_{\{s\geq \mathcal{J}^{\lambda t}\}} \big) + \big( \epsilon_{s\vee \mathcal{J}^{\lambda t}}-\epsilon_{\mathcal{J}^{\lambda t}}  \big),
 \eeqnn
 which follows that 
 \beqnn
 \sup_{0\leq s\leq T} \big| \epsilon_s- \Delta\epsilon_{\mathcal{J}^{\lambda t}} \cdot  \mathbf 1_{\{s\geq \mathcal{J}^{\lambda t}\}} \big| 
 \leq \sup_{0\leq s\leq T} \epsilon_{(s\wedge \mathcal{J}^{\lambda t})-} +  \sup_{0\leq s\leq T}\big| \epsilon_{s\vee \mathcal{J}^{\lambda t}}-\epsilon_{\mathcal{J}^{\lambda t}}  \big|
 \eeqnn
 This induces that the left-hand side of \eqref{eqn.906} can be bounded by 
 \beqnn 
 \underline n\Big(  \sup_{0\leq s\leq T} \epsilon_{(s\wedge \mathcal{J}^{\lambda t})-} \geq \frac{\varepsilon t}{2}
 \, \Big|\, \zeta\in t+\varDelta_\delta \Big)   +  \underline n\Big(  \sup_{0\leq s\leq T}\big| \epsilon_{s\vee \mathcal{J}^{\lambda t}}-\epsilon_{\mathcal{J}^{\lambda t}}  \big| \geq   \frac{\varepsilon t}{2}
 \, \Big|\, \zeta\in t+\varDelta_\delta \Big)  ,
 \eeqnn 
 which  tends to $0$ as $t\to\infty$ by Lemma~\ref{Lemma.903} and \ref{Lemma.904}.

 \textbf{Step 2.}  We extend the convergence uniformly in
 $\delta\in[\theta_0,\theta_t]$. It suffices to prove that  for any $T>0$ and any bounded, non-negative, uniformly continuous functional $F$ on $D\big([0,T];\mathbb{R})$, as $t\to\infty$, 
 \beqlb\label{eqn.9009}
 \sup_{ \delta\in[\theta_0,\theta_t] }   \Big| \underline{n}\big( F(\epsilon_s/t:s\in[0,T])\,\big|\, \zeta\in t+\varDelta_\delta \big)- \mathbf{E}\big[ F( \beta\cdot \mathbf{1}_{\{s\geq \mathcal{T} \}}:s\in[0,T])\big] \Big| \to 0.
 \eeqlb 
 For any large integer $n\geq 1$, we have 
 \beqlb\label{eqn.9010}
 \lefteqn{\underline{n}\big( F(\epsilon_s/t:s\in[0,T])\,\big|\, \zeta\in t+\varDelta_\delta \big)}\cr\cr
 \ar \leq \ar \frac{1}{\underline{n}\big(   \zeta\in t+\varDelta_\delta \big)}\cdot \sum_{k=0}^{[n\delta]}\underline{n}\Big( F(\epsilon_s/t:s\in[0,T]) ,\, \zeta\in t+\frac{k}{n}+\varDelta_{1/n} \Big),
 \eeqlb 
 which can further be decomposed into the following two terms:
 \beqnn
 I_n(t,\delta) \ar:=\ar   \frac{1}{\underline{n}\big(   \zeta\in t+\varDelta_\delta \big)} \sum_{k=0}^{[n\delta]} \underline{n}\Big( F\Big(\frac{\epsilon_s}{t+k/n}:s\in[0,T]\Big) ,\, \zeta\in t+\frac{k}{n}+\varDelta_{1/n} \Big),\cr\cr
 \varepsilon_n(t,\delta) \ar:=\ar \frac{1}{\underline{n}\big(   \zeta\in t+\varDelta_\delta \big)} \sum_{k=0}^{[n\delta]}\underline{n}\Big( F\Big(\frac{\epsilon_s}{t}:s\in[0,T]
 \Big)-  F\Big(\frac{\epsilon_s}{t+k/n}:s\in[0,T]\Big),\, \zeta\in t+\frac{k}{n}+\varDelta_{1/n} \Big). 
 \eeqnn
 Recall that the Skorokhod topology is weaker than the uniform topology. 
 The boundedness and uniform continuity of $F$ yield that for some constant $C>0$ independent of $\epsilon$ and $t,k,n$ such that
 \beqnn
 \bigg|  F\Big(\frac{\epsilon_s}{t}:s\in[0,T]
 \Big)-  F\Big(\frac{\epsilon_s}{t+k/n}:s\in[0,T]\Big)  \bigg| 
 \ar\leq\ar C\cdot \mathbf{1}_{\{ \overline{\epsilon} \geq 2\beta t \}} + C\cdot   \frac{k/n}{t^2}\cdot \overline{\epsilon} \cdot \mathbf{1}_{\{ \overline{\epsilon} \leq 2\beta t \}}  \cr\cr
 \ar\leq\ar C\cdot \Big(\mathbf{1}_{\{ \overline{\epsilon} \geq 2\beta t \}} +  \frac{k/n}{t} \Big) . 
 \eeqnn  
 Applying this and the law of total probability to $ \big|\varepsilon_n(t,\delta) \big| $ yields
 \beqlb\label{eqn.9012}
 \sup_{ \delta\in[\theta_0,\theta_t]} \big|\varepsilon_n(t,\delta) \big|  
 \ar\leq\ar C\cdot \sup_{ \delta\in[\theta_0,\theta_t]} \frac{\underline{n}\big( \overline{\epsilon} \geq 2\beta t ,  \zeta\in t+\varDelta_{([n\delta]+1)/n} \big)}{\underline{n}\big(   \zeta\in t+\varDelta_\delta \big)}\cr
 \ar\ar +\frac{C\cdot \theta_t}{t}\cdot \sup_{ \delta\in[\theta_0,\theta_t]}  \frac{\underline{n}\big(  \zeta\in t+\varDelta_{([n\delta]+1)/n} \big)}{\underline{n}\big(   \zeta\in t+\varDelta_\delta \big)} ,
 \eeqlb
 which goes to $0$ as $t\to\infty$ by using Theorem~\ref{MainThm04} and~\ref{MainThm01} to the two terms on the right-hand side respectively. 
 For $I_n(t,\delta)$,  our result established in the first step states that there exists a non-negative function $\rho_n$ on $\mathbb R_+$
 going to $0$ at infinity such that, for all $t>0$ and $k\geq0$,
 \beqnn
 \underline{n}\Big( F\Big(\frac{\epsilon_s}{t+k/n}:s\in[0,T])\,\Big|\, \zeta\in t+\frac{k}{n}+\varDelta_{1/n} \Big)
 \leq  \mathbf{E}\big[ F( \beta\cdot \mathbf{1}_{\{s\geq \mathcal{T} \}}:s\in[0,T])\big]+\rho_n(t);
 \eeqnn
 equivalently, 
 \beqnn
 \lefteqn{\underline{n}\Big( F\Big(\frac{\epsilon_s}{t+k/n}:s\in[0,T]) ,\, \zeta\in t+\frac{k}{n}+\varDelta_{1/n} \Big)}\cr\cr
 \ar\leq\ar  \Big[\mathbf{E}\big[ F( \beta\cdot \mathbf{1}_{\{s\geq \mathcal{T} \}}:s\in[0,T])\big]+\rho_n(t)\Big] \cdot \underline{n}\Big(  \zeta\in t+\frac{k}{n}+\varDelta_{1/n} \Big) . 
 \eeqnn
 Substituting this back into the summation in \eqref{eqn.9010} yields 
 \beqnn 
 I_n(t,\delta)  \ar\leq \ar  \sum_{k=0}^{[n\delta]} \Big[\mathbf{E}\big[ F( \beta\cdot \mathbf{1}_{\{s\geq \mathcal{T} \}}:s\in[0,T])\big]+\rho_n(t)\Big] \cdot \frac{\underline{n}\big(  \zeta\in t+k/n+\varDelta_{1/n} \big)}{\underline{n}\big(   \zeta\in t+\varDelta_\delta \big)}\cr
 \ar=\ar  \Big[\mathbf{E}\big[ F( \beta\cdot \mathbf{1}_{\{s\geq \mathcal{T} \}}:s\in[0,T])\big]+\rho_n(t)\Big]  \cdot \frac{\underline{n}\big(  \zeta\in t+ \varDelta_{([n\delta]+1)/n} \big)}{\underline{n}\big(   \zeta\in t+\varDelta_\delta \big)}. 
 \eeqnn
 This induces that 
 \beqnn 
 \limsup_{t\to\infty}    \sup_{ \delta\in[\theta_0,\theta_t]} I_n(t,\delta)  \ar\leq \ar   \mathbf{E}\big[ F( \beta\cdot \mathbf{1}_{\{s\geq \mathcal{T} \}}:s\in[0,T])\big]  \cdot \limsup_{t\to\infty}    \sup_{ \delta\in[\theta_0,\theta_t]}  \frac{\underline{n}\big(  \zeta\in t+ \varDelta_{([n\delta]+1)/n} \big)}{\underline{n}\big(   \zeta\in t+\varDelta_\delta \big)}\cr\cr
 \ar\leq\ar \mathbf{E}\big[ F( \beta\cdot \mathbf{1}_{\{s\geq \mathcal{T} \}}:s\in[0,T])\big]  \cdot \Big( 1+\frac{1/n}{\theta_0} \Big).
 \eeqnn
 Combining this together with \eqref{eqn.9012} and then letting $n\to\infty$, we obtain
 \beqnn
 \limsup_{t\to\infty}    \sup_{ \delta\in[\theta_0,\theta_t]} \underline{n}\big( F(\epsilon_s/t:s\in[0,T])\,\big|\, \zeta\in t+\varDelta_\delta \big) 
 \leq \mathbf{E}\big[ F( \beta\cdot \mathbf{1}_{\{s\geq \mathcal{T} \}}:s\in[0,T])\big] . 
 \eeqnn
 On the other hand, for the lower bound, an analogous argument gives
 \beqnn
 \liminf_{t\to\infty}    \inf_{ \delta\in[\theta_0,\theta_t] } \underline{n}\big( F(\epsilon_s/t:s\in[0,T]) \,\big|\, \zeta\in t+\varDelta_\delta \big)  
 \geq \mathbf{E}\big[ F( \beta\cdot \mathbf{1}_{\{s\geq \mathcal{T} \}}:s\in[0,T])\big] . 
 \eeqnn
 Combining them together yields \eqref{eqn.9009} immediately. 
 \qed 
 
 \medskip
 
 \textbf{\textit{Proof of Theorem~\ref{MainThm05}(2).}} Similarly as in the proof of claim (1), here we still just prove that the weak convergence in claim (2) under $\underline{n}\big(\cdot\,|\, \zeta\in t+\varDelta_\delta\big)$ holds for any fixed $\delta$. It can be extended to the uniform convergence in
 $\delta\in[\theta_0,\theta_t]$ as before.
 The uniform weak convergence under $\underline{n}\big(\cdot\,|\, \overline\epsilon\in\beta(t+\varDelta_\delta)\big)$ can be proved in the same way.
 
 We start to prove the weak convergence under $\underline{n}\big(\cdot\,|\, \zeta\in t+\varDelta_\delta\big)$.
 By Lemma~\ref{Lemma.701}, it suffices to prove that for any $\varepsilon\in(0,1)$, $0<s_1<s_2<\infty$ and $0<\eta\leq \frac{\varepsilon(s_1\wedge 1)}{4\beta+1}<1$, as $t\to\infty$,
 \beqlb\label{eqn.9013}
 \underline{n}\bigg(
 \sup_{s_1\leq s\leq s_2}
 \Big|\frac{\epsilon_{ts}}{t}-\beta(1-s)^+\Big|
 \geq\varepsilon,\,
 \mathcal{J}^{\lambda t}\leq\eta t
 \,\bigg|\,\zeta\in t+\varDelta_\delta
 \bigg)\to0.
 \eeqlb
 On the event $\{\zeta\in t+\varDelta_\delta\}$, if
 $s\in[s_1,s_2]$ satisfies $ts\geq\zeta$, then
 $s\geq\zeta/t>1$ and hence
 $\big|\epsilon_{ts}/t-\beta(1-s)^+\big|=0$.
 It remains to consider $s\in[s_1,s_2]$ with $ts<\zeta$.
 On the further event $\{\mathcal J^{\lambda t}\leq\eta t\}$, we observe that for any $s\in[s_1,s_2]$,
 \beqnn
 \frac{\epsilon_{ts}}{t}
 \ar=\ar
 \Big(\frac{\epsilon_{\mathcal{J}^{\lambda t}}}{t}-\beta\Big)
 +\bigg[
 \frac{\epsilon_{ts}-\epsilon_{\mathcal{J}^{\lambda t}}}{t}
 +\beta\Big(s-\frac{\mathcal{J}^{\lambda t}}{t}\Big)
 \bigg]
 +\beta(1-s)+\beta\cdot\frac{\mathcal{J}^{\lambda t}}{t},
 \eeqnn
 which along with  the triangle inequality yields
 \beqnn
 \Big|\frac{\epsilon_{ts}}{t}-\beta(1-s)^+\Big|
 \ar\leq\ar
 \Big|\frac{\epsilon_{\mathcal{J}^{\lambda t}}}{t}-\beta\Big|
 +\bigg|
 \frac{\epsilon_{ts}-\epsilon_{\mathcal{J}^{\lambda t}}}{t}
 +\beta\cdot\frac{ts-\mathcal{J}^{\lambda t}}{t}
 \bigg|
 +\beta\big|(1-s)-(1-s)^+\big|
 +\beta\cdot\frac{\mathcal{J}^{\lambda t}}{t}.
 \eeqnn
 Consequently, the left-hand side of \eqref{eqn.9013} can be bounded by the sum of the following four terms:
 \beqnn
 A_1(t)\ar:=\ar
 \underline{n}\bigg(
 \Big|\frac{\epsilon_{\mathcal{J}^{\lambda t}}}{t}-\beta\Big|
 \geq\frac{\varepsilon}{4},\,
 \mathcal{J}^{\lambda t}\leq\eta t
 \,\bigg|\,\zeta\in t+\varDelta_\delta
 \bigg),\cr
 A_2(t)\ar:=\ar
 \underline{n}\bigg(
 \sup_{s_1\leq s\leq s_2}
 \mathbf{1}_{\{ts<\zeta\}}\cdot 
 \Big|
 \frac{\epsilon_{ts}-\epsilon_{\mathcal{J}^{\lambda t}}}{t}
 +\beta\Big(s-\frac{\mathcal{J}^{\lambda t}}{t}\Big)
 \Big|
 \geq\frac{\varepsilon}{4},\,
 \mathcal{J}^{\lambda t}\leq\eta t
 \,\bigg|\,\zeta\in t+\varDelta_\delta
 \bigg),\cr
 A_3(t)\ar:=\ar
 \underline{n}\bigg(
 \sup_{s_1\leq s\leq s_2}
 \mathbf{1}_{\{ts<\zeta\}}\cdot 
 \beta\cdot \big|(1-s)-(1-s)^+\big|
 \geq\frac{\varepsilon}{4},\,
 \mathcal{J}^{\lambda t}\leq\eta t
 \,\bigg|\,\zeta\in t+\varDelta_\delta
 \bigg),\cr
 A_4(t)\ar:=\ar
 \underline{n}\bigg(
 \beta\cdot\frac{\mathcal{J}^{\lambda t}}{t}
 \geq\frac{\varepsilon}{4},\,
 \mathcal{J}^{\lambda t}\leq\eta t
 \,\bigg|\,\zeta\in t+\varDelta_\delta
 \bigg).
 \eeqnn
 It is obvious that $A_4(t)=0$ since
 $\beta\eta\leq\frac{\beta\varepsilon(s_1\wedge 1)}{4\beta+1}<\frac{\varepsilon}{4}$.
 For $A_1(t)$, by the triangle inequality we have
 \beqnn
 A_1(t)\ar\leq\ar
 \underline{n}\bigg(
 \Big|\frac{\Delta\epsilon_{\mathcal{J}^{\lambda t}}}{t}-\beta\Big|
 \geq\frac{\varepsilon}{8},\,
 \mathcal{J}^{\lambda t}\leq\eta t
 \,\bigg|\,\zeta\in t+\varDelta_\delta
 \bigg)
 +\underline{n}\bigg(
 \frac{\epsilon_{\mathcal{J}^{\lambda t}-}}{t}
 \geq\frac{\varepsilon}{8},\,
 \mathcal{J}^{\lambda t}\leq\eta t
 \,\bigg|\,\zeta\in t+\varDelta_\delta
 \bigg).
 \eeqnn
 The first term goes to $0$ as $t\to\infty$ by Lemma~\ref{Lemma.902}.
 For the second term, we choose $R>0$ and obtain that for all $t>R/\eta$,
 \beqnn
 \lefteqn{
 	\underline n\bigg(
 	\frac{\epsilon_{\mathcal J^{\lambda t}-}}{t}
 	\geq\frac{\varepsilon}{8},\,
 	\mathcal J^{\lambda t}\leq\eta t
 	\,\bigg|\,\zeta\in t+\varDelta_\delta
 	\bigg)
 }\cr\cr
 \ar\leq\ar
 \underline n\bigg(
 \frac{\epsilon_{\mathcal J^{\lambda t}-}}{t}
 \geq\frac{\varepsilon}{8},\,
 \mathcal J^{\lambda t}\leq R
 \,\bigg|\,\zeta\in t+\varDelta_\delta
 \bigg)
 +
 \underline n\bigg(
 R<\mathcal J^{\lambda t}\leq\eta t
 \,\bigg|\,\zeta\in t+\varDelta_\delta
 \bigg)\cr\cr
 \ar\leq\ar
 \underline n\bigg(
 \sup_{s\leq R}
 \frac{\epsilon_{(s\wedge\mathcal J^{\lambda t})-}}{t}
 \geq\frac{\varepsilon}{8}
 \,\bigg|\,\zeta\in t+\varDelta_\delta
 \bigg)
 +
 \underline n\bigg(
 \mathcal J^{\lambda t}>R
 \,\bigg|\,\zeta\in t+\varDelta_\delta
 \bigg),
 \eeqnn
 which tends to $0$ as $t\to\infty$ and then $R\to\infty$ by Lemma~\ref{Lemma.901} and \ref{Lemma.903}.
 Consequently, we have $A_1(t)\to0$ as $t\to\infty$. For $A_3(t)$, note that
 $\beta\big|(1-s)-(1-s)^+\big|
 =\beta(s-1)^+
 \leq\beta(\zeta/t-1)$
 for all $s\in[s_1,s_2]$ with $ts<\zeta$, we have for sufficiently large $t$,
 \beqnn
 A_3(t)\ar\leq\ar
 \underline{n}\bigg(
 \beta\Big(\frac{\zeta}{t}-1\Big)
 \geq\frac{\varepsilon}{4},\,
 \mathcal{J}^{\lambda t}\leq\eta t
 \,\bigg|\,\zeta\in t+\varDelta_\delta
 \bigg)=0.
 \eeqnn
 It remains to prove $A_2(t)\to0$ as $t\to\infty$. On the event
 $\{\mathcal J^{\lambda t}\leq\eta t\}$, we have
 $\mathcal J^{\lambda t}\leq ts$ for all $s\in[s_1,s_2]$, and hence
 \beqlb\label{eqn.9.10}
 A_2(t)
 \ar\leq\ar
 \underline n\bigg(
 \sup_{0\leq u\leq s_2t}
 \mathbf 1_{\{\mathcal J^{\lambda t}+u<\zeta\}}\cdot 
 \bigg|
 \frac{\epsilon_{\mathcal J^{\lambda t}+u}
 	-\epsilon_{\mathcal J^{\lambda t}}}{t}
 +\beta\cdot\frac{u}{t}
 \bigg|
 \geq\frac{\varepsilon}{4},\,
 \mathcal J^{\lambda t}\leq\eta t
 \,\bigg|\,\zeta\in t+\varDelta_\delta\bigg). 
 \eeqlb
 Obviously, the last expression can be bounded by
 \beqnn
 \frac{1}{\underline n(\zeta\in t+\varDelta_\delta)}
 \cdot\underline n\bigg(
 \sum_{0<r\leq\eta t\wedge\zeta}
 \mathbf 1_{\{\Delta\epsilon_r>\lambda t\}}
 \cdot\mathbf 1_{\{
 	\sup_{0\leq u\leq s_2t}
 	\mathbf 1_{\{r+u<\zeta\}}
 	|\epsilon_{r+u}-\epsilon_r+\beta u|
 	\geq\varepsilon t/4,\,
 	\zeta\in t+\varDelta_\delta
 	\}}
 \bigg).
 \eeqnn
 By using the Markov property \eqref{eqn.n-Markov} and then the compensation formula \eqref{eqn.Compensation}, the second factor in the preceding expression  can further be bounded by
 \beqnn
 \lefteqn{
 	\underline n\Bigg(
 	\sum_{0<r\leq\eta t\wedge\zeta}
 	\mathbf 1_{\{\Delta\epsilon_r>\lambda t\}}
 	\cdot\mathbf P_{\epsilon_r}\bigg(
 	\sup_{0\leq u\leq s_2t}
 	|X_u-\epsilon_r+\beta u|
 	\geq\frac{\varepsilon t}{4},\,
 	\tau_0^-\in t-r+\varDelta_\delta
 	\bigg)
 	\Bigg)
 }\cr\cr
 \ar=\ar
 \int_0^{\eta t}dr\int_0^\infty\underline n_r(dz)
 \int_{\lambda t}^{\infty}\nu(dw)\,
 \mathbf P_{z+w}\bigg(
 \sup_{0\leq u\leq s_2t}
 |X_u-(z+w)+\beta u|
 \geq\frac{\varepsilon t}{4},\,
 \tau_0^-\in t-r+\varDelta_\delta
 \bigg)\cr\cr
 \ar=\ar
 \int_0^{\eta t}dr\int_0^\infty\underline n_r(dz)
 \int_{\lambda t}^{\infty}\nu(dw)\,
 \mathbf P\bigg(
 \sup_{0\leq u\leq s_2t}|X_u+\beta u|
 \geq\frac{\varepsilon t}{4},\,
 \tau_{-(z+w)}^-\in t-r+\varDelta_\delta
 \bigg)\cr\cr
 \ar\leq \ar \int_0^{\eta t}dr \int_0^\infty  \underline n_r(dz)\,  \mathbf E\bigg[  \nu\Big(\big(-\underline X_{t-r}-z,\, -\underline X_{t-r+\delta}-z \big]\cap(\lambda t,\infty) \Big);\, \sup_{0\leq u\leq s_2t} |X_u+\beta u| \geq\frac{\varepsilon t}{4} \bigg] . 
 \eeqnn
 By Corollary~\ref{Corollary.202} and the monotonicity of $\overline{\nu}$, there exists a constant $C>0$ such that for all sufficiently large $t$, all $z\geq0$ and $0\leq r\leq\eta t$,
 \beqnn
 \nu\Big(\big(-\underline X_{t-r}-z,\, -\underline X_{t-r+\delta}-z \big]\cap(\lambda t,\infty) \Big)
 \leq
 C\cdot\frac{\overline{\nu}(\lambda t)}{\lambda t}
 \cdot\big(
 \underline X_{t-r}-\underline X_{t-r+\delta}+2
 \big).
 \eeqnn
 Combining these two estimates together and then taking them back into \eqref{eqn.9.10}, we obtain that $A_2(t)$ is bounded by
 \beqnn
 A_2(t)\leq
 \frac{C\cdot\overline\nu(\lambda t)/(\lambda t)}
 {\underline n(\zeta\in t+\varDelta_\delta)}
 \cdot
 \int_0^{\eta t}
 \underline n(\zeta>r)\cdot
 \sup_{0\leq r\leq\eta t}
 \mathbf E\bigg[
 \underline X_{t-r}
 -\underline X_{t-r+\delta}+2;\,
 \sup_{0\leq u\leq s_2t}
 |X_u+\beta u|
 \geq\frac{\varepsilon t}{4}
 \bigg]\,dr.
 \eeqnn
 By Lemma~\ref{Lemma.310} and Proposition~4.2 in
 \cite{CuiWuXu2025}, the last supremum is uniformly bounded in $t\geq0$ and also converges to $0$ as
 $t\to\infty$.
 By using Theorem~\ref{MainThm01}, \eqref{eqn.30} and then the dominated convergence theorem along with \eqref{eqn.501}, we obtain that $A_2(t)\to0$ as
 $t\to\infty$.
 Combining all preceding estimates together yields the limit \eqref{eqn.9013}. This completes the whole proof.
 \qed

 \subsection{Proofs of Lemma~\ref{Lemma.901},~\ref{Lemma.902},~\ref{Lemma.903} and~\ref{Lemma.904}}

 \textbf{\textit{Proof of Lemma~\ref{Lemma.901}.}} The convergence of the first absolute value in \eqref{eqn.7000} to $0$  follows directly from Lemma~\ref{Lemma.701}. 	
 We now consider the second one. 
 Recall  the set  $\mathcal{A}_\delta(x,K)$ defined in \eqref{eqn.5033}. 
 In view of Lemma~\ref{Lemma.501}-\ref{Lemma.503} and Theorem~\ref{MainThm02}, we see that the two events $\{ \overline\epsilon\in \beta (t+\varDelta_\delta)  \}$ and $\{ \#\mathcal{A}_{\beta \delta}(\beta t,K)=1\}$ are asymptotically equivalent, which induces that  
 \beqlb\label{eqn.90011}
 \limsup_{K\to\infty}\limsup_{t\to\infty}\bigg|\underline{n}\big(\mathcal{J}^{\lambda t}\leq R \,\big|\,  \overline\epsilon\in \beta (t+\varDelta_\delta)   \big)  -\frac{\underline{n}\big(\mathcal{J}^{\lambda t}\leq R, \#\mathcal{A}_{\beta \delta}(\beta t,K)=1 \big)}{ \underline{n}\big( \overline\epsilon\in \beta (t+\varDelta_\delta)   \big) } \bigg| =0.
 \eeqlb
 Recall the two sets $\mathcal{A}^{\leq R}_\delta(x,K)$ and $\mathcal{A}^{>R}_\delta(x,K)$ defined in \eqref{eqn.600}. The law of total probability yields 
 \beqnn
 \underline{n}\Big(\mathcal{J}^{\lambda t}\leq R,\, \#\mathcal{A}_{\beta \delta}(\beta t,K)=1 \Big) 
 \ar=\ar \underline{n}\Big(\mathcal{J}^{\lambda t}\leq R,\, \#\mathcal{A}^{\leq R}_{\beta \delta}(\beta t,K)=1 \Big) \cr\cr
 \ar\ar - \underline{n}\Big(\mathcal{J}^{\lambda t}\leq R,\, \#\mathcal{A}^{\leq R}_{\beta \delta}(\beta t,K)=1,\, \#\mathcal{A}^{> R}_{\beta \delta}(\beta t,K)\geq 1\Big)\cr\cr
 \ar\ar + \underline{n}\Big(\mathcal{J}^{\lambda t}\leq R,\, \#\mathcal{A}^{\leq R}_{\beta \delta}(\beta t,K)=0,\, \#\mathcal{A}^{> R}_{\beta \delta}(\beta t,K)=1 \Big).
 \eeqnn
 In view of \eqref{eqn.5018} again, we obtain that $\{\mathcal{J}^{\lambda t}> R\}$ and $\{ \#\mathcal{A}_{\beta \delta}^{\leq R}(\beta t,K)=1 \}$  are two mutually exclusive events for all sufficiently large $t$, and hence 
 $\underline{n}\big(\mathcal{J}^{\lambda t}\leq R,\, \#\mathcal{A}^{\leq R}_{\beta \delta}(\beta t,K)=1 \big)= \underline{n}\big(\#\mathcal{A}^{\leq R}_{\beta \delta}(\beta t,K)=1 \big)$. 
 By the triangle inequality, 
 \beqlb\label{eqn.9002}
 \bigg|  \frac{\underline{n}\big(\mathcal{J}^{\lambda t}\leq R,\, \#\mathcal{A}_{\beta \delta}(\beta t,K)=1 \big)}{ \underline{n}\big( \overline\epsilon\in \beta (t+\varDelta_\delta)   \big) } - \frac{ \underline{n}\big( \#\mathcal{A}^{\leq R}_{\beta \delta}(\beta t,K)=1 \big)}{ \underline{n}\big( \overline\epsilon\in \beta (t+\varDelta_\delta)   \big) }\bigg| 
 \leq  \frac{ \underline{n}\big(\mathcal{J}^{\lambda t}\leq R ,\, \#\mathcal{A}^{> R}_{\beta \delta}(\beta t,K)\geq 1\big) }{ \underline{n}\big( \overline\epsilon\in \beta (t+\varDelta_\delta)   \big) } .
 \eeqlb
 For any $S>R$, using the law of total probability again, the last expression can be bounded by 
 \beqlb\label{eqn.9001}
 \frac{\underline{n}\big(\mathcal{J}^{\lambda t}\leq R ,\, \#\mathcal{A}^{> R}_{\beta \delta}(\beta t,K)-\#\mathcal{A}^{> R+S}_{\beta \delta}(\beta t,K)\geq 1\big)}{ \underline{n}\big( \overline\epsilon\in \beta (t+\varDelta_\delta)   \big) }  + \frac{\underline{n}\big(\mathcal{J}^{\lambda t}\leq R ,\,  \#\mathcal{A}^{> R+S}_{\beta \delta}(\beta t,K)\geq 1\big)}{ \underline{n}\big( \overline\epsilon\in \beta (t+\varDelta_\delta)   \big) }  .
 \eeqlb
 Recall the set $\mathcal{B}^x_{\leq R}$ defined in \eqref{eqn.7001}. 
 For sufficiently large $t$, by \eqref{eqn.5018} we observe that  on the event $\{ \mathcal{J}^{\lambda t}\leq R ,\, \#\mathcal{A}^{> R}_{\beta \delta}(\beta t,K)-\#\mathcal{A}^{> R+S}_{\beta \delta}(\beta t,K)\geq 1 \}$,  at least two jumps larger than $\lambda t$ occur in the time interval $[0,R+S]$. 
 Then by \eqref{eqn.701} the first term in \eqref{eqn.9001} can be bounded by  
 \beqnn
 \frac{\underline{n}\big( \# \mathcal{B}^{\lambda t}_{\leq R+S}\geq 2\big)}{ \underline{n}\big( \overline\epsilon\in \beta (t+\varDelta_\delta)   \big) }  
 \leq 
 \frac{\underline{n}\big( \# \mathcal{B}^{\lambda t}_{\leq R+S} \cdot \big(\# \mathcal{B}^{\lambda t}_{\leq R+S}-1  \big)\big)}{ \underline{n}\big( \overline\epsilon\in \beta (t+\varDelta_\delta)   \big) } 
 \leq 2 (R+S) \cdot \underline{n}\big( (R+S)\wedge \zeta \big)\cdot \frac{\big| \overline{\nu}(\lambda t) \big|^2}{ \underline{n}\big( \overline\epsilon\in \beta (t+\varDelta_\delta)   \big) } ,
 \eeqnn
 which goes to $0$ as $t\to \infty$ by Theorem~\ref{MainThm02} and $\overline{\nu}\in \mathrm{RV}^\infty_{-\alpha}$ with $\alpha>1$.  
 Moreover,  the second term  can be bounded by 
 \beqnn
 \frac{\underline{n}\big(\mathcal{J}^{\lambda t}\leq R ,\,  \#\mathcal{A}^{> R+S}_{\beta \delta}(\beta t,K)\geq 1\big)}{ \underline{n}\big( \overline\epsilon\in \beta (t+\varDelta_\delta)   \big) } 
 \leq \frac{\underline{n}\big(  \#\mathcal{A}^{> R+S}_{\beta \delta}(\beta t,K)\geq 1\big)}{ \underline{n}\big( \overline\epsilon\in \beta (t+\varDelta_\delta)   \big) } ,
 \eeqnn
 which also goes to $0$ as $t\to\infty$, then $K\to\infty $ and finally $S\to\infty$ by  Lemma \ref{Lemma.600} and Theorem \ref{MainThm02}. 
 Substituting these two limits back into \eqref{eqn.9002} and then \eqref{eqn.90011} yields that 
 \beqlb\label{eqn.9003}
 \limsup_{K\to\infty}\limsup_{t\to\infty}\bigg| \underline{n}\big(\mathcal{J}^{\lambda t}\leq R \,\big|\,  \overline\epsilon\in \beta (t+\varDelta_\delta)   \big)  -\frac{\underline{n}\big( \#\mathcal{A}_{\beta \delta}^{\leq R}(\beta t,K)=1 \big)}{ \underline{n}\big( \overline\epsilon\in \beta (t+\varDelta_\delta)   \big) } \bigg| =0.
 \eeqlb
 Similarly as in the proof of Lemma~\ref{Lemma.504}, one can also prove that 
 \beqnn
  \limsup_{K\to\infty} \limsup_{t\to\infty} \bigg|\frac{\underline{n}\big( \#\mathcal{A}_{\beta \delta}^{\leq R}(\beta t,K)=1 \big)}{ \underline{n}\big( \overline\epsilon\in \beta (t+\varDelta_\delta)   \big) } - \frac{\underline{n}(R\wedge \zeta)}{\underline{n}(\zeta)}\bigg|=0.
 \eeqnn
 Combining these two limits together with Theorem~\ref{MainThm02} immediately yields  the convergence of the second absolute value in \eqref{eqn.7000} to $0$. This completes the proof.  
 \qed

 \textbf{\textit{Proof of Lemma~\ref{Lemma.902}}.} For any constants $R,K>0$, the law of total probability gives
 \beqnn
 \underline{n}\bigg(\Big|\frac{\Delta\epsilon_{ \mathcal{J}^{\lambda t}}}{t} -\beta\Big|\geq \varepsilon \,\bigg|\,\zeta\in t+\varDelta_\delta\bigg) 
 \ar\leq \ar  \underline{n}\bigg(\Big|\frac{\Delta\epsilon_{ \mathcal{J}^{\lambda t}}}{t} -\beta\Big|\geq \varepsilon,\, \mathcal{J}^{\lambda t}>R \,\bigg|\,\zeta\in t+\varDelta_\delta\bigg)\cr
 \ar\ar +  \underline{n}\bigg(\Big|\frac{\Delta\epsilon_{ \mathcal{J}^{\lambda t}}}{t} -\beta\Big|\geq \varepsilon  ,\, \mathcal{J}^{\lambda t}\leq R,\, \sup_{0\leq r<\mathcal{J}^{\lambda t}} \epsilon_r>K\,\bigg|\,\zeta\in t+\varDelta_\delta\bigg)\cr
 \ar\ar +  \underline{n}\bigg(\Big|\frac{\Delta\epsilon_{ \mathcal{J}^{\lambda t}}}{t} -\beta\Big|\geq \varepsilon  ,\, \mathcal{J}^{\lambda t}\leq R,\, \sup_{0\leq r<\mathcal{J}^{\lambda t}} \epsilon_r\leq K \,\bigg|\,\zeta\in t+\varDelta_\delta\bigg).
 \eeqnn
 The first term can be bounded by $ \underline{n}\big( \mathcal{J}^{\lambda t}>R \,\big|\,\zeta\in t+\varDelta_\delta\big)$, 
 which tends to $0$ as $t\to\infty$ and then $R\to\infty$ by Lemma~\ref{Lemma.701}. 
 By Lemma~\ref{Lemma.702}, the second term can be bounded by 
 \beqnn
 \underline{n}\Big(  \mathcal{J}^{\lambda t}\leq R,\, \sup_{0\leq r<\mathcal{J}^{\lambda t}} \epsilon_r>K\,\Big|\,\zeta\in t+\varDelta_\delta\Big),
 \eeqnn
 which also goes to $0$ as $t\to\infty$ and then  $K\to\infty$. 
 Finally, by Lemma~\ref{Lemma.704}, the third term on the right-hand side tends to $0$ as $t\to\infty$.
 Combining these estimates yields the convergence of the first term in \eqref{eqn.9020} to $0$.
 
 We now consider the second term in \eqref{eqn.9020}. 
 For any $R>0$, the law of total probability also gives
 \beqnn
 \underline{n}\bigg(\Big|\frac{\Delta\epsilon_{ \mathcal{J}^{\lambda t}}}{t} -\beta\Big|\geq \varepsilon \,\bigg|\,  \overline\epsilon\in \beta (t+\varDelta_\delta)\bigg)  
 \ar=\ar    \underline{n}\bigg(\Big|\frac{\Delta\epsilon_{ \mathcal{J}^{\lambda t}}}{t} -\beta\Big|\geq \varepsilon,\, \mathcal{J}^{\lambda t}>R  \,\bigg|\,  \overline\epsilon\in \beta (t+\varDelta_\delta)\bigg) \cr
 \ar\ar +  \underline{n}\bigg(\Big|\frac{\Delta\epsilon_{ \mathcal{J}^{\lambda t}}}{t} -\beta\Big|\geq \varepsilon,\, \mathcal{J}^{\lambda t}\leq R \,\bigg|\,  \overline\epsilon\in \beta (t+\varDelta_\delta)\bigg) .
 \eeqnn
 The first term on the right-hand side can be bounded by  $\underline{n}\big(  \mathcal{J}^{\lambda t}>R \,\big|\,  \overline\epsilon\in \beta (t+\varDelta_\delta)\big)$, which goes to $0$ as $t\to\infty $ and then $R\to\infty$ by Lemma~\ref{Lemma.901}. 
 For the second term, by \eqref{eqn.9003} we have 
 \beqnn
 \bigg|   \underline{n}\Big(\Big|\frac{\Delta\epsilon_{ \mathcal{J}^{\lambda t}}}{t} -\beta\Big|\geq \varepsilon,\, \mathcal{J}^{\lambda t}\leq R \,\Big|\,  \overline\epsilon\in \beta (t+\varDelta_\delta)\Big)- \frac{ \underline{n}\big(\big|  \Delta\epsilon_{ \mathcal{J}^{\lambda t}}/t -\beta\big|\geq \varepsilon,\,  \#\mathcal{A}_{\beta \delta}^{\leq R}(\beta t,K) =1\big)}{\underline{n}\big( \overline\epsilon\in \beta (t+\varDelta_\delta) \big)}  \bigg| \to 0,
 \eeqnn
 as $t\to\infty$ and then $K\to \infty$. 
 It remains to prove that 
 \beqlb\label{eqn.9005}
 \limsup_{K\to\infty}  \limsup_{t\to\infty}  \frac{ \underline{n}\big(\big| \Delta\epsilon_{ \mathcal{J}^{\lambda t}}/t -\beta\big|\geq \varepsilon,\,  \# \mathcal{A}_{\beta \delta}^{\leq R}(\beta t,K) =1\big)}{\underline{n}\big( \overline\epsilon\in \beta (t+\varDelta_\delta) \big)}  =0. 
 \eeqlb
 By the law of total probability, the fraction on the left-hand side can be bounded by 
 \beqlb\label{eqn.9004}
 \frac{\underline{n}\big(\big| \Delta\epsilon_{ \mathcal{J}^{\lambda t}}/t -\beta\big|\geq \varepsilon,\,   \mathcal{J}^{\lambda t}\in\mathcal{A}_{\beta \delta}^{\leq R}(\beta t,K) \big) }{ \underline{n}\big( \overline\epsilon\in \beta (t+\varDelta_\delta)   \big) } + \frac{ \underline n\Big( \mathcal J^{\lambda t}\notin \mathcal A_{\beta\delta}^{\leq R}(\beta t,K),\, \#\mathcal A_{\beta\delta}^{\leq R}(\beta t,K)=1 \Big) }{ \underline n\big( \overline\epsilon\in\beta(t+\varDelta_\delta) \big) }. 
 \eeqlb 
 For sufficiently large $t$, on the event $ \mathcal{J}^{\lambda t}\in\mathcal{A}_{\beta \delta}^{\leq R}(\beta t,K)$, by \eqref{eqn.5018} we have  
 $\Delta\epsilon_{ \mathcal{J}^{\lambda t}} -\beta t \in [-2K, \beta \delta]$ and hence the first term in \eqref{eqn.9004} equals  $0$. 
 Moreover, on the event $\big\{  \mathcal J^{\lambda t}\notin \mathcal A_{\beta\delta}^{\leq R}(\beta t,K),\, \#\mathcal A_{\beta\delta}^{\leq R}(\beta t,K)=1 \big\}$, we have that at least two jumps larger than $\lambda t$ occur in the time interval $[0,R]$.
 This induces that 
 \beqnn
 \frac{ \underline n\Big( \mathcal J^{\lambda t}\notin \mathcal A_{\beta\delta}^{\leq R}(\beta t,K),\, \#\mathcal A_{\beta\delta}^{\leq R}(\beta t,K)=1 \Big) }{ \underline n\big( \overline\epsilon\in\beta(t+\varDelta_\delta) \big) }
 \leq   \frac{\underline{n}\big( \# \mathcal{B}^{\lambda t}_{\leq R}\geq 2\big)}{ \underline{n}\big( \overline\epsilon\in \beta (t+\varDelta_\delta)   \big) } ,
 \eeqnn
 which goes to $0$ as $t\to\infty$ by \eqref{eqn.701}.  
 These two estimates together yield the limit \eqref{eqn.9005}. 
 This completes the proof. 
 \qed

 \textbf{\textit{Proof of Lemma~\ref{Lemma.903}}.} For any constant $R>0$, the law of total probability yields that the first term in \eqref{eqn.903} can be bounded by 
 \beqnn
 \underline{n}\Big(  \sup_{s\leq T }  \epsilon_{(s\wedge  \mathcal{J}^{\lambda t})-}  \geq \varepsilon t,\, \mathcal{J}^{\lambda t}\leq R \,\Big|\,\zeta\in t+\varDelta_\delta\Big) + \underline{n}\big(  \mathcal{J}^{\lambda t}> R \,\big|\,\zeta\in t+\varDelta_\delta\big) .
 \eeqnn
 The first term on the right-hand side goes to $0$ by Lemma~\ref{Lemma.702}. The second one goes to $0$ as $t\to\infty$ and then $R\to\infty$ by Lemma~\ref{Lemma.701}. Consequently,  the first term in \eqref{eqn.903} converges to $0$ as $t\to\infty$.
 
 We now consider the second term in \eqref{eqn.903}.
 For any $R>0$, the law of total probability also shows that it can be bounded by 
 \beqnn
 \underline{n}\Big(\sup_{s\leq T\wedge \zeta}\epsilon_{(s\wedge  \mathcal{J}^{\lambda t})-}  \geq  \varepsilon  t ,\, \mathcal{J}^{\lambda t}\leq R \,\Big|\, \overline\epsilon\in \beta (t+\varDelta_\delta) \Big)  + \underline{n}\Big(   \mathcal{J}^{\lambda t}>R \,\Big|\, \overline\epsilon\in \beta (t+\varDelta_\delta) \Big) .
 \eeqnn
 The second term goes to $0$ as $t\to\infty$ and then $R\to\infty$ by Lemma~\ref{Lemma.901}. 
 It remains to prove the first term also goes to $0$ as $t\to\infty$.  
 An argument analogous to  \eqref{eqn.9005} yields that it suffices to prove that as $t\to\infty$, 
 \beqlb\label{eqn.9007}
 \frac{\underline{n}\big(\sup_{s\leq T\wedge \zeta}\epsilon_{(s\wedge  \mathcal{J}^{\lambda t})-}  \geq  \varepsilon  t ,\,   \# \mathcal{A}_{\beta \delta}^{\leq R}(\beta t,K) =1,\, \mathcal{J}^{\lambda t} \in \mathcal{A}_{\beta \delta}^{\leq R}(\beta t,K) \big) }{\underline{n}\big( \overline\epsilon\in \beta (t+\varDelta_\delta)  \big) } \to0.
 \eeqlb
 By  the definition of $ \mathcal{A}_{\beta \delta}^{\leq R}(\beta t,K)$ and \eqref{eqn.5018}, for some $\eta\in(0,1)$ and sufficiently large $t$,  the preceding fraction can be  bounded by 
 \beqnn
 \frac{1}{\underline{n}\big( \overline\epsilon\in \beta (t+\varDelta_\delta)  \big)}\cdot  \underline{n}\bigg(  \sum_{s\leq R\wedge \zeta} \mathbf{1}_{\{  \sup_{r\leq T\wedge s} \epsilon_{r-}  \geq  \varepsilon  t,\, \epsilon_{s-} \leq K,\, \Delta \epsilon_s \geq (\beta t)^\eta,\, \sup_{s\leq r\leq \zeta} \epsilon_r \in \beta (t+\varDelta_\delta ),\, \sup_{s\leq r\leq \zeta} \epsilon_r-\epsilon_s\leq K \}}\bigg). 
 \eeqnn
 Similarly as in the proof of Lemma~\ref{Lemma.502} , we can control it by
 \beqlb\label{eqn.9008}
 \frac{\nu\big( \beta t +\varDelta_\delta \big)}{\underline{n}\big( \overline\epsilon\in \beta (t+\varDelta_\delta)  \big)}\int_0^R ds \int_0^K \frac{G_{\beta\delta}(\beta t,K,z)}{\nu\big( \beta t +\varDelta_\delta \big)} \underline{n}\Big(  \sup_{r\leq T\wedge s} \epsilon_{r-}   \geq  \varepsilon  t ,\epsilon_{s-} \in dz,\, \zeta>s \Big),
 \eeqlb
 where $G_{\beta\delta}(x,K,z)$ is defined as in \eqref{eqn.50123} with $\vartheta_x=x^\eta$. 
 By Theorem~\ref{MainThm02} and Proposition~\ref{Prop.508}, there exists some constant $C>0$ such that uniformly for sufficiently large $t$,  
 \beqnn
 \eqref{eqn.9008} \leq C\cdot \int_0^R ds \int_0^K \underline{n}\Big(  \sup_{r\leq T\wedge s} \epsilon_{r-}   \geq  \varepsilon  t ,\epsilon_{s-} \in dz ,\, \zeta>s \Big) \leq C\cdot \underline{n}\big(\overline{\epsilon}\geq   \varepsilon  t\big),
 \eeqnn
 which goes to $0 $ as $t\to\infty$ and then the limit \eqref{eqn.9007} holds.
 This completes the proof.
 \qed

 \textbf{\textit{Proof of Lemma~\ref{Lemma.904}}.}
 On the event $ \{\sup_{s\leq T} |\epsilon_{s\vee  \mathcal{J}^{\lambda t}}-\epsilon_{  \mathcal{J}^{\lambda t}}|\geq \varepsilon t \}$, we have $\mathcal{J}^{\lambda t}\leq T$ and hence 
 \beqnn
 \underline{n}\Big(\sup_{s\leq T } \big| \epsilon_{s\vee  \mathcal{J}^{\lambda t}}-\epsilon_{  \mathcal{J}^{\lambda t}}  \big|\geq \varepsilon t ,\,\zeta\in t+\varDelta_\delta\Big) 
 \ar\leq \ar  \underline{n}\bigg(\sum_{s\leq T } \mathbf{1}_{\{  \Delta\epsilon_s> \lambda t,\, \sup_{s\leq r\leq T} |\epsilon_r-\epsilon_s| \geq \varepsilon t, \, \zeta\in t+\varDelta_\delta \}} \bigg) 
 \eeqnn
 By using the  Markov property and then the compensation formula \eqref{eqn.Compensation}, the right-hand side equals  
 \beqnn
 \int_0^T ds \int_0^\infty \underline{n}_s(dz) \int_{\lambda t}^\infty \mathbf{P}_{z+w} \Big( \sup_{0\leq r\leq T-s} \big|X_r-z-w \big|\geq \varepsilon t , \tau_0^-\in t-s+\varDelta_\delta\Big) \,\nu(dw)   ,
 \eeqnn
 which can be bounded by 
 \beqnn
 \underline{n}(\zeta)\cdot \sup_{s\in[0,T]} \sup_{z\geq 0} \int_{\lambda t}^\infty \mathbf{P}_{z+w} \Big( \sup_{0\leq r\leq T-s} \big|X_r-z-w \big|\geq \varepsilon t , \tau_0^-\in t-s+\varDelta_\delta\Big) \,\nu(dw) .
 \eeqnn
 By the spatial homogeneity of $X$ along with the definition of $\tau_{-z-w}^-$,  and then Fubini's theorem
 \beqnn
 \lefteqn{ \sup_{s\in[0,T]} \sup_{z\geq 0}  \int_{\lambda t}^\infty \mathbf{P}_{z+w} \Big( \sup_{0\leq r\leq T-s} \big|X_r-z-w \big|\geq \varepsilon t , \tau_0^-\in t-s+\varDelta_\delta\Big) \,\nu(dw)  }\cr
 \ar=\ar  \sup_{s\in[0,T]} \sup_{z\geq 0} \int_{\lambda t}^\infty \mathbf{P} \Big( \sup_{0\leq r\leq T-s} \big|X_r \big|\geq \varepsilon t , \underline{X}_{t-s} \geq -z-w >  \underline{X}_{t-s+\delta} \Big) \,\nu(dw)  \cr
 \ar=\ar  \sup_{s\in[0,T]} \sup_{z\geq 0}  \mathbf{E} \Big[ \nu \big( [-\underline{X}_{t-s}-z, -\underline{X}_{t-s+\delta}-z  )\cap (\lambda t,\infty) \big); \sup_{0\leq r\leq T-s} |X_r|\geq \varepsilon t \Big].
 \eeqnn  
 By Corollary~\ref{Corollary.202}  and \eqref{eqn.30}, the last expression can be bounded for all sufficiently large $t$  by, 
 \beqnn
 \lefteqn{C\cdot  \sup_{s\in[0,T]} \mathbf{E} \Big[  \underline{X}_{t-s}-\underline{X}_{t-s+\delta} +1, \sup_{0\leq r\leq T-s} |X_r|\geq \varepsilon t \Big] \cdot \nu\big(\beta (t+\varDelta_\delta)\big) }\cr
 \ar\leq\ar C\cdot   \mathbf{E} \Big[  \underline{X}_{t-T}-\underline{X}_{t+T+\delta} +1, \sup_{0\leq r\leq T} |X_r|\geq \varepsilon t \Big] \cdot \nu\big(\beta (t+\varDelta_\delta)\big)
 \eeqnn
 Combining these estimates together and then using Theorem~\ref{MainThm01}, we obtain  for all sufficiently large $t$, 
 \beqnn
 \underline{n}\bigg(\sup_{s\leq T } \Big|\frac{\epsilon_{s\vee  \mathcal{J}^{\lambda t}}-\epsilon_{  \mathcal{J}^{\lambda t}}}{t} \Big|\geq \varepsilon \,\bigg|\,\zeta\in t+\varDelta_\delta\bigg) 
 \leq  C\cdot   \mathbf{E} \Big[  \underline{X}_{t-T}-\underline{X}_{t+T+\delta} +1, \sup_{0\leq r\leq T} |X_r|\geq \varepsilon t \Big] ,
 \eeqnn 
 which goes to $0$ as $t\to\infty$, since $\sup_{0\leq r\leq T} |X_r|<\infty$ a.s. and the family $\{  \underline{X}_{t-T}-\underline{X}_{t+T+\delta}\}_{t\geq T}$   is uniformly integrable; see Lemma~\ref{Lemma.310}. This proves the convergence of the first term in \eqref{eqn.904} to $0$. 
 
 We now consider the second term in \eqref{eqn.904}. 
 Similarly as in the proof of Lemma~\ref{Lemma.903}, it suffices to prove that as $t\to\infty$, 
 \beqnn
 \frac{\underline{n}\big(\sup_{s\leq T\wedge \zeta} |\epsilon_{s\vee  \mathcal{J}^{\lambda t}}- \epsilon_{ \mathcal{J}^{\lambda t}}|  \geq  \varepsilon  t ,\,   \# \mathcal{A}_{\beta \delta}^{\leq R}(\beta t,K) =1,\, \mathcal{J}^{\lambda t} \in \mathcal{A}_{\beta \delta}^{\leq R}(\beta t,K) \big) }{\underline{n}\big( \overline\epsilon\in \beta (t+\varDelta_\delta)  \big) } \to 0.
 \eeqnn 
 In view of \eqref{eqn.5018}, on the event $\{ \# \mathcal{A}_{\beta \delta}^{\leq R}(\beta t,K) =1,\, \mathcal{J}^{\lambda t} \in \mathcal{A}_{\beta \delta}^{\leq R}(\beta t,K) \}$ we have $\Delta\epsilon_{\mathcal{J}^{\lambda t}} \in [\beta t-2K,\beta t+\beta \delta]$ and hence the preceding fraction can be bounded by
 \beqnn
 \lefteqn{ \frac{1 }{\underline{n}\big( \overline\epsilon\in \beta (t+\varDelta_\delta)  \big) }\cdot \underline{n}\big(\sup_{s\leq T\wedge \zeta} \big|\epsilon_{s\vee  \mathcal{J}^{\lambda t}}- \epsilon_{ \mathcal{J}^{\lambda t}} \big| \geq  \varepsilon  t ,\,   \Delta\epsilon_{\mathcal{J}^{\lambda t}} \in [\beta t-2K,\beta t+\beta \delta] ,\, \mathcal{J}^{\lambda t}\leq R\big) }\cr
 \ar\leq\ar \frac{1}{\underline{n}\big( \overline\epsilon\in \beta (t+\varDelta_\delta)  \big)} \cdot \underline{n}\bigg( \sum_{0<s\leq R\wedge \zeta} \mathbf{1}_{\{\Delta \epsilon_s \in [\beta t-2K,\beta t+\beta \delta], \, \sup_{s\leq r\leq T} |\epsilon_{r}- \epsilon_{s} | \geq  \varepsilon  t   \}} \bigg)\cr
 \ar\leq \ar  \frac{\nu\big( (\beta t-2K,\beta t +\beta\delta] \big)}{\underline{n}\big( \overline\epsilon\in \beta (t+\varDelta_\delta)  \big)} \cdot \int_0^R   \underline{n}(\zeta >s)\, ds  \cdot   \mathbf{P}\Big( \sup_{0\leq r\leq T} \big|X_r \big| \geq  \varepsilon  t   \Big)  ,
 \eeqnn
 which goes to $0$ as $t\to\infty$, by Theorem \ref{MainThm02}, \eqref{eqn.30} and $\sup_{0\leq r\leq T}|X_r|<\infty$ a.s.
 \qed

 \section{Proof of Theorem \ref{MainThm06}} 
 \label{Sec.ProofThm06}
 \setcounter{equation}{0} 
 
 For notational convenience in the proof of Theorem~\ref{MainThm06}, we introduce the following family of finite measures on $\mathbb R_+$:
 \beqlb\label{eqn.8011}
 \mu_{t,\delta}(dz) :=  \int_{t}^{t+\delta} \frac{  \underline{n}_s(dz)}{ \nu\big( \beta t+\varDelta_\delta\big) } \, ds,\quad t,\delta>0 .
 \eeqlb 
 The vague convergence of these measures as $t\to\infty$ is established in the following proposition.
 
 \begin{proposition}\label{Prop.801}
 	For any constants $\delta>0$, $0\leq x<y<\infty$ and any continuous function $f$ on $\mathbb{R}_+$, we have as $t\to\infty$,
 	\beqnn
 	\bigg|\int_x^y f(z)\,\mu_{t,\delta}(dz)
 	-\int_x^y f(z)V(z)\,dz\bigg|\to0.
 	\eeqnn
 \end{proposition}
 \proof
 By \eqref{eqn.8011} and the change of variable $r=s-t$, we have
 \beqnn
 \bigg|\int_x^y f(z)\,\mu_{t,\delta}(dz)
 -\int_x^y f(z)V(z)\,dz\bigg|
 \ar=\ar
 \bigg|\int_0^\delta dr\int_x^y f(z)
 \frac{\underline n_{t+r}(dz)}
 {\nu\big(\beta t+\varDelta_\delta\big)}
 -\int_x^y f(z)V(z)\,dz\bigg|\cr\cr
 \ar\leq\ar
 \frac{1}{\delta}\int_0^\delta
 \bigg|\int_x^y f(z)
 \frac{\delta\cdot\underline n_{t+r}(dz)}
 {\nu\big(\beta t+\varDelta_\delta\big)}
 -\int_x^y f(z)V(z)\,dz\bigg|\,dr\cr\cr
 \ar\leq\ar
 \sup_{r\in[0,\delta]}
 \bigg|\int_x^y f(z)
 \frac{\delta\cdot\underline n_{t+r}(dz)}
 {\nu\big(\beta t+\varDelta_\delta\big)}
 -\int_x^y f(z)V(z)\,dz\bigg|.
 \eeqnn
 By the triangle inequality, the last supremum can be bounded by the sum of the following two terms:
 \beqnn
 \sup_{r\in[0,\delta]}
 \bigg|
 \frac{\alpha\cdot\overline\nu\big(\beta(t+r)\big)\cdot\delta}
 {\beta(t+r)\cdot\nu\big(\beta t+\varDelta_\delta\big)}
 -1\bigg|
 \cdot\int_x^y |f(z)|V(z)\,dz
 \eeqnn
 and
 \beqnn
 \sup_{r\in[0,\delta]}
 \frac{\alpha\cdot\overline\nu\big(\beta(t+r)\big)\cdot\delta}
 {\beta(t+r)\cdot\nu\big(\beta t+\varDelta_\delta\big)}
 \bigg|
 \frac{\beta(t+r)}
 {\alpha\cdot\overline\nu\big(\beta(t+r)\big)}
 \int_x^y f(z)\,\underline n_{t+r}(dz)
 -\int_x^y f(z)V(z)\,dz
 \bigg|.
 \eeqnn
 By \eqref{eqn.30}, the first term goes to $0$ as $t\to\infty$.
 Moreover, by Corollary~1.6 in \cite{CuiWuXu2026b},
 \beqnn
 \lim_{t\to\infty}\sup_{r\geq0}
 \bigg|
 \frac{\beta(t+r)}
 {\alpha\cdot\overline\nu\big(\beta(t+r)\big)}
 \int_x^y f(z)\,\underline n_{t+r}(dz)
 -\int_x^y f(z)V(z)\,dz
 \bigg|=0,
 \eeqnn
 and then the second term also vanishes as $t\to\infty$.
 This completes the proof.
 \qed
 
 \textbf{\textit{Proof of Theorem \ref{MainThm06}.}}
 By the identity \eqref{eqn.EH} and Fubini's theorem, we first have for all $\delta\in [\delta_0,\infty]$, 
 \beqnn
 \underline n\big(f(0,0)\,\big|\,\zeta\in t+\varDelta_\delta\big)
 =
 f(0,0)
 =
 \frac{1}{\mathbf{E}\big[\widehat{H}_1\big]}
 \Big(
 \underline{\tt a}\cdot f(0,0)
 +
 \int_0^\infty dx\int_0^\infty
 f(0,0)V(x)\,\nu(-x-dy)
 \Big).
 \eeqnn
 Consequently, it suffices to prove that the uniform convergence \eqref{eqn.MainThm06} holds for any non-negative, bounded and
 continuous function $f$ on $[0,\infty)^2$ with
 $\|f\|_\infty\leq1$ and $f(0,0)=f(x,y)=0$ if $x<0$ or $y<0$;
 that is, as $t\to\infty$,
 \beqlb\label{eqn.801}
 \bigg|
 \underline n\Big(
 f\big(\epsilon_{\zeta-},-\epsilon_\zeta\big)
 \,\Big|\,\zeta\in t+\varDelta_\delta
 \Big)
 -
 \frac{1}{\mathbf{E}[\widehat{H}_1]}
 \int_0^\infty dx\int_0^\infty
 f(x,y)V(x)\,\nu(-x-dy)
 \bigg|
 \to0.
 \eeqlb
 The proof consists of two steps.
 We first establish that the limit \eqref{eqn.801} holds for any
 fixed $\delta>0$ and then extend this convergence uniformly to all
 $\delta\in[\delta_0,\infty]$ by the method of interval partition.
 
 \textbf{Step 1.}
 We now prove that the limit \eqref{eqn.801} holds for any fixed $\delta>0$.
 By the triangle inequality, the left-hand side of \eqref{eqn.801}
 can be bounded by the sum of the following two terms:
 \beqnn
 \int_0^\infty dx\int_0^\infty
 f(x,y)V(x)\,\nu(-x-dy)
 \cdot
 \bigg|
 \frac{\nu\big(\beta t+\varDelta_\delta\big)}
 {\underline n\big(\zeta\in t+\varDelta_\delta\big)}
 -
 \frac{1}{\mathbf{E}\big[\widehat{H}_1\big]}
 \bigg|
 \eeqnn
 and
 \beqnn
 \frac{\nu\big(\beta t+\varDelta_\delta\big)}
 {\underline n\big(\zeta\in t+\varDelta_\delta\big)}
 \cdot
 \bigg|
 \frac{
 	\underline n\big(
 	f(\epsilon_{\zeta-},-\epsilon_\zeta),
 	\zeta\in t+\varDelta_\delta
 	\big)}
 {\nu\big(\beta t+\varDelta_\delta\big)}
 -
 \int_0^\infty dx\int_0^\infty
 f(x,y)V(x)\,\nu(-x-dy)
 \bigg|.
 \eeqnn
 By the fact that $\|f\|_\infty\leq1$ and \eqref{eqn.EH},
 the first term can be bounded by
 \beqnn
 \mathbf{E}\big[\widehat{H}_1\big]
 \cdot
 \bigg|
 \frac{\nu\big(\beta t+\varDelta_\delta\big)}
 {\underline n\big(\zeta\in t+\varDelta_\delta\big)}
 -
 \frac{1}{\mathbf{E}\big[\widehat{H}_1\big]}
 \bigg|,
 \eeqnn
 which goes to $0$ as $t\to\infty$; see
 Theorem~\ref{MainThm01} and \eqref{eqn.EH0}.
 To obtain the limit \eqref{eqn.801}, it remains to prove that
 as $t\to\infty$,
 \beqlb\label{eqn.802}
 \bigg|
 \frac{
 	\underline n\big(
 	f(\epsilon_{\zeta-},-\epsilon_\zeta),
 	\zeta\in t+\varDelta_\delta
 	\big)}
 {\nu\big(\beta t+\varDelta_\delta\big)}
 -
 \int_0^\infty dx\int_0^\infty
 f(x,y)V(x)\,\nu(-x-dy)
 \bigg|
 \to0.
 \eeqlb
 For $t\geq1$, by using the Markov property at time $1$ we get
 \beqnn
 \underline n\big(
 f(\epsilon_{\zeta-},-\epsilon_\zeta),
 \zeta\in t+\varDelta_\delta
 \big)
 \ar=\ar
 \int_0^\infty
 \mathbf{E}_z\Big[
 f\big(X_{\tau_0^--},-X_{\tau_0^-}\big),
 \tau_0^-\in t-1+\varDelta_\delta
 \Big]\,\underline n_1(dz).
 \eeqnn
 For any $z>0$, it is easy to identify that
 \beqnn
 f\big(X_{\tau_0^--},-X_{\tau_0^-}\big)
 \cdot
 \mathbf{1}_{\{\tau_0^-\in t-1+\varDelta_\delta\}}
 \ar=\ar
 \sum_{s\in t-1+\varDelta_\delta}
 f\big(X_{s-},-X_s\big)
 \cdot
 \mathbf{1}_{\{\underline X_{s-}>0,X_s<0\}}
 \cr\cr
 \ar=\ar
 \sum_{s\in t-1+\varDelta_\delta}
 f\big(X_{s-},-X_{s-}-\Delta X_s\big)
 \cdot
 \mathbf{1}_{\{
 	\underline X_{s-}>0,\Delta X_s<-X_{s-}
 	\}},
 \eeqnn
 which along with the compensation formula
 \eqref{eqn.CompenstationX} yields that
 \beqnn
 \mathbf{E}_z\Big[
 f\big(X_{\tau_0^--},-X_{\tau_0^-}\big),
 \tau_0^-\in t-1+\varDelta_\delta
 \Big]
 \ar=\ar
 \mathbf{E}_z\bigg[
 \sum_{s\in t-1+\varDelta_\delta}
 f\big(X_{s-},-X_{s-}-\Delta X_s\big)
 \cdot
 \mathbf{1}_{\{
 	\underline X_{s-}>0,\Delta X_s<-X_{s-}
 	\}}
 \bigg]
 \cr\cr
 \ar=\ar
 \int_{t-1}^{t-1+\delta}
 \mathbf{E}_z\Big[
 \int_{X_{s-}}^\infty
 f(X_{s-},-X_{s-}+w)\,\nu(-dw),
 \underline X_{s-}>0
 \Big]\,ds.
 \eeqnn
 Recall that $f(0,0)=f(x,y)=0$ if $x<0$ or $y<0$.
 For convenience, we define
 \beqlb\label{eqn.8001}
 g(x):=
 \int_x^\infty f(x,-x+w)\,\nu(-dw),
 \quad x\geq0.
 \eeqlb
 Using Fubini's theorem and then the Markov property,
 \beqnn
 \underline n\big(
 f(\epsilon_{\zeta-},-\epsilon_\zeta),
 \zeta\in t+\varDelta_\delta
 \big)
 \ar=\ar
 \int_0^\infty\underline n_1(dz)
 \int_{t-1}^{t-1+\delta}
 \mathbf{E}_z\big[
 g(X_{s-}),\underline X_{s-}>0
 \big]\,ds
 \cr\cr
 \ar=\ar
 \int_{t-1}^{t-1+\delta}
 \underline n\Big(
 \mathbf{E}_{\epsilon_1}\big[
 g(X_{s-}),\underline X_{s-}>0
 \big],\zeta>1
 \Big)\,ds
 \cr\cr
 \ar=\ar
 \int_{t-1}^{t-1+\delta}
 \underline n\big(
 g(\epsilon_{(s+1)-}),\zeta>s+1
 \big)\,ds
 \cr\cr
 \ar=\ar
 \int_{t-1}^{t-1+\delta}ds
 \int_0^\infty g(x)\,\underline n_{s+1}(dx)
 =
 \int_0^\infty g(x)
 \int_t^{t+\delta}\underline n_s(dx)\,ds
 \eeqnn
 and
 \beqnn
 \int_0^\infty V(x)\cdot g(x)\,dx
 =
 \int_0^\infty dx\int_0^\infty
 f(x,y)V(x)\,\nu(-x-dy).
 \eeqnn
 Recall the measure $\mu_{t,\delta}(dz)$ defined in \eqref{eqn.8011}.
 The left-hand side of \eqref{eqn.802} equals
 \beqnn
 \bigg|
 \int_0^\infty g(x)\,\mu_{t,\delta}(dx)
 -
 \int_0^\infty g(x)V(x)\,dx
 \bigg|.
 \eeqnn
 For any $\varepsilon\in(0,1)$ and integer $M\geq1$,
 it can be bounded by the sum of the following three terms:
 \beqnn
 I_1(t,\varepsilon)
 \ar:=\ar
 \bigg|
 \int_0^\varepsilon g(x)\,\mu_{t,\delta}(dx)
 -
 \int_0^\varepsilon g(x)V(x)\,dx
 \bigg|,
 \cr\cr
 I_2(t,\varepsilon,M)
 \ar:=\ar
 \bigg|
 \int_\varepsilon^M g(x)\,\mu_{t,\delta}(dx)
 -
 \int_\varepsilon^M g(x)V(x)\,dx
 \bigg|,
 \cr\cr
 I_3(t,M)
 \ar:=\ar
 \bigg|
 \int_M^\infty g(x)\,\mu_{t,\delta}(dx)
 -
 \int_M^\infty g(x)V(x)\,dx
 \bigg|.
 \eeqnn
 Firstly, note that $g$ is bounded on $[\varepsilon,M]$,
 Proposition~\ref{Prop.801} yields that
 $I_2(t,\varepsilon,M)\to0$ as $t\to\infty$.
 To obtain the limit \eqref{eqn.802} and then the limit
 \eqref{eqn.MainThm06} for any fixed $\delta>0$, it remains to prove
 \beqlb\label{eqn.804}
 \limsup_{\varepsilon\to0+}
 \limsup_{t\to\infty}I_1(t,\varepsilon)=0
 \quad\mbox{and}\quad
 \limsup_{M\to\infty}
 \limsup_{t\to\infty}I_3(t,M)=0.
 \eeqlb
 
 \textbf{(i)}
 Firstly, by the triangle inequality we have
 \beqlb\label{eqn.80111}
 I_1(t,\varepsilon)
 \leq
 \int_0^\varepsilon |g(x)|\,\mu_{t,\delta}(dx)
 +
 \int_0^\varepsilon |g(x)|V(x)\,dx.
 \eeqlb
 Note that $g(x)\leq\nu\big((-\infty,-x)\big)$,
 by \eqref{eqn.EH} the second term on the right-hand side goes to
 $0$ as $\varepsilon\to0+$.
 For the first term, in view of \eqref{eqn.8001} we decompose $g$
 into the following two parts: for any $\varepsilon'\in(0,1)$,
 \beqnn
 g_1(\varepsilon',x):=
 \int_x^{x+\varepsilon'}f(x,-x+w)\,\nu(-dw)
 \quad\mbox{and}\quad
 g_2(\varepsilon',x):=
 \int_{x+\varepsilon'}^\infty f(x,-x+w)\,\nu(-dw).
 \eeqnn
 This along with the triangle inequality yields that the first term
 on the right-hand side of \eqref{eqn.80111} can be bounded by
 $I_{1,1}(t,\varepsilon,\varepsilon')
 +I_{1,2}(t,\varepsilon,\varepsilon')$ with
 \beqnn
 I_{1,i}(t,\varepsilon,\varepsilon')
 \ar:=\ar
 \int_0^\varepsilon
 g_i(\varepsilon',x)\,\mu_{t,\delta}(dx),
 \quad i=1,2.
 \eeqnn
 For $i=1$, since
 $g_1(\varepsilon',x)
 \leq\sup_{z\in[0,\varepsilon']}|f(x,z)|
 \cdot\nu\big((-\infty,-x)\big)$, we have
 \beqnn
 I_{1,1}(t,\varepsilon,\varepsilon')
 \ar\leq\ar
 \sup_{y\in[0,\varepsilon]}
 \sup_{z\in[0,\varepsilon']}|f(y,z)|
 \cdot
 \int_0^\varepsilon
 \nu\big((-\infty,-x)\big)\,\mu_{t,\delta}(dx).
 \eeqnn
 It follows from  \eqref{eqn.8011} and the compensation formula
 \eqref{eqn.Compensation} applied to terminal jumps that
 \beqnn
 \int_0^\varepsilon
 \nu\big((-\infty,-x)\big)\,\mu_{t,\delta}(dx)
 \ar\leq\ar
 \int_0^\infty
 \nu\big((-\infty,-x)\big)\,\mu_{t,\delta}(dx)
 \cr\cr
 \ar=\ar
 \frac{1}{\nu\big(\beta t+\varDelta_\delta\big)}
 \int_t^{t+\delta}ds\int_0^\infty
 \nu\big((-\infty,-x)\big)\,\underline n_s(dx)
 \cr\cr
 \ar=\ar
 \frac{1}{\nu\big(\beta t+\varDelta_\delta\big)}
 \underline n\bigg(
 \sum_{0<s\leq\zeta}
 \mathbf{1}_{\{
 	s\in t+\varDelta_\delta,\,
 	\epsilon_{s-}>0,\,
 	\Delta\epsilon_s<-\epsilon_{s-}
 	\}}
 \bigg)
 \cr\cr
 \ar\leq\ar
 \frac{\underline n\big(\zeta\in t+\varDelta_\delta\big)}
 {\nu\big(\beta t+\varDelta_\delta\big)}.
 \eeqnn
 Hence, by  Theorem~\ref{MainThm01} and \eqref{eqn.30},
 we have as $t\to\infty$,
 \beqnn
 I_{1,1}(t,\varepsilon,\varepsilon')
 \ar\leq\ar
 \sup_{y\in[0,\varepsilon]}
 \sup_{z\in[0,\varepsilon']}|f(y,z)|
 \cdot
 \frac{\underline n\big(\zeta\in t+\varDelta_\delta\big)}
 {\nu\big(\beta t+\varDelta_\delta\big)}
 \to
 \beta\underline n(\zeta)\cdot
 \sup_{y\in[0,\varepsilon]}
 \sup_{z\in[0,\varepsilon']}|f(y,z)|,
 \eeqnn
 which goes to $0$ as $\varepsilon\to0$ and then
 $\varepsilon'\to0$ because of the continuity of $f$ and $f(0,0)=0$.
 For $i=2$, we have
 $g_2(\varepsilon',x)
 \leq\nu\big((-\infty,-x-\varepsilon']\big)
 \leq\nu\big((-\infty,-\varepsilon']\big)$
 for all $x\geq0$ and then
 \beqnn
 I_{1,2}(t,\varepsilon,\varepsilon')
 \ar\leq\ar
 \nu\big((-\infty,-\varepsilon']\big)
 \cdot
 \mu_{t,\delta}\big((0,\varepsilon]\big).
 \eeqnn
 By Proposition~\ref{Prop.801} with $x=0$ and $y=\varepsilon$,
 we have
 \beqnn
 \limsup_{t\to\infty}I_{1,2}(t,\varepsilon,\varepsilon')
 \ar\leq\ar
 \nu\big((-\infty,-\varepsilon']\big)
 \cdot
 \int_0^\varepsilon V(z)\,dz,
 \eeqnn
 which goes to $0$ as $\varepsilon\to0$.
 These two estimates together yield the first limit in
 \eqref{eqn.804} by letting first $t\to\infty$,
 then $\varepsilon\to0+$ and finally $\varepsilon'\to0+$.
 
 \textbf{(ii)} 
 Since  $\|f\|_\infty\leq 1$ and $V(\infty)<\infty$, we also have 
 \beqlb\label{eqn.803}
 I_3(t,M)
 \ar\leq\ar    \int_{t}^{t+\delta} ds \int_M^\infty \nu\big((-\infty,-x]\big) \frac{\underline{n}_s(dx) }{\nu\big( \beta t+\varDelta_\delta\big)} + \int_M^\infty \nu\big((-\infty,-x]\big)\,V(x)\,dx  \cr
 \ar\leq\ar   \sum_{k=M}^{\infty} \nu\big((-\infty,-k]\big) \int_{t}^{t+\delta}
 \frac{\underline{n}_s\big( k+\varDelta_1 \big)}{\nu\big( \beta t+\varDelta_\delta\big)}\, ds + V(\infty)\cdot\int_M^\infty \nu\big((-\infty,-x]\big) \,dx .
 \eeqlb
 By Theorem~1.5 in \cite{CuiWuXu2026b} and $V(\infty)<\infty$, there exists $C>0$ such that, for any sufficiently large $t$, uniformly in $s\in[t,t+\delta]$ and $k\geq M$,
 \beqnn
 \underline n_s\big(k+\varDelta_1\big)
 \ar\leq\ar
 C\cdot\frac{\overline\nu(\beta s+k)}{\beta s+k}
 \int_k^{k+1}V(z)\,dz
 \leq
 C\cdot\frac{\overline\nu(\beta s)}{\beta s}.
 \eeqnn
 By using the monotonicity of $\overline{\nu}$ and then  \eqref{eqn.30} with any fixed $\delta>0$, we obtain for all sufficiently large $t$ 
 \beqnn
 \sum_{k=M}^{\infty} \nu\big((-\infty,-k]\big) \int_{t}^{t+\delta}
 \frac{\underline{n}_s\big( k+\varDelta_1 \big)}{\nu\big( \beta t+\varDelta_\delta\big)}\, ds 
 \ar\leq\ar \frac{C}{\nu\big( \beta t+\varDelta_\delta\big)} \sum_{k=M}^{\infty} \nu\big((-\infty,-k]\big) \int_{t}^{t+\delta} \frac{\overline\nu(\beta s)}{\beta s}\, ds \cr
 \ar\leq\ar C\delta\cdot \frac{\overline\nu(\beta t)/(\beta t)}{\nu\big( \beta t+\varDelta_\delta\big)} \cdot \sum_{k=M}^{\infty} \nu\big((-\infty,-k]\big)  
 \leq C   \int_{M-1}^{\infty} \nu\big((-\infty,-z]\big) dz. 
 \eeqnn 
 Plugging this back into the right-hand side of the second inequality in \eqref{eqn.803}, 
 \beqnn
 \limsup_{t\to\infty} I_3(t,M) \ar\leq \ar  C\cdot  	\int_{M-1}^\infty \nu\big((-\infty,-x]\big)dx \to 0,
 \eeqnn 
 as $M\to\infty$, since  $ \int_1^\infty  \nu\big((-\infty,-y)\big)\, dy =    \int_{-\infty}^{-1} (|y|-1) \,\nu(dy) <\infty$. 
 
 \textbf{Step 2.} We now extend \eqref{eqn.MainThm06} to hold uniformly for $\delta\in[\delta_0,\infty]$.
 For any large integer $n>1/\delta_0$ and $\delta\in[\delta_0,\infty]$, since $f\geq0$, we have
 \beqnn
 \sum_{k=0}^{[n\delta]-1}
 \underline n\big(f(\epsilon_{\zeta-},-\epsilon_\zeta),
 \zeta\in t+k/n+\varDelta_{1/n}\big)
 \ar\leq\ar
 \underline n\big(f(\epsilon_{\zeta-},-\epsilon_\zeta),
 \zeta\in t+\varDelta_\delta\big)
 \cr\cr
 \ar\leq\ar
 \sum_{k=0}^{[n\delta]}
 \underline n\big(f(\epsilon_{\zeta-},-\epsilon_\zeta),
 \zeta\in t+k/n+\varDelta_{1/n}\big).
 \eeqnn
 For convenience, we write $A$ for the limit on the right-hand side of \eqref{eqn.MainThm06} with $f(0,0)=0$.
 Applying the limit \eqref{eqn.MainThm06} with $\delta=1/n$, we obtain a positive function $\rho_n$ on $\mathbb R_+$
 that decreases to $0$ at infinity and satisfies, for all sufficiently large $t$ and all $k\geq0$,
 \beqnn
 \big(A-\rho_n(t)\big)\cdot
 \underline n\big(\zeta\in t+k/n+\varDelta_{1/n}\big)
 \ar\leq\ar
 \underline n\big(f(\epsilon_{\zeta-},-\epsilon_\zeta),
 \zeta\in t+k/n+\varDelta_{1/n}\big)
 \cr\cr
 \ar\leq\ar
 \big(A+\rho_n(t)\big)\cdot
 \underline n\big(\zeta\in t+k/n+\varDelta_{1/n}\big).
 \eeqnn
 Combining these two estimates together yields
 \beqnn
 \underline n\Big(f(\epsilon_{\zeta-},-\epsilon_\zeta)
 \,\Big|\,\zeta\in t+\varDelta_\delta\Big)
 \ar\leq\ar
 \frac{A+\rho_n(t)}
 {\underline n\big(\zeta\in t+\varDelta_\delta\big)}
 \sum_{k=0}^{[n\delta]}
 \underline n\big(\zeta\in t+k/n+\varDelta_{1/n}\big)
 \cr\cr
 \ar\leq\ar
 \big(A+\rho_n(t)\big)\cdot
 \frac{\underline n\big(\zeta\in t+\varDelta_{\delta+1/n}\big)}
 {\underline n\big(\zeta\in t+\varDelta_\delta\big)}.
 \eeqnn
 In view of Theorem~\ref{MainThm01}, the monotonicity of $\nu(x+\varDelta_\delta)$ in $\delta$ and \eqref{eqn.30}, this yields
 \beqnn
 \lefteqn{\limsup_{t\to\infty}
 	\sup_{\delta\in[\delta_0,\infty]}
 	\underline n\Big(f(\epsilon_{\zeta-},-\epsilon_\zeta)
 	\,\Big|\,\zeta\in t+\varDelta_\delta\Big)}
 \cr\cr
 \ar\leq\ar
 A\cdot\limsup_{t\to\infty}
 \sup_{\delta\in[\delta_0,\infty]}
 \frac{\underline n\big(\zeta\in t+\varDelta_{\delta+1/n}\big)}
 {\underline n\big(\zeta\in t+\varDelta_\delta\big)}
 \cr\cr
 \ar=\ar
 A\cdot\limsup_{t\to\infty}
 \sup_{\delta\in[\delta_0,\infty]}
 \frac{\nu\big(\beta(t+\varDelta_{\delta+1/n})\big)}
 {\nu\big(\beta(t+\varDelta_\delta)\big)}
 \cr\cr
 \ar\leq\ar
 A\cdot\Big(
 1+\limsup_{t\to\infty}
 \sup_{\delta\in[\delta_0,\infty]}
 \frac{\nu\big(\beta(t+\delta)+\varDelta_{\beta/n}\big)}
 {\nu\big(\beta t+\varDelta_{\beta\delta_0}\big)}
 \Big)
 \leq
 A\cdot\Big(1+\frac{1}{n\delta_0}\Big).
 \eeqnn
 On the other hand, an analogous argument gives
 \beqnn
 \liminf_{t\to\infty}
 \inf_{\delta\in[\delta_0,\infty]}
 \underline n\Big(f(\epsilon_{\zeta-},-\epsilon_\zeta)
 \,\Big|\,\zeta\in t+\varDelta_\delta\Big)
 \ar\geq\ar
 A\cdot\liminf_{t\to\infty}
 \inf_{\delta\in[\delta_0,\infty]}
 \frac{\nu\big(\beta(t+\varDelta_{\delta-1/n})\big)}
 {\nu\big(\beta(t+\varDelta_\delta)\big)}\geq 
 A\cdot\Big(1-\frac{1}{n\delta_0}\Big).
 \eeqnn
 Combining these bounds and then letting $n\to\infty$ yields the desired uniform convergence \eqref{eqn.MainThm06}.
 \qed

 \bibliographystyle{plain}
 
 \bibliography{Reference}

@Article{Alili2001,
  author  = {Alili, L. and Chaumont, L.},
  title   = {A new fluctuation identity for {L}{\'e}vy processes and some applications},
  journal = {Bernoulli},
  year    = {2001},
  volume  = {7},
  number  = {3},
  pages   = {557--569},
}

@Article{AsmussenFossKorshunov2003,
	author  = {Asmussen, S. and Foss, S. and Korshunov, D.},
	title   = {Asymptotics for sums of random variables with local subexponential behaviour},
	journal = {J. Theoret. Probab.},
	year    = {2003},
	volume  = {16},
	number  = {2},
	pages   = {489--518},
 }

@Book{Bertoin1996,
	author    = {Bertoin, J.},
	publisher = {Cambridge University Press},
	title     = {L\'evy {P}rocesses},
	year      = {1996},
	address   = {Cambridge},
}

@Book{Billingsley1999,
	title     = {Convergence of {P}robability {M}easures},
	publisher = {John Wiley \& Sons, Inc.},
	year      = {1999},
	author    = {Billingsley, P.},
	address   = {New York},
	edition   = {2nd},
}

@Book{BinghamGoldieTeugels1987,
 	title     = {Regular {V}ariation},
 	publisher = {Cambridge University Press},
 	year      = {1987},
 	author    = {Bingham, N.H. and Goldie, C.M. and Teugels, J.L.},
 	volume    = {27},
 	address   = {Cambridge},
 }

@Article{CuiWuXu2025,
 	author  = {Cui, Z. and Wu, H. and Xu, W.},
 	title   = {{Large excursions of reflected L\'evy processes: asymptotic shapes}},
 	journal = {arXiv:2512.08171},
 	year    = {2025}, 
 }

@article{CuiWuXu2026b,
  	title={{Local asymptotics for entrance laws of reflected L\'evy excursions}},
  	author={Cui, Z. and Wu, H. and Xu, W.},
  	journal={arXiv:2610.03449},
  	year={2026}
  }

@Article{WuXu2026c,
  	author  = {Wu, H. and Xu, W.},
  	title   = {{Local limit theorems for heavy-tailed L\'evy processes with negative drift}},
  	journal = {In preparation},
  	year    = {2026}, 
  }

@Article{DenisovShneer2013,
 	author  = {Denisov, D. and Shneer, V.},
 	journal = {J. Appl. Probab.},
 	title   = {Asymptotics for the first passage times of {L}\'{e}vy processes and random walks},
 	year    = {2013},
 	volume  = {50},
 	number  = {1},
 	pages   = {64--84},
 }

@Article{DenisovVatutinWachtel2014,
 	author  = { Denisov, D. and Vatutin, V. and Wachtel, V. },
 	journal = {Electron. J. Probab.},
 	title   = {Local probabilities for random walks with negative drift conditioned to stay nonnegative},
 	year    = {2014},
 	volume  = {19},
 	pages   = {1--17},
 }

@Book{Doney2007,
  	title     = {{Fluctuation Theory for L\'evy Processes}},
  	publisher = {Springer},
  	year      = {2007},
  	author    = {Doney, R.A.}, 
  	address   = {Berlin},
  }

@article{Doney2012,
 	title={{Local behaviour of first passage probabilities}},
 	author={Doney, R.A.},
 	journal={Probab. Theory Relat. Fields},
 	volume={152},
 	number={3},
 	pages={559--588},
 	year={2012}, 
 }

@article{DoneyRivero2013,
 	title={{Asymptotic behaviour of first passage time distributions for L{\'e}vy processes}},
 	author={Doney, R.A. and Rivero, V.},
 	journal={Probab. Theory Relat. Fields},
 	volume={157},
 	number={1},
 	pages={1--45},
 	year={2013}, 
 }

@Article{DoneySavov2010,
  author  = {Doney, R.A. and Savov, M.S.},
  title   = {The asymptotic behavior of densities related to the supremum of a stable process},
  journal = {Ann. Probab.},
  year    = {2010},
  volume  = {38},
  number  = {1},
  pages   = {316--326},
  doi     = {10.1214/09-AOP479},
 }

@article{DoneyRivero2016Erratum,
 	title={{Erratum to: Asymptotic behaviour of first passage time distributions for L{\'e}vy processes}},
 	author={Doney, R.A. and Rivero, V.},
 	journal={Probab. Theory Relat. Fields},
 	volume={164},
 	number={3},
 	pages={1079--1083},
 	year={2016}, 
 }

@Book{JacodShiryaev2003,
 	title     = {Limit {T}heorems for {S}tochastic {P}rocesses},
 	publisher = {Springer},
 	year      = {2003},
 	author    = {Jacod, J. and Shiryaev, A. N.},
 	address={Berlin},
 }

@Book{Kyprianou2014,
 	title     = {Fluctuations of {L}\'evy {P}rocesses with {A}pplications: {I}ntroductory {L}ectures},
 	publisher = {Springer Science \& Business Media},
 	year      = {2014},
 	author    = {Kyprianou, A. E.},
 }

@Book{Petrov2012,
 	title={{Sums of Independent Random Variables}},
 	author={Petrov, V.V.},
 	year={1975},
 	publisher={Springer},
 	address={Berlin},
 }

@Book{Resnick1987,
	author    = {Resnick, S. I.},
	publisher = {Springer},
	title     = {{Extreme Values, Regular Variation and Point Processes}},
	year      = {1987}, 
	address={ New York},
}

@Book{Sato1999,
 	title     = {L\'{e}vy {P}rocesses and {I}nfinitely {D}ivisible {D}istributions},
 	publisher = {Cambridge University Press},
 	year      = {1999},
 	author    = {Sato, K.I.},
 	address={Cambridge},
 }

@Article{VatutinWachtel2009,
 	title={Local probabilities for random walks conditioned to stay positive},
 	author={Vatutin, V. and Wachtel, V.},
 	journal={Probab. Theory Relat. Fields},
 	volume={143},
 	number={1},
 	pages={177--217},
 	year={2009},
 }

@Article{WuXu2026a,
 	author    = {Wu, H. and Xu, W.},
 	title     = {{Uniform local asymptotics for L\'evy processes with subexponential jumps}},
 	journal   = {arXiv:2608.11637},
 	year      = {2026}, 
 }

 \end{document}